\documentclass{amsart}
\usepackage[foot]{amsaddr}
\usepackage{float, graphicx}
\usepackage[]{epsfig}
\usepackage{amsmath, amsthm, amssymb,}
\usepackage{epsfig}
\usepackage{verbatim}
\usepackage{url}
\usepackage{latexsym}
\usepackage{mathrsfs}
\usepackage{hyperref}
\usepackage{graphicx}
\usepackage{enumerate}
\usepackage[normalem]{ulem}
\usepackage{bm}
\usepackage{stmaryrd}
\usepackage{enumitem}
\usepackage{mathtools}
\usepackage[letterpaper]{geometry}
\usepackage[normalem]{ulem}

\numberwithin{equation}{section}

\DeclareMathOperator{\dist}{dist}

\DeclareMathOperator{\sgn}{sgn}

\DeclareMathOperator{\argmax}{argmax}

\DeclareMathOperator{\wf}{wf}
 
\DeclareMathOperator{\ttop}{top}
\DeclareMathOperator{\btm}{btm}
\DeclareMathOperator{\cwf}{cwf}

\newcommand{\fP}{{\mathfrak P}}
\newcommand{\fA}{{\mathfrak A}}

\newcommand{\bZ}{{\mathbb Z}}
\usepackage{tikz}
\usetikzlibrary{arrows}
\usepackage{cleveref}

\DeclareMathOperator{\Imaginary}{Im}

\DeclareMathOperator{\Ai}{Ai}

\DeclareMathOperator{\Adm}{Adm}

\DeclareMathOperator{\TV}{TV}
\DeclareMathOperator{\mix}{mix}
\DeclareMathOperator{\crf}{crf}
\DeclareMathOperator{\rf}{rf}
\DeclareMathOperator{\alt}{alt}
\DeclareMathOperator{\fl}{fl}
\DeclareMathOperator{\diam}{diam}
\DeclareMathOperator{\north}{no}
\DeclareMathOperator{\so}{so}
\DeclareMathOperator{\ea}{ea}
\DeclareMathOperator{\we}{we}

\newtheorem{thm}{Theorem}[section]
\newtheorem{prop}[thm]{Proposition}
\newtheorem{lem}[thm]{Lemma}
\newtheorem{cor}[thm]{Corollary}

\theoremstyle{remark}
\newtheorem{rem}[thm]{Remark}
\theoremstyle{definition}
\newtheorem{definition}[thm]{Definition}
\newtheorem{assumption}[thm]{Assumption}

\newtheorem*{thm1}{Theorem}

\title{Edge Statistics for Lozenge Tilings of Polygons, II: Airy Line Ensemble}

    \author{Amol Aggarwal$^{1, 2}$}
    \address{$^1$Columbia University, NY}
    \address{$^2$Clay Mathematics Institute} 
    \email{amolaggarwal@math.columbia.edu}
    \author{Jiaoyang Huang$^3$}
    \address{$^3$University of Pennsylvania, PA}
    \email{huangjy@wharton.upenn.edu}

\begin{document}

\begin{abstract}
		
		We consider uniformly random lozenge tilings of simply connected polygons subject to a technical assumption on their limit shape. We show that the edge statistics around any point on the arctic boundary, that is not a cusp or tangency location, converge to the Airy line ensemble. Our proof proceeds by locally comparing these edge statistics with those for a random tiling of a hexagon, which are well understood. To realize this comparison, we require a nearly optimal concentration estimate for the tiling height function, which we establish by exhibiting a certain Markov chain on the set of all tilings that preserves such concentration estimates under its dynamics. 
		
	\end{abstract}

\maketitle
{
  \hypersetup{linkcolor=black}
  \tableofcontents
}

	\section{Introduction}

	\label{Introduction}

	A central feature of random lozenge tilings is that they exhibit boundary-induced phase transitions. Depending on the shape of the domain, they can admit \emph{frozen} regions, where the associated height function is flat almost deterministically, and \emph{liquid} regions, where the height function appears more rough and random; the curve separating these two phases is called an \emph{arctic boundary}. We refer to the papers \cite{TSP} and \cite{ECM} for some of the earlier analyses of this phenomenon in lozenge tilings of hexagonal domains, and to the book \cite{RT} for a comprehensive review. A thorough study of arctic boundaries on arbitrary polygons was pursued in \cite{LSCE,DMCS}, where it was shown that their limiting trajectories are algebraic curves. 
	
	After realizing that these phase boundaries exist and admit limits, the next question is to understand their fluctuations, known as the \emph{edge statistics}. On domains of diameter order $n$, the general prediction is that their fluctuations are of order $n^{1/3}$ and $n^{2/3}$ in the directions transverse and parallel to their limiting trajectories, respectively. Upon scaling by these exponents, it is further predicted that the boundary converges to the \emph{Airy$_2$ process}, a universal scaling limit introduced in \cite{SIDP} that is believed to govern various phenomena related to the Kardar--Parisi--Zhang universality class. See \cite{EFLS} for a detailed survey. 
	
	Following the initial works \cite{SFRM,NPRTRM,ACP} (where it was first proven in the related context of domino tilings for the Aztec diamond), this prediction has been established for random lozenge tilings of various families of domains. For example, we refer to \cite{CFPALG,RSPPP,SFFC} for certain $q$-weighted random plane partitions; \cite{DP} for tilings of hexagons; and \cite{ARS,UEFDIPS} for tilings of trapezoids (hexagons with cuts along a single side). These results are all based on exact and analyzable formulas, specific to the domain of study, for the correlation kernel for which the tiling forms a determinantal point process. Although for lozenge tilings of arbitrary polygonals such explicit formulas are not known, it is believed under such generality that convergence to the Airy$_2$ process under the above scaling still holds; see \cite[Conjecture 18.7]{RT} and \cite[Conjecture 9.1]{DMCS}. 
	
	In this paper we prove this statement for simply connected polygonal domains subject to a certain technical assumption on their limit shape that we believe to hold generically (see \Cref{pa} and \Cref{ptypical} below). Under the interpretation of lozenge tilings as non-intersecting random Bernoulli walks, we in fact more broadly consider the family of Bernoulli walks around the arctic boundary (not only the extreme one); we prove under the above scaling that it converges to the \emph{Airy line ensemble}, a multi-level generalization of the Airy$_2$ process. An informal formulation of this result is provided as follows; we refer to \Cref{walkspconverge} below for a more precise statement.
	
	\begin{thm1}[\Cref{walkspconverge} below]
		
		\nonumber  
		
		\label{paths1}

		Consider a uniformly random lozenge tiling of a simply connected polygonal domain, whose arctic boundary does not exhibit any of the four configurations depicted in \Cref{curve}. Under appropriate rescaling, the family of associated non-intersecting Bernoulli walks in a neighborhood of any point (that is neither a cusp nor a tangency location) of the limiting arctic boundary converges to the Airy line ensemble.
		
	\end{thm1}

	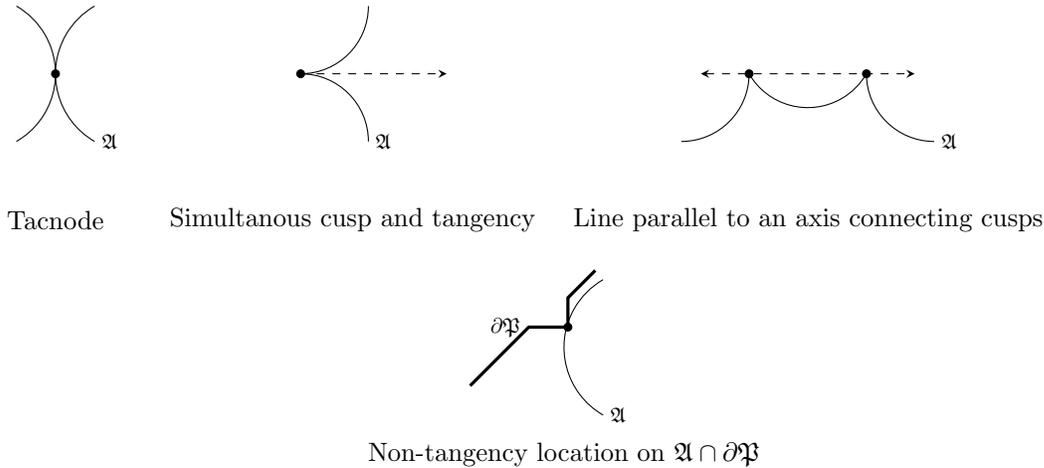
\begin{figure}
		
		\begin{center}		
			
			\begin{tikzpicture}[
				>=stealth,
				auto,
				style={
					scale = .52
				}
				]
				
				\draw[black] (-5, 0) arc(-60:60:2);
				\draw[black] (-3, 0) node[right, scale = .8]{$\mathfrak{A}$} arc(240:120:2);
				\filldraw[fill=black] (-4, 1.732) circle [radius = .1];
				
				\draw[] (4, 0) node[right, scale = .8]{$\mathfrak{A}$} arc(0:90:1.732);
				\draw[] (4, 3.464) arc(0:-90:1.732);
				\filldraw[fill = black] (2.268, 1.732) circle[radius = .1];
				\draw[dashed, ->] (2.268, 1.732) -- (6, 1.732);
				
				\draw[black] (12, 0) arc(-90:0:1.732);
				\draw[black] (13.732, 1.732) arc(-150:-30:1.732);
				\draw[black] (16.732, 1.732) arc(180:270:1.732) node[right, scale = .8]{$\mathfrak{A}$};
				
				\filldraw[fill = black] (13.732, 1.732) circle [radius = .1];
				\filldraw[fill = black] (16.732, 1.732) circle [radius = .1];

				\draw[black] (6.07, -7) node[right, scale = .8]{$\mathfrak{A}$} arc(240:120:2);
				
				\draw[black] (13.6, -6.3) node[left, scale = .8]{$\mathfrak{A}$} arc(0:90:2);
				\draw[very thick] (2.6, -6.25) -- (4.1, -4.75)  node[left, scale = .8]{$\partial \mathfrak{P}$}-- (5.1, -4.75) -- (5.1, -4) -- (5.8, -3.3);
				
				\draw[very thick, dashed]  (13.9, -4.75)  -- (12.9, -4.75) -- (12.9, -4) ;
				\draw[very thick, dashed]  (12.9, -4.75) -- (13.65, -4) ;
				\draw[very thick] (13.9, -7.25) -- (13.9, -4.75)  node[right, scale = .8]{$\partial \mathfrak{P}$}-- (14.65, -4) -- (13.65, -4) -- (13.65, -3.25)-- (12.9, -4) -- (11, -4);
				
				\filldraw[fill = black] (5.1, -4.75) circle [radius = .1];

				\filldraw[fill = black] (12.9, -4.75) circle [radius = .1];

				\draw[dashed, <->] (12.5, 1.732) -- (17.964, 1.732);

				\draw[] (-4, -2) circle [radius = 0] node[]{Tacnode};
				\draw[] (3.6, -2) circle [radius = 0] node[]{Simultanous cusp and tangency};
				\draw[] (15.25, -2) circle [radius = 0] node[]{Line parallel to an axis connecting cusps};
				\draw[] (9, -8) circle [radius = 0] node[]{Non-tangency location on $\mathfrak{A} \cap \partial \mathfrak{P}$, and non-tangency discontinuity of $\nabla H^*$ on $\mathfrak{A}$};
				
			\end{tikzpicture}
			
		\end{center}
		
		\caption{\label{curve} Depicted above are the four scenarios for arctic curve $\mathfrak{A}$ forbidden by \Cref{pa}.}
		
	\end{figure} 
	
	In the above theorem, we forbade specific (presumably non-generic) behaviors for singular points of the arctic boundary associated with the domain. These include the presence of tacnodes and cuspidal turning points; see \Cref{pa} for the exact condition. At some of these non-generic singularities, the edge scaling limit is more exotic; see \cite{RM,TCP,TNSDP,THCAF,PDTP} for more information. Still, it is believed that such behaviors should not disrupt the convergence to the Airy line ensemble elsewhere along the arctic boundary; that our theorem does not apply for these non-generic polygons therefore seems to be an artifact of our proof method. Generic singularities along the arctic boundary (which do appear in almost any polygonal domain) are ordinary cusps. The scaling limits at such points are believed to be given by the Pearcey process \cite{RSPPP}; we do not address this intriguing question here.

	The above theorem can be viewed as a \emph{universality} result for random lozenge tilings, since it shows that their statistics converge to the Airy line ensemble at any point (that is not a cusp or tangency location) around the arctic boundary, regardless of the polygonal shape bounding the domain. Recently, universality results for lozenge tilings at other points inside the domain (where different limiting statistics appear) have been established. For example, it was shown that local statistics in the interior of the liquid region converge to the unique translation-invariant, ergodic Gibbs measure of the appropriate slope for hexagons \cite{DP,gorin2008nonintersecting,borodin2010q}, domains covered by finitely many trapezoids \cite{ARS,gorin2017bulk}, bounded perturbations of these \cite{LLTS}, and finally for general domains \cite{ULS}. It was also recently shown in \cite{ERT} that, at tangency locations between the arctic boundary and sides of general domains, the limiting statistics converge to the corners process of the Gaussian Unitary Ensemble. Both of these phenomena were proven to be quite robust, and they in fact apply on domains beyond polygonal ones.
	
	Although the Airy$_2$ process also serves as the edge scaling limit in random lozenge tilings for certain classes of domains beyond polygons, the precise conditions under which it appears seem subtle. They are not determined by information about the macroscopic shape of the domain alone; microscopic perturbations of it can affect the edge statistics. Indeed, placing a single microscopic defect on an edge of a hexagonal domain corresponds to inserting a new walk in the associated non-intersecting Bernoulli walk ensemble. At the point where this new walk meets the arctic boundary for the original hexagon, the edge statistics should instead be given by the Airy$_2$ process with a wanderer, introduced in \cite{PWC}. 
	
	We now outline our proof of the theorem. We will show a concentration estimate for the tiling height function on a simply connected polygon $\mathsf{P}$ (satisfying \Cref{pa} below) of diameter order $n$, stating that with high probability it is within $n^{\delta}$ of its limit shape, for any $\delta > 0$. Given such a bound, we establish the theorem by locally comparing the random tiling of $\mathsf{P}$ with that of a hexagon, around their arctic boundaries. More specifically, the concentration estimate implies that the extreme paths in the non-intersecting Bernoulli walk ensemble $\mathsf{X}$ associated with a random tiling of $\mathsf{P}$ remains close to its limiting trajectory. We then match the slope and curvature of this limiting curve with those of the arctic boundary for a suitably chosen hexagon $\mathsf{P}'$. Using this, we exhibit a coupling between $\mathsf{X}$ and the non-intersecting Bernoulli walk ensemble associated with a random tiling of $\mathsf{P}'$, in such a way that they likely nearly coincide, up to error $\mathrm{o}(n^{1/3})$, around their arctic boundaries. Known results for random tilings on hexagonal domains \cite{DP,ARS,UEFDIPS}, coming from their exact solvability, show that the edge statistics of the random tiling of $\mathsf{P}'$ are given by the Airy line ensemble. It follows that the same holds for the random tiling of $\mathsf{P}$. 
	
 The remainder of this paper is devoted to proving the above mentioned concentration estimate, given by \Cref{mh} below. Such a concentration phenomenon is ubiquitous in random matrix theory, and is known as \emph{eigenvalue rigidity}. In the context of Wigner matrices, it was first proven in \cite{erdHos2012rigidity}, and later for more general classes of random matrices \cite{erdHos2013local,knowles2013isotropic,alex2014isotropic,he2018isotropic,bao2017convergence, ajanki2019stability, bao2020spectral}. To show this, we begin with a ``preliminary'' concentration bound, \Cref{estimategamma} below, for a family of $n$ non-intersecting random discrete bridges conditioned to start and end at specified locations (equivalently, random lozenge tilings of a strip). Assuming the initial and ending data for these Bernoulli walks are such that the limiting arctic boundary has at most one cusp (see the left and middle of \Cref{walkscusp} for examples), this bound states that with high probability the associated height function is within $n^{\delta}$ of its limit shape. Its proof is presented in part I of this series \cite{H1}, which proceeds by first using results of \cite{HFRTNRW} to approximate the random bridge model by a family of non-intersecting Bernoulli random walks with space and time dependent drifts. The latter walk model can then be studied through a dynamical version of the loop equations and an analysis of the complex Burgers equation through the characteristic method.

	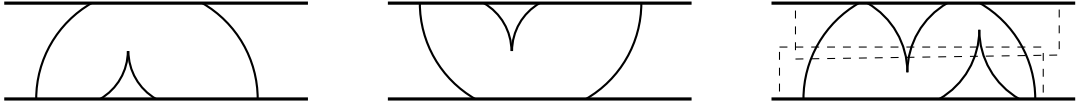
\begin{figure}
		
		\begin{center}		
			
			\begin{tikzpicture}[
				>=stealth,
				auto,
				style={
					scale = .425
				}
				]
				
				\draw[black, very thick] (-4, 0) -- (5.5, 0); 
				\draw[black, very thick] (-4, 3) -- (5.5, 3);
				
				\draw[black, thick] (-3, 0) arc (180:120:3.464);
				\draw[black, thick] (3.928, 0) arc (0:60:3.464);
				
				\draw[black, thick] (-1, 0) arc (-60:0:1.732);
				\draw[black, thick] (.75, 0) arc (240:180:1.732);
				
				\draw[black, very thick] (8, 0) -- (17.5, 0); 
				\draw[black, very thick] (8, 3) -- (17.5, 3);
				
				\draw[black, thick] (9, 3) arc (180:240:3.464);
				\draw[black, thick] (15.928, 3) arc (0:-60:3.464);
				
				\draw[black, thick] (11, 3) arc (60:0:1.732);
				\draw[black, thick] (12.75, 3) arc (120:180:1.732);
				
				\draw[black, very thick] (20, 0) -- (29.5, 0); 
				\draw[black, very thick] (20, 3) -- (29.5, 3);
				
				\draw[black, thick] (21, 0) arc (180:120:3.464);
				\draw[black, thick] (28.25, 0) arc (0:60:3.464);
				
				\draw[black, thick] (25.25, 0) arc (-60:0:2.5);
				\draw[black, thick] (27.75, 0) arc (240:180:2.5);
				
				\draw[black, thick] (23, 3) arc (60:0:2.5);
				\draw[black, thick] (25.5, 3) arc (120:180:2.5);
				
				\draw[dashed] (20.75, 1.25) -- (29, 1.375) -- (29, 3) -- (20.75, 3) -- (20.75, 1.375);
				\draw[dashed] (20.25, 1.625) -- (28.5, 1.625) -- (28.5, 0) -- (20.25, 0) -- (20.25, 1.625);
				
			\end{tikzpicture}
			
		\end{center}
		
		\caption{\label{walkscusp} Shown to the left and middle are arctic boundaries exhibiting a single cusp. Shown to the right is an arctic boundary exhibiting two cusps that point in opposite directions, and a decomposition of that strip into overlapping regions that each have (at most) one cusp.}
		
	\end{figure} 

	Imposing that the arctic boundary has at most one cusp is a substantial constraint; it does not hold for lozenge tilings of most polygons. Its origin can be heuristically attributed to the fact that, while disjoint families of non-intersecting Bernoulli walks often merge, merged ones do not separate, unless they are driven by a diverging drift (which much less amenable to analysis). As such, when comparing the bridge model to a family of non-intersecting Bernoulli walks with drift, one must ensure that these Bernoulli walks only merge and never separate. If the arctic boundary has only one cusp, then by suitably orienting the Bernoulli walks, this cusp can be interpreted as a location where families of Bernoulli walks merge; for example, this is the case if we orient the Bernoulli walks associated with the left and middle diagrams in \Cref{walkscusp} north and south, respectively. If the arctic boundary for the bridge model exhibits two cusps ``pointing in opposite directions,'' as in the right side of \Cref{walkscusp}, then any choice of orientation will lead to at least one cusp serving as a point where the Bernoulli walks separate. This issue was circumvented in \cite{HFRTNRW} by restricting to a family of domains in which all cusps point in the same direction. However, on generic polygons, cusps pointing in opposite directions do appear, so this point must be addressed here.

	To that end, we decompose our domain into a bounded number of (possibly overlapping) subregions that each have at most one cusp; see the right side of \Cref{walkscusp}. We then introduce a Markov chain, called the alternating dynamics (a form of the block dynamics), that uniformly resamples the tiling in one subregion and leaves it fixed in the others. Known estimates \cite{ADCC} for mixing times of Glauber dynamics, together with the censoring inequality of \cite{EUM}, imply that this Markov chain mixes to the uniform measure in time that is polynomial in $n$ (for example, $\mathcal{O} (n^{22})$).
	
	Initiating the alternating dynamics from a profile approximating the limit shape, we show that the $n^{\delta}$ concentration bound is with high probability preserved at each step of the alternating dynamics (from which the result follows by running these dynamics until they mix). The preliminary concentration result alone is insufficient to prove this, since the $n^{\delta}$ error it admits could in principle accumulate at each step. To overcome this, we introduce deterministic barrier functions, which we refer to as \emph{tilted profiles}, and show (with the assistance of the preliminary concentration bound) that they likely bound the tiling height function from above and below throughout the dynamics. To prove that such tilted profiles exist, we exhibit them by perturbing solutions to the complex Burgers equation in a specific way.
	
	The remainder of this paper is organized as follows. In \Cref{Walks} we define the model and state our main results. In \Cref{CompareProbability} we state the concentration result for the tiling height function on the polygon $\mathsf{P}$ and establish the theorem assuming it. In \Cref{TimeDynamics} we state the preliminary concentration bound for non-intersecting Bernoulli walks; introduce the alternating dynamics Markov chain; and bound its mixing time. In \Cref{Profile1} we introduce and discuss properties of tilted height functions. In \Cref{ProofH} we establish the concentration result for the tiling height function on $\mathsf{P}$. In \Cref{WalkLimit}, we give the proof for the existence of tilted height functions. 
	
\subsection*{Notation}	Throughout, we let $\overline{\mathbb{C}} = \mathbb{C} \cup \{ \infty \}$, $\mathbb{H}^+ = \big\{ z \in \mathbb{C} : \Imaginary z > 0 \big\}$, and $\mathbb{H}^- = \{ z \in \mathbb{C} : \Imaginary z < 0 \}$ denote the  compactified complex plane, upper complex plane, and lower complex plane, respectively. We further denote by $|u - v|$ the Euclidean distance between any elements $u, v \in \mathbb{R}^2$. For any subset $\mathfrak{R} \subseteq \mathbb{R}^2$, we let $\partial \mathfrak{R}$ denote its boundary, $\overline{\mathfrak{R}}$ denote its closure, and $\diam \mathfrak{R} = \sup_{r, r' \in \mathfrak{R}} |r - r'|$ denote its diameter. For any additional subset $\mathfrak{R}' \subseteq \mathbb{R}^2$, we let $\dist (\mathfrak{R}, \mathfrak{R}') = \sup_{r \in \mathfrak{R}} \inf_{r' \in \mathfrak{R}'} |r - r'|$ denote the distance between $\mathfrak{R}$ and $\mathfrak{R}'$. For any real number $c \in \mathbb{R}$, we also define the rescaled set $c \mathfrak{R} = \{ cr : r \in \mathfrak{R} \}$, and for any $u \in \mathbb{R}^2$, we define the shifted set $\mathfrak{R} + u = \big\{ r + u : r \in \mathfrak{R} \big\}$. For any $u \in \mathbb{R}^2$ and $r \geq 0$, let $\mathfrak{B}_r (u) = \mathfrak{B} (u; r) = \big\{ v \in \mathbb{R}^2 : |v - u| \leq r \big\}$ denote the disk centered at $u$ of radius $r$.

	\subsection*{Acknowledgements}

	The work of Amol Aggarwal was partially supported by a Clay Research Fellowship. The work of Jiaoyang Huang was partially supported by the Simons Foundation as a Junior Fellow at the Simons Society of
	Fellows and NSF grant DMS-2054835. The authors heartily thank Shirshendu Ganguly, Vadim Gorin and Lingfu Zhang for very helpful comments on this paper, as well as Erik Duse for highly enlightening discussions on \cite{LSCE,DMCS}.

	\section{Results} 
	
	\label{Walks}

	\subsection{Tilings and Height Functions}
	
	\label{FunctionWalks}
	
	We denote by $\mathbb{T}$ the \emph{triangular lattice}, namely, the graph whose vertex set is $\mathbb{Z}^2$ and whose edge set consists of edges connecting $(x, y), (x', y') \in \mathbb{Z}^2$ if $(x' - x, y' - y) \in \{ (1, 0), (0, 1), (1, 1)\}$. The axes of $\mathbb{T}$ are the lines $\{ x = 0 \}$, $\{ y = 0 \}$, and $\{ x  = y \}$, and the faces of $\mathbb{T}$ are triangles with vertices of the form $\big\{ (x, y), (x + 1, y), (x + 1, y + 1) \big\}$ or $\big\{ (x, y), (x, y + 1), (x + 1, y + 1) \big\}$. A \emph{domain} $\mathsf{R} \subseteq \mathbb{T}$ is a simply connected induced subgraph of $\mathbb{T}$. The \emph{boundary} $\partial \mathsf{R} \subseteq \mathsf{R}$ is the set of vertices $\mathsf{v} \in \mathsf{R}$ adjacent to a vertex in $\mathbb{T} \setminus \mathsf{R}$.

	A \emph{dimer covering} of a domain $\mathsf{R} \subseteq \mathbb{T}$ is defined to be a perfect matching on the dual graph of $\mathsf{R}$. A pair of adjacent triangular faces in any such matching forms a parallelogram, which we will also refer to as a \emph{lozenge} or \emph{tile}. Lozenges can be oriented in one of three ways; see the right side of \Cref{tilinghexagon} for all three orientations. We refer to the topmost lozenge there (that is, one with vertices of the form $\big\{ (x, y), (x, y + 1), (x + 1, y + 2), (x + 1, y + 1) \big\}$) as a \emph{type $1$} lozenge. Similarly, we refer to the middle (with vertices of the form $\big\{ (x, y), (x + 1, y), (x + 2, y + 1), (x + 1, y + 1) \big\}$) and bottom (vertices of the form $\big\{ (x, y), (x, y + 1), (x + 1, y + 1), (x + 1, y) \big\}$) ones there as \emph{type $2$} and \emph{type $3$} lozenges, respectively. A dimer covering of $\mathsf{R}$ can equivalently be interpreted as a tiling of $\mathsf{R}$ by lozenges of types $1$, $2$, and $3$. Therefore, we will also refer to a dimer covering of $\mathsf{R}$ as a \emph{(lozenge) tiling}. We call $\mathsf{R}$ \emph{tileable} if it admits a tiling.
	
	Associated with any tiling of $\mathsf{R}$ is a \emph{height function} $\mathsf{H}: \mathsf{R} \rightarrow \mathbb{Z}$, namely, a function on the vertices of $\mathsf{R}$ that satisfies 
	\begin{flalign*}
		\mathsf{H} (\mathsf{v}) - \mathsf{H} (\mathsf{u}) \in \{ 0, 1 \}, \quad \text{whenever $\mathsf{u} = (x, y)$ and $\mathsf{v} \in \big\{ (x + 1, y), (x, y - 1), (x + 1, y + 1) \big\}$},
	\end{flalign*}
	
	\noindent for some $(x, y) \in \mathbb{Z}^2$. We refer to the restriction $\mathsf{h} = \mathsf{H}|_{\partial \mathsf{R}}$ as a \emph{boundary height function}. For any boundary height function $\mathsf{h} : \partial \mathsf{R} \rightarrow \mathbb{Z}$, let $\mathscr{G} (\mathsf{h})$ denote the set of all height functions $\mathsf{h}: \mathsf{R} \rightarrow \mathbb{Z}$ with $\mathsf{R}|_{\partial \mathsf{R}} = \mathsf{h}$. 
	
	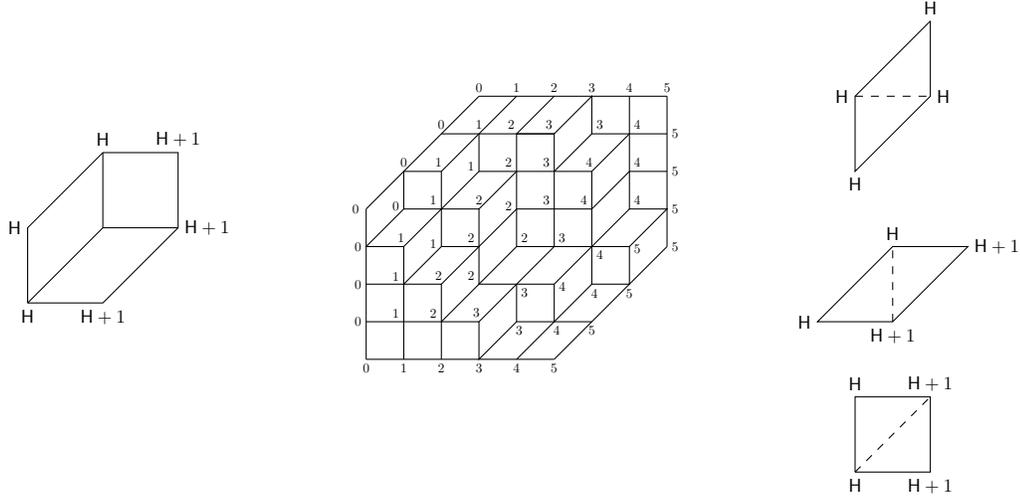
\begin{figure}
		
		\begin{center}		
			
			\begin{tikzpicture}[
				>=stealth,
				auto,
				style={
					scale = .5
				}
				]		
				
				\draw[-, black] (15, 9) node[above, scale = .7]{$\mathsf{H}$}-- (15, 7) node[right, scale = .7]{$\mathsf{H}$} -- (13, 5) node[below, scale = .7]{$\mathsf{H}$} -- (13, 7) node[left, scale = .7]{$\mathsf{H}$} -- (15, 9); 
				\draw[-, dashed, black] (13, 7) -- (15, 7);
				
				\draw[-, black] (14, 3) node[above, scale = .7]{$\mathsf{H}$} -- (16, 3) node[right, scale = .7]{$\mathsf{H} + 1$} -- (14, 1) node[below, scale = .7]{$\mathsf{H} + 1$} -- (12, 1) node[left, scale = .7]{$\mathsf{H}$}-- (14, 3);
				\draw[-, dashed, black] (14, 1) -- (14, 3);
				
				\draw[-, black] (13, -3) node[below, scale = .7]{$\mathsf{H}$} -- (13, -1) node[above, scale = .7]{$\mathsf{H}$} -- (15, -1) node[above, scale = .7]{$\mathsf{H} + 1$} -- (15, -3) node[below, scale = .7]{$\mathsf{H} + 1$} -- (13, -3);
				\draw[-, dashed, black] (13, -3) -- (15, -1);
				
				\draw[-, black] (-9, 1.5) node[below, scale = .7]{$\mathsf{H}$} -- (-7, 1.5) node[below, scale = .7]{$\mathsf{H} + 1$} -- (-5, 3.5) node[right, scale = .7]{$\mathsf{H} + 1$} -- (-5, 5.5) node[above, scale = .7]{$\mathsf{H} + 1$} -- (-7, 5.5) node[above, scale = .7]{$\mathsf{H}$} -- (-9, 3.5) node[left, scale = .7]{$\mathsf{H}$} -- (-9, 1.5);
				\draw[-, black] (-9, 1.5) -- (-7, 3.5) -- (-7, 5.5);
				\draw[-, black] (-7, 3.5) -- (-5, 3.5);
				
				\draw[-, black] (0, 0) -- (5, 0);
				\draw[-, black] (0, 0) -- (0, 4);
				\draw[-, black] (5, 0) -- (8, 3);
				\draw[-, black] (0, 4) -- (3, 7);
				\draw[-, black] (8, 3) -- (8, 7); 
				\draw[-, black] (8, 7) -- (3, 7);	
				
				\draw[-, black] (1, 0) -- (1, 3) -- (0, 3) -- (1, 4) -- (1, 5) -- (2, 5) -- (4, 7);
				\draw[-, black] (0, 2) -- (2, 2) -- (2, 0);
				\draw[-, black] (0, 1) -- (3, 1) -- (3, 0) -- (4, 1) -- (6, 1);
				\draw[-, black] (2, 6) -- (5, 6) -- (6, 7) -- (6, 6) -- (8, 6);
				\draw[-, black] (4, 6) -- (5, 7);
				\draw[-, black] (1, 4) -- (2, 4) -- (2, 5); 
				\draw[-, black] (2, 4) -- (3, 5) -- (3, 6);
				\draw[-, black] (1, 3) -- (2, 4);
				\draw[-, black] (2, 2) -- (3, 3) -- (3, 4) -- (2, 4) -- (2, 3) -- (1, 2);
				\draw[-, black] (2, 3) -- (3, 3) -- (3, 2) -- (4, 3) -- (4, 6) -- (5, 6) -- (5, 3) -- (4, 3);
				\draw[-, black] (3, 5) -- (5, 5) -- (6, 6);
				\draw[-, black] (7, 7) -- (7, 4) -- (8, 4) -- (7, 3) -- (7, 2) -- (6, 2) -- (4, 0);
				\draw[-, black] (8, 5) -- (7, 5) -- (6, 4) -- (6, 5) -- (7, 6);
				\draw[-, black] (3, 3) -- (4, 4) -- (6, 4);
				\draw[-, black] (3, 4) -- (4, 5);
				\draw[-, black] (6, 5) -- (5, 5);
				\draw[-, black] (7, 4) -- (6, 3) -- (7, 3);
				\draw[-, black] (2, 1) -- (3, 2) -- (5, 2) -- (6, 3);
				\draw[-, black] (3, 1) -- (5, 3) -- (6, 3) -- (6, 4); 
				\draw[-, black] (4, 1) -- (4, 2); 
				\draw[-, black] (5, 1) -- (5, 2); 
				\draw[-, black] (6, 2) -- (6, 3); 
				
				\draw[] (0, 0) circle [radius = 0] node[below, scale = .5]{$0$};
				\draw[] (1, 0) circle [radius = 0] node[below, scale = .5]{$1$};
				\draw[] (2, 0) circle [radius = 0] node[below, scale = .5]{$2$};
				\draw[] (3, 0) circle [radius = 0] node[below, scale = .5]{$3$};
				\draw[] (4, 0) circle [radius = 0] node[below, scale = .5]{$4$};
				\draw[] (5, 0) circle [radius = 0] node[below, scale = .5]{$5$};
				
				\draw[] (0, 1) circle [radius = 0] node[left, scale = .5]{$0$};
				\draw[] (1, 1) circle [radius = 0] node[above = 3, left = 0, scale = .5]{$1$};
				\draw[] (2, 1) circle [radius = 0] node[above = 3, left = 0, scale = .5]{$2$};
				\draw[] (3, 1) circle [radius = 0] node[left = 1, above = 0, scale = .5]{$3$};
				\draw[] (4, 1) circle [radius = 0] node[right = 1, below = 0, scale = .5]{$3$};
				\draw[] (5, 1) circle [radius = 0] node[right = 1, below = 0, scale = .5]{$4$};
				\draw[] (6, 1) circle [radius = 0] node[below, scale = .5]{$5$};
				
				\draw[] (0, 2) circle [radius = 0] node[left, scale = .5]{$0$};
				\draw[] (1, 2) circle [radius = 0] node[above = 3, left = 0, scale = .5]{$1$};
				\draw[] (2, 2) circle [radius = 0] node[left = 1, above = 0, scale = .5]{$2$};
				\draw[] (3, 2) circle [radius = 0] node[left = 3, above = 0, scale = .5]{$2$};
				\draw[] (4, 2) circle [radius = 0] node[below = 3, right = 0, scale = .5]{$3$};
				\draw[] (5, 2) circle [radius = 0] node[below = 1, right = 0, scale = .5]{$4$};
				\draw[] (6, 2) circle [radius = 0] node[right = 1, below = 0, scale = .5]{$4$};
				\draw[] (7, 2) circle [radius = 0] node[below, scale = .5]{$5$};
				
				\draw[] (0, 3) circle [radius = 0] node[left, scale = .5]{$0$};
				\draw[] (1, 3) circle [radius = 0] node[left = 1, above = 0, scale = .5]{$1$};
				\draw[] (2, 3) circle [radius = 0] node[above = 1, left = 0, scale = .5]{$1$};
				\draw[] (3, 3) circle [radius = 0] node[left = 3, above = 0, scale = .5]{$2$};
				\draw[] (4, 3) circle [radius = 0] node[right = 3, above = 0, scale = .5]{$2$};
				\draw[] (5, 3) circle [radius = 0] node[right = 3, above = 0, scale = .5]{$3$};
				\draw[] (6, 3) circle [radius = 0] node[below = 3, right = 0, scale = .5]{$4$};
				\draw[] (7, 3) circle [radius = 0] node[below = 1, right = 0, scale = .5]{$5$};
				\draw[] (8, 3) circle [radius = 0] node[right, scale = .5]{$5$};
				
				\draw[] (0, 4) circle [radius = 0] node[left = 1, scale = .5]{$0$};
				\draw[] (1, 4) circle [radius = 0] node[above = 1, left = 0, scale = .5]{$0$};
				\draw[] (2, 4) circle [radius = 0] node[left = 3, above = 0, scale = .5]{$1$};
				\draw[] (3, 4) circle [radius = 0] node[above, scale = .5]{$2$};
				\draw[] (4, 4) circle [radius = 0] node[above = 1, left = 0, scale = .5]{$2$};
				\draw[] (5, 4) circle [radius = 0] node[left = 3, above = 0, scale = .5]{$3$};
				\draw[] (6, 4) circle [radius = 0] node[left = 3, above = 0, scale = .5]{$4$};
				\draw[] (7, 4) circle [radius = 0] node[right = 3, above = 0, scale = .5]{$4$};
				\draw[] (8, 4) circle [radius = 0] node[right, scale = .5]{$5$};
				
				\draw[] (1, 5) circle [radius = 0] node[above, scale = .5]{$0$};
				\draw[] (2, 5) circle [radius = 0] node[left = 1, above = 0, scale = .5]{$1$};
				\draw[] (3, 5) circle [radius = 0] node[above = 2, left = 0, scale = .5]{$1$};
				\draw[] (4, 5) circle [radius = 0] node[left = 3, above = 0, scale = .5]{$2$};
				\draw[] (5, 5) circle [radius = 0] node[left = 3, above = 0, scale = .5]{$3$};
				\draw[] (6, 5) circle [radius = 0] node[left = 1, above = 0, scale = .5]{$4$};
				\draw[] (7, 5) circle [radius = 0] node[right = 3, above = 0, scale = .5]{$4$};
				\draw[] (8, 5) circle [radius = 0] node[right, scale = .5]{$5$};			
				
				\draw[] (2, 6) circle [radius = 0] node[above, scale = .5]{$0$};
				\draw[] (3, 6) circle [radius = 0] node[above, scale = .5]{$1$};
				\draw[] (4, 6) circle [radius = 0] node[left = 2, above = 0, scale = .5]{$2$};
				\draw[] (5, 6) circle [radius = 0] node[left = 2, above = 0, scale = .5]{$3$};
				\draw[] (6, 6) circle [radius = 0] node[right = 3, above = 0, scale = .5]{$3$};
				\draw[] (7, 6) circle [radius = 0] node[right = 3, above = 0, scale = .5]{$4$};
				\draw[] (8, 6) circle [radius = 0] node[right, scale = .5]{$5$};
				
				\draw[] (3, 7) circle [radius = 0] node[above, scale = .5]{$0$};
				\draw[] (4, 7) circle [radius = 0] node[above, scale = .5]{$1$};
				\draw[] (5, 7) circle [radius = 0] node[above, scale = .5]{$2$};
				\draw[] (6, 7) circle [radius = 0] node[above, scale = .5]{$3$};
				\draw[] (7, 7) circle [radius = 0] node[above, scale = .5]{$4$};
				\draw[] (8, 7) circle [radius = 0] node[above, scale = .5]{$5$};
				
			\end{tikzpicture}
			
		\end{center}
		
		\caption{\label{tilinghexagon} Depicted to the right are the three types of lozenges. Depicted in the middle is a lozenge tiling of a hexagon. One may view this tiling as a packing of boxes (of the type depicted on the left) into a large corner, which gives rise to a height function (shown in the middle). }
		
	\end{figure}

	For a fixed vertex $\mathsf{v} \in \mathsf{R}$ and integer $m \in \mathbb{Z}$, one can associate with any tiling of $\mathsf{R}$ a height function $\mathsf{H}: \mathsf{R} \rightarrow \mathbb{Z}$ as follows. First, set $\mathsf{H} (\mathsf{v}) = m$, and then define $\mathsf{H}$ at the remaining vertices of $\mathsf{R}$ in such a way that the height functions along the four vertices of any lozenge in the tiling are of the form depicted on the right side of \Cref{tilinghexagon}. In particular, we require that $\mathsf{H} (x + 1, y) = \mathsf{H} (x, y)$ if and only if $(x, y)$ and $(x + 1, y)$ are vertices of the same type $1$ lozenge, and that $\mathsf{H} (x, y) - \mathsf{H} (x, y + 1) = 1$ if and only if $(x, y)$ and $(x, y + 1)$ are vertices of the same type $2$ lozenge. Since $\mathsf{R}$ is simply connected, a height function on $\mathsf{R}$ is uniquely determined by these conditions (and the value of $\mathsf{H}(\mathsf{v}) = m$).
	
	We refer to the right side of \Cref{tilinghexagon} for an example; as depicted there, we can also view a lozenge tiling of $\mathsf{R}$ as a packing of $\mathsf{R}$ by boxes of the type shown on the left side of \Cref{tilinghexagon}. In this case, the value $\mathsf{H} (\mathsf{u})$ of the height function associated with this tiling at some vertex $\mathsf{u} \in \mathsf{R}$ denotes the height of the stack of boxes at $\mathsf{u}$. Observe in particular that, if there exists a tiling $\mathscr{M}$ of $\mathsf{R}$ associated with some height function $\mathsf{H}$, then the boundary height function $\mathsf{h} = \mathsf{H} |_{\partial \mathsf{R}}$ is independent of $\mathscr{M}$ and is uniquely determined by $\mathsf{R}$ (except for a global shift, which was above fixed by the value of $\mathsf{H}(\mathsf{v}) = m$).

	\subsection{Non-Intersecting  Bernoulli Walk Ensembles} 
	
	\label{WalkModel} 
	
	In this section we explain the correspondence between tilings and non-intersecting Bernoulli walk ensembles, to which end we begin by defining the latter. A  \emph{Bernoulli walk} is a sequence $\mathsf{q} = \big( \mathsf{q} (s), \mathsf{q} (s + 1), \ldots , \mathsf{q} (t) \big) \in \mathbb{Z}_{\geq 0}^{t - s + 1}$ such that $\mathsf{q} (r + 1) - \mathsf{q} (r) \in \{ 0, 1 \}$ for each $r \in [s, t - 1]$; viewing $r$ as a time index, $\big( \mathsf{q} (r) \big)$ denotes the space-time trajectory for a discrete walk, which may either not move or jump to the right at each step. For this reason, the interval $[s, t]$ is called the \emph{time span} of the Bernoulli walk $\mathsf{q}$, and a step $(r,r + 1)$ of this Bernoulli walk may be interpreted as an ``non-jump'' or a ``right-jump'' if $\mathsf{q} (r + 1) = \mathsf{q} (r)$  or $\mathsf{q} (r + 1) = \mathsf{q} (r) + 1$, respectively. A family of Bernoulli walks $\mathsf{Q} = \big( \mathsf{q}_l, \mathsf{q}_{l + 1}, \ldots , \mathsf{q}_m \big)$ is called \emph{non-intersecting} if $\mathsf{q}_i (r) < \mathsf{q}_j (r)$ whenever $l \leq i < j \leq m$ and $r$ is the in time span of $\mathsf{q}_i$ and $\mathsf{q}_j$. 
	
	Now fix some tileable domain $\mathsf{R} \subset \mathbb{T}$, with a height function $\mathsf{H}: \mathsf{R} \rightarrow \mathbb{T}$ corresponding to a tiling $\mathscr{M}$ of $\mathsf{R}$. We may interpret $\mathscr{M}$ as a family of non-intersecting Bernoulli walks by first omitting all type $1$ lozenges from $\mathscr{M}$, and then viewing any type $2$ or type $3$ tile as a right-jump or non-jump of a Bernoulli walk, respectively; see \Cref{walksfigure} for a depiction. 
	
	It will be useful to set more precise notation on this correspondence. Since $\partial_x \mathsf{H} (x, t) \in \{ 0, 1 \}$ for all $(x, t)$, there exist integers $\mathsf{x}_{a(t)} (t) < \mathsf{x}_{a(t) + 1} (t) < \cdots < \mathsf{x}_{b(t)} (t)$ such that 
	\begin{flalign*}
		\partial_x \mathsf{H} (x, t) = \displaystyle\sum_{i = a(t)}^{b(t)} \textbf{1} \Big( x \in \big[ \mathsf{x}_i (t), \mathsf{x}_i (t) + 1 \big] \Big),
	\end{flalign*}

	\noindent which are those such that $\mathsf{H} \big( \mathsf{x}_i (t) + 1, t \big) = \mathsf{H} \big( \mathsf{x}_i (t), t \big) + 1$. This fixes the locations of $\mathsf{x}_i (t)$, but in this way the indices $a(t)$ and $b(t)$ are defined up to an overall shift; we will fix this shift by stipulating that $\mathsf{H} \big( \mathsf{x}_{a(t)}+1, t \big) = a(t)$. This defines a Bernoulli walk $\mathsf{x}_i = \big( \mathsf{x}_i (t) \big)$,\footnote{It might in fact be a union of disconnected walks, but this point will have no effect on our discussion.} and a non-intersecting ensemble of Bernoulli walks $\mathsf{X} = \big( \mathsf{x}_l, \mathsf{x}_{l + 1}, \ldots , \mathsf{x}_m \big)$ that are indexed through the height function $\mathsf{H}$.

	\begin{figure}	
		\begin{center}		
			\begin{tikzpicture}[
				>=stealth,
				auto,
				style={
					scale = .55
				}
				]
				
				\draw[-, black] (1.5, 1) -- (1.5, 2) -- (1.5, 3) -- (2.5, 4) -- (2.5, 5); 
				\draw[-, black] (2.5, 2) -- (3.5, 3) -- (3.5, 4); 
				\draw[-, black] (4.5, 2) -- (4.5, 3) -- (5.5, 4);
				\draw[-, black] (5.5, 0) -- (5.5, 1) -- (6.5, 2) -- (6.5, 3) -- (7.5, 4) -- (7.5, 5); 
				\draw[-, black] (6.5, 0) -- (7.5, 1) -- (7.5, 2) -- (8.5, 3) -- (8.5, 4) -- (9.5, 5);
				\draw[-, black] (9.5, 1) -- (9.5, 2) -- (9.5, 3) -- (10.5, 4) -- (10.5, 5);
				
				\filldraw[fill=black] (1.5, 1) circle [radius = .1] node[below = 2, scale = .6]{$\mathsf{q}_{-2} (1)$};	
				\filldraw[fill=black] (2.5, 2) circle [radius = .1] node[left = 2, below, scale = .6]{$\mathsf{q}_{-1} (2)$};	
				\filldraw[fill=black] (4.5, 2) circle [radius = .1] node[right = 2, below, scale = .6]{$\mathsf{q}_0 (2)$};	
				\filldraw[fill=black] (5.5, 0) circle [radius = .1] node[left = 2, below, scale = .6]{$\mathsf{q}_1 (0)$};	
				\filldraw[fill=black] (6.5, 0) circle [radius = .1] node[right = 2, below, scale = .6]{$\mathsf{q}_2 (0)$};	
				\filldraw[fill=black] (9.5, 1) circle [radius = .1] node[below = 2, scale = .6]{$\mathsf{q}_3 (1)$};
				
				\filldraw[fill=black] (1.5, 1) circle [radius = .1];
				\filldraw[fill=black] (5.5, 1) circle [radius = .1];
				\filldraw[fill=black] (7.5, 1) circle [radius = .1];
				\filldraw[fill=black] (9.5, 1) circle [radius = .1];	
				
				\filldraw[fill=black] (1.5, 2) circle [radius = .1];
				\filldraw[fill=black] (2.5, 2) circle [radius = .1];
				\filldraw[fill=black] (4.5, 2) circle [radius = .1];
				\filldraw[fill=black] (6.5, 2) circle [radius = .1];
				\filldraw[fill=black] (7.5, 2) circle [radius = .1];
				\filldraw[fill=black] (9.5, 2) circle [radius = .1];
				
				\filldraw[fill=black] (1.5, 3) circle [radius = .1];
				\filldraw[fill=black] (3.5, 3) circle [radius = .1];
				\filldraw[fill=black] (4.5, 3) circle [radius = .1];
				\filldraw[fill=black] (6.5, 3) circle [radius = .1];
				\filldraw[fill=black] (8.5, 3) circle [radius = .1];
				\filldraw[fill=black] (9.5, 3) circle [radius = .1];
				
				\filldraw[fill=black] (2.5, 4) circle [radius = .1];
				\filldraw[fill=black] (3.5, 4) circle [radius = .1];
				\filldraw[fill=black] (5.5, 4) circle [radius = .1];
				\filldraw[fill=black] (7.5, 4) circle [radius = .1];
				\filldraw[fill=black] (8.5, 4) circle [radius = .1];
				\filldraw[fill=black] (10.5, 4) circle [radius = .1];
				
				\filldraw[fill=black] (2.5, 5) circle [radius = .1] node[above = 2, scale = .6]{$\mathsf{q}_{-2} (5)$};
				\filldraw[fill=black] (3.5, 4) circle [radius = .1] node[left = 2, above, scale = .6]{$\mathsf{q}_{-1} (4)$};
				\filldraw[fill=black] (5.5, 4) circle [radius = .1] node[right = 2, above, scale = .6]{$\mathsf{q}_0 (4)$};
				\filldraw[fill=black] (7.5, 5) circle [radius = .1] node[left = 2, above, scale = .6]{$\mathsf{q}_1 (5)$};
				\filldraw[fill=black] (9.5, 5) circle [radius = .1] node[left = 2, above, scale = .6]{$\mathsf{q}_2 (5)$};
				\filldraw[fill=black] (10.5, 5) circle [radius = .1] node[above = 2, scale = .6]{$\mathsf{q}_3 (5)$};
				
				\draw[-, black, thick] (15, 1) -- (15, 3) -- (16, 4) -- (16, 5) -- (17, 5) -- (18, 6) -- (18, 5) -- (17, 4) -- (18, 4) -- (19, 5) -- (19, 4) -- (20, 4) -- (20, 5) -- (21, 6) -- (21, 5) -- (22, 5) -- (23, 6) -- (23, 5) -- (25, 5) -- (25, 4) -- (24, 3) -- (24, 1) -- (23, 1) -- (23, 2) -- (21, 0) -- (19, 0) -- (18, -1) -- (18, 0) -- (18, 1) -- (19, 2) -- (18, 2) -- (18, 3) -- (17, 2) -- (16, 2) -- (16, 1) -- (15, 1);
				
				\draw[-, black, dashed] (16.5, 2) node[below = 0, scale = .6]{$\mathsf{q}_{-1}$} -- (17.5, 3) -- (17.5, 4); 
				\draw[-, black, dashed] (18.5, 2) node[left = 1, below = 0, scale = .6]{$\mathsf{q}_0$} -- (18.5, 3) -- (19.5, 4);
				\draw[-, black, dashed] (19.5, 0) node[below, scale = .6]{$\mathsf{q}_1$} -- (19.5, 1) -- (20.5, 2) -- (20.5, 3) -- (21.5, 4) -- (21.5, 5); 
				\draw[-, black, dashed] (20.5, 0) node[below, scale = .6]{$\mathsf{q}_2$} -- (21.5, 1) -- (21.5, 2) -- (22.5, 3) -- (22.5, 4) -- (23.5, 5);
				
				\draw[-, black, dashed] (15. 5, 1) node[below, scale = .6]{$\mathsf{q}_{-2}$} -- (15.5, 2) -- (15.5, 3) -- (16.5, 4) -- (16.5, 5); 
				\draw[-, black] (18, 3) -- (18, 4);
				\draw[-, black] (16, 2) -- (17, 3) -- (17, 4) -- (18, 5); 
				\draw[-, black] (18, 2) -- (18, 3) -- (19, 4) -- (19, 5);
				\draw[-, black] (19, 0) -- (19, 1) -- (20, 2) -- (20, 3) -- (21, 4) -- (21, 5); 
				\draw[-, black] (20, 0) -- (21, 1) -- (21, 2) -- (22, 3) -- (22, 4) -- (23, 5);
				\draw[-, black ](23, 1) -- (23, 2) -- (23, 3) -- (24, 4);
				\draw[-, black, dashed] (23.5, 1) node[below, scale = .6]{$\mathsf{q}_3$} -- (23.5, 2) -- (23.5, 3) -- (24.5, 4) -- (24.5, 5);
				
				\draw[-, black] (15, 3) -- (16, 3); 
				\draw[-, black] (16, 4) -- (17, 4); 
				\draw[-, black] (16, 3) -- (17, 4);
				\draw[-, black] (17, 5) -- (17, 4);
				\draw[-, black] (17, 2) -- (18, 3) -- (18, 2);
				\draw[-, black] (19, 5) -- (18, 4);
				\draw[-, black] (20, 5) -- (20, 4) -- (19, 4);
				\draw[-, black] (20, 4) -- (19, 3) -- (19, 2) -- (18, 1) -- (18, 0);
				\draw[-, black] (18, 2) -- (19, 2);
				\draw[-, black] (18, 3) -- (19, 3);
				\draw[-, black] (18, 0) -- (19, 1) -- (19, 2) -- (20, 3) -- (20, 4) -- (21, 5);
				\draw[-, black] (20, 0) -- (20, 1) -- (21, 2) -- (21, 3) -- (22, 4) -- (22, 5) -- (21, 5);
				\draw[-, black] (21, 0) -- (22, 1) -- (22, 2) -- (23, 3) -- (23, 4) -- (24, 5) -- (23, 5);
				\draw[-, black] (19, 0) -- (21, 0);
				\draw[-, black] (19, 1) -- (20, 1);
				\draw[-, black] (20, 2) -- (22, 2);
				\draw[-, black] (20, 3) -- (21, 3);
				\draw[-, black] (21, 4) -- (23, 4);
				\draw[-, black] (22, 3) -- (23, 3);
				\draw[-, black] (18, 0) -- (18, -1) -- (19, 0); 
				\draw[-, black] (17, 5) -- (18, 6) -- (18, 5);  
				\draw[-, black] (20, 5) -- (21, 6) -- (21, 5);  
				\draw[-, black] (22, 5) -- (23, 6) -- (23, 5);  
				\draw[-, black] (21, 1) -- (22, 1);
				\draw[-, black] (17, 3) -- (18, 3);
				\draw[-, black] (24, 3) -- (25, 4) -- (25, 5) -- (24, 5);
				\draw[-, black] (23, 1) -- (24, 1); 
				\draw[-, black] (23, 2) -- (24, 2); 
				\draw[-, black] (23, 3) -- (24, 3); 
				\draw[-, black] (24, 4) -- (25, 4);
				\draw[-, black] (16, 2) -- (16, 3);		
			\end{tikzpicture}	
		\end{center}
		\caption{\label{walksfigure} Depicted to the left is an ensemble $\mathsf{Q} = \big( \mathsf{q}_{-2}, \mathsf{q}_{-1}, \mathsf{q}_0, \mathsf{q}_1, \mathsf{q}_2, \mathsf{q}_3 \big)$ consisting of six non-intersecting Bernoulli walks. Depicted to the right is an associated lozenge tiling. }		
	\end{figure}
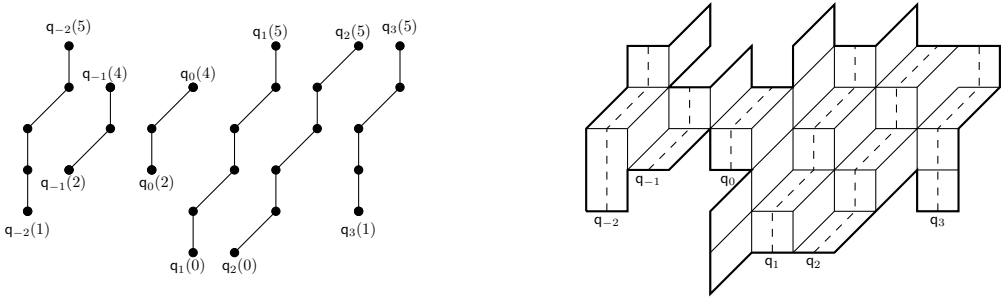

	\subsection{Limit Shapes and Arctic Boundaries}

	\label{HeightLimit} 
	
	To analyze the limits of height functions of random tilings, it will be useful to introduce continuum analogs of the notions considered in \Cref{FunctionWalks}. So, set 
	\begin{flalign} 
		\label{t}
		\mathcal{T} = \big\{ (s, t) \in (0, 1) \times \mathbb{R}_{< 0}: s + t > 0 \big\} \subset \mathbb{R}^2, 
	\end{flalign} 

	\noindent and its closure $\overline{\mathcal{T}} = \big\{ (s, t) \in [0, 1] \times \mathbb{R}_{\leq 0} :  s + t \ge 0 \big\}$. We interpret $\overline{\mathcal{T}}$ as the set of possible gradients, also called \emph{slopes}, for a continuum height function; $\mathcal{T}$ is then the set of ``liquid'' slopes, whose associated tilings contain tiles of all types. For any simply connected open subset $\mathfrak{R} \subset \mathbb{R}^2$, we say that a function $H : \mathfrak{R} \rightarrow \mathbb{R}$ is \emph{admissible} if $H$ is $1$-Lipschitz and $\nabla H(u) \in \overline{\mathcal{T}}$ for almost all $u \in \mathfrak{R}$. We further say a function $h: \partial \mathfrak{R} \rightarrow \mathbb{R}$ \emph{admits an admissible extension to $\mathfrak{R}$} if $\Adm (\mathfrak{R}; h)$, the set of admissible functions $H: \mathfrak{R} \rightarrow \mathbb{R}$ with $H |_{\partial \mathfrak{R}} = h$, is not empty.
	
	We say that a sequence of domains $\mathsf{R}_1, \mathsf{R}_2, \ldots \subset \mathbb{T}$ \emph{converges} to a simply connected subset $\mathfrak{R} \subset \mathbb{R}^2$ if $n^{-1} \mathsf{R}_n \subseteq \mathfrak{R}$ for each $n \geq 1$ and  $\lim_{n \rightarrow \infty} \dist (n^{-1} \mathsf{R}_n, \mathfrak{R}) = 0$. We further say that a sequence $\mathsf{h}_1, \mathsf{h}_2, \ldots $ of boundary height functions on $\mathsf{R}_1, \mathsf{R}_2, \ldots $, respectively, \emph{converges} to a boundary height function $h : \partial \mathfrak{R} \rightarrow \mathbb{R}$ if $\lim_{n \rightarrow \infty} n^{-1} \mathsf{h}_n (n v) = h (v)$ if $v$ is any point in $n^{-1} \mathsf{R}_n$ for all sufficiently large $n$.  
	
	To state results on the limiting height function of random tilings, for any $x \in \mathbb{R}_{\geq 0}$ and $(s, t) \in \overline{\mathcal{T}}$ we denote the \emph{Lobachevsky function} $L: \mathbb{R}_{\geq 0} \rightarrow \mathbb{R}$ and the \emph{surface tension} $\sigma : \overline{\mathcal{T}} \rightarrow \mathbb{R}^2$ by 
	\begin{flalign}
		\label{sigmal} 
		L(x) = - \displaystyle\int_0^x \log |2 \sin z| \mathrm{d} z; \qquad \sigma (s, t) = \displaystyle\frac{1}{\pi} \Big( L \big(\pi (1-s) \big) + L (- \pi t) + L \big( \pi (s + t) \big) \Big).
	\end{flalign}
	
	\noindent For any $H \in \Adm (\mathfrak{R})$, we further denote the \emph{entropy functional}
	\begin{flalign}
		\label{efunctionh} 
		\mathcal{E} (H) = \displaystyle\int_{\mathfrak{R}} \sigma \big( \nabla H (z) \big) \mathrm{d}z.
	\end{flalign}
	
	The following variational principle of \cite{VPT}  states that the height function associated with a uniformly random tiling of a sequence of domains corresponding to $\mathfrak{R}$ converges to the maximizer of $\mathcal{E}$ with high probability.

	\begin{lem}[{\cite[Theorem 1.1]{VPT}}]
		
		\label{hzh} 
		
		Let $\mathsf{R}_1, \mathsf{R}_2, \ldots \subset \mathbb{T}^2$ denote a sequence of tileable domains, with associated boundary height functions $\mathsf{h}_1, \mathsf{h}_2, \ldots $, respectively. Assume that they converge to a simply connected subset $\mathfrak{R} \rightarrow \mathbb{R}^2$ with piecewise smooth boundary, and a boundary height function $h : \partial \mathfrak{R} \rightarrow \mathbb{R}$, respectively. Denoting the height function associated with a uniformly random tiling of $\mathsf{R}_n$ by $\mathsf{H}_n$, we have
		\begin{flalign*}
			\displaystyle\lim_{n \rightarrow \infty} \mathbb{P} \bigg( \displaystyle\max_{\mathsf{v} \in \mathsf{R}_n} \big| n^{-1} \mathsf{H}_n (\mathsf{v}) - H^* (n^{-1} \mathsf{v}) \big| > \varepsilon \bigg) = 0,
		\end{flalign*}  
	
		\noindent where $H^*$ is the unique maximzer of $\mathcal{E}$ on $\mathfrak{R}$ with boundary data $h$,
		\begin{flalign}
			\label{hmaximum}
			H^* = \displaystyle\argmax_{H \in \Adm (\mathfrak{R}; h)} \mathcal{E} (H).
		\end{flalign}
	\end{lem} 

	\noindent The fact that there is a unique maximizer described as in \eqref{hmaximum} follows from \cite[Proposition 4.5]{MCFARS}. Under a suitable change of coordinates, this maximizer $H^*$ solves a complex variant of the Burgers equation \cite{LSCE}, which makes it amenable to further analysis; we will discuss this point in more detail in \Cref{Slopeft} below.
	
	For any simply connected open subset $\mathfrak{R} \subset \mathbb{R}^2$ with Lipschitz boundary, and boundary height function $h: \partial \mathfrak{R} \rightarrow \mathbb{R}^2$ admitting an admissible extension to $\mathfrak{R}$, define the \emph{liquid region} $\mathfrak{L} = \mathfrak{L} (\mathfrak{R}; h) \subseteq \mathfrak{R}$ and \emph{arctic boundary} $\mathfrak{A} = \mathfrak{A} (\mathfrak{R}; h) \subseteq \overline{\mathfrak{R}}$ by
	\begin{flalign}
		\label{al} 
		\mathfrak{L} = \big\{ u \in \mathfrak{R}: \big( \partial_x H^* (u), \partial_t H^* (u) \big) \in \mathcal{T} \big\}, \quad \text{and} \quad \mathfrak{A} = \partial \mathfrak{L}, \qquad \text{where $H^*$ is as in \eqref{hmaximum}}.
	\end{flalign}
	
	\noindent By \cite[Proposition 4.1]{MCFARS}, the set $\mathfrak{L}$ is open.
	
 	We will commonly be interested in the case when $\mathfrak{R}$ is a polygonal domain, given as follows. 
 	
 	\begin{definition} 
 		
 	\label{p} 
 	
 	A subset $\mathfrak{P} \subset \mathbb{R}^2$ is called \emph{polygonal} if it is a simply connected polygon, and the boundary edges are in the axes directions of the triangular lattice. We assume the domain\footnote{We assume throughout that all vertices of $n \partial \mathfrak{P}$ are in $\mathbb{Z}^2$.} $\mathsf{P} = \mathsf{P}_n = n \overline{\mathfrak{P}} \cap \mathbb{T}$ is  tileable and thus associated with a (unique, up to global shift) boundary height function $\mathsf{h} = \mathsf{h}_n$. By translating $\mathfrak{P}$ if necessary, we will assume that $\overline{\mathfrak{P}} \subset \mathbb{R} \times \mathbb{R}_{\geq 0}$ and that $(0, 0) \in \partial \mathfrak{P}$. Then by shifting $\mathsf{h}$ if necessary, we will further suppose that $\mathsf{h} (0, 0) = 0$. Under this notation, we set $h: \partial \mathfrak{P} \rightarrow \mathbb{R}$ by $h (u) = n^{-1} \mathsf{h} (nu)$ for each $u \in \partial \mathfrak{P}$. Moreover, we abbreviate $\Adm (\mathfrak{P}) = \Adm (\mathfrak{P}; h)$, $\mathfrak{L} (\mathfrak{P}) = \mathfrak{L} (\mathfrak{P}; h)$, and $\mathfrak{A} (\mathfrak{P}) = \mathfrak{A} (\mathfrak{P}; h)$; they do not depend on the above choice of global shift fixing $h$. We further define the maximizer $H^* \in \Adm (\mathfrak{R}; h)$ as in \eqref{hmaximum}.
 	
 	\end{definition} 
 
 	We will make use of the following results from \cite{LSCE,DMCS} on the behavior of the limit shape $H^*$ and arctic boundary $\mathfrak{A}$ when $\mathfrak{R}$ is polygonal. The first statement in the below lemma is given by \cite[Theorem 1.9]{DMCS} and the second by \cite[Theorem 1.2, Theorem 1.10]{DMCS} (see also \cite[Theorem 2, Proposition 5]{LSCE}).
 	
 	\begin{lem}[{\cite{LSCE,DMCS}}]
	
	\label{pla}
	
	Adopt the notation of \Cref{p}, and assume that the domain $\mathfrak{R} = \mathfrak{P}$ is polygonal with at least $6$ sides. Then  the following two statements hold.
	
	\begin{enumerate}
		
		\item On $\mathfrak{P} \setminus \mathfrak{L} (\mathfrak{P})$, $\nabla H^* (x, t)$ is piecewise constant, taking values in $\big\{ (0, 0), (1, 0), (1, -1) \big\}$. 
		\item The arctic boundary $\mathfrak{A} (\mathfrak{P})$ is an algebraic curve, and its singularities are all either ordinary cusps or tacnodes.
	\end{enumerate}
 
	\end{lem} 

	The following is an integrality result for the limiting height function $H^*$ outside of the associated liquid region. We provide its proof in \Cref{HeightInteger} below. 
	
	\begin{prop}
		
		\label{p:frozenr}
		
		Adopt the notation of \Cref{p}.
		
		\begin{enumerate}
			\item Fix $(x,t)\in\mathfrak P\setminus \mathfrak L$, such that $\nabla H^* (x, t) = (s, r)$ exists and $\nabla H^*$ is continuous at $(x, t)$. If $(s, r) \in \big\{ (0, 0), (1, 0), (1, -1) \big\}$, then $n(H^*(x,t)-sx-rt)\in \bZ$.
			\item For any point $v\in(\mathfrak P\setminus \overline{\mathfrak L})\cap (n^{-1} \cdot \bZ)^2$, we have $n \cdot H^*(v)\in \bZ$.
		\end{enumerate}
	\end{prop}

	It will also be useful to further set notation for the local parabolic shape of $\mathfrak{A} (\mathfrak{P})$ around any nonsingular point $(x_0, y_0) \in \mathfrak{A}$.
	
	\begin{definition}
		
		\label{sr} 
		
		Fix a nonsingular point $(x_0, y_0) \in \mathfrak{A} = \mathfrak{A} (\mathfrak{P})$; assume it is not a \emph{tangency location} of $\mathfrak{A}$, which is a point on $\mathfrak{A}$ whose tangent line to $\mathfrak{A}$ has slope in $\{ 0, 1, \infty \}$. Define the \emph{curvature parameters} $(\mathfrak{l}, \mathfrak{r}) = \big( \mathfrak{l} (x_0, y_0; \mathfrak{A}), \mathfrak{r} (x_0, y_0; \mathfrak{A}) \big) \in \mathbb{R}^2$ associated with $(x_0, y_0)$ so that 
		\begin{flalign}
			\label{xyql}
			x - x_0 = \mathfrak{l} (y - y_0) + \mathfrak{q} (y - y_0)^2 + \mathcal{O} \big( (y - y_0^3) \big),
		\end{flalign} 
	
		\noindent for all $(x, y) \in \mathfrak{A}$ in a sufficiently small neighborhood of $(x_0, y_0)$. Since $(x_0, y_0)$ is nonsingular and is not a tangency location, the parameters $(\mathfrak{l}, \mathfrak{q})$ exist, with $\mathfrak{l} \notin \{ 0, 1,\infty \}$ and $\mathfrak{q} \notin \{ 0, \infty \}$.

\end{definition}
	
	\subsection{Edge Statistics Results}
	
	\label{Ensemble} 
	
	In order to state our results, we first require some notation on edge statistics. 
	
	\begin{definition} 
		
		\label{kernellimit} 
		
		For any $s, t, x, y \in \mathbb{R}$, the \emph{extended Airy kernel} is given by
	\begin{flalign*}
		\mathcal{K} (s, x; t, y) = \displaystyle\int_0^{\infty} e^{u (t - s)} \Ai (x + u) \Ai (y + u) \mathrm{d}u, \qquad & \text{if $s \geq t$}; \\
		\mathcal{K} (s, x; t, y) = - \displaystyle\int_{-\infty}^0 e^{u (t - s)} \Ai (x + u) \Ai (y + u) \mathrm{d}u, \qquad & \text{if $s < t$},
	\end{flalign*}

	\noindent where we recall that the Airy function $\Ai: \mathbb{R} \rightarrow \mathbb{R}$ is given by
	\begin{flalign*}
		\Ai (x) = \displaystyle\frac{1}{\pi} \displaystyle\int_{-\infty}^{\infty} \cos \Big( \displaystyle\frac{z^3}{3} + xz \Big) \mathrm{d}z.
	\end{flalign*}
	
	\end{definition} 

	From this, we define the Airy line ensemble, which will be limits for our edge statistics. 

	\begin{definition}
		
		\label{ensemblewalks}
		
		The \emph{Airy line ensemble} $\mathcal{A} = (\mathcal{A}_1, \mathcal{A}_2, \ldots )$ is an infinite collection of continuous curves $\mathcal{A}_i: \mathbb{R} \rightarrow \mathbb{R}$, ordered so that $\mathcal{A}_1 (t) > \mathcal{A}_2 (t) > \cdots$ for each $t \in \mathbb{R}$, such that 
		\begin{flalign}
			\label{probabilityaxjtj}
		\mathbb{P} \Bigg( \bigcap_{j = 1}^m \{ (x_j, t_j) \in \mathcal{A} \} \Bigg) = \det \big[ \mathcal{K} (t_i, x_i; t_j, x_j) \big]_{1 \leq i, j \leq m} \displaystyle\prod_{j = 1}^m dx_j,
		\end{flalign}
	
		\noindent for any $(x_1, t_1), (x_2, t_2), \ldots , (x_m, t_m) \in \mathbb{R}^2$. Here, we have written $(x, t) \in \mathcal{A}$ if there exists some integer $k \geq 1$ such that $\mathcal{A}_k (t) = x$. The existence of such an ensemble was shown as \cite[Theorem 3.1]{PLE} (and the uniqueness follows from the explicit form \eqref{probabilityaxjtj} of its multi-point distributions).\footnote{Its top curve $\mathcal{A}_1$ is the Airy$_2$ process.} We abbreviate $\mathcal{R} = \big(\mathcal{A}_1 (t) - t^2, \mathcal{A}_2 (t) - t^2, \ldots \big)$, which may be viewed as a function $\mathcal{R}: \mathbb{Z}_{> 0} \times \mathbb{R} \rightarrow \mathbb{R}$ by setting $\mathcal{R} (i, t) = \mathcal{R}_i (t) = \mathcal{A}_i (t) - t^2$. 
		
	\end{definition}

	We next impose the following assumption of a polygonal subset $\mathfrak{P} \subset \mathbb{R}^2$, which excludes certain conditions on its arctic boundary.
	
	\begin{assumption} 
		
		\label{pa} 
		
		Under the notation of \Cref{p}, assume the following four properties hold. 
		\begin{enumerate} 
			\item The arctic boundary $\mathfrak{A} = \mathfrak{A} (\mathfrak{P})$ has no tacnode singularities. 
			\item No cusp singularity of $\mathfrak{A}$ is also a tangency location of $\mathfrak{A}$. 
			\item  There exists an axis $\ell$ of $\mathbb{T}$ such that any line connecting two distinct cusp singularities of $\mathfrak{A}$ is not parallel to $\ell$. 
			\item  Any intersection point between $\mathfrak{A}$ and $\partial \mathfrak{P}$ must be a tangency location of $\mathfrak{A}$. Moreover, $\nabla H^*(x,t)$ is continuous at any point on $\mathfrak{A}$ that is not a tangency location.
		\end{enumerate}
	
	\end{assumption} 
	
	We refer to \Cref{curve} above for depictions of the four forbidden scenarios. 
	
	\begin{rem}
	
	\label{ptypical}

		It seems likely to us that the constraints listed in \Cref{pa} hold for a generic polygonal domain with a fixed number of sides, since each constraint should impose an algebraic relation between the side lengths of $\mathfrak{P}$. However, we will not pursue a rigorous proof of this here. We refer to \Cref{2polygon} for the arctic boundaries on a generic octagon and 12-gon (obtained by analytically solving for, and then plotting, the algebraic curves of the appropriate degrees tangent to all sides of these polygons); it is quickly seen that these arctic boundaries satisfy our assumption.

	\end{rem}

	\begin{figure}
	\begin{center}
	 \includegraphics[scale=0.7]{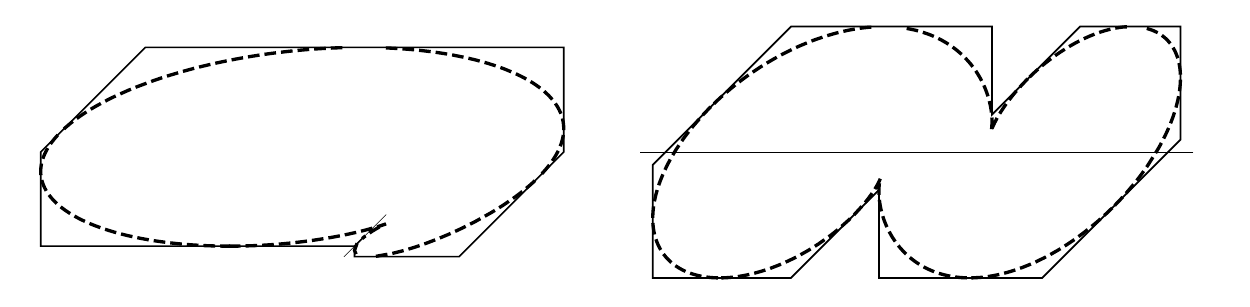}
	 \caption{Shown to the left is the arctic boundary of an octagon, and shown to the right is the arctic boundary of a $12$-gon. Both examples satisfy the constraints listed in \Cref{pa}.}
	 \label{2polygon}
	 \end{center}
	 \end{figure}

Now we can state the following theorem on the convergence to the Airy line ensemble for edge statistics of uniformly random tilings on polygonal domains satisfying \Cref{pa}. In what follows, we recall the non-intersecting Bernoulli walk ensemble associated with any tiling of a domain from \Cref{WalkModel} and the curvature parameters from \Cref{sr} (observe that the quantity $K$ defined in the below theorem is an integer, by \Cref{p:frozenr}).

	\begin{thm} 
		
		\label{walkspconverge} 
		
		Adopt the notation of \Cref{p} and the constraint from \Cref{pa}. Fix some point $(x_0, t_0) \in \mathfrak{A} (\mathfrak{P})$ that is not a tangency or cusp location of $\mathfrak{A} (\mathfrak{P})$, and assume that $\nabla H^* (x_0 + \varepsilon, t_0) = (0, 0)$ for sufficiently small $\varepsilon$. Denote the curvature parameters associated with $(x_0, t_0)$ by $(\mathfrak{l}, \mathfrak{q})$, and set
		\begin{flalign}
			\label{sr1} 
			\mathfrak{s} = \bigg| \displaystyle\frac{\mathfrak{l}^{2/3} (1 - \mathfrak{l})^{2/3}}{4^{1/3} \mathfrak{q}^{1/3}} \bigg|; \qquad \mathfrak{r} = \bigg| \displaystyle\frac{\mathfrak{l}^{1/3} (1 - \mathfrak{l})^{1/3}}{2^{1/3} \mathfrak{q}^{2/3}} \bigg|.
		\end{flalign} 
	
		\noindent Let $\mathscr{M}$ denote a uniformly random tiling of $\mathsf{P}$, which is associated with a (random) family $\big( \mathsf{x}_j (t) \big)$ of non-intersecting Bernoulli walks. Denote $K = n H^* (x_0, t_0)$, and define the family of functions $\mathcal{X}_n = (\mathsf{X}_1, \mathsf{X}_2, \ldots )$ by, for each $i \geq 0$, setting
		\begin{flalign} 
			\label{xit1}
			\mathsf{X}_{i+1 } (t) = \mathfrak{s}^{-1} n^{-1/3} \Big( \mathsf{x}_{K - i} (t_0 n + \mathfrak{r} n^{2/3} t) - n x_0 - \mathfrak{l} n^{2/3} t \Big).
		\end{flalign}
	
		\noindent Then $\mathcal{X}_n$ converges to $\mathcal{R}$, uniformly on compact subsets of $\mathbb{Z}_{> 0} \times \mathbb{R}$, as $n$ tends to $\infty$.
		
	\end{thm}

	Furthermore, observe that \Cref{walkspconverge} stipulates $\nabla H^* (x_0 + \varepsilon, t_0) = (0, 0)$ for small $\varepsilon$. Since for a polygonal domain $\mathfrak{P} \subset \mathbb{R}^2$ we have $\nabla H^* (x, y) \in \big\{ (0, 0), (1, 0), (1, -1) \big\}$ for almost any $(x, y) \notin \mathfrak{L} (\mathfrak{P})$ (by the first statement of \Cref{pla}), there are six possibilities for the behavior of $\nabla H^*$ around any $(x_0, t_0) \in \mathfrak{A} (\mathfrak{P})$. Specifically, we either have $\nabla H^* (x_0 + \varepsilon, t_0) \in \big\{ (0, 0), (1, 0), (1, -1) \big\}$ or $\nabla H^* (x_0 - \varepsilon, t_0) \in \big\{ (0, 0), (1, 0), (1, -1) \big\}$, with the former if $(x_0, t_0)$ is on a ``right part'' of the arctic boundary and the latter if it is on a ``left part.'' By rotating or reflecting the tileable domain $\mathsf{P}$ if necessary, establishing convergence for the edge statistics in any one of these six situations also shows it for the remaining cases, and so for brevity we only stated \Cref{walkspconverge} when $\nabla H^* (x_0, t_0 + \varepsilon) = (0, 0)$.

	\section{Convergence of Edge Statistics}
	
	\label{CompareProbability}
	
	In this section we establish \Cref{walkspconverge}, assuming the concentration estimate \Cref{mh} below. We begin in \Cref{Slopeft} by recalling complex analytic properties of tiling limit shapes in relation to the complex Burgers equation; in \Cref{EstimateWalks} we discuss classical locations of these limit shapes. Next, in \Cref{HeightP} we state a concentration bound for the tiling height function of polygonal domains satisfying \Cref{pa}, which we use in \Cref{CompareX} to compare the edge statistics on such polygons to those on hexagonal domains. We then establish \Cref{walkspconverge} in \Cref{ProofPaths}.

	\subsection{Complex Slopes and Complex Burgers Equation} 
	
	\label{Slopeft}
	
	In this section we recall from \cite{LSCE,DMCS} various complex analytic aspects of the tiling limit shapes discussed in \Cref{HeightLimit}; they will be briefly used in the proof of \Cref{lqq} below, and then more extensively in our discussion of tilted height profiles later. In what follows, we fix a simply connected open subset $\mathfrak{R} \subset \mathbb{R}^2$ and a boundary height function $h: \partial \mathfrak{R} \rightarrow \mathbb{R}$. We recall the maximizer $H^* \in \Adm (\mathfrak{R}; h)$ of $\mathcal{E}$ defined in \eqref{hmaximum}, as well as the liquid region $\mathfrak{L} = \mathfrak{L} (\mathfrak{R}; h)$ and arctic boundary $\mathfrak{A} = \mathfrak{A} (\mathfrak{R}; h)$ defined in \eqref{al}. 
	
	Then define the \emph{complex slope} $f: \mathfrak{L} \rightarrow \mathbb{H}^-$ by, for any $u \in \mathfrak{L}$, setting $f(u) \in \mathbb{H}^-$ to be the unique complex number satisfying 
	\begin{flalign}
		\label{fh}
		\arg^* f(u) = - \pi \partial_x H^* (u); \qquad \arg^* \big( f(u) + 1 \big) = \pi \partial_y H^* (u),
	\end{flalign}
	
	\noindent where for any $z \in \overline{\mathbb{H}^-}$ we have set $\arg^* z = \theta \in [-\pi, 0]$ to be the unique number in $[-\pi, 0]$ sastifying $e^{-\mathrm{i} \theta} z \in \mathbb{R}$; see \Cref{slope1} for a depiction, where there we interpret $1 - \partial_x H^* (u)$ and $-\partial_y H^* (u)$ as the approximate proportions of tiles of types $1$ and $2$ around $nu \in \mathsf{R}_n$, respectively (which follows from the definition of the height function from \Cref{FunctionWalks}).

	\begin{figure}
		\begin{center}
			\includegraphics[scale=0.3,trim={0cm 5cm 0 7cm},clip]{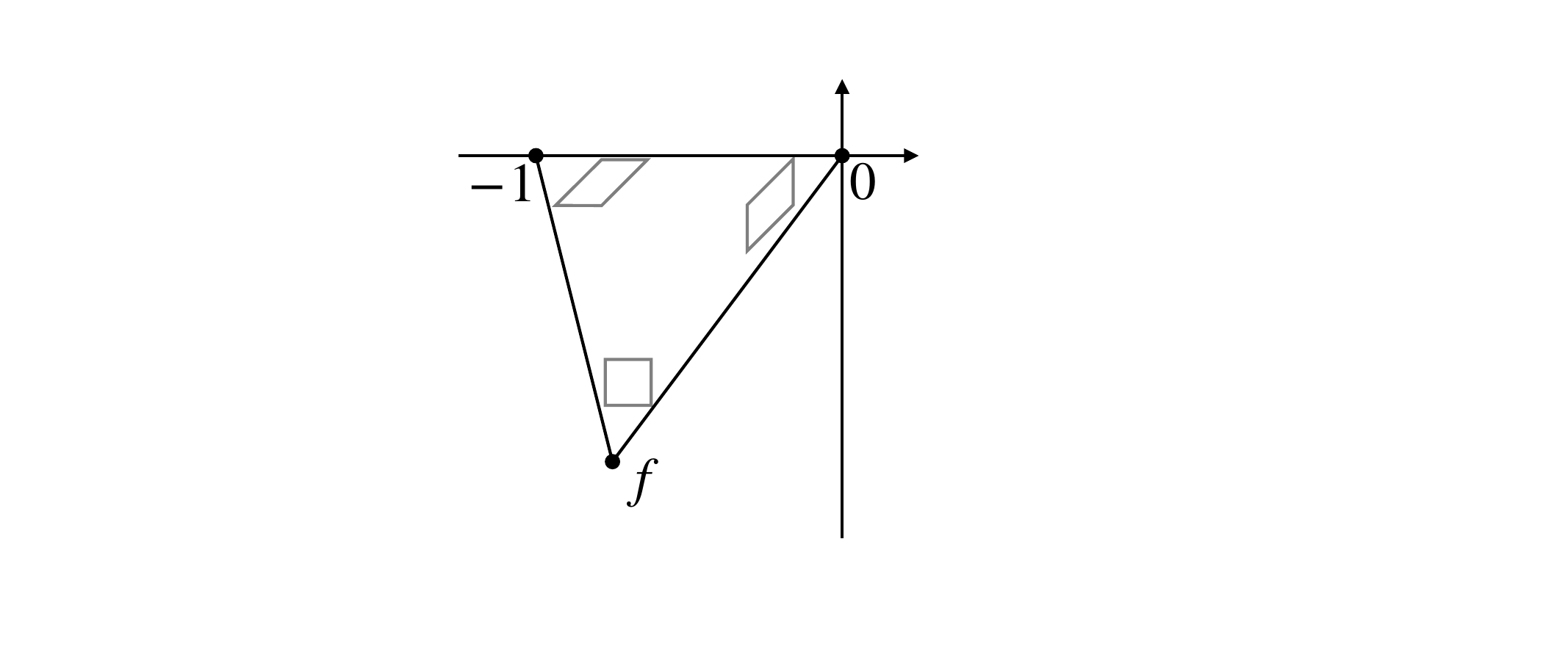}
			\caption{Shown above the complex slope $f = f (u)$.}
			\label{slope1}
		\end{center}
	\end{figure}
	
	The following result from \cite{LSCE} indicates that $f$ satisfies the complex Burgers equation.
	\begin{prop}[{\cite[Theorem 1]{LSCE}}]
		
		\label{fequation}
		
		For any $(x, t) \in \mathfrak{L}$, let $f_t (x) = f (x, t)$ we have
		\begin{flalign}
			\label{ftx}
			\partial_t f_t (x) + \partial_x f_t (x) \displaystyle\frac{f_t (x)}{f_t (x) + 1} = 0.
		\end{flalign} 
		
	\end{prop} 
	
	\begin{rem} 
		
		As explained in \cite[Section 3.2.2]{DMCS}, the composition $\widetilde{f} = M \circ f$ of $f_t (z)$ with a certain M\"{o}bius transformation $M$ solves the Beltrami equation $\partial_{\overline{z}} \widetilde{f} = \widetilde{f} \cdot \partial_z \widetilde{f}$. Thus, after a change of variables, any solution of the complex Burgers equation also solves the Beltrami equation.
		
	\end{rem}

	The following result from \cite{LSCE,DMCS} describes properties of the complex slope $f_t (x)$ when $\mathfrak{R}$ is polygonal. 
	
	\begin{prop}[{\cite{LSCE,DMCS}}]
		
		\label{pa1}
		
		Adopt the notation of \Cref{p}, and assume that the domain $\mathfrak{R} = \mathfrak{P}$ polygonal with at least $6$ sides. Then the following three statements hold.
		
		\begin{enumerate}
			
			\item The complex slope $f_t (x)$ extends continuously to the arctic boundary $\mathfrak{A} (\mathfrak{P})$. 
			
			\item Fix $(x_0, t_0) \in \overline{\mathfrak{L}}$. There exists a neighborhood $\mathfrak{U} \subset \mathbb{C}^2$ of $(x_0, t_0)$ and a real analytic function $Q_0 : \mathfrak{U} \rightarrow \mathbb{C}$ such that, for any $(x, t) \in \mathfrak{U} \cap \overline{\mathfrak{L}}$, we have
			\begin{flalign}
				\label{q0f} 
				Q_0 \big( f_t (x) \big) = x \big( f_t (x) + 1 \big) - t f_t (x).
			\end{flalign}
			 There exists a nonzero rational function $Q : \mathbb{C}^2 \rightarrow \mathbb{C}^2$ such that, for any $(x,t)\in \overline{\mathfrak{L}}$, we have   
	 	\begin{flalign}\label{e:qfh}
	 		Q \bigg( f_t (x), x - \displaystyle\frac{t f_t (x)}{f_t (x) + 1} \bigg) = 0.
	 	\end{flalign}

			\item For any $(x, t) \in \mathfrak{U} \cap \overline{\mathfrak{L}}$, $f_t (x)$ is a double root of \eqref{q0f} if and only if $(x, t) \in \partial \mathfrak{L}$.
			
		\end{enumerate}
		
	\end{prop}

	\begin{rem} 
		
		\label{pa2}
		
		The first statement of \Cref{pa1} is \cite[Theorem 1.10]{DMCS}. The local existence of $Q_0$ in \eqref{q0f} in the second is \cite[Theorem 10.5]{RT} (see also \cite[Corollary 2]{LSCE} or \cite[Theorem 5.2]{DMCS}), and the global existence of $Q$ in \eqref{e:qfh} is a quick consequence of the first part and \eqref{q0f}; see \cite[Proposition A.2(3)]{H1}.
		The third statement follows from the facts that $(x, t) \in \mathfrak{A} (\mathfrak{P})$ if and only if $f_t (x) \in \mathbb{R}$, by \eqref{fh}, and that any root of $Q_0$ is real if and only if it is a double root, as $Q_0$ is real anaytic (see also the discussion at the end of \cite[Section 1.6]{LSCE}).
		
	\end{rem}

	\subsection{Classical Locations}	

	\label{EstimateWalks}
	
	In the remainder of this section, we adopt the notation of \Cref{walkspconverge}. To establish \Cref{walkspconverge}, we will use a concentration estimate for the Bernoulli walk locations $\mathsf{x}_i$ associated with the uniformly random tiling $\mathscr{M}$ of $\mathsf{P}$. To state this result, we require some additional notation that will be in use throughout the remainder of this paper. 
	
	\begin{definition} 
		
	\label{gammait}
	
	For any integer $i \in \mathbb{Z}$ and real number $t \geq 0$, define the \emph{classical location} $\gamma_i (t)$ to be the (deterministic) real number 
	\begin{align}
	\gamma_i(t):=\inf \big\{x\in \mathbb{R}: n H^* ( x, t ) = i \big\},
	\end{align}
	
	\noindent if it exists (whenever this quantity is used, we will always implicitly assume that the parameters $(i, t)$ are such that it exists).
 
	\end{definition} 

	We will use an estimate for the classical locations $\gamma_i (t)$ around the arctic boundary. 
	
	\begin{lem}
		
		\label{gammaikx0t0}
		
		Adopt the notation of \Cref{walkspconverge}. For any integer $j \geq 0$ and $t \in \mathbb{R}_{\geq 0}$, we have 
		\begin{flalign*}
			\gamma_{K - j} (t) = x_0 + \mathfrak{l} (t - t_0) + \mathfrak{q} (t - t_0)^2 - \mathfrak{s}^{3 / 2} \bigg( \displaystyle\frac{3 \pi j}{2 n} \bigg)^{2/3} + \mathcal{O} \big( j n^{-1} + |t - t_0|^3 \big),
		\end{flalign*}
		
		\noindent where the implicit constant in the error is uniform if $(x_0, t_0)$ is bounded away from a singularity or tangency location of $\mathfrak{A} (\mathfrak{P})$.
	\end{lem}

	To establish \Cref{gammaikx0t0}, we require the following lemma expressing the curvature parameters $(\mathfrak{l}, \mathfrak{r})$ in terms of the analytic function $Q_0$ from \Cref{pa1} (associated with some point $(x_0, t_0) \in \overline{\mathfrak{L}}$). Its proof, which essentially follows from a Taylor expansion, is given in \Cref{Equation} below.
	
	\begin{lem}
		
		\label{lqq}
		
		Adopting the notation of \Cref{walkspconverge} and abbreviating $(x, t) = (x_0, t_0)$, we have  
		\begin{flalign*}
			\mathfrak{l} = \displaystyle\frac{f_t (x)}{f_t (x) + 1}; \qquad \mathfrak{q} = - \frac{1}{2} \big( f_t (x) + 1 \big)^{-3} Q_0'' \big( f_t (x) \big)^{-1}.
		\end{flalign*}
	\end{lem} 
 
	Now we can establish \Cref{gammaikx0t0}.

	\begin{proof}[Proof of \Cref{gammaikx0t0}]
		
		We may assume throughout that $|t - t_0|$ and $jn^{-1}$ are sufficiently small, for otherwise $|t - t_0|^3 + jn^{-1}$ is of order $1$ (and thus of $\diam (\mathfrak{P}) \geq \gamma_{K - j} (t)$). Then, observe that $\big( \gamma_K (s), s \big) \in \mathfrak{A} (\mathfrak{P})$ for each $s$ in a neighborhood of $t_0$. Indeed, since $\nabla H^* (w, s) = (0, 0)$ for all $(w, s) \in \mathfrak{P} \setminus \mathfrak{L} (\mathfrak{P})$ sufficiently close to $(x_0, t_0)$, we have $n H^* (w, s) = n H (x_0, t_0) = K$ for each $(w, s) \in \mathfrak{A} (\mathfrak{P})$ in a neighborhood of $(x_0, t_0)$, implying $\big( \gamma_K (s), s \big) = (w, s) \in \mathfrak{A} (\mathfrak{P})$.

		Throughout this proof, set $\widetilde{x} = \gamma_{K - j} (t)$. We first consider the case $t = t_0$. Fix some $x \in [\widetilde{x}, x_0]$, and abbreviate $f_0 = f_{t_0} (x_0)$ and $f = f_{t_0} (x)$. We will approximately express $f$ in terms of $f_0$, and then we will use this with \eqref{fh} to compare the classical locations $\gamma_{K - j} (t_0)$ and $\gamma_K (t_0) = x_0$. To that end, the second part of \Cref{pa1} implies
		\begin{flalign*}
		Q_0 (f_0) = x_0 (f_0 + 1) - t_0 f_0; \qquad Q_0 (f) = x (f + 1) - t_0 f.
		\end{flalign*}

	\noindent Subtracting these and applying a Taylor expansion yields
	\begin{flalign*}
		(f - f_0) Q_0' (f_0) + \displaystyle\frac{(f - f_0)^2}{2} Q_0'' (f_0) + \mathcal{O} \big( |f - f_0|^3 \big) & = Q_0 (f) - Q_0 (f_0) \\
		& = (f + 1) (x - x_0) + (x_0 - t_0) (f - f_0),
	\end{flalign*} 
	
	\noindent where the error depends on the first three derivatives of $Q_0$ at $f$, which is uniformly bounded if $(x_0, t_0)$ is bounded away from a singularity or tangency location of $\mathfrak{A} (\mathfrak{P})$. Since the third part of \Cref{pa1} gives $Q_0' (f_0) = x_0 - t_0$, we find that 
	\begin{flalign*}
		(f - f_0)^2 = \displaystyle\frac{2 (f_0 + 1)}{Q_0'' (f_0)} (x - x_0) + \mathcal{O} \big( |f - f_0|^3  \big).
	\end{flalign*}

	\noindent In particular, $|f_0 - f| = \mathcal{O} \big( |x - x_0|^{1/2} \big)$ and, more specifically,
	\begin{flalign}
		\label{ff0} 
		f - f_0 = \bigg( \displaystyle\frac{2 (f_0 + 1)}{Q_0'' (f_0)} \bigg)^{1/2} (x - x_0)^{1/2} + \mathcal{O} \big( |x - x_0|^{3 / 2} \big).
	\end{flalign} 

	\noindent Since $\gamma_{K - j} (t_0) = \widetilde{x} \leq x \leq x_0 \leq \gamma_K (t_0)$, we have $(x - x_0)^{1/2} \in \mathrm{i} \mathbb{R}$. Since moreover $f_0 \in \mathbb{R}$, which implies $Q_0'' (f_0) \in \mathbb{R}$, we deduce
	\begin{flalign*}
		\arg^* f & = f_0^{-1} \Imaginary (f - f_0) + \mathcal{O} \big( |f - f_0|^3 + |x - x_0|^{3/2} \big) \\
		& = \bigg( \displaystyle\frac{2 |f_0 + 1|}{f_0^2 \big| Q_0'' (f_0) \big|} \bigg)^{1/2} (x_0 - x)^{1/2} + \mathcal{O} \big( |x_0 - x|^{3 / 2} \big).
	\end{flalign*} 
	
	\noindent In particular, since $n H^* (x_0, t_0) = K$ and $n H^* (\widetilde{x}, t_0) = K - j$, this implies by \eqref{fh} that
	\begin{flalign*}
		\displaystyle\frac{j}{n} = H^* (x_0, t_0) - H^* (\widetilde{x}, t_0) & = \displaystyle\int_{\widetilde{x}}^{x_0} \partial_x H^* (w, t_0) \mathrm{d}w \\
		& = \displaystyle\frac{1}{\pi} \displaystyle\int_{\widetilde{x}}^{x_0} \arg^* f_{t_0} (w) \mathrm{d} w \\
		& = \bigg( \displaystyle\frac{8 |f_0 + 1|}{9 \pi^2 f_0^2 \big| Q_0'' (f_0) \big|} \bigg)^{1 / 2} (x_0 - \widetilde{x})^{3/2} + \mathcal{O} \big( |x_0 - \widetilde{x}|^{5 / 2} \big).
	\end{flalign*}
	
	\noindent Hence,
	\begin{flalign}
		\label{gammakt0}
				\gamma_K (t_0) - \gamma_{K - j} (t_0) = x_0 - \widetilde{x} = \bigg( \displaystyle\frac{f_0^2 \big| Q_0'' (f_0) \big|}{2 |f_0 + 1|} \bigg)^{1/3} \bigg( \displaystyle\frac{3 \pi j}{2 n} \bigg)^{2/3} + \mathcal{O} (j n^{-1}),
	\end{flalign}

	\noindent which by \Cref{lqq} and the definition of $\mathfrak{s}$ from \eqref{sr1} implies the lemma when $t = t_0$.
	
	If $t \ne t_0$, then set $\widehat{x}_0 = \gamma_K (t) \in \mathfrak{A} (\mathfrak{P})$ and $\widehat{f}_0 = f_t (\widehat{x}_0)$. Then, the same reasoning as used to deduce \eqref{gammakt0} implies
	\begin{flalign*}
		\gamma_K (t) - \gamma_{K - j} (t) & = \Bigg( \displaystyle\frac{\widehat{f}_0^2 \big| Q_0'' (\widehat{f}_0)\big|}{2 |\widehat{f}_0 + 1|} \Bigg)^{1/3} \bigg( \displaystyle\frac{3 \pi j}{2 n} \bigg)^{2/3} + \mathcal{O} (jn^{-1}) \\
		& = \bigg( \displaystyle\frac{f_0^2 \big| Q_0'' (f_0) \big|}{2 |f_0 + 1|} \bigg)^{1/3} \bigg( \displaystyle\frac{3 \pi j}{2 n} \bigg)^{2/3} + \mathcal{O} \big( |t - t_0| j^{2/3} n^{-2/3} + jn^{-1} \big),
	\end{flalign*}

	\noindent where in the last equality we used the fact that $f_t$ is uniformly smooth in $t$ around $t_0$ along $\mathfrak{A} (\mathfrak{P})$. By \eqref{xyql}, it follows that
	\begin{flalign*}
		\gamma_{K - j} (t) & = \gamma_K (t) - \bigg( \displaystyle\frac{f_0^2 \big| Q_0'' (f_0) \big|}{2 |f_0 + 1|} \bigg)^{1/3} \bigg( \displaystyle\frac{3 \pi j}{2 n} \bigg)^{2/3} + \mathcal{O} \big( |t - t_0| j^{2/3} n^{-2/3} + jn^{-1} \big) \\
		& = \gamma_K (t_0) + \mathfrak{l} (t - t_0) + \mathfrak{q} (t - t_0)^2 - \bigg( \displaystyle\frac{f_0^2 \big| Q_0'' (f_0) \big|}{2 |f_0 + 1|} \bigg)^{1/3} \bigg( \displaystyle\frac{3 \pi j}{2 n} \bigg)^{2/3} + \mathcal{O} \big( jn^{-1} + |t - t_0|^3 \big),
	\end{flalign*}
	
	\noindent which implies the lemma due to \Cref{lqq} and \eqref{sr1} again. 
	\end{proof}

	\subsection{Concentration Estimate for the Height Function}
	
	\label{HeightP}
		
	In this section we state a concentration estimate for the height function of a random tiling of $\mathsf{P}$. We begin with the following definition for events that hold with very high probability.
	
	\begin{definition} 
		
		\label{e}
		
		We say that an event $\mathscr{E}_n$ \emph{occurs with overwhelming probability} if the following holds. For any real number $D > 1$, there exists a constant $C > 1$ (dependent on $D$ and also possibly on other implicit parameters, but not $n$, involved in the definition of $\mathscr{E}_n$) such that $\mathbb{P} (\mathscr{E}_n) \geq 1 - n^{-D}$ for any integer $n > C$.
		
	\end{definition} 

	Recalling the notation of \Cref{walkspconverge}, and letting $\mathsf{H}$ denote the height function associated with the random tiling $\mathsf{M}$ of $\mathsf{P}$, our concentration estimate will state that the following two points with overwhelming probability. First, $\mathsf{H}$ is within $n^{\delta}$ of the deterministic function $n H^*$ everywhere on $\mathfrak{P}$. Second, $\mathsf{H}$ is frozen (deterministic) at a ``sufficiently far mesoscopic distance'' from the liquid region $\mathfrak{L} (\mathfrak{P})$. To make the latter point precise, we require the following definition. 
	
	\begin{definition}
		
		\label{llarge}
		
		Adopt the notation of \Cref{walkspconverge}, and abbreviate $\mathfrak{L} = \mathfrak{L} (\mathfrak{P})$ and $\mathfrak{A} = \mathfrak{A} (\mathfrak{P})$. Then, define the augmented liquid region 
		\begin{flalign*}
			\mathfrak{L}_+ (\mathfrak{P}) = \mathfrak{L}_+^{\delta} (\mathfrak{P}) = \mathfrak{L} \cup \bigcup_{u \in \mathfrak{A}} \mathfrak{B} (u; n^{\delta - 2/3}).
		\end{flalign*}

	\end{definition}

	 Under this notation, the following theorem then provides a concentration bound for the height function associated with $\mathscr{M}$; it will be established in \Cref{TilingEstimate} below.
	
	\begin{thm}
		
		\label{mh}
		
		Adopt the notation of \Cref{walkspconverge}, and let $\mathsf{H}: \mathsf{P} \rightarrow \mathbb{Z}$ denote the height function associated with $\mathscr{M}$. For any real number $\delta > 0$, the following two statements hold with overwhelming probability. 
		
		\begin{enumerate}
			\item We have $\big| \mathsf{H} (nu) - n H^* (u) \big| < n^{\delta}$ for any $u \in \overline{\mathfrak{P}}$.
			\item For any $u \in \overline{\mathfrak{P}} \setminus \mathfrak{L}_+^{\delta} (\mathfrak{P})$, we have $\mathsf{H} (nu) = n H^* (u)$.
		\end{enumerate}
	\end{thm}

	Together with \Cref{gammaikx0t0}, \Cref{mh} implies the following corollary that estimates trajectories for the random Bernoulli walks associated with $\mathscr{M}$ (recall the Bernoulli walk locations associated with the uniformly random tiling $\mathscr{M}$ of $\mathsf{P}$ from \Cref{walkspconverge} and \Cref{WalkModel}) near the arctic boundary.

	\begin{cor}
		
		\label{xjpath} 
		
		Adopt the notation of \Cref{walkspconverge}, and fix a real number $\delta \in \big( 0, 1/100 \big)$. For any integers $j \in [1, 2n^{10 \delta}]$ and $s \in [-n^{2/3 + 20 \delta}, n^{2/3 + 20 \delta}]$, we have with overwhelming probability that
		\begin{flalign*} 
			\Bigg| \mathsf{x}_{K - j + 1} (s + t_0 n) - \bigg( x_0 n + \mathfrak{l} s + \mathfrak{q} n^{-1} s^2 - \mathfrak{s}^{3 / 2} \Big( \displaystyle\frac{3 \pi j}{2} \Big)^{2/3} n^{1/3} \bigg) \Bigg| \leq j^{-1/3} n^{1/3 + \delta}.
		\end{flalign*}
	\end{cor}
	
	\begin{proof}
		
		We first show that \Cref{mh} implies, for any $t$ in a sufficiently small (independent of $n$) neighborhood of $t_0$, that with overwhelming probability we have
		\begin{flalign}
			\label{xgamma1}
			 \gamma_{K - j - n^{\delta} + 1} (t) - n^{-1} \leq n^{-1} \mathsf{x}_{K - j + 1} (tn) \leq \min \big\{ \gamma_{K - j + n^{\delta} + 1} (t), \gamma_K (t) + n^{\delta / 2 - 2/3} \},
		\end{flalign}	
	
		\noindent where we recall the classical locations $\gamma_i(t)$ from \Cref{gammait} (and we assume that $tn \in \mathbb{Z}$ for notational convenience). Let us only show the second bound in \eqref{xgamma1}, for the proof of the first is entirely analogous. Then, from the bijection between tilings and non-intersecting Bernoulli walk ensembles described in \Cref{WalkModel}, we have $n^{-1} (\mathsf{x}_{K - j + 1} (tn)+1) \leq x$ if and only if $\mathsf{H} (xn, tn) \geq K - j + 1$.
		
		So, setting $\gamma = \gamma_{K - j + n^{\delta} + 1} (t)$, the first part of \Cref{mh} implies with overwhelming probability that $\mathsf{H} (\gamma n, tn) \geq n H^* (\gamma n, tn) - n^{\delta} = K - j + 1$. Hence, $\mathsf{x}_{K - j + 1} (tn) \leq \gamma_{K - j + n^{\delta} + 1} (t)$ holds with overwhelming probability. Moreover, denoting $x' = \gamma_K (t) + N^{\delta  / 2 - 2/3}$, we have by the second part of \Cref{mh} that $\mathsf{H} (x' n, tn) = n H^* (x', t) = n H^* \big( \gamma_K (t), t \big) = K$ with overwhelming probability, where in the second equality we used the fact that $\nabla H^* (x, t) = (0, 0)$ for $(x, t)$ in a neighborhood of $(x_0, t_0)$ to the right of $\mathfrak{A}$. Hence, $n^{-1} \mathsf{x}_{K - j + 1} (t) \leq n^{-1} \mathsf{x}_K (t) \leq x' = \gamma_K (t) + N^{\delta / 2 - 2/3}$ with overwhelming probability. This confirms \eqref{xgamma1}. 
		
		Now, \eqref{xgamma1} and \Cref{gammaikx0t0} together imply that
		\begin{flalign}
			\label{y1} 
			\begin{aligned} 
				n^{-1} \mathsf{x}_{K - j + 1} (tn) & =  x_0 + \mathfrak{l} (t - t_0) + \mathfrak{q} (t - t_0)^2 - \mathfrak{s}^{3 / 2} \Big( \displaystyle\frac{3 \pi j}{2n} \Big)^{2/3} \\
				& \qquad + \mathcal{O} \big( j n^{-1} + |t - t_0|^3 + j^{-1/3} n^{\delta - 2/3}  \big),
			\end{aligned} 
		\end{flalign}
		
		\noindent holds for each $j \in \mathbb{Z}$ and $t \in \mathbb{R}$, with overwhelming probability. Here, we have also used the fact that \Cref{gammaikx0t0} implies the classical locations $\gamma_j (t)$ (from \Cref{gammait}) with respect to $\mathfrak{P}$ satisfy $\gamma_j (t) - \gamma_{j - n^{\delta}} (t) = \mathcal{O} (j^{-1/3} n^{- 2/3})$. Since $\delta \in \big( 0, 1/100 \big)$ and $j \in [1, 2 n^{10 \delta}]$, we have $jn^{-1} + |t - t_0|^3 + j^{-1/3} n^{\delta - 2/3} \leq 3 j^{-1/3} n^{\delta - 2/3}$ for $j \in [1, 2n^{10 \delta}]$ and $|t - t_0| \leq n^{20 \delta - 2/3}$. Setting $s = (t - t_0) n$ in \eqref{y1} then yields the corollary.	
	\end{proof}

	\subsection{Comparison to Hexagonal Edge Statistics} 
	
	\label{CompareX}
	
	We will prove \Cref{walkspconverge} through a local comparison of a random tiling of $\mathsf{P}$ with one of a suitably chosen hexagonal domain, whose universality of edge statistics has been proven in \cite{ARS,UEFDIPS,DP,UCLE}. In this section we set notation and state known properties for such hexagonal domains.
	
	\begin{definition}
		
		\label{xdomain}
		
		For any real numbers $a, b, c > 0$, let $\mathfrak{E}_{a, b, c} \subset \mathbb{R}^2$ denote the \emph{$a \times b \times c$ hexagon}, that is, the polygon with vertices $\big\{ (0, 0), (a, 0), (a + c, c), (a + c, b + c), (c, b + c), (0, b) \big\}$. By \cite[Theorem 1.1]{TSP}, its liquid region $\mathfrak{L}_{a, b, c} = \mathfrak{L} (\mathfrak{E}_{a, b, c})$ is bounded by the ellipse inscribed in $\mathfrak{E}_{a, b, c}$.

	\end{definition} 
	
	We refer to the middle of \Cref{tilinghexagon} for a depiction when $(a, b, c) = (5, 4, 3)$. The following result from \cite{ARS,UEFDIPS,DP,UCLE} is the case of \Cref{walkspconverge} when $\mathsf{P}$ is a hexagon.

	\begin{prop}[{\cite{ARS,UEFDIPS,DP,UCLE}}]
		
		\label{walksdomain1} 
		
		Let $a = a_n$, $b = b_n$, and $c = c_n$ be real numbers bounded away from $0$ and $\infty$, and set $(\mathsf{a}, \mathsf{b}, \mathsf{c}) = (\mathsf{a}_n, \mathsf{b}_n, \mathsf{c}_n) = (na, nb, nc)$; assume that $\mathsf{a}, \mathsf{b}, \mathsf{c} \in \mathbb{Z}$. Then \Cref{walkspconverge} holds with the $\mathfrak{P}$ there equal to the $a \times b \times c$ hexagon (and $\mathsf{P}$ equal to the $\mathsf{a} \times \mathsf{b} \times \mathsf{c}$ hexagon).
		
	\end{prop}

	\begin{rem}
		
		\label{kernelwalks}
		
		Since \Cref{walksdomain1} does not appear to have been stated exactly in above form in the literature, let us briefly outline how it follows from known results. First, \cite[Theorem 4.1]{UCLE} (see also the proof of \cite[Theorem 1.5]{UCLE}) indicates that, to show uniform convergence of the normalized discrete non-intersecting Bernoulli walks $\mathsf{X}_i (t)$ from \eqref{xit1} to the shifted Airy line ensemble $\mathcal{R}$ from \Cref{ensemblewalks}, it suffices to establish convergence in the sense of distributions, that is, 
		\begin{flalign}
			\label{xr} 
			\displaystyle\lim_{n \rightarrow \infty} \mathbb{P} \Bigg(\bigcap_{i = 1}^m \big\{ \mathsf{X}_{j_i} (t_i) \leq z_i \big\} \Bigg) = \mathbb{P} \Bigg( \bigcap_{i = 1}^m \big\{ \mathcal{R}_{j_i} (t_i) \leq z_i \big\} \Bigg),
		\end{flalign}
		
		\noindent for any $j_1, j_2, \ldots , j_m \in \mathbb{Z}_{\geq 1}$ and $t_1, t_2, \ldots , t_m, z_1, z_2, \ldots , z_m \in \mathbb{R}$. Next, \cite[Theorem 8.1]{ARS} and \cite[Theorem 1.12]{UEFDIPS} show that the non-intersecting Bernoulli walk ensemble $\big( \mathsf{x}_j (t) \big)$ is a determinantal point process, whose correlation kernel under the scaling \eqref{xit1} converges to the extended Airy kernel from \Cref{kernellimit}. Since probabilities as in the left side of \eqref{xr} are expressible in terms of unbounded sums involving this correlation kernel (see, for example, \cite[Equation (3.9)]{ACP}), to conclude the distributional convergence \eqref{xr} from the kernel limit, it suffices to show one-point tightness of the extremal Bernoulli walk $\mathsf{X}_1$ (in order to effectively cut off\footnote{One could alternatively prove sufficient decay of the kernel, as in \cite[Lemma 3.1(b)]{ACP}.} the sum mentioned above). This tightness is provided by \cite[Theorem 3.14]{DP}, which in fact shows that the one-point law of $\mathsf{X}_1$ converges to the Tracy--Widom GUE distribution.
		
	\end{rem} 

	To proceed, we require some additional notation on non-intersecting Bernoulli walk ensembles. Let $\mathsf{X} = (\mathsf{x}_l, \mathsf{x}_{l + 1}, \ldots , \mathsf{x}_m)$ denote a family of non-intersecting Bernoulli walks, each with time span $[s, t]$, so that $\mathsf{x}_j = \big( \mathsf{x}_j (s), \mathsf{x}_j (s + 1), \ldots , \mathsf{x}_j (t) \big)$ for each $j \in [l, m]$. Given functions $\mathsf{f}, \mathsf{g}: [s, t] \rightarrow \mathbb{R}$, we say that $\mathsf{X}$ has $(\mathsf{f}; \mathsf{g})$ as a \emph{boundary condition} if $\mathsf{f} (r) \leq \mathsf{x}_j (r) \leq \mathsf{g} (r)$ for each $r \in [s, t]$. We refer to $\mathsf{f}$ and $\mathsf{g}$ as a \emph{left boundary} and \emph{right boundary} for $\mathsf{X}$, respectively, and allow $\mathsf{f}$ and $\mathsf{g}$ to be $-\infty$ or $\infty$. We further say that $\mathsf{X}$ has \emph{entrance data} $\mathsf{d} = (\mathsf{d}_l, \mathsf{d}_{l + 1}, \ldots , \mathsf{d}_m)$ and \emph{exit data} $\mathsf{e} = (\mathsf{e}_l, \mathsf{e}_{l + 1}, \ldots , \mathsf{e}_m)$ if $\mathsf{x}_j (s) = \mathsf{d}_j$ and $\mathsf{x}_j (t) = \mathsf{e}_j$, for each $j \in [l, m]$; see \Cref{pathsfigure} for a depiction. Then, there is a finite number of non-intersecting Bernoulli walk ensembles with any given entrance and exit data $(\mathsf{d}; \mathsf{e})$ and (possibly infinite) boundary conditions $(\mathsf{f}; \mathsf{g})$.

	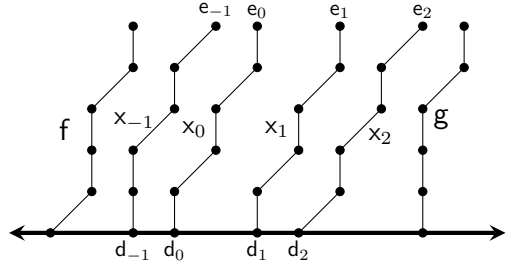
\begin{figure}
		
		\begin{center}		
			
			\begin{tikzpicture}[
				>=stealth,
				auto,
				style={
					scale = .55
				}
				]
				
				\draw[<->, black, ultra thick] (-.5, 0) -- (11.5, 0);	
				
				\draw[-, black] (.5, 0) -- (1.5, 1) -- (1.5, 2) -- (1.5, 3) -- (2.5, 4) -- (2.5, 5); 
				\draw[-, black] (2.5, 0) -- (2.5, 1) -- (2.5, 2) -- (3.5, 3) -- (3.5, 4) -- (4.5, 5); 
				\draw[-, black] (3.5, 0) -- (3.5, 1) -- (4.5, 2) -- (4.5, 3) -- (5.5, 4) -- (5.5, 5);
				\draw[-, black] (5.5, 0) -- (5.5, 1) -- (6.5, 2) -- (6.5, 3) -- (7.5, 4) -- (7.5, 5); 
				\draw[-, black] (6.5, 0) -- (7.5, 1) -- (7.5, 2) -- (8.5, 3) -- (8.5, 4) -- (9.5, 5);
				\draw[-, black] (9.5, 0) -- (9.5, 1) -- (9.5, 2) -- (9.5, 3) -- (10.5, 4) -- (10.5, 5);
				
				\filldraw[fill=black] (1.25, 2.5) circle [radius = .0] node[left, scale = 1.15]{$\mathsf{f}$};	
				\filldraw[fill=black] (3.25, 2.7) circle [radius = .0] node[left, scale = 1]{$\mathsf{x}_{-1}$};	
				\filldraw[fill=black] (4.5, 2.5) circle [radius = .0] node[left, scale = 1]{$\mathsf{x}_0$};	
				\filldraw[fill=black] (6.5, 2.5) circle [radius = .0] node[left, scale = 1]{$\mathsf{x}_1$};
				\filldraw[fill=black] (9, 2.35) circle [radius = .0] node[left, scale = 1]{$\mathsf{x}_2$};		
				\filldraw[fill=black] (.5, 0) circle [radius = .1];
				\filldraw[fill=black] (2.5, 0) circle [radius = .1] node[below, scale = .85]{$\mathsf{d}_{-1}$};	
				\filldraw[fill=black] (3.5, 0) circle [radius = .1] node[below, scale = .85]{$\mathsf{d}_0$};	
				\filldraw[fill=black] (5.5, 0) circle [radius = .1] node[below, scale = .85]{$\mathsf{d}_1$};	
				\filldraw[fill=black] (6.5, 0) circle [radius = .1] node[below, scale = .85]{$\mathsf{d}_2$};	
				\filldraw[fill=black] (9.5, 0) circle[radius = .1];
				\filldraw[fill=black] (9.5, 2.75) circle [radius = .0] node[right, scale = 1.15]{$\mathsf{g}$};
				
				\filldraw[fill=black] (1.5, 1) circle [radius = .1];
				\filldraw[fill=black] (2.5, 1) circle [radius = .1];
				\filldraw[fill=black] (3.5, 1) circle [radius = .1];
				\filldraw[fill=black] (5.5, 1) circle [radius = .1];
				\filldraw[fill=black] (7.5, 1) circle [radius = .1];
				\filldraw[fill=black] (9.5, 1) circle [radius = .1];	
				
				\filldraw[fill=black] (1.5, 2) circle [radius = .1];
				\filldraw[fill=black] (2.5, 2) circle [radius = .1];
				\filldraw[fill=black] (4.5, 2) circle [radius = .1];
				\filldraw[fill=black] (6.5, 2) circle [radius = .1];
				\filldraw[fill=black] (7.5, 2) circle [radius = .1];
				\filldraw[fill=black] (9.5, 2) circle [radius = .1];
				
				\filldraw[fill=black] (1.5, 3) circle [radius = .1];
				\filldraw[fill=black] (3.5, 3) circle [radius = .1];
				\filldraw[fill=black] (4.5, 3) circle [radius = .1];
				\filldraw[fill=black] (6.5, 3) circle [radius = .1];
				\filldraw[fill=black] (8.5, 3) circle [radius = .1];
				\filldraw[fill=black] (9.5, 3) circle [radius = .1];
				
				\filldraw[fill=black] (2.5, 4) circle [radius = .1];
				\filldraw[fill=black] (3.5, 4) circle [radius = .1];
				\filldraw[fill=black] (5.5, 4) circle [radius = .1];
				\filldraw[fill=black] (7.5, 4) circle [radius = .1];
				\filldraw[fill=black] (8.5, 4) circle [radius = .1];
				\filldraw[fill=black] (10.5, 4) circle [radius = .1];
				
				\filldraw[fill=black] (2.5, 5) circle [radius = .1];
				\filldraw[fill=black] (4.5, 5) circle [radius = .1] node[above, scale = .85]{$\mathsf{e}_{-1}$};
				\filldraw[fill=black] (5.5, 5) circle [radius = .1] node[above, scale = .85]{$\mathsf{e}_0$};
				\filldraw[fill=black] (7.5, 5) circle [radius = .1] node[above, scale = .85]{$\mathsf{e}_1$};
				\filldraw[fill=black] (9.5, 5) circle [radius = .1] node[above, scale = .85]{$\mathsf{e}_2$};
				\filldraw[fill=black] (10.5, 5) circle [radius = .1];	
				
			\end{tikzpicture}
			
		\end{center}
		
		\caption{\label{pathsfigure} Shown above is an ensemble of non-intersecting Bernoulli walks $\mathsf{X} = ( \mathsf{x}_{-1}, \mathsf{x}_0, \mathsf{x}_1, \mathsf{x}_2)$ with initial data $\mathsf{d} = (\mathsf{d}_{-1}, \mathsf{d}_0, \mathsf{d}_1, \mathsf{d}_2)$; ending data $\mathsf{e} = (\mathsf{e}_{-1}, \mathsf{e}_0, \mathsf{e}_1, \mathsf{e}_2)$; left boundary $\mathsf{f}$; and right boundary $\mathsf{g}$. }
		
	\end{figure}

	The below lemma from \cite{LSRT} provides a monotonicity property for non-intersecting Bernoulli walk ensembles randomly sampled under the uniform measure on the set of such families with prescribed entrance, exit, and boundary conditions. In what follows, for any functions $\mathsf{f}, \mathsf{f}': [s, t] \rightarrow \mathbb{R}$ we write $\mathsf{f} \leq \mathsf{f}'$ if $\mathsf{f} (r) \leq \mathsf{f}' (r)$ for each $r \in [s, t]$. Similarly, for any sequences $\mathsf{d} = (\mathsf{d}_l, \mathsf{d}_{l + 1}, \ldots , \mathsf{d}_m) \subset \mathbb{R}$ and $\mathsf{d}' = (\mathsf{d}_l', \mathsf{d}_{l + 1}', \ldots , \mathsf{d}_m') \subset \mathbb{R}$, we write $\mathsf{d} \leq \mathsf{d}'$ if $\mathsf{d}_j \leq \mathsf{d}_j'$ for each $j \in [l, m]$.

	\begin{lem}[{\cite[Lemma 18]{LSRT}}]
		
		\label{comparewalks}
		
		Fix integers $s \leq t$ and $l \leq m$; functions $\mathsf{f}, \mathsf{f}, \mathsf{g}, \mathsf{g}' : [s, t] \rightarrow \mathbb{R}$; and $(m - l + 1)$-tuples $\mathsf{d}, \mathsf{d}', \mathsf{e}, \mathsf{e}'$ with coordinates indexed by $[l, m]$. Let $\mathsf{X} = (\mathsf{x}_l, \mathsf{x}_{l + 1}, \ldots , \mathsf{x}_m)$ denote a uniformly random non-intersecting Bernoulli walk ensemble with boundary data $(\mathsf{f}; \mathsf{g})$; entrance data $\mathsf{d}$; and exit data $\mathsf{e}$. Define $\mathsf{X}' = (\mathsf{x}_l', \mathsf{x}_{l + 1}', \ldots , \mathsf{x}_m')$ similarly, but with respect to $(\mathsf{f}'; \mathsf{g}')$ and $ (\mathsf{d}'; \mathsf{e}')$. If $\mathsf{f} \leq \mathsf{f}'$, $\mathsf{g} \leq \mathsf{g}'$, $\mathsf{d} \leq \mathsf{d}'$, and $\mathsf{e} \leq \mathsf{e}'$, then there exists coupling between $\mathsf{X}$ and $\mathsf{X}'$ such that $\mathsf{x}_j \leq \mathsf{x}_j'$ almost surely, for each $j \in [l, m]$.
		
	\end{lem}

	\begin{rem}
		
		\label{heightcompare}
		
		An equivalent way of stating \Cref{comparewalks} (as was done in \cite{LSRT}) is through the height functions associated with the Bernoulli walk ensembles $\mathsf{X}$ and $\mathsf{X}'$. Specifically, let $\mathsf{D} \subset \mathbb{T}$ be a finite domain, and let $\mathsf{h}, \mathsf{h}' : \partial \mathsf{H} \rightarrow \mathbb{Z}$ denote two boundary height functions such that $\mathsf{h} (v) \geq \mathsf{h}' (v)$, for each $v \in \partial \mathsf{D}$. Let $\mathsf{H}, \mathsf{H}' : \mathsf{D} \rightarrow \mathbb{Z}$ denote two uniformly random height functions on $\mathsf{D}$ with boundary data $\mathsf{H} |_{\partial \mathsf{D}} = \mathsf{h}$ and $\mathsf{H}' |_{\partial \mathsf{D}} = \mathsf{h}'$. Then, \Cref{comparewalks} implies (and is equivalent to) the existence of a coupling between $\mathsf{H}$ and $\mathsf{H}'$ such that $\mathsf{H} (\mathsf{u}) \geq \mathsf{H}' (\mathsf{u})$ almost surely, for each $\mathsf{u} \in \mathsf{D}$. 
		
	\end{rem}
	
	\begin{rem}
		
		\label{measurepaths} 
		
		Due to the correspondence from \Cref{WalkModel} between tilings and non-intersecting Bernoulli walk ensembles, the uniform measure on the set of (free) tilings of a strip domain of the form $\mathbb{Z} \times [s, t] \subset \mathbb{T}$ is equivalent to that on the set of non-intersecting Bernoulli walk ensembles with time spans $[s, t]$ under specified entrance, exit, and boundary conditions. Moreover, if $\mathsf{X} = (\mathsf{x}_l, \mathsf{x}_{l + 1}, \ldots , \mathsf{x}_m)$ is sampled under the uniform measure then it satisfies the following \emph{Gibbs property}. For any $s \leq s' \leq t' \leq t$ and $l \leq l' \leq m' \leq m$, the law of $(\mathsf{x}_{l'}, \mathsf{x}_{l' + 1}, \ldots , \mathsf{x}_{m'})$ restricted to $\mathbb{Z} \times [s', t']$ is uniform measure on those non-intersecting Bernoulli walk ensembles with entrance data $\big( \mathsf{x}_{l'} (s'), \mathsf{x}_{l' + 1} (s'), \ldots , \mathsf{x}_{m'} (s') \big)$; exit data $\big( \mathsf{x}_{l'} (t'), \mathsf{x}_{l' + 1} (t'), \ldots , \mathsf{x}_{m'} (t') \big)$; and boundary conditions $(\mathsf{x}_{l' - 1}; \mathsf{x}_{m' + 1})$.
		
	\end{rem}

	\subsection{Proof of \Cref{walkspconverge}}
	
	\label{ProofPaths}

	In this section we establish \Cref{walkspconverge}. We begin with the following proposition that provides edge statistics for non-intersecting random Bernoulli walks with an approximately quadratic boundary condition.
	
	\begin{prop} 
		
		\label{xiwalks1} 
		
		Fix $\delta \in \big( 0, 1/100 \big)$ and real numbers $\mathfrak{q} = \mathfrak{q}_n$ and $\mathfrak{l} = \mathfrak{l}_n$ bounded away from $0$ and $\infty$. Define $\mathfrak{s}, \mathfrak{r} \in \mathbb{R}$ from $(\mathfrak{l}, \mathfrak{q})$ through \eqref{sr1}; set $m = \lfloor n^{10 \delta} \rfloor$ and $\mathsf{T} = \lfloor n^{2 / 3 + 20 \delta} \rfloor$; and let $\mathsf{f}: [-\mathsf{T}, \mathsf{T}] \rightarrow \mathbb{R}$ be a function satisfying
		\begin{flalign}
			\label{klqf}
			\displaystyle\sup_{s \in [-\mathsf{T}, \mathsf{T}]} \big| \mathsf{f} (s) + K_0 - \mathfrak{l} s- \mathfrak{q} s^2 n^{-1} \big| < n^{1/3 - \delta}, \qquad \text{where} \qquad K_0= \mathfrak{s}^{3/2} n^{1/3} \Big( \displaystyle\frac{3 \pi m}{2} \Big)^{2 / 3}.
		\end{flalign} 
		
		\noindent Further let $\mathsf{d} = (\mathsf{d}_1, \mathsf{d}_2, \ldots, \mathsf{d}_m)$ and $\mathsf{e} = (\mathsf{e}_1, \mathsf{e}_2, \ldots , \mathsf{e}_m)$ be integer sequences satisfying
		\begin{flalign}
			\label{de} 
			\big| \mathsf{d}_j - \mathsf{f} (- \mathsf{T}) \big| < n^{1/3 + 10 \delta}; \qquad \big| \mathsf{e}_j - \mathsf{f} (\mathsf{T}) \big| & < n^{1/3 + 10 \delta},
		\end{flalign}
	
		\noindent for each $j \in [1, m]$. Let $\mathsf{X} = (\mathsf{x}_1, \mathsf{x}_2, \ldots , \mathsf{x}_m)$ denote a uniformly random ensemble of non-intersecting Bernoulli walks with time span $[-\mathsf{T}, \mathsf{T}]$; boundary data $(\mathsf{f}; \infty)$; entrance data $\mathsf{d}$; and exit data $\mathsf{e}$. Define the family of functions $\mathcal{X}_n = (\mathsf{X}_1, \mathsf{X}_2, \ldots , \mathsf{X}_n)$ by
		\begin{flalign}
			\label{xiconstant}
			\mathsf{X}_i (t) = \mathfrak{s}^{-1} n^{-1/3} \Big( \mathsf{x}_{m - i + 1} (\mathfrak{r} n^{2/3} t) - \mathfrak{l} n^{2 / 3} t \Big).
		\end{flalign}
		
		\noindent Then $\mathcal{X}_n$ converges to $\mathcal{R}$, uniformly on compact subsets of $\mathbb{Z}_{> 0} \times \mathbb{R}$, as $n$ tends to $\infty$. 
		
	\end{prop} 
	
	\begin{proof} 
		
		 Throughout, we assume $\mathfrak{q} > 0$, as the proof when $\mathfrak{q} < 0$ is entirely analogous. This proposition will follow from a comparison between the random Bernoulli walk ensemble $\mathsf{X}$ and the one associated with a random tiling of a suitably chosen hexagon. So, we begin by identifying real numbers $a, b, c > 0$ and a point $(x_0, t_0) \in \mathfrak{A} (a, b, c) = \partial \mathfrak{L}_{a, b, c}$ on the ellipse inscribed in the hexagon $\mathfrak{E}_{a, b, c}$ (recall \Cref{xdomain}) whose curvature parameters are given by $(\mathfrak{l}, \mathfrak{q})$. To that end, first observe that there are two points on $\mathfrak{A} (1, 1, 1)$ at which a line with inverse slope $\mathfrak{l}$ is tangent. Let $(\widetilde{x}, \widetilde{t}) \in \mathfrak{A} (1, 1, 1)$ be the one whose curvature parameters $(\widetilde{\mathfrak{l}}, \widetilde{\mathfrak{q}}) = (\mathfrak{l}, \widetilde{\mathfrak{q}})$ are such that $\sgn \widetilde{\mathfrak{q}} = \sgn \mathfrak{q} = 1$. Since $\widetilde{\mathfrak{l}} = \mathfrak{l}$ is bounded away from $0$ and $\infty$, as is $\widetilde{\mathfrak{q}}$. Then setting $r = \widetilde{\mathfrak{q}} \mathfrak{q}^{-1}$, $(a, b, c) = (r, r, r)$, and $(x_0, t_0) = (r \widetilde{x}, r \widetilde{t}) \in \mathfrak{A} (a, b, c)$, the curvature parameters at $(x_0, t_0)$ are $(\mathfrak{l}, \mathfrak{q})$.
		
		We will now perturb the quadratic curvature parameter of $(x_0, t_0)$ with respect to $\mathfrak{A} (\mathfrak{E}_{a, b, c})$ through scaling by $\kappa$ and $\nu$, where  
		\begin{flalign*}
			\kappa = 1 + n^{-20 \delta}; \qquad \nu = 1 - n^{-20 \delta}.
		\end{flalign*}
	
		\noindent Set $(a', b', c') = (\kappa^{-1} a, \kappa^{-1} b, \kappa^{-1} c)$ and $(a'', b'', c'') = (\nu^{-1} a, \nu^{-1} b, \nu^{-1} c)$; further let $(x_0', t_0') = (\kappa^{-1} x_0, \kappa^{-1} t_0)$ and $(x_0'', t_0'') = (\nu^{-1} x_0, \nu^{-1} t_0)$. Observe that the curvature parameters at $(x_0', t_0')$ with respect to $\mathfrak{A} (a', b', c')$ and at $(x_0'', t_0'')$ with respect to $\mathfrak{A} (a'', b'', c'')$ are given by $(\mathfrak{l}', \mathfrak{q}') = (\mathfrak{l}, \kappa \mathfrak{q})$ and $(\mathfrak{l}'', \mathfrak{q}'') = (\mathfrak{l}, \nu \mathfrak{q})$, respectively. 
		
	Next, define $(\mathsf{a}', \mathsf{b}', \mathsf{c}') = (na', nb', nc')$; and $(\mathsf{a}'', \mathsf{b}'', \mathsf{c}'') = (na'', nb'', nc'')$, which we  assume for notational simplicity are integers. Denote the hexagons $\mathsf{P}' = \mathfrak{E}_{\mathsf{a}', \mathsf{b}', \mathsf{c}'}$ and $\mathsf{P}'' = \mathfrak{E}_{\mathsf{a}'', \mathsf{b}'', \mathsf{c}''}$; let $\mathscr{M}'$ and $\mathscr{M}''$ denote uniformly random tilings of $\mathsf{P}'$ and $\mathsf{P}''$, which are associated with non-intersecting Bernoulli walk ensembles $\mathsf{Y}' = (\mathsf{y}_1', \mathsf{y}_2', \ldots , \mathsf{y}_{\mathsf{a}'}')$ and $\mathsf{Y}'' = (\mathsf{y}_1'', \mathsf{y}_1'', \ldots , \mathsf{y}_{\mathsf{a}''}'')$, respectively. Define the Bernoulli walk ensembles $\mathsf{X}' = (\mathsf{x}_1', \mathsf{x}_2', \ldots , \mathsf{x}_m')$ and $\mathsf{X}'' = (\mathsf{x}_1'', \mathsf{x}_2'', \ldots , \mathsf{x}_m'')$ through a spacial and index shift of $\mathsf{Y}'$ and $\mathsf{Y}''$ respectively; specifically, for each $j$ and $t$, set
	\begin{flalign*}
		\mathsf{x}_j' (t) = \mathsf{y}_{\mathsf{a}' + j - m }' (t + t_0' n) - x_0' n + 3 n^{1/3 - \delta}; \qquad \mathsf{x}_j'' (t) = \mathsf{y}_{\mathsf{a}'' + j - m}'' (t + t_0'' n) - x_0'' n - 3 n^{1/3 - \delta}.
	\end{flalign*}

	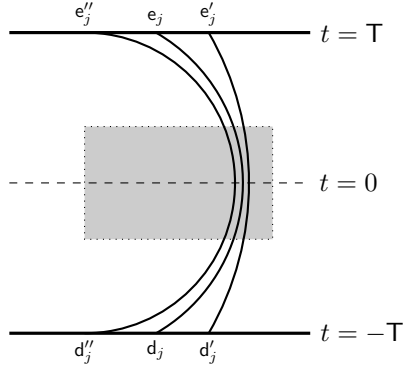
\begin{figure}
		
		\begin{center}		
			
			\begin{tikzpicture}[
				>=stealth,
				auto,
				style={
					scale = 1
				}
				]
				
				\filldraw[fill = gray!40!white, dotted] (1, 2.75) -- (3.5, 2.75) -- (3.5, 1.25) -- (1, 1.25) -- (1, 2.75);
				
				\draw[black, very thick] (0, 0) -- (4, 0) node[right]{$t = -\mathsf{T}$};	
				\draw[black, dashed] (0, 2) -- (4, 2) node[right]{$t = 0$};
				\draw[black, very thick] (0, 4) -- (4, 4) node[right]{$t = \mathsf{T}$};
		
				\draw[black, thick] (1, 0) node[below, scale = .75]{$\mathsf{d}_j''$} arc (-90:90:2) node[above, scale = .75]{$\mathsf{e}_j''$};
				\draw[black, thick] (1.95, 0) node[below, scale = .75]{$\mathsf{d}_j$} arc (-60:60:2.309) node[above, scale = .75]{$\mathsf{e}_j$};
				\draw[black, thick] (2.65, 0) node[below, scale = .75]{$\mathsf{d}_j'$} arc (-30:30:4) node[above, scale = .75]{$\mathsf{e}_j'$};

			\end{tikzpicture}
			
		\end{center}

		\caption{\label{xpathscompare} Shown above trajectories for the paths $\mathsf{x}_j' \leq \mathsf{x}_j \leq \mathsf{x}_j''$ in the proof of \Cref{xiwalks1}; they approximately coincide in the shaded region. }
		
	\end{figure} 

	Given this notation, we will first use \Cref{comparewalks} to bound the ensemble $\mathsf{X}$ between $\mathsf{X}'$ and $\mathsf{X}''$; see \Cref{xpathscompare}. Then, we will apply \Cref{walksdomain1} to show that $(\mathsf{X}', \mathsf{X}'')$ converges to the same Airy line ensemble under the normalization \eqref{xiconstant}. To implement the former, define the sequences $\mathsf{d}', \mathsf{e}', \mathsf{d}'', \mathsf{e}'' \in \mathbb{Z}^m$ by
	\begin{flalign*}
		& \mathsf{d}' =  \big( \mathsf{x}_1' (-\mathsf{T}), \mathsf{x}_2' (-\mathsf{T}), \ldots , \mathsf{x}_m' (-\mathsf{T}) \big); \quad  \mathsf{e}' = \big( \mathsf{x}_1' (\mathsf{T}), \mathsf{x}_2' (\mathsf{T}), \ldots , \mathsf{x}_m' (\mathsf{T}) \big); \\
		& \mathsf{d}'' = \big( \mathsf{x}_1'' (-\mathsf{T}), \mathsf{x}_2'' (-\mathsf{T}), \ldots , \mathsf{x}_m'' (-\mathsf{T}) \big); \quad \mathsf{e}'' = \big( \mathsf{x}_1'' (\mathsf{T}), \mathsf{x}_2'' (\mathsf{T}), \ldots , \mathsf{x}_m'' (\mathsf{T}) \big),
	\end{flalign*}
	
	\noindent and the functions $\mathsf{f}', \mathsf{f}'' : [-\mathsf{T}, \mathsf{T}] \rightarrow \mathbb{Z}$ by
	\begin{flalign*}
		\mathsf{f}' (t) = \mathsf{x}_0' (t); \qquad \mathsf{f}'' (t) = \mathsf{x}_0'' (t).
	\end{flalign*}

	\noindent Then, by the Gibbs property described in \Cref{measurepaths}, $\mathsf{X}'$ is a uniformly random non-intersecting Bernoulli walk ensemble with entrance data $\mathsf{d}'$, exit data $\mathsf{e}'$, and boundary conditions $(\mathsf{x}_0'; \infty)$; a similar statement holds for $\mathsf{X}''$. Let us show with overwhelming probability that 
	\begin{flalign}
		\label{def}
		\mathsf{d}'' \leq \mathsf{d} \leq \mathsf{d}'; \qquad \mathsf{e}'' \leq \mathsf{e} \leq \mathsf{e}'; \qquad \mathsf{f}'' \leq \mathsf{f} \leq \mathsf{f}'.
	\end{flalign}

	To verify the first statement of \eqref{def}, observe for any $j \in [1, m]$ that \eqref{klqf} and \eqref{de} imply
	\begin{flalign}
		\label{d1} 
		\mathsf{d}_j \leq \mathfrak{q} n^{-1} \mathsf{T}^2 - \mathfrak{l} \mathsf{T} + 5 (\mathfrak{s}^2 + 1) n^{1/3 + 10 \delta} \leq \mathfrak{q}' n^{-1} \mathsf{T}^2 - \mathfrak{l} \mathsf{T} - 5 (\mathfrak{s}^2 + 1) n^{1/3 + 10 \delta}.
	\end{flalign}

	\noindent To deduce the first bound, we used the facts that $n^{1/3 - \delta} \leq \mathfrak{s}^{-3/2} K_0 \leq 3 m^{2/3} n^{1/3} \leq n^{1/3 + 10 \delta}$, since $m = \lfloor n^{10 \delta} \rfloor$; to deduce the second, we used the facts that $\mathfrak{q}' = \kappa \mathfrak{q} = (1 + n^{-20 \delta}) \mathfrak{q}$ and $\mathsf{T} = \lfloor n^{2/3 + 20 \delta} \rfloor$. 
	
	Moreover, \Cref{xjpath} (applied with the $\mathsf{P}$ there equal to $\mathsf{P}'$ here) implies with overwhelming probability that
	\begin{flalign}
		\label{d2}
		\mathsf{x}_{m - j}' (-\mathsf{T}) = \mathsf{y}_{\mathsf{a}' - j}' (t_0 n - \mathsf{T}) - x_0' n + 3 n^{1/3 - \delta} \geq \mathfrak{q}' n^{-1} \mathsf{T}^2 - \mathfrak{l} \mathsf{T} - 5 (\mathfrak{s}^2 + 1) n^{1/3 + 10 \delta},
	\end{flalign} 

	\noindent where we have again used the facts that $n^{1/3 - \delta} \leq \mathfrak{s}^{-3/2} \leq K_0 \leq 3 m^{2/3} n^{1/3} \leq n^{1/3 + 10 \delta}$. Combining \eqref{d1} and \eqref{d2}, it follows that $\mathsf{d} \leq \mathsf{d}'$. The proof that $\mathsf{d}'' \leq \mathsf{d}$ is entirely analogous, thereby establishing the first statement of \eqref{def}; the second is shown similarly.
	
	The third statement of \eqref{def} follows from the fact that for any $t \in [-\mathsf{T}, \mathsf{T}]$ we have
	\begin{flalign}
		\label{ftft2} 
		\begin{aligned}
		\mathsf{f}' (t) = \mathsf{x}_0' (t) & = \mathsf{y}_{\mathsf{a}'- m}' (t + t_0' n) - x_0' n + 3 n^{1/3 - \delta} \\
		& \geq \mathfrak{l} t + \mathfrak{q}' n^{-1} t^2 - K_0' - m^{-1/3} n^{1/3 + \delta} + 3 n^{1/3 - \delta} \\
		& \geq \mathsf{f} (t) + K_0 - K_0' + (\mathfrak{q}' - \mathfrak{q}) n^{-1} t^2 + n^{1/3 - \delta} \geq \mathsf{f} (t),
		\end{aligned}
	\end{flalign}

	\noindent where we have set
	\begin{flalign*}
		K_0' = \mathfrak{s}'^{3/2} n^{1/3} \Big( \displaystyle\frac{3 \pi m}{2} \Big)^{2/3} = \kappa^{-1/2} K_0, 
	\end{flalign*}
	
	\noindent and we have denoted
		\begin{flalign}
			\label{srsr}
		\mathfrak{s}' = \Big| \displaystyle\frac{\mathfrak{l}^{2/3} (1 - \mathfrak{l})^{2/3}}{4^{1/3} \mathfrak{q}'^{1/3}} \Big|; \quad \mathfrak{r}' = \Big| \displaystyle\frac{\mathfrak{l}^{1/3} (1 - \mathfrak{l})^{1/3}}{2^{1/3} \mathfrak{q}'^{1/3}} \Big|. 
	\end{flalign}
	
	\noindent The first and second statements of \eqref{ftft2} follow from the definitions of $\mathsf{f}'$ and $\mathsf{x}'$; the third from \Cref{xjpath} (applied with $j$ there equal to $m$ here); the fourth from \eqref{klqf} and the definition $m = \lfloor n^{10 \delta} \rfloor$; and the fifth from the facts that $\mathfrak{q}' = \kappa \mathfrak{q} \geq \mathfrak{q}$ and that 
	\begin{flalign*} 
		|K_0 - K_0'| \leq |1 - \kappa^{-1/2}| K_0 \leq n^{-10 \delta} K_0 \leq 5 \mathfrak{s}^{3/2} n^{1/3 -10 \delta} m^{2/3} = \mathrm{o}(n^{1/3 - \delta}),
	\end{flalign*} 

	\noindent as $m \leq n^{10 \delta}$. Hence, $\mathsf{f} \leq \mathsf{f}'$; similarly, $\mathsf{f}'' \leq \mathsf{f}$. This verifies that \eqref{def} holds with overwhelming probability. So, it follows from \Cref{comparewalks} that there exists a coupling between $(\mathsf{X}, \mathsf{X}', \mathsf{X}'')$ such that $\mathsf{x}_j'' \leq \mathsf{x}_j \leq \mathsf{x}_j'$ holds for each $j \in [1, m]$, with overwhelming probability.
	
	Define normalizations of these Bernoulli walk ensembles, denoted $\mathcal{X}_n' = (\mathsf{X}_1', \mathsf{X}_2', \ldots , \mathsf{X}_m')$, $\mathcal{X}_n'' = (\mathsf{X}_1'', \mathsf{X}_2'', \ldots , \mathsf{X}_n'')$, $\mathcal{W}_n' = (\mathsf{W}_1', \mathsf{W}_2', \ldots , \mathsf{W}_m')$, and $\mathcal{W}_n'' = (\mathsf{W}_1'', \mathsf{W}_2'', \ldots , \mathsf{W}_m'')$, by (recall the notation from \eqref{srsr}) setting
	\begin{flalign*}
		& \mathsf{X}_i' (t) = \mathfrak{s}^{-1} n^{-1/3} \Big( \mathsf{x}_{m - i + 1}' (\mathfrak{r} n^{2/3} t) - \mathfrak{l} n^{2/3} t\Big); \qquad \mathsf{X}_i'' (t) = \mathfrak{s}^{-1} n^{-1/3} \Big( \mathsf{x}_{m -i + 1} (\mathfrak{r} n^{2/3} t) - \mathfrak{l} n^{2/3} t \Big); \\
		& \mathsf{W}_i' (t) = \mathfrak{s}'^{-1} n^{-1/3} \Big( \mathsf{x}_{m - i + 1}' (\mathfrak{r}' n^{2/3} t) - \mathfrak{l} n^{2/3} t\Big); \qquad \mathsf{W}_i'' (t) = \mathfrak{s}''^{-1} n^{-1/3} \Big( \mathsf{x}_{m - i + 1} (\mathfrak{r}'' n^{2/3} t) - \mathfrak{l} n^{2/3} t \Big).
	\end{flalign*}

	\noindent Then, \Cref{walksdomain1} implies that $\mathcal{W}_n'$ and $\mathcal{W}_n''$ converge to $\mathcal{R}$, uniformly on compact subsets of $\mathbb{Z}_{> 0} \times \mathbb{R}$, as $n$ tends to $\infty$.  Since \eqref{srsr} and the facts that $(\mathfrak{q}, \mathfrak{q}') = (\kappa \mathfrak{q}, \nu \mathfrak{q})$ and $\kappa - 1 = 1 - \nu = n^{-20 \delta} = \mathrm{o}(1)$ imply $\mathfrak{s}' = \mathfrak{s} \big( 1 + \mathrm{o}(1) \big)$; $\mathfrak{r}' = \mathfrak{r} \big( 1 + \mathrm{o}(1) \big)$; $\mathfrak{s}'' = \mathfrak{s} \big( 1 + \mathrm{o}(1) \big)$; and $\mathfrak{r}'' = \mathfrak{r} \big( 1 + \mathrm{o}(1) \big)$, we deduce that $\big| \mathsf{W}_j' (t) - \mathsf{X}_j' (t) \big|$ and $\big| \mathsf{W}_j'' (t) - \mathsf{X}_j'' (t) \big| = \mathrm{o}(1)$ for each $j \in [1, m]$, uniformly on compact subsets of $\mathbb{Z}_{> 0} \times \mathbb{R}$. Hence, $\mathcal{X}_n'$ and $\mathcal{X}_n''$ both converge to $\mathcal{R}$. Since $\mathsf{x}_j'' \leq \mathsf{x}_j \leq \mathsf{x}_j''$, it follows that $\mathsf{X}_j'' \leq \mathsf{X}_j \leq \mathsf{X}_j'$, and thus the same convergence holds for $\mathcal{X}_n$.
	\end{proof}

	We can now establish \Cref{walkspconverge}.

	\begin{proof}[Proof of \Cref{walkspconverge}]
		
		This will follow from \Cref{xiwalks1}. Fix a real number $\delta \in \big( 0, 1/100 \big)$, and define $m = n^{10 \delta}$ and $\mathsf{T} = n^{2/3 + 20 \delta}$ (as in \Cref{xiwalks1}), which we assume for notational convenience are integers. Define the non-intersecting Bernoulli walk ensemble $\mathsf{Y} = (\mathsf{y}_1, \mathsf{y}_2, \ldots , \mathsf{y}_m)$ by
		\begin{flalign}
			\label{yti}
			\mathsf{y}_i (t) = \mathsf{x}_{K + i - m} (t_0 n + t) - x_0 n.
		\end{flalign} 
	
		\noindent Denoting the sequences $\mathsf{d}, \mathsf{e} \in \mathbb{Z}^m$ and function $\mathsf{f} : [-\mathsf{T}, \mathsf{T}] \rightarrow \mathbb{Z}$ by 
		\begin{flalign*}
			\mathsf{d} = \big( \mathsf{y}_1 (- \mathsf{T}), \mathsf{y}_2 (-\mathsf{T}), & \ldots , \mathsf{y}_m (-\mathsf{T}) \big); \qquad 
			\mathsf{e} = \big( \mathsf{y}_1 (\mathsf{T}), \mathsf{y}_2 (\mathsf{T}), \ldots , \mathsf{y}_m (\mathsf{T}) \big); \\
			& \mathsf{f} (t) = \mathsf{x}_{K - m} (t_0 n + t) - x_0 n,
		\end{flalign*}
		
		\noindent it follows from the Gibbs property described in \Cref{measurepaths} that $\mathsf{Y}$ is a uniformly random non-intersecting Bernoulli walk ensemble with entrance and exit data $(\mathsf{d}; \mathsf{e})$ and boundary conditions $(\mathsf{f}; \infty)$. Let us verify that $(\mathsf{d}, \mathsf{e}, \mathsf{f})$ satisfy \eqref{de} and \eqref{klqf}. 
		
		To that end, since $m = n^{10 \delta}$, the $j = m$ case of \Cref{xjpath} implies with overwhelming probability that
		\begin{flalign}
			\label{fls} 
			\displaystyle\max_{s \in [-\mathsf{T}, \mathsf{T}]}	\bigg| \mathsf{f} (s) - \mathfrak{l} s - \mathfrak{q} n^{-1} s^2 + \mathfrak{s}^{3/2} \Big( \displaystyle\frac{3 \pi m}{2} \Big)^{3/2} n^{1/3} \bigg| < m^{-1/3} n^{1/3 + \delta} < n^{1/3 - \delta},
		\end{flalign}
		
		\noindent and so $\mathsf{f}$ satisfies \eqref{klqf}. Applying \Cref{xjpath} with $j \in [0, m - 1]$ and also using \eqref{fls} yields  
		\begin{flalign*}
			\big| \mathsf{y}_1 (-\mathsf{T}) - f(-\mathsf{T}) \big| \leq 2 \mathfrak{s}^{3/2} \Big( \displaystyle\frac{3 \pi m}{2} \Big)^{2/3} n^{1/3} + 2n^{1/3} \leq n^{1/3 + 10 \delta},
		\end{flalign*}
	
		\noindent where to deduce the last inequality we used the fact that $m = n^{10 \delta}$. This verifies that $\mathsf{d}$ satisfies \eqref{de} with overwhelming probability; the proof that $\mathsf{e}$ does as well is very similar and thus omitted. 
		
		Hence, \Cref{xiwalks1} applies and gives that $\mathcal{Y}_n = (\mathsf{Y}_1, \mathsf{Y}_2, \ldots , \mathsf{Y}_m)$ converges to $\mathcal{R}$, uniformly on compact subsets of $\mathbb{Z}_{> 0} \times \mathbb{R}$, as $n$ tends to $\infty$, where
		\begin{flalign*}
			\mathsf{Y}_i (t) = \mathfrak{s}^{-1} n^{-1/3} \Big( \mathsf{y}_{m - i + 1} (\mathfrak{r} n^{2/3} t) - \mathfrak{l} n^{2/3} t \Big).
		\end{flalign*}
		
		\noindent By \eqref{yti}, $\mathsf{Y}_i (t) = \mathsf{X}_i (t)$, meaning that $\mathcal{Y}_n = \mathcal{X}_n$; so, the same convergence holds for $\mathcal{X}_n$.
	\end{proof}

	\section{Mixing and Concentration Bounds}
	
	\label{TimeDynamics}
	
	By the content in \Cref{ProofPaths}, it remains to establish \Cref{mh}. In this section we collect several miscellaneous results that will be used in its proof, to appear in \Cref{ProofH} below. More specifically, in \Cref{EstimateDomain} we state a preliminary concentration estimate for a class of tilings whose arctic boundaries are constrained to only have one cusp (in addition to other, less essential conditions); in \Cref{Dynamics} we state a mixing time bound for certain dynamics on the set of tilings, which we prove in \Cref{Prooft}.

	\subsection{Preliminary Concentration Estimate}
	
	\label{EstimateDomain} 
	
	In this section we state a concentration estimate for tiling height functions on ``double-sided trapezoid domains'' (one may also view these as tilings of a strip). These domains are different from the ones considered in earlier works, such as \cite{ARS,AURTPF,UEFDIPS}, since they will accomodate non-frozen boundary conditions along both their north and south edges (instead of only their south ones).
	
	Throughout this section, we fix real numbers $\mathfrak{t}_1 < \mathfrak{t}_2$ and denote $\mathfrak{t} = \mathfrak{t}_2 - \mathfrak{t}_1$. We fix linear functions $\mathfrak{a}, \mathfrak{b} : [\mathfrak{t}_1, \mathfrak{t}_2]$ with $\mathfrak{a}' (s), \mathfrak{b}' (s) \in \{ 0, 1 \}$, such that $\mathfrak{a} (s) \leq \mathfrak{b} (s)$ for each $s \in [\mathfrak{t}_1, \mathfrak{t}_2]$. Define the trapezoid domain
	\begin{flalign}
		\label{d}
		\mathfrak{D} = \mathfrak{D} (\mathfrak{a}, \mathfrak{b}; \mathfrak{t}_1, \mathfrak{t}_2) = \big\{ (x, y) \in \mathbb{R} \times (\mathfrak{t}_1, \mathfrak{t}_2) : \mathfrak{a} (t) < x < \mathfrak{b} (t) \big\},
	\end{flalign}

	\noindent and denote its four boundaries by 
	\begin{flalign}
		\label{dboundary}
		\begin{aligned} 
		&\partial_{\so} (\mathfrak{D}) = \big\{ (x, y) \in \overline{\mathfrak{D}}: y = \mathfrak{t}_1 \big\}; \qquad \quad
		\partial_{\north} (\mathfrak{D}) = \big\{ (x, y) \in \overline{\mathfrak{D}}: y = \mathfrak{t}_2 \big\}; \\ 
		& \partial_{\we} (\mathfrak{D}) = \big\{ (x, y) \in \overline{\mathfrak{D}}: x = \mathfrak{a} (y) \big\}; \qquad 
		\partial_{\ea} (\mathfrak{D}) = \big\{ (x, y) \in \overline{\mathfrak{D}}: x = \mathfrak{b} (y) \big\}.
		\end{aligned} 
	\end{flalign}

	\noindent We refer to \Cref{ddomain} for a depiction.

	\begin{figure}
		
		\begin{center}		
			
			\begin{tikzpicture}[
				>=stealth,
				auto,
				style={
					scale = .52
				}
				]
				
				\draw[black, very thick] (0, 0) node[left, scale = .7]{$y = \mathfrak{t}_1$} -- (4, 0) -- (7, 3) -- (0, 3) node[left, scale = .7]{$y = \mathfrak{t}_2$} -- (0, 0);
				\draw[black, very thick] (6.5, 0) -- (10.5, 0) -- (13.5, 3) -- (9.5, 3) -- (6.5, 0);
				\draw[black, very thick] (15, 0) -- (19, 0) -- (19, 3) -- (15, 3) -- (15, 0);
				\draw[black, very thick] (20.5, 0) -- (25.5, 0) -- (25.5, 3) -- (23.5, 3) -- (20.5, 0);
				
				\draw[] (3.5, 3) node[above, scale = .7]{$\partial_{\north} (\mathfrak{D})$};
				\draw[] (5.5, 1.5) node[below = 2, right, scale = .7]{$\partial_{\ea} (\mathfrak{D})$};
				\draw[] (0, 1.5) node[left, scale = .7]{$\partial_{\we} (\mathfrak{D})$};
				\draw[] (2, 0) node[below, scale = .7]{$\partial_{\so} (\mathfrak{D})$};	
				
			\end{tikzpicture}
			
		\end{center}
		
		\caption{\label{ddomain} Shown above are the four possibilities for $\mathfrak{D}$.}
		
	\end{figure}
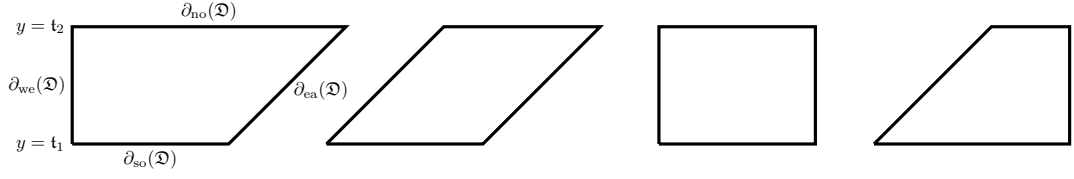 

	Let $h: \partial \mathfrak{D} \rightarrow \mathbb{R}$ denote a function admitting an admissible extension to $\mathfrak{D}$. We assume throughout that $h$ is constant along both $\partial_{\we} (\mathfrak{D})$ and $\partial_{\ea} (\mathfrak{D})$. Let $H^* \in \Adm (\mathfrak{D}; h)$ denote the maximizer of $\mathcal{E}$ from \eqref{efunctionh}, as in \eqref{hmaximum}, and let the liquid region $\mathfrak{L} = \mathfrak{L} (\mathfrak{D}; h)$ and arctic boundary $\mathfrak{A} = \mathfrak{A} (\mathfrak{D}; h)$ be as in \eqref{al}. Recall that a point on $\mathfrak{A}$ is a tangency location if the tangent line to $\mathfrak{A}$ through it has slope either $\{ 0, 1, \infty \}$. 
	
	We may then define the complex slope $f: \mathfrak{L} \rightarrow \mathbb{H}^-$ as in \eqref{fh}, which upon denoting $f_t (x) = f (x, t)$ satisfies the complex Burgers equation \eqref{ftx}. Further let $\mathfrak{L}_{\north} = \mathfrak{L}_{\north} (\mathfrak{D}; h)$ denote the interior of $\overline{\mathfrak{L}} \cap \partial_{\north} (\mathfrak{D})$, and let $\mathfrak{L}_{\so} = \mathfrak{L}_{\so} (\mathfrak{D}; h)$ denote the interior of $\overline{\mathfrak{L}} \cap \partial_{\so} (\mathfrak{D})$; these are the extensions of the liquid region $\mathfrak{R}$ to the north and south boundaries of $\mathfrak{L}$. For all $t \in (\mathfrak{t}_1, \mathfrak{t}_2)$, we define slices of the liquid region (along the horizontal line $y = t$) by
	\begin{flalign*}
		I_t = \big\{ x : (x, t) \in \overline{\mathfrak{L}} \big\}; \qquad I_{\mathfrak{t}_1} = \overline{\mathfrak{L}}_{\so}; \qquad I_{\mathfrak{t}_2} = \overline{\mathfrak{L}_{\north}}.
	\end{flalign*}
	
	For any real number $\delta > 0$, we define the augmented variant of $\mathfrak{L}$ (as in \Cref{llarge}) by
	\begin{flalign*}
		\mathfrak{L}_+ (\mathfrak{D}) = \mathfrak{L}_+^{\delta} (\mathfrak{D}) = \mathfrak{L}_+^{\delta} (\mathfrak{D}; h) = \mathfrak{L} \cup \bigcup_{u \in \mathfrak{A}} \mathfrak{B} (u; n^{\delta - 2/3}).
	\end{flalign*}
	
	Next, let us formulate certain conditions on the limit shape $H^*$. For any $\mathfrak{t}' \geq \mathfrak{t}_2$, we say that $H^*$ can be \emph{extended to time $\mathfrak{t}'$} if there exists a simply connected, open subset $\widetilde{\mathfrak{L}} \subset \mathbb{R}^2$ containing $\mathfrak{L}$, such that the set $\big\{ x: (x, t) \in \widetilde{\mathfrak{L}} \big\}$ is non-empty and connected for each $t \in [\mathfrak{t}_2, \mathfrak{t}']$, and there exists an extension $\widetilde{f}: \widetilde{\mathfrak{L}} \rightarrow \mathbb{H}^-$ of $f_t (x)$ satisfying the complex Burgers equation \eqref{ftx}. In this case, $\mathfrak{L}_{\north} = I_{\mathfrak{t}_2}$ is a single interval. We also call $\partial_{\north} (\mathfrak{D})$ \emph{packed} (with respect to $h$) if $\partial_x h (v) = 1$ for each $v \in \partial_{\north} (\mathfrak{D})$; in this case, $\overline{\mathfrak{L}} \cap \partial_{\north} (\mathfrak{D})$ consistutes at most a single point, and so $\mathfrak{L}_{\north}$ is empty. We refer to \Cref{dldomain} for a depiction. 
	
	Now let $n \geq 1$ be an integer; denote $\mathsf{t}_1 = \mathfrak{t}_1 n$, and $\mathsf{t}_2 = \mathfrak{t}_2 n$. Suppose $\mathsf{D} = n \mathfrak{D} \subset \mathbb{T}^2$, so that $\mathsf{t}_1, \mathsf{t}_2 \in \mathbb{Z}_{> 0}$. Let $\mathsf{h}: \partial \mathsf{D} \rightarrow \mathbb{Z}$ denote a boundary height function. We next stipulate the following assumption on the continuum limit shape $H^*$. Here, we fix a real number $\delta > 0$ and a (large) positive integer $n$. Below, we view the quantities $\mathfrak{t}_1 < \mathfrak{t}_2 < \mathfrak{t}'$; functions $\mathfrak{a}, \mathfrak{b}$; and polygonal domain $\mathfrak{P}$ as independent of $n$. In what follows, a \emph{horizontal tangency location} of $\partial \mathfrak{L}$ is a tangency location on $\mathfrak{A}$ at which the tangent line is horizontal (parallel to the $x$-axis).
	
		\begin{assumption}
		\label{xhh} 
		
		Assume the following constraints hold.
		
		\begin{enumerate} 
			\item The boundary height function $h$ is constant along both $\partial_{\ea} (\mathfrak{D})$ and $\partial_{\we} (\mathfrak{D})$.
			\item Either $\partial_{\north} (\mathfrak{D})$ is packed with respect to $h$ or there exists $\mathfrak{t}' > \mathfrak{t}_2$ such that $H^*$ admits an extension to time $\mathfrak{t}'$.
			\item There exists $\widetilde{\mathfrak{t}} \in [\mathfrak{t}_1, \mathfrak{t}_2]$ such that the following holds. For $t \in [\mathfrak{t}_1, \widetilde{\mathfrak{t}}]$, the set $I_t$ consists of one nonempty interval, and for $t \in (\widetilde{\mathfrak{t}}, \mathfrak{t}_2]$, the set $I_t$ consists of two nonempty disjoint intervals.\footnote{Observe if $\widetilde{\mathfrak{t}} = \mathfrak{t}_2$, then $I_t$ always only consists of one nonempty interval.}
			\item Any tangency location along $\partial \overline{\mathfrak{L}}$ is of the form $\min I_t$ or $\max I_t$, for some $t \in (\mathfrak{t}_1, \mathfrak{t}_2)$. At most one tangency location is of the form $\min I_t$, and at most one is of the form $\max I_t$.
			\item There exists an algebraic curve $Q$ such that, for any $(x, t) \in \overline{\mathfrak{L}}$, we have 
			\begin{flalign*}
				Q \bigg( f_t (x), x - \displaystyle\frac{t f_t (x)}{f_t (x) + 1} \bigg) = 0.
			\end{flalign*}
				
			\noindent Furthermore, the curve $Q$ ``approximately comes from a polygonal domain'' in the following sense. There exists a polygonal domain $\mathfrak{P}$ satisfying \Cref{pa} with liquid region $\mathfrak{L} (\mathfrak{P})$ and a real number $\alpha \in \mathbb{R}$ with $|\alpha - 1| < n^{-\delta}$ such that, if $Q_{\mathfrak{L} (\mathfrak{P})}$ is the algebraic curve associated with $\mathfrak{L} (\mathfrak{P})$ from \Cref{pa1}, then
			\begin{flalign*}
				Q (u, v) = \displaystyle\frac{u + 1}{\alpha^{-1} u + 1} Q_{\mathfrak{L} (\mathfrak{P})} (\alpha^{-1} u, v).
			\end{flalign*}
			
		\end{enumerate}
		
	\end{assumption}

	Let us briefly comment on these constraints. The first guarantees that the associated tiling is one of a trapezoid, as depicted in \Cref{ddomain}. The second and third guarantee that the arctic boundary for the tiling has only one cusp. The fourth implies that are are most two tangency locations along the arctic boundary (and are along the leftmost and rightmost components of the arctic curve, if they exist); the fifth implies that the limit shape for the tiling is part of (an explicit perturbation of) one given by a polygonal domain. The last two conditions could in principle be weakened; we impose them since doing so will substantially simplify notation in the proofs later. 
		
	The next assumption indicates how the tiling boundary data $\mathsf{h}$ approximates $h$ along $\partial \mathfrak{D}$. 
		
	\begin{assumption} 
		\label{xhh2} 
		
		Adopt \Cref{xhh}, and assume the following on how $\mathsf{h}$ converges to $h$.
		
		\begin{enumerate}
			\item For each $v \in \partial \mathfrak{D}$, we have $\big| \mathsf{h} (nv) - n h(v) \big| < n^{\delta / 2}$.
			\item For each $v \in \partial_{\ea} (\mathfrak{D}) \cup \partial_{\we} (\mathfrak{D})$, and each $v \in \partial_{\north} \mathfrak{D}\cup \partial_{\so} \mathfrak{D}$ such that $v \notin \mathfrak{L}_+^{\delta / 2} (\mathfrak{D})$, we have $\mathsf{h} (nv) = n h(v)$. 
		\end{enumerate}
	
	\end{assumption}

	The first assumption states that $\mathsf{h}$ approximates its limit shape, and the second states that it coincides with its limit shape in the frozen region.\footnote{Observe since $\partial_x H \le 1$ that, if $\partial_{\north} (\mathfrak{D})$ is packed, then the second part of \Cref{xhh2} implies that $\mathsf{h} (nv) = n h(v)$ for each $v \in \partial_{\north} (\mathfrak{D})$.} Recalling that $\mathscr{G} (\mathsf{h})$ denotes the set of height functions on $\mathsf{D}$ with boundary data $\mathsf{h}$. We can now state the following concentration estimate for a uniformly element of $\mathscr{G} (\mathsf{h})$ from part I of this series \cite{H1}. In particular, the below result appears as \cite[Theorem 2.5]{H1}, where Assumption 2.3 there is verified by \Cref{xhh} and \cite[Proposition A.4]{H1}.

	\begin{figure}
		\begin{center}			
			\begin{tikzpicture}[
				>=stealth,
				auto,
				style={
					scale = .52
				}
				]
	
				\draw[black, very thick] (-5, 0) node[left, scale = .7]{$y = \mathfrak{t}_1$}-- (5.5, 0) -- (5.5, 3) -- (-2, 3) node[left, scale = .7]{$y = \mathfrak{t}_2$} -- (-5, 0);
				\draw[black, very thick] (9.5, 0) -- (18.5, 0) node[right, scale = .7]{$y = \mathfrak{t}_1$} -- (18.5, 3) node[right, scale = .7]{$y = \mathfrak{t}_2$} -- (12.5, 3) -- (9.5, 0);
				\draw[dotted] (-2, 3.3) -- (5.7, 3.3) node[right, scale = .7]{$y = \mathfrak{t}'$};
				
				\draw[black, thick] (-3, 0) arc (180:120:3.464);
				\draw[black, thick] (3.928, 0) arc (0:60:3.464);
				\draw[black, thick] (11.5, 0) arc (180:0:3);
				
				\draw[black, thick] (-1, 0) arc (-60:0:1.732);
				\draw[black, thick] (.75, 0) arc (240:180:1.732);
				
				\draw[black, dashed] (-1.278, 3) arc (120:60:3.464);
				
				\draw[] (0, 2) circle [radius = 0] node[]{$\mathfrak{L}$};
				\draw[] (14.75, 2) circle [radius = 0] node[]{$\mathfrak{L}$};
				
				\draw[-, dashed, thick] (-4.15, 1) node[left, scale = .7]{$y = \mathfrak{s}$} -- (5.7, 1);
				\draw[-, dashed, thick]  (10.35, 1) -- (18.7, 1) node[right, scale = .7]{$y = \mathfrak{s}$};
			\end{tikzpicture}	
		\end{center}
		\caption{\label{dldomain} Shown to the left is an example of limit shape admitting an extension to time $\mathfrak{t}' > \mathfrak{t}_2$; shown to the right is a liquid region that is packed with respect to $h$.}
	\end{figure}
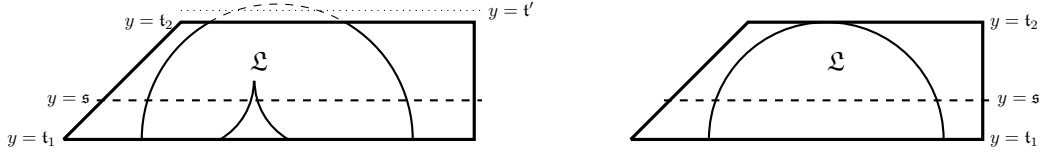

	\begin{thm}[{\cite[Theorem 2.5]{H1}}]
		
		\label{estimategamma}
		
		There exists a constant $\mathfrak{c} = \mathfrak{c} (\mathfrak{P}) > 0$ such that the following holds. Adopt Assumption 4.1 and Assumption 4.2, and further assume that $\mathfrak{t}_2 - \mathfrak{t}_1 < \mathfrak{c}$. Let $\mathsf{H} : \mathsf{D} \rightarrow \mathbb{Z}$ denote a uniformly random element of $\mathscr{G} (\mathsf{h})$. Then, the following two statements hold with overwhelming probability. 
		
		\begin{enumerate} 
			\item We have $\big| \mathsf{H} (nu) - n H^* (u) \big| < n^{\delta}$, for any $u \in \mathfrak{D}$.
			\item For any $u \in \mathfrak{D} \setminus \mathfrak{L}_+^{\delta} (\mathfrak{D})$, we have $\mathsf{H} (nu) =n H^* (u)$.
		\end{enumerate}  
	
	\end{thm}

	\begin{rem}
		
		\label{pathsgamma} 
		
		Recall from \Cref{WalkModel} that the height function $\mathsf{H}$ from \Cref{estimategamma} can equivalently be interpreted as a family of non-intersecting Bernoulli walks on $\mathsf{D}$. The fact that $h$ is constant along both $\partial_{\ea} (\mathfrak{D})$ and $\partial_{\we} (\mathfrak{D})$ implies that $\mathsf{h}$ is constant along the east and west boundaries of $\mathsf{D}$. This is equivalent to not imposing any left or right boundary constraints $(\mathsf{f}; \mathsf{g})$ for these non-intersecting random Bernoulli walks (in the sense described in \Cref{CompareX}).
		
	\end{rem}

	\subsection{Mixing Time Estimates} 
	
	\label{Dynamics} 
	
	Fix a domain $\mathsf{R} \subset \mathbb{T}$ and a boundary height function $\mathsf{h}: \partial \mathsf{R} \rightarrow \mathbb{Z}$. Let us introduce the following Markov dynamics on $\mathscr{G} (\mathsf{h})$ that, given a certain decomposition of $\mathsf{R}$ as a union $\mathsf{R} = \bigcup_{i = 1}^k \mathsf{R}_i$ of domains, ``alternate'' between uniformly resampling $\mathsf{H} \in \mathscr{G} (\mathsf{h})$ on each of the $\mathsf{R}_i$.
	
	\begin{definition}
		
		\label{r1r2r} 
		
		Fix an integer $k \geq 1$, a domain $\mathsf{R} \subset \mathbb{T}$, and a boundary height  function $\mathsf{h}: \partial \mathsf{R} \rightarrow \mathbb{Z}$. Suppose that $\mathsf{R}\setminus \partial \mathsf{R}$ and $\mathscr{G} (\mathsf{h})$ are both nonempty, and that $k \le \diam \mathsf{R}$. Let $\mathsf{R}_1, \mathsf{R}_2, \ldots , \mathsf{R}_k \subseteq \mathsf{R}$ denote domains such that $\mathsf{R} = \bigcup_{i = 1}^k \mathsf{R}_i$, and such that any interior vertex of $\mathsf{R}$ is an interior vertex of some $\mathsf{R}_i$, that is, for each $\mathsf{v} \in \mathsf{R} \setminus \partial \mathsf{R}$, there exists $i \in [1, k]$ for which $\mathsf{v} \in \mathsf{R}_i \setminus \partial \mathsf{R}_i$. The \emph{alternating dynamics} on $\mathsf{R}$ with respect to $(\mathsf{R}_1, \mathsf{R}_2, \ldots , \mathsf{R}_k)$, denoted by $M_{\alt}$, is the discrete-time Markov chain on $\mathscr{G} (\mathsf{h})$, whose state $\mathsf{H}_{t+1} \in \mathscr{G} (\mathsf{h})$ at any time $t+1 \in \mathbb{Z}_{\geq 1}$ is defined from $\mathsf{H}_t$ as follows. 
		
		Let $i \in [1, k]$ denote the integer such that $k$ divides $t -  i + 1$, and let $\mathsf{h}_{t+1} = \mathsf{H}_t |_{\partial \mathsf{R}_i}$. Further let $\mathsf{F}_{t+1} \in \mathscr{G} (\mathsf{h}_{t+1})$ denote a uniformly random height function on $\mathsf{R}_i$. Then, define $\mathsf{H}_{t+1}: \mathsf{R} \rightarrow \mathbb{Z}$ by setting $\mathsf{H}_{t+1} |_{\mathsf{R}_i} = \mathsf{F}_{t+1}$ and $\mathsf{H}_{t+1} |_{\mathsf{R} \setminus \mathsf{R}_i} = \mathsf{H}_t |_{\mathsf{R}\setminus \mathsf{R}_i}$. 
		
	\end{definition} 

	Observe (as is quickly verified by induction on $|\mathsf{R}|$) that the alternating dynamics are irreducible. Thus, they admit a unique stationary measure \cite[Corollary 1.17]{CMT}, which is the uniform one $\mathscr{G} (\mathsf{h})$. 
	
	We will bound the rate of convergence to stationarity for these alternating dynamics, so let us recall some notion on mixing times. Fix a discrete state space $\mathscr{S}$, and let $\mathscr{P} (\mathscr{S})$ denote the set of probability measures on $\mathscr{S}$. The total variation distance between two measures $\nu, \nu' \in \mathscr{P} (\mathscr{S})$ is  
	\begin{flalign*}
		d_{\TV} (\nu, \nu') = \displaystyle\max_{\mathscr{A} \subseteq \mathscr{S}} \big| \nu (\mathscr{A}) - \nu' (\mathscr{A}) \big|.
	\end{flalign*}
	
	\noindent In addition, fix an irreducible Markov chain $M: \mathscr{P} (\mathscr{S}) \rightarrow \mathscr{P} (\mathscr{S})$ on $\mathscr{S}$, whose unique stationary measure is denoted by $\rho$. For any real number $\varepsilon > 0$, the \emph{mixing time with respect to $M$} is given by 
	\begin{flalign}\label{e:mix}
		t_{\mix} (\varepsilon; M) = \min \Big\{ t \in \mathbb{Z}_{\geq 0}: \displaystyle\max_{\nu \in \mathscr{P} (\mathscr{S})}  d_{\TV} (M^t \nu, \rho) \le \varepsilon \Big\},
	\end{flalign}

	\noindent which by \cite[Exercise 4.3]{CMT} (which is a quick consequence of \cite[Proposition 4.7]{CMT}) satisfies 
	\begin{flalign} 
		\label{d02} 
		d_{\TV} (M^t \nu, \rho) \le d_{\TV} (M^{t_{\mix} (\varepsilon; M)} \nu, \rho) \le \varepsilon, \qquad \text{whenever $t \ge t_{\mix} (\varepsilon; M)$}.
	\end{flalign} 
	
	We now state the following (very coarse) estimate on the mixing time for the dynamics $M_{\alt}$. Its proof will appear in \Cref{Prooft} below. 
	
	\begin{prop}
		
		\label{r1r2estimate}
		
		There exists a constant $C > 1$ such that the following holds. Adopt the notation of \Cref{r1r2r}, and set $A = (\diam \mathsf{R})^2$. If $A \geq C$, then $t_{\mix} (e^{-A}; M_{\alt}) \leq A^{11}$. 
	\end{prop} 
	
	\subsection{Proof of \Cref{r1r2estimate}} 
	
	\label{Prooft} 
	
	There are likely many ways of establishing \Cref{r1r2estimate}; the proof below will proceed through a comparison between the alternating dynamics and the Glauber (``flip'') dynamics. To define the latter, given a height function $\mathsf{H} : \mathsf{R} \rightarrow \mathbb{Z}$ and an interior vertex $\mathsf{v} \in \mathsf{R} \setminus \partial \mathsf{R}$, we say $\mathsf{v}$ is \emph{increasable with respect to $\mathsf{H}$} if the function $\mathsf{H}' : \mathsf{R} \rightarrow \mathbb{Z}$ defined by 
	\begin{flalign*}
		\mathsf{H}' (\mathsf{v}) = \mathsf{H} (\mathsf{v}) + 1; \qquad \mathsf{H}' (\mathsf{u}) = \mathsf{H} (\mathsf{u}), \quad \text{for $\mathsf{u} \in \mathsf{R}\setminus \{ v \}$},
	\end{flalign*}
	
	\noindent is a height function on $\mathsf{R}$. In this case, we say that $\mathsf{H}'$ is the \emph{unit increase} of $\mathsf{H}$ at $\mathsf{v}$. We define decreasable vertices and unit decreases of $\mathsf{H}$ analogously. Observe that a vertex of $\mathsf{R}$ cannot simultaneously be increasable and decreasable with respect to a given height function $\mathsf{H}$ on $\mathsf{R}$. 
	
	\begin{definition}
		
		\label{hh} 
		
		Given a height function $\mathsf{H} \in \mathscr{G} (\mathsf{h})$ and an interior vertex $\mathsf{v} \in \mathsf{R} \setminus \partial \mathsf{R}$, the \emph{random flip of $\mathsf{H}$ at $\mathsf{v}$} is the random height function $\mathsf{H}' \in \mathscr{G} (\mathsf{h})$ defined as follows. 
		
		\begin{itemize}
			\item If $\mathsf{v}$ is neither increasable nor decreasable with respect to $\mathsf{H}$, then set $\mathsf{H}' = \mathsf{H}$.
			\item If $\mathsf{v}$ is increasable with respect to $\mathsf{H}$, then with probability $\frac{1}{2}$ set $\mathsf{H}'$ to be the unit increase of $\mathsf{H}$ at $\mathsf{v}$. Otherwise, set $\mathsf{H}' = \mathsf{H}$. 
			\item If $\mathsf{v}$ is decreasable with respect to $\mathsf{H}$, then with probability $\frac{1}{2}$ set $\mathsf{H}'$ to be the unit decrease of $\mathsf{H}$ at $\mathsf{v}$. Otherwise, set $\mathsf{H}' = \mathsf{H}$.
		\end{itemize} 	
	\end{definition} 
	
	Now we define two Markov chains on $\mathscr{G} (\mathsf{h})$, the flip dynamics and the region-flip dynamics. Each update of either chain is obtained by applying a random flip to an interior vertex $\mathsf{v}$ of $\mathsf{R}$. In the flip dynamics, $\mathsf{v} \in \mathsf{R} \setminus \partial \mathsf{R}$ is selected uniformly at random; in the region-flip dynamics $\mathsf{v}$ is selected uniformly at random from $\mathsf{R}_i \setminus \partial \mathsf{R}_i$, where $i$ is determined from the time of the update.

	\begin{definition}
		
		\label{rh} 
		
		The \emph{flip dynamics} on $\mathsf{R}$, denoted by $M_{\fl}$, is the discrete-time Markov chain on $\mathscr{G} (\mathsf{h})$ whose state $\mathsf{H}_{t+1}$ at time $t+1 \geq 1$ is defined from $\mathsf{H}_t$ as follows. Select an interior vertex $\mathsf{v} \in \mathsf{R} \setminus \partial \mathsf{R}$ uniformly at random, and set $\mathsf{H}_{t+1}$ to be the random flip of $\mathsf{H}_t$ at $\mathsf{v}$. 
		
		The \emph{region-flip dynamics} on $\mathsf{R}$ with respect to $(\mathsf{R}_1, \mathsf{R}_2, \ldots , \mathsf{R}_k)$, denoted by $M_{\rf}$, is the discrete-time Markov chain on $\mathscr{G} (\mathsf{h})$ whose state $\mathsf{H}_{t+1}$ at time $t+1 \geq 1$ is defined from $\mathsf{H}_t$ as follows. Let $i \in [1, k]$ denote the integer such that $(k t_0 + i - 1) A^5 < t+1 \leq (k t_0 + i) A^5$, for some $t_0 \in \mathbb{Z}_{\geq 0}$. Select a vertex $\mathsf{v} \in \mathsf{R}_i \setminus \partial \mathsf{R}_i$ uniformly at random, and set $\mathsf{H}_{t+1}$ to be the random flip of $\mathsf{H}_t$ at $\mathsf{v}$. 
		
	\end{definition}
	
	We then have the following result from \cite{ADCC} bounding the mixing time for the flip dynamics.

	\begin{prop}[{\cite[Theorem 5]{ADCC}}]
		
		\label{rt1} 
		
		Adopting the notation of \Cref{r1r2estimate}, we have for any real number $\varepsilon \in (0, 1)$ that 
		\begin{flalign*}
			t_{\mix} (\varepsilon; M_{\fl}) < C A^4 \log A + C A^3 \log A \log \varepsilon^{-1}.
		\end{flalign*}
		
	\end{prop} 
	
	We now state the following two lemmas, which will be established below. 
	\begin{lem} 
		
		\label{rt3}
		
		Under the notation of \Cref{r1r2estimate}, 
		\begin{flalign*} 
			t_{\mix} (e^{-2A}; M_{\rf}) \leq 144A^{11} \cdot \bigg( t_{\mix} \Big( \frac{e^{-2A}}{32A^4}; M_{\fl} \Big) + 1 \bigg).
		\end{flalign*} 
	\end{lem}

	\begin{lem}
		
		\label{rt2}
		
		Under the notation of \Cref{r1r2estimate}, $A^5 t_{\mix} (e^{-A}; M_{\alt}) \leq t_{\mix} (e^{-2A}; M_{\rf})$. 
	\end{lem}
	
	Given these two results, we can quickly establish \Cref{r1r2estimate}.
	
	\begin{proof}[Proof of \Cref{r1r2estimate}]
		
		This follows from \Cref{rt2}, \Cref{rt3}, and \Cref{rt1} (the last applied with the $\varepsilon$ there equal to $\frac{e^{-2A}}{32A^4}$ here).		
	\end{proof}

	The proofs of \Cref{rt3} and \Cref{rt2} will use ``weighted'' and ``censored'' forms of the flip dynamics from \Cref{rh}. To define these, adopting the notation of \Cref{r1r2r}, for each $i \in [1, k]$ let
	\begin{flalign}
		\label{pi} 
		p_i = |\mathsf{R}_i \setminus \partial \mathsf{R}_i| \cdot \Bigg(\displaystyle\sum_{j=1}^k |\mathsf{R}_j \setminus \partial \mathsf{R}_j| \Bigg)^{-1} \in [0, 1].
	\end{flalign} 

	\noindent For any vertex $v \in \mathsf{R} \setminus \partial \mathsf{R}$, we further let 
	\begin{flalign*} 
		\mathfrak{m}(\mathsf{v}) = \# \big\{ j \in [1, k] : \mathsf{v} \in \mathsf{R}_j \setminus \partial \mathsf{R}_j \big\}; \quad \mathfrak{M} = \displaystyle\sum_{\mathsf{v} \in \mathsf{R} \setminus \partial \mathsf{R}} \mathfrak{m} (\mathsf{v}) = \displaystyle\sum_{j = 1}^k |\mathsf{R}_j \setminus \partial \mathsf{R}_j|; \quad \mathfrak{m}= \displaystyle\sum_{\mathsf{v} \in \mathsf{R} \setminus \partial \mathsf{R}} \mathfrak{m} (\mathsf{v})^{-1}.
	\end{flalign*}
	
	\begin{definition} 
		
		\label{mcf}
		
		The \emph{weighted flip dynamics}, denoted by $M_{\wf}$, is the discrete-time Markov chain on $\mathscr{G} (\mathsf{h})$, whose state $\mathsf{H}_{t+1}$ at time $t+1$ is defined from $\mathsf{H}_t$ as follows. Select an interior vertex $\mathsf{v} \in \mathsf{R} \setminus \partial \mathsf{R}$ with probability $\mathfrak{m} (\mathsf{v}) \cdot \mathfrak{M}^{-1}$, and set $\mathsf{H}_{t+1}$ to be the random flip of $\mathsf{H}_t$ at $\mathsf{v}$.		
		
		The \emph{censored weighted flip dynamics}, denoted by $M_{\cwf}$, is the discrete-time Markov chain on $\mathscr{G} (\mathsf{h})$, whose state $\mathsf{H}_{t+1}$ at time $t+1$ is defined from $\mathsf{H}_t$ as follows. Select an interior vertex $\mathsf{v} \in \mathsf{R} \setminus \partial \mathsf{R}$ with probability $\mathfrak{m} (\mathsf{v}) \cdot \mathfrak{M}^{-1}$. Then set $\mathsf{H}_{t+1}$ to be the random flip of $\mathsf{H}_t$ at $\mathsf{v}$ with probability $\mathfrak{m}(\mathsf{v})^{-1} \cdot \mathfrak{m}^{-1}$, and set $\mathsf{H}_{t+1} = \mathsf{H}_t$ with the complementary probability $1 - \mathfrak{m} (\mathsf{v})^{-1} \cdot \mathfrak{m}^{-1}$. 
		
	\end{definition}

	\begin{rem} 
		
		\label{mcwf2} 
		
		By \Cref{mcf}, we may interpret the censored weighted flip dynamics $M_{\cwf}$ as the following ``lazy'' version of the flip dynamics from \Cref{rh}. With probability $1 - |\mathsf{R} \setminus \partial \mathsf{R}| \cdot (\mathfrak{M} \cdot \mathfrak{m})^{-1}$, we perform a \emph{lazy} step and set $\mathsf{H}_{t+1} = \mathsf{H}_t$. Otherwise we perform a \emph{active} step, by selecting a vertex $\mathsf{v} \in \mathsf{R} \setminus \partial \mathsf{R}$ uniformly at random and setting $\mathsf{H}_{t+1}$ to be the random flip of $\mathsf{H}_t$ at $\mathsf{v}$. 
	
	\end{rem}

	The following lemma compares the mixing time for $M_{\cwf}$ to that of the flip dynamics $M_{\fl}$, using \Cref{mcwf2}. 

	\begin{lem} 
		
		\label{tmcf2} 
		
		For any real number $\varepsilon \in \big( 0, \frac{1}{2} \big]$, we have 
		\begin{flalign*} 
			t_{\mix} (2 \varepsilon; M_{\cwf}) \le 4 A^3 (\log \varepsilon^{-1})^2 \cdot \big( t_{\mix} (\varepsilon; M_{\fl}) + 1 \big).
		\end{flalign*} 
	\end{lem} 

	\begin{proof}
		
		Throughout this proof, we recall the interpretation of $M_{\cwf}$ as a lazy version of $M_{\fl}$ from \Cref{mcwf2}; we also recall the notation of that remark and set $T_0 = 4A^3 (\log \varepsilon^{-1})^2 \cdot \big( t_{\mix} (\varepsilon; M_{\fl}) \big) + 2 \le 4 A^3 (\log \varepsilon^{-1})^2 \cdot \big( t_{\mix} (\varepsilon; M_{\fl}) + 1 \big)$. By Chernoff's inequality, with probability at least $1 - \varepsilon$, the number of active steps in this walk after some time $T \ge T_0 \ge 4A^3 (\log \varepsilon^{-1})^2$ is at least 
		\begin{flalign*}
			\displaystyle\frac{|\mathsf{R} \setminus \partial \mathsf{R}|}{\mathfrak{m} \cdot \mathfrak{M}} \cdot T - T^{1/2} \log \varepsilon^{-1} \ge \displaystyle\frac{T}{\mathfrak{M}} - T^{1/2} \log \varepsilon \ge A^{-3/2} T - (\log \varepsilon^{-1}) T^{1/2} \ge \displaystyle\frac{T_0}{2A^{3/2}},
		\end{flalign*}
	
		\noindent where we used the fact that $\mathfrak{m} \le |\mathsf{R} \setminus \partial \mathsf{R}|$ (as $\mathfrak{m}(\mathsf{v}) \ge 1$ for each $\mathsf{v}\in \mathsf{R} \setminus \partial \mathsf{R}$) in the first bound; the fact that $\mathfrak{M} \le k \cdot |\mathsf{R} \setminus \mathsf{R}| \le (\diam \mathsf{R})^3 = A^{3/2}$ (as $\mathfrak{m} (\mathsf{v}) \le k \le \diam \mathsf{R}$ for each $\mathsf{V} \in \mathsf{R} \setminus \partial \mathsf{R}$) in the second; and the fact that $T \ge T_0 \ge 4 A^3 (\log \varepsilon)^2$ in the third. It follows for $T \ge T_0$ that the number of active steps in these dynamics is at least $t_{\mix} (\varepsilon; M_{\fl})$, with probability at least $1 - \varepsilon$. Conditioning on this event, we find for any $\mathsf{H}_0 \in \mathscr{G} (\mathsf{h})$ that we may couple $M_{\cwf}^{\lceil T_0 \rceil} \mathsf{H}_0$ to coincide with a uniformly random element $\mathsf{F} \in \mathscr{G} (\mathsf{h})$ with probability $1 - \varepsilon$. Hence, by a union bound, we may couple $M_{\cwf}^{\lceil T_0 \rceil} \mathsf{H}_0$ to coincide with $\mathsf{F}$ with probability $1 - 2 \varepsilon$, from which the lemma follows, as $T_0 \le 4A^3 (\log \varepsilon^{-1})^2 \cdot \big( t_{\mix} (\varepsilon; M_{\fl}) + 1 \big)$.
	\end{proof}
	
	We further require a censored version of the region-flip dynamics from \Cref{rh}.

	\begin{definition} 
		
		\label{mcrf} 
		
		The \emph{censored region-flip dynamics}, denoted by $M_{\crf}$, is the discrete-time Markov chain on $\mathscr{G} (\mathsf{h})$ whose state $\mathsf{H}_{t+1}$ at time $t+1$ is defined from $\mathsf{H}_t$ as follows. First, let $\mathcal{X} = (X_1, X_2, \ldots) \in \mathbb{Z}_{\geq 1}$ denote the sequence of integer-valued random variables defined by first setting $\mathbb{P} (X_1 = i) = p_i$ for each $i \in [1, k]$. Then, given $X_r$, we define $X_{r + 1} \in \{ X_r + 1, X_r + 2, \ldots , X_r + k \}$ by setting $\mathbb{P} (X_{r + 1} = X_r + i) = p_j$, where $j \in [1, k]$ is such that $k$ divides $X_r + i - j$. 
		
		Now, let $s \geq 1$ denote the integer such that $(s - 1) A^{5} < t+1 \leq s A^{5}$ and let $i' \in [1, k]$ denote the integer such that $k t_0 + i' = s$ for some $t_0 \in \mathbb{Z}_{\geq 0}$. If $s \in \mathcal{X}$ and $t = (s - 1) A^{5}$, then select an interior $\mathsf{v} \in \mathsf{R}_{I'} \setminus \partial \mathsf{R}_{I'}$ uniformly at random, and let $\mathsf{H}_{t + 1}$ denote the random flip of $\mathsf{H}_t$ at $\mathsf{v}$. If instead $s \notin \mathcal{X}$ or $t - (s - 1) A^{5} > 0$, then set $\mathsf{H}_{t + 1} = \mathsf{H}_t$. 
		
	\end{definition} 
	
	The process $M_{\crf}$ censors any step in the region-flip dynamics $M_{\rf}$ from \Cref{rh} in the time interval $\big[ (s - 1) A^{5} + 2, s A^{5} \big]$, and also the step at time $(s - 1) A^{5}$ if $s \notin \mathcal{X}$. 		
	
	Denote the  maximal and minimal configurations of $\mathscr{G} (\mathsf{h})$ by $\mathsf H^{\rm top}, \mathsf H^{\rm btm}\in \mathscr{G} (\mathsf{h})$, which for any $\mathsf H\in \mathscr{G} (\mathsf{h})$ satisfy $\mathsf H^{\rm top}(\mathsf v)\geq \mathsf H(\mathsf v)\geq \mathsf H^{\rm btm}(\mathsf v)$, for each $\mathsf v\in \mathsf R$. Further let $\delta_{\mathsf H^{\rm top}}, \delta_{\mathsf H^{\rm btm}} \in \mathscr{P} \big(\mathscr{G} (\mathsf{h}) \big)$ denote the delta masses at $\mathsf H^{\rm top}$ and $\mathsf H^{\rm btm}$ respectively. 
	The following result from \cite{EUM} shows that the above censorings (weakly) increase the mixing time for these dynamics when started from the top configuration or the bottom one.\footnote{This statement was only explicitly made in \cite[Therorem 1.1]{EUM} for starting the dynamics from the top configuration, but the fact that it also holds when started from the bottom one follows by symmetry.}

	\begin{prop}[{\cite[Theorem 1.1]{EUM}}]
		
		\label{piriestimate}
		
		Letting $\rho$ denote the uniform measure on $\mathscr{G} (\mathsf{h})$, we have for any integer $t \ge 0$ that 
		\begin{flalign*} 
			& d_{\TV} (M_{\wf}^t \delta_{\mathsf{H}^{\ttop}}, \rho) \le d_{\TV} (M_{\cwf}^t \delta_{\mathsf{H}^{\ttop}}, \rho); \qquad d_{\TV} (M_{\wf}^t \delta_{\mathsf{H}^{\btm}}, \rho) \le d_{\TV} (M_{\cwf}^t \delta_{\mathsf{H}^{\btm}}, \rho); \\
			& d_{\TV} (M_{\rf}^t \delta_{\mathsf{H}^{\ttop}}, \rho) \le d_{\TV} (M_{\crf}^t \delta_{\mathsf{H}^{\ttop}}, \rho); \qquad d_{\TV} (M_{\rf}^t \delta_{\mathsf{H}^{\btm}}, \rho) \le d_{\TV} (M_{\crf}^t \delta_{\mathsf{H}^{\btm}}, \rho).
		\end{flalign*} 

	\end{prop}

	Recalling the notation from \eqref{e:mix}, we also denote for any irreducible Markov chain $M : \mathscr{P} \big( \mathscr{G} (h) \big) \rightarrow \mathscr{P} \big(\mathscr{G} (h) \big)$ the quantities (where below $\rho$ denotes the stationary measure of $M$)
	\begin{flalign*}
		&R (\varepsilon; M) = \min \Big\{ t \in \mathbb{Z}_{\geq 0}:  d_{\TV} (M^t \delta_{\mathsf H^{\rm top}}, \rho) < \varepsilon \Big\}; \\
		&S (\varepsilon; M) = \min \Big\{ t \in \mathbb{Z}_{\geq 0}:  d_{\TV} (M^t \delta_{\mathsf H^{\rm btm}}, \rho) < \varepsilon \Big\},
	\end{flalign*}

	The following lemma bounds the mixing times of weighted flip and region-flip dynamics by the associated $R$ and $S$.

\begin{lem}
	
	\label{l:Mrfmixing}
	
	For any $\varepsilon \in (0, 1)$ we have 
\begin{align}
	\label{twfrf}
t_{\mix} (\varepsilon, M_{\wf}) \le R \Big( \displaystyle\frac{\varepsilon}{4A^2}; M_{\wf} \Big) + S \Big( \displaystyle\frac{\varepsilon}{4A^2}; M_{\wf} \Big); \quad t_{\mix} (\varepsilon, M_{\rf}) \le R \Big( \displaystyle\frac{\varepsilon}{4A^2}; M_{\rf} \Big) + S \Big( \displaystyle\frac{\varepsilon}{4A^2}; M_{\rf} \Big).
\end{align}
\end{lem}
\begin{proof}
	
	We only establish the second statement of \eqref{twfrf}, as the proof of the first is entirely analogous. First observe that there exists a grand coupling the region-flip dynamics $M_{\rf}$ over all choices of initial data in $\mathscr{G} (\mathsf{h})$, by running them under the same choices of sequences $\mathcal{X} = (X_1, X_2, \ldots )$ and vertices $\mathsf{v}$ (at which each flip is made) from \Cref{mcrf}. It is quickly verified (see \cite[Proposition 25.7]{RT}, for example) that this coupling is \emph{monotone}, meaning that if for some $\mathsf{H}_1, \mathsf{H}_2 \in \mathscr{G} (\mathsf{h})$ we have $\mathsf{H}_1 (\mathsf{v}) \le \mathsf{H}_2 (\mathsf{v})$ for each $\mathsf{v} \in \mathsf{R}$, then it holds that $M_{\rf}^t \mathsf{H}_1 (\mathsf{v}) \le M_{\rf}^t \mathsf{H}_2 (\mathsf{v})$ for each $t \ge 0$ and $v \in \mathsf{R}$. 
	
	Observe that it suffices to show under these coupled dynamics that, with probability at least $1 - \varepsilon$, the models started at $\mathsf{H}^{\ttop}$ and at $\mathsf{H}^{\btm}$ coincide after time $R \big( \frac{\varepsilon}{4A^2},  M_{\rf} \big) + S \big( \frac{\varepsilon}{4A^2}, M_{\rf} \big)$, that is, 
	\begin{flalign}
		\label{mrft2} 
		\mathbb{P} [ M_{\rf}^T \mathsf{H}^{\ttop} = M_{\rf}^T \mathsf{H}^{\btm}] \ge 1 - \varepsilon, \qquad \text{if $T \ge R \Big( \displaystyle\frac{\varepsilon}{4A^2}; M_{\rf} \Big) + S \Big( \displaystyle\frac{\varepsilon}{4A^2}; M_{\rf} \Big)$}.
	\end{flalign}

	\noindent Indeed, given \eqref{mrft2}, it follows since $\mathsf{H}^{\btm} \le \mathsf{F} \le \mathsf{H}^{\ttop}$ for each $\mathsf{F} \in \mathscr{G}(\mathsf{h})$ that with probability at least $1 - \varepsilon$ the $M_{\rf}^T \mathsf{F}$ all, over every $\mathsf{F} \in \mathscr{G} (\mathsf{h})$, coincide for $T \ge R \big( \frac{\varepsilon}{4A^2},  M_{\rf} \big) + S \big( \frac{\varepsilon}{4A^2}, M_{\rf} \big)$. In particular, sampling $\mathsf{F}$ under the stationary measure $\rho$ for $M_{\rf}$, we deduce for any $\mathsf{H} \in \mathscr{G}(\mathsf{h})$ that one can couple $M_{\rf}^T \mathsf{H}$ to coincide with a height function sampled under $\rho$, with probability $1 - \varepsilon$; hence, $t_{\mix} (\varepsilon; M_{\rf}) \le R \big( \frac{\varepsilon}{4A^2},  M_{\rf} \big) + S \big( \frac{\varepsilon}{4A^2}, M_{\rf} \big)$, confirming the lemma.
	
	It remains to verify \eqref{mrft2}. Since $T \ge R \big( \frac{\varepsilon}{4A^2}; M_{\rf} \big)$, we have by \eqref{d02} that it is possible to couple $M_{\rf}^T \mathsf{H}^{\ttop}$ with a height function $\mathsf{F}$ sampled under the stationary measure $\rho$ of $M_{\rf}$ such that $M_{\rf}^T \mathsf{H}^{\ttop} = \mathsf{F}$ with probability at least $1 - \frac{\varepsilon}{4A^2}$. Moreover, since $M_{\rf}^T \mathsf{H}^{\ttop} (\mathsf{u}) = \mathsf{h} (\mathsf{u}) = \mathsf{F}(\mathsf{u})$ for each $\mathsf{u} \in \partial \mathsf{R}$ and since $\diam \mathsf{R} \le A$, it follows (as $\mathsf{H}$ is $1$-Lipschitz) that $\big| M_{\rf}^T \mathsf{H}^{\ttop} (\mathsf{v}) - \mathsf{F} (\mathsf{v}) \big| \le 2A$. Combining these two statements, we deduce that $\mathbb{E} \big[ M_{\rf}^T \mathsf{H}^{\ttop} (\mathsf{v}) \big] \le \mathbb{E} \big[ \mathsf{F} (\mathsf{v}) \big] + \frac{\varepsilon}{2A}$, for each $\mathsf{v} \in \mathsf{R}$. Similarly, we have $\mathbb{E} \big[ M_{\rf}^T \mathsf{H}^{\btm} (\mathsf{v}) \big] \le \mathbb{E} \big[ \mathsf{F} (\mathsf{v}) \big] - \frac{\varepsilon}{2A}$. 

	Therefore, $\mathbb{E} \big[ M_{\rf}^T \mathsf{H}^{\ttop} (\mathsf{v}) \big] \le \mathbb{E} \big[ M_{\rf}^T \mathsf{H}^{\btm} (\mathsf{v})| \big] + \frac{\varepsilon}{A}$ for any $\mathsf{v} \in \mathsf{R}$. Together with the above grand coupling satisfying $M_{\rf}^T \mathsf{H}^{\btm} \le M_{\rf}^T \mathsf{H}^{\ttop}$; the fact that any height function is integer-valued; and a Markov inequality, we deduce that $\mathbb{P} \big[ M_{\rf}^T \mathsf{H}^{\ttop} (\mathsf{v}) \ne M_{\rf}^T \mathsf{H}^{\btm} (\mathsf{v}) \big] \le \frac{\varepsilon}{A}$, under this grand coupling. A union bound over all $|\mathsf{R}| \le A$ vertices $\mathsf{v} \in \mathsf{R}$ then yields \eqref{mrft2} and thus the lemma.
\end{proof}
	
	Next we have the following lemma that compares the mixing times of the flip and censored region-flip dynamics.

	\begin{lem}
		
		\label{mfmcrf} 
		
		Adopting the notation of \Cref{r1r2estimate}, and fixing a real number $\varepsilon \in \big( 0, \frac{1}{2} \big]$, we have $t_{\mix} (8A^2 \varepsilon; M_{\crf}) \leq 8A^9 (\log \varepsilon^{-1})^2 \cdot \big( t_{\mix} (\varepsilon; M_{\fl}) + 1 \big)$.
		
	\end{lem}
	
	\begin{proof}
		
		We first bound the mixing time of $M_{\crf}$ in terms of that of $M_{\wf}$. To that end, recall that the state $\mathsf{H}_t$ at time $t \geq 1$ under $M_{\wf}$ is defined from $\mathsf{H}_{t - 1}$ by performing a random flip at a vertex $\mathsf{v} \in \mathsf{R} \setminus \partial \mathsf{R}$ chosen with probability $\mathfrak{m} (\mathsf{v}) \cdot \mathfrak{M}^{-1}$. Observe that we equivalently sample $\mathsf{v}$ by first selecting an index $i \in [1, k]$ with probability $p_i$ and then selecting $\mathsf{v} \in \mathsf{R}_i \setminus \partial \mathsf{R}_i$ uniformly at random. Recalling the random sequence $\mathcal{X} = (X_1, X_2, \ldots )$ from \Cref{mcrf}, this is in turn equivalent to sampling $\mathsf{v} \in \mathsf{R}_{X_t} \setminus \partial \mathsf{R}_{X_t}$ uniformly at random, where we have denoted $\mathsf{R}_j = \mathsf{R}_i$ for $i \in [1, k]$ the integer such that $k$ divides $j - i$. 
		
		It follows that $M_{\fl}$ and $M_{\crf}$ can be coupled so that the former at time $t$ coincides with the latter at time $(X_t - 1) A^{5} \le k t A^{5}$, where in the last equality we used the fact that $X_t \leq kt$ (as $X_t - X_{t - 1} \leq k$). Hence, 
		\begin{flalign}
			\label{tcrf2} 
			t_{\mix} (8A^2 \varepsilon; M_{\crf}) \leq k A^{5} t_{\mix} (8A^2 \varepsilon; M_{\wf}) \le A^6 t_{\mix} (8 A^2 \varepsilon; M_{\wf}).
		\end{flalign}
	
		\noindent Moreover, we have 
		\begin{flalign*} 
			t_{\mix} (8 A^2 \varepsilon; M_{\wf}) \le R (2 \varepsilon; M_{\wf}) + S(2 \varepsilon; M_{\wf}) & \le 2 t_{\mix} (2\varepsilon; M_{\cwf}) \\
			& \le 8A^3 (\log \varepsilon^{-1})^2 \cdot \big( t_{\mix} (\varepsilon; M_{\fl}) + 1 \big),
		\end{flalign*} 
	
		\noindent where in the first inequality we applied \Cref{l:Mrfmixing}; in the second we applied \Cref{piriestimate}; and in the third we applied \Cref{tmcf2}. Combining this with \eqref{tcrf2} yields the lemma.
	\end{proof} 
	
	Given the above, we can quickly establish \Cref{rt3} and \Cref{rt2}.
	
	\begin{proof}[Proof of \Cref{rt3}]
		
		 By \Cref{l:Mrfmixing}, \Cref{piriestimate}, and \Cref{mfmcrf}, we have for sufficiently large $A$ that
		\begin{align*}
		t_{\mix} (e^{-2A}; M_{\rf}) 
		&\leq R \Big(\frac{e^{-2A}}{4A^2}, M_{\rf} \Big) + S \Big(\frac{e^{-2A}}{4A^2}, M_{\rf} \Big)\\
		&\leq R \Big(\frac{e^{-2A}}{4A^2}, M_{\crf} \Big) + S \Big(\frac{e^{-2A}}{4A^2}, M_{\crf} \Big)\\
		&\leq 2t_{\mix} \Big(\frac{e^{-2A}}{4A^2}, M_{\crf} \Big)\leq 16A^9 (3A)^2 \cdot \bigg( t_{\mix} \Big(\frac{e^{-2A}}{32A^4}; M_{\fl} \Big) + 1 \bigg),
		\end{align*}
	
		\noindent which yields the lemma.
	\end{proof} 
	
	\begin{proof}[Proof of \Cref{rt2}]
		
		First observe that for sufficiently large $A$ we have $t_{\mix} \big(\frac{e^{-2A}}{4A^2}; M_{\fl} \big) \leq \frac{A^5}{200}$, by \Cref{rt1}; thus, \Cref{rt3} implies that $t_{\mix} (e^{-2A}; M_{\rf}) \leq A^{16}$. So, to establish the lemma it suffices to couple the dynamics $M_{\alt}$ at time $t$ to coincide with $M_{\rf}$ at time $A^5 t$ for each $t \in [0, A^{11}]$, away from an event of probability at most $A^{11} e^{-2A} \leq e^{-A} - e^{-2A}$. 
		
		To that end, let $\mathsf{H}_t$ denote the state after $t \geq 0$ steps of the dynamics $M_{\rf}$. Furthermore, for any integer $s \geq 0$ and $i \in [1, k]$ such that $k$ divides $s - i + 1$, set $\mathsf{H}_s' = \mathsf{H}_{s A^5}$ and $\mathsf{h}_{s+1}' = \mathsf{H}_{s A^5} |_{\partial \mathsf{R}_i}$. Then, \Cref{rh} implies $\mathsf{H}_{s+1}' \in \mathscr{G} (\mathsf{h}_{s+1}')$ is obtained from $\mathsf{H}_s' \in \mathscr{G} (\mathsf{h}_s')$ from applying the flip dynamics $M_{\fl}$ on $\mathsf{R}_i$ for time $A^5$. Hence, since $t_{\mix} (e^{-2A}; M_{\fl})\leq t_{\mix} \big( \frac{e^{-2A}}{4A^2}; M_{\fl} \big) \leq A^5$, we may couple $\mathsf{H}_s' |_{\mathsf{R}_i}$ with a uniformly random element of $\mathscr{G} (\mathsf{h}_s')$, away from an event of probability at most $e^{-2A}$. 
		
		It follows that the sequence $\{ \mathsf{H}_0', \mathsf{H}_1', \ldots , \mathsf{H}_s' \}$ can be coupled with $s$ steps of the alternating dynamics $M_{\alt}$ with initial data $\mathsf{H}_0$, away from an event of probability at most $s e^{-2A}$. Taking $s = t_{\mix} (e^{-2A}; M_{\fl}) \leq A^{11}$ and recalling that $\mathsf{H}_s' = \mathsf{H}_{s A^5}$, we deduce the lemma. 
	\end{proof}

	\section{Tilted Height Functions and Comparison Estimates}
	
	\label{Profile1} 
	
	 In this section we discuss how height functions can be ``tilted'' in a specific way. Section \ref{HeightOmega} introduces the notion of a tilted height function and states results comparing tilted height functions to random tiling height functions; we prove the latter comparison results in \Cref{ProofHt1t2} and \Cref{ProofHt1t22}.
	 
	\subsection{Tilted Height Functions} 
	
	\label{HeightOmega}

	In this section we describe a way of ``tilting'' the height function of a random tiling that will enable us to apply \Cref{estimategamma} in an effective way. Throughout this section, we recall the notation from \Cref{p}, and more generally from \Cref{HeightLimit} and \Cref{Slopeft}; this includes the polygonal domain $\mathfrak{P}$ and associated boundary height function $h : \partial \mathfrak{P} \rightarrow \mathbb{R}$; the liquid region $\mathfrak{L} = \mathfrak{L} (\mathfrak{P})$ and arctic curve $\mathfrak{A} = \mathfrak{A} (\mathfrak{P})$ from \eqref{al}; the associated complex slope $f_t (x) = f(x, t)$ from \eqref{fh}; the polygonal domain $\mathsf{P} = n \overline{\mathfrak{P}} \cap \mathbb{T}^2$ and its associated boundary height function $\mathsf{h} : \mathsf{P} \rightarrow \mathbb{Z}$. For any $(x, s) \in \overline{\mathfrak{L}}$, we also set
	\begin{flalign}
		\label{omegasx} 
		\Omega_s (x) = \frac{1}{\pi} \Imaginary \displaystyle\frac{f_s (x)}{f_s (x) + 1}=\frac{1}{\pi}\frac{\Imaginary f_s (x)}{|f_s (x) + 1|^2}; \qquad \Upsilon_s (x) = \displaystyle\frac{f_s (x)}{\big( f_s (x) + 1 \big)^2},
	\end{flalign}

	\noindent and further set $\Omega_s (x) = 0$ when $(x, s) \notin \overline{\mathfrak{L}}$. The parameters $\Omega_s (x)$ and $\Upsilon_s (x)$ will quantitatively govern how height functions change under ``tilts,'' to be described further in \Cref{omegaxizeta} below.

	\begin{rem}
		\label{omegafsx}
		
		Observe that $\Omega_s (x) \leq 0$ since $f_s (x) \in \mathbb{H}^- \cup \mathbb{R}$. Moreover, if $(x, s) \in \mathfrak{A}$, then $f_s (x) \in \mathbb{R}$, so $\Upsilon_s (x) \in \mathbb{R}$. Then $\Upsilon_s (x) > 0$ holds if $f_s (x) > 0$, and $\Upsilon_s (x) < 0$ holds if $f_s (x) < 0$. The former implies $\arg^* f_s (x) = 0$, which by \eqref{fh} implies $\partial_x H^* (x, s) = 0$. Similarly, the latter implies $\partial_x H^* (x, s) = 1$.
		
	\end{rem}

	 Throughout this section, we fix real numbers $\mathfrak{t}_1 < \mathfrak{t}_2$ with $\mathfrak{t} = \mathfrak{t}_2 - \mathfrak{t}_1$; linear functions $\mathfrak{a}, \mathfrak{b} : [\mathfrak{t}_1, \mathfrak{t}_2] \rightarrow \mathbb{R}$, with slopes in $\{ 0, 1 \}$; and the domain $\mathfrak{D} = \mathfrak{D} (\mathfrak{a}, \mathfrak{b}; \mathfrak{t}_1, \mathfrak{t}_2)$ from \eqref{d} with boundaries \eqref{dboundary}, as in \Cref{EstimateDomain}. We view them all as independent from $n$. We will impose the following condition on $\mathfrak{D}$ concerning its relation to $\mathfrak{P}$ (see \Cref{dldomain} for possible depictions). 
	 
	 \begin{assumption}
	 	
	 \label{assumptiond}

	 Adopt the notation of \Cref{walkspconverge}, and suppose $\mathfrak{D} \subseteq \mathfrak{P}$, with $\mathsf{D} = n \mathfrak{D} \subseteq \mathbb{T}$. Assume that the second, third, and fourth constraints listed in \Cref{xhh} hold for $\mathfrak{D}$ (with respect to $H^*$). Further suppose that either $\partial_{\ea} (\mathfrak{D})$ is disjoint with $\overline{\mathfrak{L}}$ or that $\partial_{\ea} (\mathfrak{D}) \subset \partial \mathfrak{P}$; similarly, suppose that either $\partial_{\we} (\mathfrak{D})$ is disjoint with $\overline{\mathfrak{L}}$ or that $\partial_{\we} (\mathfrak{D}) \subset \partial \mathfrak{P}$. Additionally, fix a real number $\mathfrak{s} \in [\mathfrak{t}_1, \mathfrak{t}_2]$, and assume that no cusp or tangency location of $\mathfrak{A}$ in $\overline{\mathfrak{D}}$ is of the form $(x, y)$, with $x \in \mathbb{R}$ and $y \in \{ \mathfrak{t}_1, \mathfrak{s}, \mathfrak{t}_2 \}$.
	 
	 \end{assumption}

 	\begin{rem}
 	
 	\label{omegad}

 	Under \Cref{assumptiond}, $\Upsilon_t (x)$ is uniformly bounded away from $0$ and $\infty$, for any $(x, t) \in \mathfrak{A}$ with $t \in \{ \mathfrak{t}_1, \mathfrak{s}, \mathfrak{t}_2 \}$. This holds since $(f_t (x) + 1)/f_t (x)$ is the slope of the tangent line to $\mathfrak{A}$ at $(x, t)$ (by \Cref{lqq}), and since no tangency location of $\mathfrak{A}$ has $y$-coordinate in $\big\{ \mathfrak{t}_1, \mathfrak{s}, \mathfrak{t}_2 \}$. Moreover, if $u = (x, t) \in \mathfrak{L}$ satisfies $t \in \{ \mathfrak{t}_1, \mathfrak{s}, \mathfrak{t}_2 \}$, and we  denote $d = \dist (u, \mathfrak{A})$, then there exists a constant $c = c(\mathfrak{P}, \mathfrak{D}) > 0$ such that $c d^{1/2} \leq -\Omega_t (x) \leq c^{-1} d^{1/2}$. This follows from the square root decay of $\Imaginary f_t (x)$ around smooth points of $\mathfrak{A}$ (see \Cref{fdomain} below) and the fact that no cusp or tangency location of $\mathfrak{A}$ in $\overline{\mathfrak{D}}$ has $y$-coordinate in $\{ \mathfrak{t}_1, \mathfrak{s}, \mathfrak{t}_2 \}$.
 	
 	\end{rem}
 	
 	Next we state the following proposition, to be established in \Cref{WalkLimit} below, indicating how a height function can be ``tilted.'' Here, the parameters $\Omega_s (x)$ and $\Upsilon_s (x)$ from \eqref{omegasx} will govern how the height function and edge of the liquid region change under such a tilt, respectively. In what follows, all implicit constants (including notions of being ``sufficiently small'') will only depend on the parameters $\mathfrak{P}$, $\mathfrak{D}$, and $\varepsilon$ in the statement of the proposition. We also recall maximizers of $\mathcal{E}$ from \eqref{hmaximum}, and the liquid regions $\mathfrak{L} (\mathfrak{D}; h)$, $\mathfrak{L}_{\north} (\mathfrak{D}; h)$, and $\mathfrak{L}_{\so} (\mathfrak{D}; h)$ from \Cref{EstimateDomain}. 
	
	\begin{prop} 
		
	\label{omegaxizeta}
	 
	Fix $\varepsilon > 0$, and adopt \Cref{assumptiond}. Also let $\xi_1, \xi_2 \in \mathbb{R}$ be real numbers of the same sign (that is, $\xi_1 \xi_2 \geq 0$), with $|\xi_1|, |\xi_2|$ sufficiently small. Further assume that $ |\xi_2 - \xi_1| \geq \varepsilon \max \big\{ |\xi_1|, |\xi_2| \big\}$, and define the function $\omega: [\mathfrak{t}_1, \mathfrak{t}_2] \rightarrow \mathbb{R}$ interpolating $\xi_1$ and $\xi_2$:
	\begin{flalign}
		\label{omegay}
		\omega (y) = \xi_2 \displaystyle\frac{y - \mathfrak{t}_1}{\mathfrak{t}_2 - \mathfrak{t}_1} + \xi_1 \displaystyle\frac{\mathfrak{t}_2 - y}{\mathfrak{t}_2 - \mathfrak{t}_1}.
	\end{flalign}

	\noindent Then there exists a function $\widehat{h} : \partial \mathfrak{D} \rightarrow \mathbb{R}$ admitting an admissible extension to $\mathbb{R}^2$, such that the maximizer $\widehat{H}^* = \argmax_{F \in \Adm (\mathfrak{D}, \widehat{h})} \mathcal{E} (F)$ of $\mathcal{E}$ satisfies the following properties. In the below, we fix one of three real numbers $y \in \{ \mathfrak{t}_1, \mathfrak{s}, \mathfrak{t}_2 \}$, and we abbreviate $\widehat{\mathfrak{L}} = \widehat{\mathfrak{L}} (\mathfrak{D};  \widehat{h}) \cup \mathfrak{L}_{\north} (\mathfrak{D}; \widehat{h}) \cup \mathfrak{L}_{\so} (\mathfrak{D}; \widehat{h})$.

	\begin{enumerate}
		\item For any $(x, y) \in \mathfrak{L} \cap \widehat{\mathfrak{L}}$, we have  
		\begin{flalign}
			\label{hxyhxy1} 
			\big| \widehat{H}^* (x, y) - H^* (x, y) - \omega (y) \Omega_y (x) \big| = \mathcal{O} \big( |\xi_1|^{3/2} + |\xi_2|^{3/2} \big).
		\end{flalign}
	
		\item \label{lx} Suppose $\big\{ x: (x, y) \in \mathfrak{L} \big\}$ is a union of $k \geq 1$ disjoint open intervals $(x_1, x_1') \cup (x_2, x_2') \cup \cdots \cup (x_k, x_k')$. Then, $\big\{ x : (x, y) \in \widehat{\mathfrak{L}} \big\}$ is also a union of $k$ disjoint open intervals $(\widehat{x}_1, \widehat{x}_1') \cup (\widehat{x}_2, \widehat{x}_2') \cup \cdots \cup (\widehat{x}_k, \widehat{x}_k')$. Moreover, for any index $1\leq j\leq k$, we have
		\begin{flalign}
			\label{hxyhxy2}
			\widehat{x}_j - x_j = \omega (y) \Upsilon_y (x_j) + \mathcal{O} (\xi_1^2 + \xi_2^2); \qquad \widehat{x}_j' - x_j' = \omega (y) \Upsilon_y (x_j') + \mathcal{O} ( \xi_1^2 + \xi_2^2), 
		\end{flalign}
	
		\noindent and $\widehat{H}^* (x, y) = H^* (x, y)$ whenever $(x, y) \in \overline{\mathfrak{D}}$ and $(x, y) \notin \mathfrak{L} \cup \widehat{\mathfrak{L}}$. 
		
		\item Under the notation of Item \eqref{lx}, fix any endpoint $x \in \bigcup_{i = 1}^k \{ x_i, x_i' \}$; set $\widehat{x} = \widehat{x}_i$ or $\widehat{x} = \widehat{x}_i'$ if $x = x_i$ or $x = x_i'$, respectively. For any real number $\Delta$ with $|\Delta|$ sufficiently small, we have
		\begin{flalign}
			\label{hxdelta} 
			\widehat{H}^* (\widehat{x} + \Delta) - \widehat{H}^* (\widehat{x}) = H^* (x + \Delta) - H^* (x) + \mathcal{O} \big( (|\xi_1| + |\xi_2|) |\Delta|^{3/2} + \Delta^2 \big).
		\end{flalign}

		\item The domain $\mathfrak{D}$ satisfies five assumptions listed in \Cref{xhh}, with respect to $\widehat{H}^*$.

	\end{enumerate}

	\end{prop} 

		Let us briefly comment on \Cref{omegaxizeta}. We view $\omega (y)$ as quantifying how ``tilted'' $\widehat{H}^*$ is with respect to $H^*$ along a fixed horizontal slice $y \in \{ \mathfrak{t}_1, \mathfrak{s}, \mathfrak{t}_2 \}$. In particular, $\xi_1$ and $\xi_2$ parameterize this tiltedness along the north and south boundaries of $\mathfrak{D}$, respectively, and \eqref{omegay} implies that this tiltedness linearly interpolates between these two boundaries. The first part of \Cref{omegaxizeta} quantifies how the height function in the liquid region tilts in terms of $\Omega_s$, and the second part quantify how the edges of the liquid region tilt in terms of $\Upsilon_s$. The tilted function $\widehat{H}^*$ will eventually be obtained by perturbing solutions of the complex Burgers equation \eqref{ftx}, and these functions $\Omega_s$ and $\Upsilon_s$ can be viewed as derivatives arising from this procedure; see \Cref{ftalphaft1} and \Cref{x0fq} below. The third part of \Cref{omegaxizeta} states that the gradient around the edges does not change too much under the tilting. The fourth part verifies properties of $\widehat{H}^*$ that will enable us to later apply \Cref{estimategamma}.

\begin{rem}
We will use notation such as $H^*$ and $\widehat H^*$ for deterministic height functions (which will be maximizers of $\mathcal{E}$), and notation such as $H$ and $\widehat H$ for random height functions (which are associated with random tilings).
\end{rem}

		In view of \eqref{hxyhxy1} and \eqref{hxyhxy2}, we introduce the following more precise notion of tiltedness. It will be useful to define it through estimates, instead of the close approximations provided by \Cref{omegaxizeta}.
		
	\begin{definition} 
		
	\label{lambdamud} 
	
	Suppose $\mathfrak{D} \subseteq \mathfrak{P}$; fix real numbers $\xi, \mu, \zeta \geq 0$; and fix an admissible function $H \in \Adm (\mathfrak{D})$. For any $(x, s) \in \overline{\mathfrak{D}}$, we say that $H$ is \emph{$(\xi; \mu)$-tilted with respect to $H^*$ at $(x, s)$} if (recalling that $\Omega_s (x) \leq 0$ by \Cref{omegafsx}) we have
	\begin{flalign*}
		\big| H (x, s) - H_s^* (x, s) \big| \leq \mu - \xi \Omega_s (x),
	\end{flalign*}
	
	\noindent We also say that \emph{the edge of $H$ is $\zeta$-tilted with respect to $H^*$ at level $s$} if the following two conditions hold. Here, we set $u = (x, s)$, and we let $(x_0, s) \in \mathfrak{A}$ denote any point on $\mathfrak{A}$ with $|x - x_0|$ minimal, so that $\partial_x H^* (x_0, s) \in \{ 0, 1 \}$ (that is, $H^*$ is frozen at $(x_0, s)$).
		
	\begin{enumerate} 
		\item If $u = (x, s) \notin \mathfrak{L}$ and $|x_0 - x| \geq \zeta \big| \Upsilon_s (x_0) \big|$, then $H (u) = H^* (u)$.
		\item Assume\footnote{The power $\zeta^{8/9}$ is taken for convenience and should not be viewed as optimal in any way.} $|x_0 - x| \leq \zeta^{8/9}$. If $\partial_x H^* (x_0, s) = 0$ (so $\Upsilon_s (x_0) > 0$ by \Cref{omegafsx}), then
		\begin{flalign*}
			H^* \big(x - \zeta \Upsilon_s (x_0), s \big)\leq H (x, s) \leq H^* \big( x  + \zeta \Upsilon_s (x_0), s \big).
		\end{flalign*} 
	
		\noindent If instead $\partial_x H^* (x_0, s) = 1$ (so $\Upsilon_s (x_0) < 0$ by \Cref{omegafsx}), then
		\begin{flalign*}
			H^* \big(x - \zeta \Upsilon_s (x_0), s \big) + \zeta \Upsilon_s (x_0) \leq H (x, s) \leq H^* \big( x + \zeta \Upsilon_s (x_0),s \big) - \zeta \Upsilon_s (x_0).
		\end{flalign*}
	\end{enumerate} 
	
	 \end{definition}
 
 	We sometimes refer to the former notion described in \Cref{lambdamud} as a ``bulk'' form of tiltedness, and the latter as an ``edge'' form. The bulk form imposes a bound on $|H - H^*|$ of a similar form to Item \eqref{lx} in \Cref{omegaxizeta}. The edge form constitutes two parts. The first is an estimate for the edge ponts of $H$, of a similar form to \eqref{hxyhxy2}; the second bounds $|H - H^*|$ near these edge points (this is eventually related to \eqref{hxdelta}). We will often view the tiltedness parameters $\xi, \mu, \zeta$ as small (decaying as a negative power of $n$), even though this was not needed to formulate \Cref{lambdamud}.
 	
 	 To proceed, for any real number $\delta > 0$, we require the ``reduced'' version of the liquid region
	\begin{flalign}
		\label{ldelta} 
		\mathfrak{L}_-^{\delta} = \mathfrak{L}_-^{\delta} (\mathfrak{P}) = \big\{ u \in \mathfrak{L} : \dist ( u, \mathfrak{A}) > n^{\delta - 2/3} \big\}.
	\end{flalign}
	
	\noindent Given this notation, we will state two results concerning the tiltedness of a random height function on $\mathsf{D}$ along a middle horizontal slice, given its tiltedness on the north and south boundaries of $\mathsf{D}$. Let us introduce the following notation and assumption to set this context. 
	
	\begin{assumption}
	
	\label{assumptiondomega} 
	
	Adopt \Cref{assumptiond}; fix $\varepsilon, \varsigma, \delta \in ( 0, 1/50)$; and suppose that $\mathfrak{s} \in [ \mathfrak{t}_1 + \varepsilon \mathfrak{t}, \mathfrak{t}_2 - \varepsilon \mathfrak{t}]$. Let $\widetilde{\mathsf{h}}: \partial \mathsf{D} \rightarrow \mathbb{Z}$ denote a boundary height function that is constant along the east and west boundaries of $\mathsf{D}$; if $\partial_{\north} (\mathfrak{D})$ is packed with respect to $h$, then we further assume that $\widetilde{\mathsf{h}} = \mathsf{h}$ along $\partial_{\north} (\mathfrak{D})$. Let $\widetilde{\mathsf{H}} : \mathsf{D} \rightarrow \mathbb{Z}$ denote a uniformly random element of $\mathscr{G} (\widetilde{\mathsf{h}})$, and define $\widetilde{H} \in \Adm(\mathfrak{D})$ by setting $\widetilde{H} (u) = n^{-1} \widetilde{\mathsf{H}} (nu)$ for each $u \in \overline{\mathfrak{D}}$. Further let $\xi_1, \xi_2, \zeta_1, \zeta_2, \mu \in [0, n^{\delta - 2/3}]$ be real numbers satisfying the inequalities $\min \big\{ \xi_1, \xi_2, |\xi_1 - \xi_2| \big\} \geq \varsigma (\xi_1 + \xi_2)$ and $\min \big\{ \zeta_1, \zeta_2, |\zeta_1 - \zeta_2| \big\} \geq \varsigma (\zeta_1 + \zeta_2)$. Assume the following two statements hold for each $j \in \{ 1, 2 \}$. 

	\begin{enumerate} 
		\item At each $(x, \mathfrak{t}_j) \in \mathfrak{L}_-^{\delta}$, we have that $\widetilde{H}$ is $(\xi_j; \mu)$-tilted with respect to $H^*$.
		\item The edge of $\widetilde{H}$ is $\zeta_j$-tilted with respect to $H^*$ at level $\mathfrak{t}_j$. 
	\end{enumerate} 
	
	\end{assumption}

	Observe that the latter two points in the above assumption are more constraints on the deterministic boundary data $\widetilde{\mathsf{h}}$ (equivalently, $\widetilde{h}$) than on the random height function $\widetilde{H}$. Indeed, the restriction of $\widetilde{H}$ to levels $\mathfrak{t}_1$ and $\mathfrak{t}_2$ is fully determined by $\widetilde{h}$, since these levels constitute the south and north boundaries of $\mathfrak{D}$, respectively. \\

	Now we state the following two results to be established in \Cref{ProofHt1t2} and \Cref{ProofHt1t22} below. Qualtitatively, they both state that the tiltedness of $\widetilde{H}$ along the intermediate horizontal slice $y = \mathfrak{s}$ lies between its tiltedness along the top and bottom boundaries of $\mathfrak{D}$. The two statements differ in that \Cref{hhestimate1} addresses both the bulk and edge forms of tiltedness, but imposes that its tiltedness parameters $\zeta_1, \zeta_2 \gg n^{-2/3}$; \Cref{hhestimate2} only addresses the bulk form of tiltedness, but allows for smaller tiltedness parameters $\xi_1, \xi_2 \gg n^{-1}$.
		
	\begin{prop} 
		
	\label{hhestimate1} 
	
	Adopt \Cref{assumptiondomega}, and set
	\begin{flalign}
		\label{zeta1} 
		 \zeta = \displaystyle\max \bigg\{ \displaystyle\frac{\varepsilon}{2} \zeta_1 + \Big( 1 - \displaystyle\frac{\varepsilon}{2} \Big) \zeta_2, \Big( 1 - \displaystyle\frac{\varepsilon}{2} \Big) \zeta_1 + \displaystyle\frac{\varepsilon}{2} \zeta_2 \bigg\}.
	\end{flalign}
	
	\noindent Assume that $\mu = 0$, and that $\xi_j \leq \zeta_j \leq n^{\delta / 2 - 2/3}$ and $\zeta_j \geq n^{\delta / 100 - 2/3}$ for each $j \in \{ 1, 2 \}$. Then, the following two statements hold with overwhelming probability. 
	\begin{enumerate} 
		\item At each $(x, \mathfrak{s}) \in \mathfrak{L}_-^{\delta}$, we have that $\widetilde{H}$ is $(\zeta; 0)$-tilted with respect to $H^*$.
		\item The edge of $\widetilde{H}$ is $\zeta$-tilted with respect to $H^*$ at level $\mathfrak{s}$. 
	\end{enumerate} 
		
	\end{prop} 

	\begin{prop} 
		
	\label{hhestimate2} 
	
		Adopt \Cref{assumptiondomega}, and set 
		\begin{flalign*}
			\xi = \displaystyle\max \bigg\{ \displaystyle\frac{\varepsilon}{2} \xi_1 + \Big( 1 - \displaystyle\frac{\varepsilon}{2} \Big) \xi_2, \Big( 1 - \displaystyle\frac{\varepsilon}{2} \Big) \xi_1 + \displaystyle\frac{\varepsilon}{2} \xi_2 \bigg\}.
		\end{flalign*}
	
		\noindent Assume that $\xi_j \leq n^{- 2/3}$ and $\zeta_j \leq n^{\delta / 50 - 2/3}$ for each $j \in \{ 1, 2 \}$. Further fix $U_0 = (X_0, \mathfrak{s}) \in \mathfrak{L}_-^{\delta}$, and assume that $\xi_j \geq n^{\delta / 4 - 1} \dist (U_0, \mathfrak{A})^{-1/2}$ for each $j \in \{ 1, 2 \}$. Then, $\widetilde{H}$ is $(\xi; \mu)$-tilted with respect to $H^*$ at $U_0$, with overwhelming probability.

	\end{prop}

	\begin{rem}
		
		\label{northsouthd}
		
		By rotating $\mathfrak{D}$, \Cref{hhestimate1} and \Cref{hhestimate2} also hold if in \Cref{assumptiond} the constraints on $\partial_{\north} (\mathfrak{D})$ are instead imposed on $\partial_{\so} (\mathfrak{D})$.
		
	\end{rem}

	\subsection{Proof of \Cref{hhestimate1}} 
	
	\label{ProofHt1t2}
	
	In this section we establish \Cref{hhestimate1}. Throughout this section, we adopt the notation of that proposition. For any $u = (x, \mathfrak{s}) \in \mathfrak{D}$, with $(x_0, \mathfrak{s}) \in \mathfrak{A}$ denoting a point with $|x - x_0|$ minimal, it suffices to show with overwhelming probability that 
		\begin{flalign}
			\label{huzetaomega}
			\begin{aligned}
			H^* (u) + \zeta \Omega_{\mathfrak{s}} (x) \leq \widetilde{H} (u) \leq H^* (u) - \zeta \Omega_{\mathfrak{s}} (x), \qquad & \text{if $u \in \mathfrak{L}_-^{\delta}$}; \\
			\widetilde{H} (u) = H^* (u), \qquad & \text{if $u \notin \mathfrak{L}$ and $|x - x_0| \geq \zeta \big| \Upsilon_{\mathfrak{s}} (x_0) \big|$},
			\end{aligned} 
		\end{flalign} 
	
		\noindent and, if $|x - x_0| \leq \zeta^{8/9}$, that
		\begin{flalign}
			\label{hhxx0}
			\begin{aligned} 
			H^* \big( x - \zeta \Upsilon_{\mathfrak{s}} (x_0), \mathfrak{s}) \leq \widetilde{H} (u) \leq H^* \big( x + \zeta \Upsilon_{\mathfrak{s}} (x_0), \mathfrak{s} \big), \qquad \text{if $\partial_x H^* (x_0, \mathfrak{s}) = 0$}; \\
			H^* \big( x - \zeta \Upsilon_{\mathfrak{s}} (x_0), \mathfrak{s} \big) + \zeta \Upsilon_{\mathfrak{s}} (x_0) \leq \widetilde{H} (u) \leq H^* \big( x + \zeta \Upsilon_{\mathfrak{s}} (x_0), \mathfrak{s} \big) - \zeta \Upsilon_{\mathfrak{s}} (x_0), \qquad \text{if $\partial_x H^* (x_0, \mathfrak{s}) = 1$}.
			\end{aligned} 
		\end{flalign}
		
		\noindent We only establish the upper bounds in \eqref{huzetaomega} and \eqref{hhxx0}, as the proofs of the lower bounds are entirely analogous. In what follows, we will assume that $\zeta_1 \geq \zeta_2$, as the proof in the complementary case $\zeta_1 < \zeta_2$ is entirely analogous. Let us also fix a small real number $\theta \in ( 0, 1/50)$ (it suffices to take $\theta = 1/100$), and we define slightly larger versions of $\zeta_1, \zeta_2$ by 
		\begin{flalign}
		\label{zetajzetaj}
			\zeta_1' = (1 + \theta \varepsilon) \zeta_1; \qquad \zeta_2' = (1 + \theta \varepsilon) \zeta_2.
		\end{flalign}
	
		\noindent Throughout, we further set $\widetilde{h} = \widetilde{H} |_{\partial \mathfrak{D}}$ from \Cref{assumptiondomega}. 
	
		Before continuing, let us briefly outline how we will proceed. First, we use \Cref{omegaxizeta} to obtain a ``$(-\zeta_1', -\zeta_2')$-tilted'' boundary function $\widehat{h}: \partial \mathfrak{D} \rightarrow \mathbb{R}$, with associated maximizer $\widehat{H}^* \in \Adm (\mathfrak{D}; \widehat{h})$ of $\mathcal{E}$; properties of this tilting from \Cref{omegaxizeta} will imply $\widehat{h} \geq  \widetilde{h}$. Next, we consider a tiling of $\mathsf{D}$ whose (scaled) boundary height function is given by $\widehat{h}$. Appyling \Cref{estimategamma}, we will deduce that the (scaled) height function $\widehat{H}$ associated with this tiling is close to $\widehat{H}^*$. Together with the bound $\widehat{h}\geq \widetilde{h}$ and the monotonicity result \Cref{comparewalks}, this will esentially imply that $\widehat{H}^* \approx \widehat{H} \geq \widetilde{H}$. This, with the fact (implied by \Cref{omegaxizeta}) that $\widehat{H}^*$ is approximately $\zeta$-tilted with respect to $H^*$, will yield the upper bounds in \eqref{huzetaomega} and \eqref{hhxx0}.   
		
		Now let us implement this procedure in detail. Apply \Cref{omegaxizeta} with the $(\xi_1, \xi_2)$ there equal to $(-\zeta_1', -\zeta_2')$ here. This yields a function $\widehat{h} : \partial \mathfrak{D} \rightarrow \mathbb{R}$ and its associated maximizer $\widehat{H}^* \in \Adm (\mathfrak{D}; \widehat{h})$ of $\mathcal{E}$ satisfying the four properties listed there. Define its discretization $\widehat{\mathsf{h}} : \partial \mathsf{D} \rightarrow \mathbb{Z}$ of $\widehat{h}$ by setting $\widehat{\mathsf{h}} (nv) = \big\lfloor n \widehat{h} (v) \big\rfloor$, for each $v \in \partial \mathfrak{D}$. 
		
		\begin{lem} 
			
		\label{hh3} 
		
		For each $\mathsf{v} \in \partial \mathsf{D}$, we have that $\widetilde{\mathsf{h}} (\mathsf{v}) \leq \widehat{\mathsf{h}} (\mathsf{v})$.
		
		\end{lem}  
	
		\begin{proof} 
			
		Let us first verify the lemma when $v = n^{-1} \mathsf{v} \in \partial_{\ea} (\mathfrak{D}) \cup \partial_{\we} (\mathfrak{D})$. In this case, \Cref{assumptiond} implies that the four corners $\big\{ \big( \mathfrak{a} (\mathfrak{t}_1), \mathfrak{t}_1 \big), \big( \mathfrak{b} (\mathfrak{t}_1), \mathfrak{t}_1 \big), \big( \mathfrak{a} (\mathfrak{t}_2), \mathfrak{t}_2 \big), \big( \mathfrak{b} (\mathfrak{t}_2), \mathfrak{t}_2 \big) \big\}$ of $\mathfrak{D}$ are outside of $\overline{\mathfrak{L}}$; they are thus bounded away from $\overline{\mathfrak{L}}$ (recall we view $\mathfrak{D}$ and $\mathfrak{P}$ as fixed with respect to $n$). Since $\zeta_1, \zeta_2 \leq n^{\delta  -  2/ 3}$, the edge of $\widetilde{H}$ is $n^{\delta - 2/3} \ll 1$ tilted with respect to $H^*$ at levels $\mathfrak{t}_1$ and $\mathfrak{t}_2$. Hence $\widetilde{H} (v) = H^* (v)$ at these four corners, so $\widetilde{\mathsf{h}} (nv) = \widetilde{\mathsf{H}} (nv) = n H^* (v) = n h (v) = \mathsf{h} (n v)$ there. The second part of \Cref{omegaxizeta}, together with the fact that $\zeta_j \Upsilon_{\mathfrak{t}_j} (x) + \mathcal{O} (\zeta_1^2) = \mathcal{O} (n^{\delta -2 / 3})$, implies that the edge of $\widehat{H}^*$ is also $n^{\delta - 2/3} \ll 1$ tilted with respect to $H^*$. So, similar reasoning gives $\widehat{h} (v) = h(v)$ at these four corners, yielding $\widehat{\mathsf{h}} (nv) = \mathsf{h} (nv) = \widetilde{\mathsf{h}} (nv)$ there. Since $\widetilde{\mathsf{h}}$ is constant along the east and west boundaries of $\mathsf{D}$, it follows that $\widehat{\mathsf{h}} (\mathsf{v}) = \widetilde{\mathsf{h}} (\mathsf{v}) = \widetilde{\mathsf{h}} (\mathsf{v})$ there. 
		
		We next verify $\widetilde{\mathsf{h}} (\mathsf{v}) \leq \widehat{\mathsf{h}} (\mathsf{v})$ when $v = n^{-1} \mathsf{v} \in \partial_{\north} (\mathfrak{D}) \cup \partial_{\so} (\mathfrak{D})$. Set $v = (x, \mathfrak{t}_j)$ and let $v_0 = (x_0, \mathfrak{t}_j) \in \partial \mathfrak{A}$ denote a point with $|x - x_0|$ minimal. 
		
		First suppose $|x - x_0| \geq (\zeta_j')^{8/ 9}$ and $v \notin \overline{\mathfrak{L}}$. Then, by \Cref{omegad} and the bound $|\zeta_j| = \mathcal{O} (n^{\delta - 2/3})$, we have $|x - x_0| \geq \zeta_j' \big| \Upsilon_{\mathfrak{t}_j} (x_0) \big|$ for sufficiently large $n$. Since the edges of $\widetilde{H}$ and $\widehat{H}^*$ are $\zeta_j'$-tilted with respect to $H^*$ at level $\mathfrak{t}_j$ (as $\zeta_j \leq \zeta_j'$), we have $\widetilde{h} (v) = h(v) = \widehat{h} (v)$, so $\widehat{\mathsf{h}} (\mathsf{v}) = \widetilde{\mathsf{h}} (\mathsf{v})$. 
		
		Next, suppose that $|x - x_0| \leq (\zeta_j')^{8/9}$. Then, $v_0 = (x_0, \mathfrak{t}_j)$ is either a left or right endpoint of $\mathfrak{A}$, and $\partial_x H^* (v_0) \in \{ 0, 1 \}$. We assume in what follows that $v_0$ is a right endpoint of $\mathfrak{A}$ and that $\partial_x H^* (v_0) = 0$, as the alternative cases are entirely analogous. Define $\widehat{x}_0 \in \mathbb{R}$ such that $\widehat{v}_0 = (\widehat{x}_0, \mathfrak{t}_j)$ denotes an endpoint of $ \mathfrak{L}_{\so} (\mathfrak{D}; \widehat{h})$ or $ \mathfrak{L}_{\north} (\mathfrak{D}; \widehat{h})$ (depending on whether $j = 1$ or $j = 2$, respectively), such that $|\widehat{x}_0 - x_0|$ is minimal. By \eqref{hxyhxy2}, we have $\widehat{x}_0 - x_0 + \zeta_j' \Upsilon_{\mathfrak{t}_j} (x_0) = \mathcal{O} (\zeta_1^2 + \zeta_2^2)$ (recall that the $\xi_j$ there is $-\zeta_j$ here). In particular, since $\Upsilon_{\mathfrak{t}_j} (x_0)$ is positive (by \Cref{omegafsx}) and bounded away from $0$ (by \Cref{omegad}), we have $\widehat{x}_0 \leq x_0$ and $x_0 - \widehat{x}_0 = \mathcal{O} (\zeta_1 + \zeta_2)$. 
		
		If $x \geq \widehat{x}_0$, then $v = (x, \mathfrak{t}_j) \notin \widehat{\mathfrak{L}}$ (as $(\widehat{x}_0, \mathfrak{t}_j)$ must be a right endpoint of $\widehat{\mathfrak{L}}$). Since $x_0 \geq \widehat{x}_0$ and $\partial_x \widehat{H}^* (x', \mathfrak{t}_j) = \partial_x H^* (x_0, \mathfrak{t}_j) = 0$ for $x' > \widehat{x}_0$, this implies $\widehat{H}^* (x, \mathfrak{t}_j) = \widehat{H}^* (x_0, \mathfrak{t}_j) = H^* (x_0, \mathfrak{t}_j)$ (where the last equality follows from Item \eqref{lx} of \Cref{omegaxizeta}, since $(x_0, \mathfrak{t}_j) \notin \mathfrak{L} \cup \widehat{\mathfrak{L}}$). Hence, $\widehat{h} (v) = h(x_0, \mathfrak{t}_j)$. Since the edge of $\widetilde{H}$ is $\zeta_j$-tilted with respect to $H^*$ and $\zeta_j = \mathrm{o}(n^{1/2})$, we also have $\widetilde{h} (v) \leq \widetilde{h} (x + n^{-1/2}, \mathfrak{t}_j) = H^* (x + n^{-1/2}, \mathfrak{t}_j) = H^* (x_0, \mathfrak{t}_j) = h (x_0, \mathfrak{t}_j)$. Thus $\widetilde{h} (v) \leq h(v) = \widehat{h} (v)$, meaning $\widehat{\mathsf{h}} (\mathsf{v}) \geq \widetilde{\mathsf{h}} (\mathsf{v})$ if $x \geq \widehat{x}_0$. 
		
		If instead $x < \widehat{x}_0$, then
		\begin{flalign}
			\label{hxtj1}
			\begin{aligned}
			\widehat{H}^* (x, \mathfrak{t}_j) & \geq \widehat{H}^* (\widehat{x}_0, \mathfrak{t}_j) + H^* (x + x_0 - \widehat{x}_0, \mathfrak{t}_j) - H^* (x_0, \mathfrak{t}_j) + \mathcal{O} (\zeta_1^{16 / 9}) \\
			& = H^* (x + x_0 - \widehat{x}_0, \mathfrak{t}_j) + \mathcal{O} (\zeta_1^{16 / 9}) \geq H^* \big(x + \zeta_j' \Upsilon_{\mathfrak{t}_j} (x_0), \mathfrak{t}_j \big) + \mathcal{O} (\zeta_1^{16 / 9}).
			\end{aligned} 
		\end{flalign}
		
		\noindent Here the first statement follows from \eqref{hxdelta} and the fact that $|x - \widehat{x}_0| = \mathcal{O} \big( |x - x_0| + |x_0 - \widehat{x}_0| \big) = \mathcal{O} (\zeta_1^{8/9})$; the second from the fact that $\widehat{H}^* (\widehat{x}_0, \mathfrak{t}_j) = \widehat{H}^* (x_0, \mathfrak{t}_j) = H^* (x_0, \mathfrak{t}_j)$, which holds by the equality $\partial_x \widehat{H}^* (\widehat{x}_0, \mathfrak{t}_j) = \partial_x H^* (x_0, \mathfrak{t}_j) = 0$ (as $x_0 \geq \widehat{x}_0$) and the second part of \Cref{omegaxizeta}; and the third from the facts that $H^*$ is $1$-Lipschitz and that $\widehat{x}_0 - x_0 + \zeta_j' \Upsilon_{\mathfrak{t}_j} (x_0) = \mathcal{O} (\zeta_1^2)$, by \eqref{hxyhxy2}. 
		
		 Next, observe that $(x, \mathfrak{t}_j) \in \mathfrak{L} \cap \widehat{\mathfrak{L}}$, since $x < \widehat{x}_0 \leq x_0$ and $|x - \widehat{x}_0| = \mathcal{O} (\zeta_1^{8/9})$. By \Cref{derivativeh} concerning the square root decay of $\partial_x H^*$ around $\mathfrak{A}$, and the fact that $\partial_x H^* (x_0, t) = 0$, we therefore deduce the existence of a constant $c = c(\mathfrak{P}, \mathfrak{D}, \theta) > 0$ such that
		\begin{flalign}
			\label{hxzeta}
			H^* \big(x + \zeta_j' \Upsilon_{\mathfrak{t}_j} (x_0), \mathfrak{t}_j \big) \geq H^* \big( x_0 + \zeta_j \Upsilon_{\mathfrak{t}_j} (x_0), \mathfrak{t}_j \big) + c \zeta_j^{3/2},
		\end{flalign}
	
		\noindent where we have used the fact \eqref{zetajzetaj} that $\zeta_j' - \zeta_j = \theta \zeta_j$ (as well as the fact that $\Upsilon_{\mathfrak{t}_j} (x_0)$ is bounded away from $0$, from \Cref{omegad}). Inserting \eqref{hxzeta} into \eqref{hxtj1} yields
		\begin{flalign*} 
			\widehat{H}^* (x, \mathfrak{t}_j) & \geq H^* \big( x + \zeta_j \Upsilon_{\mathfrak{t}_j} (x_0), \mathfrak{t}_j \big) + c \zeta_j^{3/2} + \mathcal{O} (\zeta_1^{16 / 9}) \geq \widetilde{H} (x, \mathfrak{t}_j) + c \zeta_j^{3/2} + \mathcal{O} (n^{-1})  \geq \widetilde{H} (x, \mathfrak{t}_j),
		\end{flalign*}
	
		\noindent where the second inequality holds since $|\zeta_1| \leq n^{\delta - 2/3}$, and the third follows from the facts that the edge of $\widetilde{H}$ is $\zeta_j$-tilted and that $\zeta_j \leq n^{\delta / 100 - 2/3}$. Hence $\widehat{h} (v) \geq \widetilde{h} (v)$, meaning $\widehat{\mathsf{h}} (\mathsf{v}) \geq \widetilde{\mathsf{h}} (\mathsf{v})$.
		
		It thus remains to consider the case when $|x - x_0| \geq (\zeta_j')^{8/9}$ and $v = (x, \mathfrak{t}_j) \in \mathfrak{L}$. In particular, $|x - x_0| \geq \zeta_j' \Upsilon_{\mathfrak{t}_j} (x_0) + \mathcal{O} (\zeta_j^2)$, which implies by \eqref{hxyhxy2} that $v = (x, \mathfrak{t}_j) \in \mathfrak{L} \cap \widehat{\mathfrak{L}}$. Thus, \eqref{hxyhxy1} yields
		\begin{flalign*}
			\widehat{H}^* (v) & \geq H^* (v) - \zeta_j' \Omega_{\mathfrak{t}_j} (x) + \mathcal{O} \big( (\zeta_1' + \zeta_2')^{3/2} \big) \geq \widetilde{H} (v) - \theta \zeta_j \Omega_{\mathfrak{t}_j} (x) + \mathcal{O} (\zeta_1^{3/2}),
		\end{flalign*}
	
		\noindent where to deduce the last inequality we used the facts that $\zeta_j' = (1 + \theta) \zeta_j$ and that $\widetilde{H}$ is $(\zeta_j; 0)$-tilted at $v$ (since $v \in \mathfrak{L}_-^{\delta}$, as $|x - x_0| \geq (\zeta_j')^{8 / 9} \geq n^{\delta - 2/3}$ and $\xi_j \leq \zeta_j$). By \Cref{omegad}, there exists a constant $c = c(\mathfrak{P}) > 0$ such that $-\Omega_{\mathfrak{t}_j} (x) \geq c  |x - x_0|^{1/2}$. In particular, since $|x - x_0| \geq (\zeta_j')^{8 / 9} \geq n^{\delta} \zeta_1$ (as $\zeta_j \leq n^{\delta / 2 - 2 / 3}$ and $\delta < 1/50$), we deduce that $-\Omega_{\mathfrak{t}_j} (x) \geq n^{\delta / 3} \zeta_1^{1/2}$. So,  
		\begin{flalign*}
			\widehat{H}^* (v) \geq \widetilde{H} (v) + c \theta \zeta_j |x - x_0|^{1/2} + \mathcal{O} (\zeta_1^{3/2}) \geq \widetilde{H} (u),
		\end{flalign*}
	
		\noindent which once again implies that $\widehat{h} (v) \geq \widetilde{h} (v)$, so that $\widehat{\mathsf{h}} (\mathsf{v}) \geq \widetilde{\mathsf{h}} (\mathsf{v})$. This verifies the lemma in all cases. 
		\end{proof} 
	
		Given this lemma, we can establish \Cref{hhestimate1}.

		\begin{proof}[Proof of \Cref{hhestimate1}] 
		
		Let $\widehat{\mathsf{H}}: \mathsf{D} \rightarrow \mathbb{Z}$ denote a uniformly random element of $\mathscr{G} (\widehat{\mathsf{h}})$. By \Cref{hh3} and \Cref{comparewalks} (alternatively, \Cref{heightcompare}), we may couple $\widehat{\mathsf{H}}$ with $\widetilde{\mathsf{H}}$ such that $\widehat{\mathsf{H}}(\mathsf{u}) \geq \widetilde{\mathsf{H}} (\mathsf{u})$, for each $\mathsf{u} \in \mathsf{D}$. In particular, denoting $\widehat{H} : \overline{\mathfrak{D}} \rightarrow \mathbb{R}$ by $\widehat{H} (u) = n^{-1} \widehat{\mathsf{H}} (nu)$ for each $u \in \overline{\mathfrak{D}}$, we have $\widehat{H} (u) \geq \widetilde{H} (u)$.
		
		Apply \Cref{estimategamma}, with the $(h; H^*; \mathsf{H}; \delta)$ there equal to $\big( \widehat{h}; \widehat{H}^*; \widehat{\mathsf{H}}; \delta/500 \big)$ here. By the fourth part of \Cref{omegaxizeta}, \Cref{xhh} applies. Moreover, \Cref{xhh2} applies, since $\widehat{\mathsf{h}} (nv) = \big\lfloor n \widehat h(v) \big\rfloor$ for each $v \in \partial \mathfrak{D}$. Then, letting $\mathscr{E}$ denote the event on which 
		\begin{flalign}
			\label{hh2} 
			\begin{aligned}
			\big| \widehat{H} (u) - \widehat{H}^* (u) \big| < n^{\delta / 500 - 1}, \qquad & \text{for each $u \in \overline{\mathfrak{D}}$}; \\
			\widehat{H} (u) = \widehat{H}^* (u), \qquad & \text{if $u \notin \widehat{\mathfrak{L}}$ and $\dist (u, \partial \widehat{\mathfrak{L}}) \geq n^{\delta / 500 - 2/3}$},
			\end{aligned}
		\end{flalign} 
	
		\noindent \Cref{estimategamma} implies that $\mathscr{E}$ holds with overwhelming probability. In what follows, let us fix $u = (x, \mathfrak{s}) \in \overline{\mathfrak{D}}$, and let $u_0 = (x_0, \mathfrak{s}) \in \mathfrak{A}$ denote a point with $|x - x_0|$ minimal. We will show that the upper bounds in \eqref{huzetaomega} and \eqref{hhxx0} hold on $\mathscr{E}$.
		
		First assume that $u \in \mathfrak{L}_-^{\delta}$, in which case $u \in \mathfrak{L}$ and $|x - x_0| \geq n^{\delta - 2/3} \geq (\zeta_1 + \zeta_2) \big| \Upsilon_{\mathfrak{s}} (x) \big|$. Hence, \eqref{hxyhxy2} implies that $u \in \widehat{\mathfrak{L}}$, and so \eqref{hxyhxy1} applies and gives
		\begin{flalign}
			\label{hx1s} 
			\widehat{H} (x, \mathfrak{s}) \leq \widehat{H}^* (x, \mathfrak{s}) + n^{\delta / 500 - 1} \leq H^* (x, \mathfrak{s}) - \omega (\mathfrak{s}) \Omega_{\mathfrak{s}} (x) + n^{\delta / 500 - 1} + \mathcal{O} (\zeta_1^{3/2}),
		\end{flalign}
	
		\noindent where $\omega: [\mathfrak{t}_1, \mathfrak{t}_2] \rightarrow \mathbb{R}$ is given by
		\begin{flalign*} 
			\omega (y) = \zeta_2' \displaystyle\frac{y - \mathfrak{t}_1}{\mathfrak{t}_2 - \mathfrak{t}_1} + \zeta_1' \displaystyle\frac{\mathfrak{t}_2 - y}{\mathfrak{t}_2 - \mathfrak{t}_1}. 
		\end{flalign*} 
		
		\noindent In particular, since $(1 - \varepsilon) \mathfrak{t}_1 + \varepsilon \mathfrak{t}_2 = \mathfrak{t}_1 + \varepsilon \mathfrak{t} \leq \mathfrak{s} \leq \mathfrak{t}_2 - \varepsilon \mathfrak{t} = \varepsilon \mathfrak{t}_1 + (1 - \varepsilon) \mathfrak{t}_2$, we have (recalling the definition \eqref{zeta1} of $\zeta$, as well as the bounds $\zeta_1 \geq \zeta_2$,  $\zeta_1 - \zeta_2 \geq \varsigma (\zeta_1 + \zeta_2)$  and $\theta < 1/50$) that 
		\begin{flalign}
			\label{omegas} 
			\omega (\mathfrak{s}) \leq (1 - \varepsilon) \zeta_1' + \varepsilon \zeta_2' \leq (1 - 3 \theta \varepsilon) \zeta.
		\end{flalign}
		
		\noindent Together, \eqref{hx1s} and \eqref{omegas} (with the fact that $n^{\delta / 500 - 1} \leq \zeta_1^{3/2}$) yield on $\mathscr{E}$ that
		\begin{flalign*}
			\widetilde{H} (u) \leq \widehat{H} (x, \mathfrak{s}) \leq H^* (x, \mathfrak{s}) - (1 - 3 \theta \varepsilon) \zeta \Omega_{\mathfrak{s}} (x) + \mathcal{O} (\zeta_1^{3/2}) \leq H^* (x, \mathfrak{s}) - \zeta \Omega_{\mathfrak{s}} (x) = H^* (u) - \zeta \Omega_{\mathfrak{s}} (x),
		\end{flalign*}
	
		\noindent where the third inequality follows from the fact that \Cref{omegad} implies $-\Omega_{\mathfrak{s}} (x) \geq |x - x_0|^{1/2}$ and $|x - x_0| \geq n^{\delta - 2/3} \geq n^{\delta / 4} (\zeta_1 + \zeta_2)$ (as $\zeta_1, \zeta_2 \leq n^{\delta / 2 - 2/3}$). This verifies the upper bound in the first statement of \eqref{huzetaomega}. 
		
		 To verify the second, assume that  $u \notin \mathfrak{L}$ and $|x - x_0| \geq \zeta \big| \Upsilon_{\mathfrak{s}} (x) \big|$. Since $\zeta_1, \zeta_2 \geq n^{\delta / 100 - 2/3}$, we deduce $\dist (u, \mathfrak{A}) \gg n^{\delta / 500 - 2/3}$; by \eqref{hh2} this implies $\widehat{H} (u) = \widehat{H}^* (u)$ on the event $\mathscr{E}$. Additionally, the second of \Cref{omegaxizeta} implies that $\widehat{H}^* (u) = H^* (u)$ if $|x - x_0| \geq \omega (\mathfrak{s}) \big| \Upsilon_{\mathfrak{s}} (x_0) \big| + \mathcal{O} (\zeta_1^2)$. Since \eqref{omegas} yields $\omega (\mathfrak{s}) \leq (1 - 3 \theta \varepsilon) \zeta$, we obtain $|x - x_0| \geq \zeta \big| \Upsilon_{\mathfrak{s}} (x_0) \big| \geq \omega (\mathfrak{s}) \big| \Upsilon_{\mathfrak{s}} (x_0) \big| + \mathcal{O} (\zeta_1^2)$, and so this condition is satisfied. Hence, on $\mathscr{E}$, we have $\widetilde{H} (u) \leq \widehat{H} (u) = H^* (u)$, and so the upper bound in the second statement of \eqref{huzetaomega} holds. 
		 
		 It thus remains to assume $|x - x_0| \leq \zeta^{8 / 9}$ and verify that the upper bound in \eqref{hhxx0} holds on $\mathscr{E}$. To that end, we assume in what follows that $(x_0, \mathfrak{s})$ is a right endpoint of $\mathfrak{A}$ and that $\partial_x H^* (x_0, \mathfrak{s}) = 0$, as the proofs in all other cases are entirely analogous. Then, let $(\widehat{x}_0, \mathfrak{s}) \in \partial \widehat{\mathfrak{L}}$ denote the point such that $|x - \widehat{x}_0|$ is minimal. By \eqref{hxyhxy2}, we have $x_0 - \widehat{x}_0 = \omega (\mathfrak{s}) \Upsilon_{\mathfrak{s}} (x_0) + \mathcal{O} (\zeta_1^2)$. In particular, since $\Upsilon_{\mathfrak{s}} (x_0) > 0$ (by \Cref{omegad}), it follows that $x_0 \geq \widehat{x}_0$ and $x_0 - \widehat{x}_0 = \mathcal{O} (\zeta)$.  
		 
		 Let us first assume that $x \geq \widehat{x}_0 - \theta \varepsilon \zeta \Upsilon_{\mathfrak{s}} (x_0)$. Then, since $\partial_x \widehat{H}^* (x', \mathfrak{s}) = \partial_x H^* (x_0, \mathfrak{s})$ for $x' \geq \widehat{x}_0$, we have 
		 \begin{flalign}
		 	\label{hxs} 
		 	\widehat{H}^* (x, \mathfrak{s}) \leq \widehat{H}^* (\widehat{x}_0, \mathfrak{s}) \leq \widehat{H}^* (x_0, \mathfrak{s}) = H^* (x_0, \mathfrak{s}) \leq H^* \big( x + \zeta \Upsilon_{\mathfrak{s}} (x_0), \mathfrak{s} \big). 
		 \end{flalign} 
	 
	 	\noindent Here, the first inequality follows from the fact that either $x \leq \widehat{x}_0$ (in which case $\widehat{H}^* (x, \mathfrak{s}) \leq \widehat{H}^* (\widehat{x}_0, \mathfrak{s})$) or $x > \widehat{x}_0$ (in which case $\widehat{H}^* (x, \mathfrak{s}) = \widehat{H}^* (\widehat{x}_0, \mathfrak{s})$); the second inequality holds since $x_0 \geq \widehat{x}_0$. The equality follows from the second statement of \eqref{omegaxizeta}, and the last inequality follows from the fact that 
	 	\begin{flalign*} 
	 		x + \zeta \Upsilon_{\mathfrak{s}} (x_0) \geq \widehat{x}_0 + (1 - \theta \varepsilon) \zeta \Upsilon_{\mathfrak{s}} (x_0) \geq \widehat{x}_0 + \omega (\mathfrak{s}) \Upsilon_{\mathfrak{s}} (x_0) + \mathcal{O} (\zeta_1^2) \geq x_0,
	 	\end{flalign*} 
 	
 		\noindent where we have used \eqref{hxyhxy2}.  Additionally, \eqref{hh2} implies that $\widehat{H} (x, \mathfrak{s}) = \widehat{H}^* (x, \mathfrak{s})$, since $|x - \widehat{x}_0| \gg \zeta^{1 + \delta / 500} \gg n^{\delta / 500 - 2/3}$. Together with \eqref{hxs}, this implies that on $\mathscr{E}$ we have $\widetilde{H} (u) \leq \widehat{H} (x, \mathfrak{s}) = \widehat{H}^* (x, \mathfrak{s}) = H^* \big( x + \zeta \Upsilon_{\mathfrak{s}} (x_0), \mathfrak{s} \big)$, thereby verifying the upper bound in \eqref{hhxx0}.
		 
		 So, let us instead assume that $x < \widehat{x}_0 - \theta \varepsilon \zeta \Upsilon_{\mathfrak{s}} (x_0)$. Then the bound $x_0 - \widehat{x}_0 = \omega (\mathfrak{s}) \Upsilon_{\mathfrak{s}} (x_0) + \mathcal{O} (\zeta^2)$, together with \eqref{hxdelta} and the fact that $H^*$ is $1$-Lipschitz, implies
		\begin{flalign}
			\label{hu1} 
			\begin{aligned} 
			\widehat{H}^* (u) & \leq \widehat{H}^* (\widehat{x}_0, \mathfrak{s}) + H^* (x + x_0 - \widehat{x}_0, \mathfrak{s}) -  H^* (x_0, \mathfrak{s}) + \mathcal{O} (\zeta_1^2) \\
			& \leq H^* \big( x + \omega (\mathfrak{s}) \Upsilon_{\mathfrak{s}} (x_0), \mathfrak{s} \big) + \widehat{H}^* (\widehat{x}_0, \mathfrak{s}) - H^* (x_0, \mathfrak{s}) + \mathcal{O} (\zeta^2) \\
			& = H^* \big( x + \omega (\mathfrak{s}) \Upsilon_{\mathfrak{s}} (x_0), \mathfrak{s} \big) + \mathcal{O} (\zeta^2).
			\end{aligned} 
		\end{flalign} 
	
		\noindent Here, to deduce the last equality we used the fact that $\widehat{H}^* (\widehat{x}_0, \mathfrak{s}) = \widehat{H}^* (x_0, \mathfrak{s}) = H^* (x_0, \mathfrak{s})$, which holds since $x_0 \geq \widehat{x}_0$, since $\partial_x \widehat{H}^* (x', \mathfrak{s}) = \partial_x H^* (x_0, \mathfrak{s}) = 0$ for $x' \geq \widehat{x}_0$, and by the second statement of \Cref{omegaxizeta}. Next, since $\omega (\mathfrak{s}) \leq (1 - 3 \varepsilon \theta) \zeta$ (by \eqref{omegas}), the square root decay of $\partial H^*$ around $\mathfrak{A}$ (see \Cref{derivativeh}) yields a constant $c = c (\mathfrak{P}, \mathfrak{D}, \varepsilon, \theta) > 0$ such that
		\begin{flalign*}
			H^* \big( x + \omega (\mathfrak{s}) \Upsilon_{\mathfrak{s}} (x_0), \mathfrak{s} \big) \leq  H^* \big( x + \zeta \Upsilon_{\mathfrak{s}} (x_0), \mathfrak{s} \big) - c \zeta^{3/2}.
		\end{flalign*}
	
		\noindent This, together with \eqref{hh2} and \eqref{hu1}, implies on $\mathscr{E}$ that
		\begin{flalign*}
			\widetilde{H}(u) \leq \widehat{H} (u) \leq \widehat{H}^* (u) + n^{\delta / 500 - 1} & \leq H^* \big( x + \zeta \Upsilon_{\mathfrak{s}} (x_0), \mathfrak{s} \big) - c \zeta^{3/2} + n^{\delta / 500 - 1} + \mathcal{O} (\zeta^2) \\
			& \leq H^* \big( x_0 + \zeta \Upsilon_{\mathfrak{s}} (x_0), \mathfrak{s} \big),
		\end{flalign*}
		
		\noindent where for the last bound we used the fact that $\zeta \geq n^{\delta / 100 - 2/3}$. Thus, the upper bound in \eqref{hhxx0} holds in $\mathscr{E}$. As mentioned earlier, the proofs of all lower bounds are entirely analogous and therefore omitted; this establishes the proposition.
		\end{proof}

		\subsection{Proof of \Cref{hhestimate2}}
	
		\label{ProofHt1t22} 
		
		In this section we establish \Cref{hhestimate2}.
		
		\begin{proof}[Proof of \Cref{hhestimate2} (Outline)]

		Since the proof of this proposition is similar to that of \Cref{hhestimate1}, we only outline it. It suffices to show that, with overwhelming probability, we have
		\begin{flalign}
			\label{omegaomegahh}
			 H^* (X_0, \mathfrak{s}) + \xi \Omega_{\mathfrak{s}} (X_0) - \mu \leq H(X_0, \mathfrak{s}) \leq  H^* (X_0, \mathfrak{s}) - \xi \Omega_{\mathfrak{s}} (X_0) + \mu.
		\end{flalign}
	
		\noindent We only establish the upper bound in \eqref{omegaomegahh}, as the proof of the lower bound is entirely analogous. Throughout, we set $\widetilde{h} = \widetilde{H} |_{\partial \mathfrak{D}}$ from \Cref{assumptiondomega}. 
		
		To do this, we apply \Cref{omegaxizeta} with the $(\xi_1, \xi_2)$ there equal to $(-\xi_1, -\xi_2)$ here. This yields a function $\widehat{h} : \partial \mathfrak{D} \rightarrow \mathbb{R}$ and its associated maximizer $\widehat{H}^* \in \Adm (\mathfrak{D}; \widehat{h})$ of $\mathcal{E}$ satisfying the four properties listed there. Define $\breve{h} : \partial \mathfrak{D} \rightarrow \mathbb{R}$; the associated maximizer $\breve{H}^* \in \Adm (\mathfrak{D}; \breve{h})$ of $\mathcal{E}$; and the discretization $\breve{\mathsf{h}} : \partial \mathsf{D} \rightarrow \mathbb{Z}$ of $\breve{h}$ by setting
		\begin{flalign}
			\label{hh1}
			\breve{h} (v) = \widehat{h} (v) + \mu + n^{\delta / 20 - 1}, \quad \text{so that} \quad  \breve{H}^* (u) = \widehat{H}^* (u) + \mu + n^{\delta / 20 - 1}, \quad \text{and} \quad \breve{\mathsf{h}} (nv) = \big\lfloor n \breve{h} (v) \big\rfloor,
		\end{flalign}
		
		\noindent for each $u \in \overline{\mathfrak{D}}$ and $v \in \partial \mathfrak{D}$. We claim that $\widetilde{\mathsf{h}} (\mathsf{v}) \leq \breve{\mathsf{h}} (\mathsf{v})$, for each $\mathsf{v} \in \partial \mathsf{D}$. 
			
		The proof that this holds when $v = n^{-1} \mathsf{v} \in \partial_{\ea} (\mathfrak{D}) \cup \partial_{\we} (\mathfrak{D})$ is very similar to that in the proof of \Cref{hh3}, so it is omitted. Thus, suppose that $v = (x, \mathfrak{t}_j) \in \partial_{\north} (\mathfrak{D}) \cup \partial_{\so} (\mathfrak{D})$. If $v \notin \overline{\mathfrak{L}}$ and $\dist (v, \mathfrak{A}) \geq \zeta_j^{8/9}$, then the proof is again entirely analogous to that in the proof of \Cref{hh3}. 
		
		So, let us first assume that $v = (x, \mathfrak{t}_j) \in \mathfrak{L}_-^{\delta}$. Since $\widetilde{H}$ is $(\xi_j; \mu)$-tilted at $v$, we have
		\begin{flalign*}
			\widetilde{H} (v) \leq H^* (v) - \xi_j \Omega_{\mathfrak{t}_j} (x) + \mu.
		\end{flalign*} 
	
		\noindent This, together with \eqref{hxyhxy1}, \eqref{hh1}, and the bound $\xi_j \leq n^{-2/3}$  gives
		\begin{flalign}
			\label{h2}  
			\begin{aligned}
			\breve{H}^* (v) = \widehat{H}^* (v) + \mu + n^{\delta / 20 - 1} & \geq H^* (v) - \xi_j \Omega_{\mathfrak{t}_j} (x) + \mu + n^{\delta / 20 - 1} + \mathcal{O} (\xi_j^{3/2}) \\
			& \geq \widetilde{H} (v) + n^{\delta / 20 - 1} + \mathcal{O} (\xi_j^{3/2}) \geq \widetilde{H} (v) + n^{- 1}.
			\end{aligned}
		\end{flalign}
		
		\noindent Hence, $n \breve{h} (v) \geq n \widetilde{h} (v) + 1$, and so  $\breve{\mathsf{h}} (\mathsf{v}) \geq \widetilde{\mathsf{h}} (\mathsf{v})$ whenever $v = n^{-1} \mathsf{v} \in \mathfrak{L}_-^{\delta}$. 
		
		Thus, assume instead $v \notin \mathfrak{L}_-^{\delta}$ and $\dist (v, \mathfrak{A}) \leq \zeta_j^{8/9}$, and let $v_0 = (x_0, \mathfrak{t}_j) \in \mathfrak{A}$ be such that $|x - x_0|$ is minimal. We suppose that $v_0$ is a right endpoint of $\mathfrak{A}$ and that $\partial_x H^* (v_0) = 0$, as the alternative cases are entirely analogous; then, $\Upsilon_{\mathfrak{t}_j} (x_0) > 0$ by \Cref{omegafsx}. The fact that the edge of $\widetilde{H}$ is $\zeta_j$-tilted with respect to $H^*$ at level $\mathfrak{t}_j$ implies that $\widetilde{H} (x, \mathfrak{t}_j) \leq H^* \big( x + \zeta_j \Upsilon_{\mathfrak{t}_j} (x_0) , \mathfrak{t}_j \big)$. Applying the square root decay of $\partial_x H^*$ around $\mathfrak{A}$ from \Cref{derivativeh} (and the fact that $\partial_x H^* (x_0, \mathfrak{t}_j) = 0$), we therefore deduce the existence of a constant $c = c (\mathfrak{P}, \mathfrak{D}) > 0$ such that 
		\begin{flalign}
			\label{hxtj3} 
			\widetilde{H} (x, \mathfrak{t}_j) \leq H^* \big( x + \zeta_j \Upsilon_{\mathfrak{t}_j} (x_0) , \mathfrak{t}_j \big) \leq H^* (x, \mathfrak{t}_j) + c \zeta_j^{3/2} \leq H^* (v) + \mathcal{O} (n^{\delta / 30 - 1}).
		\end{flalign}
		
		\noindent Now let us compare $\widehat{H}^* (v)$ and $H^* (v)$. To that end, define  $\widehat{x}_0 \in \mathbb{R}$ such that $\widehat{v}_0 = (\widehat{x}_0, \mathfrak{t}_j)$ is an endpoint of $\partial_{\so} (\mathfrak{L})$ or $\partial_{\north} (\mathfrak{L})$ and $|\widehat{x}_0 - x_0|$ is minimal. By \eqref{hxyhxy2}, we have $\widehat{x}_0 - x_0 = - \xi_j' \Upsilon_{\mathfrak{t}_j} (x_0) + \mathcal{O} (\xi_1^2 + \xi_2^2)$. In particular, $\widehat{x}_0 \leq x_0$  and $x_0 - \widehat{x}_0  = \mathcal{O} (\xi_j)$, the former of which implies that $v_0 = (x_0, \mathfrak{t}_j) \notin \mathfrak{L} \cup \widehat{\mathfrak{L}}$. Hence, the second statement of \Cref{omegaxizeta} implies that $\widehat{H}^* (x_0, \mathfrak{t}_j) = \widehat{H}^* (\widehat{x}_0, \mathfrak{t}_j)$. Since $\partial_x \widehat{H}^* (x, \mathfrak{t}_j) = 0$ for $x \geq \widehat{x}_0$,  we find that $\widehat{H}^* (\widehat{x}_0, \mathfrak{t}_j) = \widehat{H}^* (x_0, \mathfrak{t}_j) = H^* (x_0, \mathfrak{t}_j)$, and so \eqref{hxdelta} yields
		\begin{flalign}
			\label{hxtj2}
			\begin{aligned}  
			\widehat{H}^* (x, \mathfrak{t}_j) & = \widehat{H}^* (\widehat{x}_0, \mathfrak{t}_j) + H^* (x - \widehat{x}_0 + x_0, \mathfrak{t}_j) - H^* (x_0, \mathfrak{t}_j) + \mathcal{O} \big( (\xi_1 + \xi_2) (x - \widehat{x}_0)^{3/2} + (x - \widehat{x}_0)^2 \big) \\
			& = H^* (x - \widehat{x}_0 + x_0, \mathfrak{t}_j) + \mathcal{O} (\xi_1^2)  \geq H^* (x, \mathfrak{t}_j) - c (\widehat{x}_0 - x_0)^{3/2} - n^{-1} \geq H^* (v) + \mathcal{O} (n^{- 1}). 
			\end{aligned}  
		\end{flalign}

		\noindent after decreasing $c$ if necessary. Here, to deduce the third statement we used \Cref{derivativeh}, and to deduce the fourth we used the facts that $x_0 - \widehat{x}_0 = \mathcal{O} (\xi_j) = \mathcal{O} (n^{- 2/3})$. Combining \eqref{h2}, \eqref{hxtj3}, and \eqref{hxtj2} then gives
		\begin{flalign*}
			\breve{h} (v) = \breve{H}^* (v) = \widehat{H}^* (v) + \mu + n^{\delta / 20 - 1} & \geq H^* (v) + n^{\delta / 20 - 1} + \mathcal{O} (n^{\delta / 30 - 1}) \\
			& \geq \widetilde{H} (v) + n^{\delta / 20 - 1} + \mathcal{O} (n^{\delta / 30 - 1}) \\
			&  \geq \widetilde{H} (v) + n^{-1} = \widetilde{h} (v) + n^{-1},
		\end{flalign*}
	
		\noindent which again gives $\breve{\mathsf{h}} (\mathsf{v}) = \big\lfloor n \breve{h} (v) \big\rfloor \geq \big\lfloor n \widetilde{h} (v)\big\rfloor = \widetilde{\mathsf{h}} (\mathsf{v})$. This verifies $\breve{\mathsf{h}} (\mathsf{v}) \geq \widetilde{\mathsf{h}} (\mathsf{v})$, for each $\mathsf{v} \in \partial \mathsf{D}$.  
		
		Now let $\breve{\mathsf{H}} : \mathsf{D} \rightarrow \mathbb{Z}$ denote a uniformly random element of $\mathscr{G} (\breve{\mathsf{h}})$, and set $\breve{H} (u) = n^{-1} \breve{\mathsf{H}} (nu)$ for each $u \in \overline{\mathfrak{D}}$. By \Cref{comparewalks} (and \Cref{heightcompare}), we may couple $\breve{\mathsf{H}} (\mathsf{u}) \geq \widetilde{\mathsf{H}} (\mathsf{u})$, for each $\mathsf{u} \in \mathsf{D}$; thus, under this coupling we have $\breve{H} (u) \geq \widetilde{H} (u)$ for each $u \in \overline{\mathfrak{D}}$.
			
		Let us next apply \Cref{estimategamma}, with the $(h; H^*; \mathsf{H}; \delta)$ there equal to the $\big( \breve{h}; \breve{H}^*; \breve{\mathsf{H}};\delta/30 \big)$. This yields with overwhelming probability that 
		\begin{flalign}
			\label{hu3} 
			\widetilde{H} (U_0) \leq \breve{H} (U_0) \leq \breve{H}^* (X_0, \mathfrak{s}) + n^{\delta / 30 - 1} \leq \widehat{H}^* (X_0, \mathfrak{s}) + \mu + 2n^{\delta / 20 - 1}.
		\end{flalign}
	
		Defining $\omega: [\mathfrak{t}_1, \mathfrak{t}_2] \rightarrow \mathbb{R}$ as in \eqref{omegay}, the first property listed there yields  
		\begin{flalign}
			\label{hu2}
			\widehat{H}^* (X_0, \mathfrak{s}) \leq  H^* (X_0, \mathfrak{s}) - \omega (\mathfrak{s}) \Omega_{\mathfrak{s}} (X_0) + \mathcal{O} \big( \xi_1^{3/2} + \xi_2^{3/2} \big).
		\end{flalign}
	
		\noindent The hypotheses of the proposition and \Cref{omegad} together imply (after decreasing $c$ if necessary) that
		\begin{flalign*}
			\omega (\mathfrak{s}) \leq \displaystyle\max \big\{ (1 - \varepsilon) & \xi_1 + \varepsilon \xi_2, \varepsilon \xi_1 + (1 - \varepsilon) \xi_2 \big\} < \bigg(1 - \displaystyle\frac{\varepsilon}{10} \bigg) \xi \leq \bigg( 1 -  \displaystyle\frac{\varepsilon}{10} \bigg) n^{\delta / 10 - 1} \dist (U_0, \mathfrak{A})^{-1/2}; \\
			& - \Omega_{\mathfrak{s}} (X_0) \geq c \dist (U_0, \mathfrak{A})^{1/2}; \qquad \xi_1^{3/2} + \xi_2^{3/2} = \mathcal{O} (n^{\delta / 30 - 1}).
		\end{flalign*}
		
		\noindent Together with \eqref{hu3} and \eqref{hu2}, this gives
		\begin{flalign*}
			\widetilde{H} (U_0) & \leq H^* (X_0, \mathfrak{s}) - \xi \Omega_{\mathfrak{s}} (X_0) + \mu + \big( \xi - \omega (\mathfrak{s}) \big) \Omega_{\mathfrak{s}} (X_0) + \mathcal{O} (n^{\delta / 30 - 1}) \\
			& \leq H^* (X_0, \mathfrak{s}) - \xi \Omega_{\mathfrak{s}} (X_0) + \mu - \displaystyle\frac{c \varepsilon}{10} n^{\delta / 10 - 1} + \mathcal{O} (n^{\delta / 30 - 1}) \leq H^* (U_0) - \xi \Omega_{\mathfrak{s}} (X_0) + \mu,
		\end{flalign*}
		
		\noindent which confirms the upper bound in \eqref{omegaomegahh}.
		\end{proof}

	\section{Proof of Concentration Estimate on Polygons} 
	
	\label{ProofH} 
	
	In this section we establish \Cref{mh}. Before proceeding, let us briefly outline the proof; we adopt the notation of \Cref{mh}  throughout this section. 
	
	Since the preliminary concentration result \Cref{estimategamma} is in itself too restrictive to this end, we will first decompose our polygonal subset $\mathfrak{P} = \bigcup_{i = 1}^k \mathfrak{R}_i$ into subregions, such that each $\mathfrak{R}_i$ is either frozen (outside the liquid region of $\mathfrak{P}$) or is a ``double-sided trapezoid'' $\mathfrak{D} (\mathfrak{a}, \mathfrak{b}; \mathfrak{t}_1, \mathfrak{t}_2)$ from \eqref{d}. Scaling by $n$, this induces a decomposition $\mathsf{P} = \bigcup_{i = 1}^k \mathsf{R}_i$ on our (tileable) polygonal domain. We then apply the alternating dynamics from \Cref{Dynamics} to this decomposition. Each step corresponds to a resampling of our tiling on some $\mathsf{R}_i$ (conditioned on its restriction to $\mathsf{P} \setminus \mathsf{R}_i$), to which \Cref{estimategamma} applies and shows that its tiling height function is within $n^{\delta}$ of its limit shape. Unfortunately, \Cref{r1r2estimate} shows that these dynamics only mix after about $n^{22}$ steps, which could in principle allow the previously mentioned $n^{\delta}$ error to accumulate macroscopically. 
	
	To remedy this, we use the notion of tiltedness from \Cref{HeightOmega}. In particular, we introduce parameters quantifying the tiltedness of certain horizontal levels $\mathsf{P}$ (that include the north and south boundaries of any $\mathsf{R}_i$). Then \Cref{hhestimate1} and \Cref{hhestimate2} will imply that, under any step of the alternating dynamics to some $\mathsf{R}_i$, the tiltedness along a middle row of $\mathsf{R}_i$ is likely bounded between those along its north and south boundaries. Since the tiltedness along $\partial \mathsf{P}$ is zero, we will be to show in this way that ``small tiltedness'' is preserved under the alternating dynamics with high probability. Running these dynamics until they mix, this indicates that the uniformly random tiling height function $\mathsf{H} : \mathsf{P} \rightarrow \mathbb{Z}$ has small tiltedness, which will establish \Cref{mh}.
		
		\subsection{Decomposition of \texorpdfstring{$\mathsf{P}$}{}}
		
		\label{Decomposition} 
		
		In this section we explain a decomposition of $\mathsf{P} = n \mathfrak{P}$ into subregions that are either frozen or where \Cref{estimategamma} will eventually apply. Recall that we have adopted the notation of \Cref{mh}; let us abbreviate the liquid region $\mathfrak{L} = \mathfrak{L} (\mathfrak{P})$ and arctic boundary $\mathfrak{A} = \mathfrak{A} (\mathfrak{P})$. By \Cref{pa}, there exists an axis $\ell$ of $\mathbb{T}$ such that no line connecting two distinct cusp singularities of $\mathfrak{A}$ is parallel to $\ell$. By rotating $\mathfrak{P}$ if necessary, we may assume that $\ell$ is the $x$-axis. We also distinguish a tangency location of $\mathfrak{A}$ to be \emph{horizontal} if the tangent line to $\mathfrak{A}$ through it is parallel the $x$-axis (has slope $0$). In what follows, we recall the trapezoid $\mathfrak{D} = \mathfrak{D} (\mathfrak{a}, \mathfrak{b}; \mathfrak{t}_1, \mathfrak{t}_2)$ from \eqref{d}, and its boundaries \eqref{dboundary}.  
		
		We begin with the following definition that will (partially) constrain what types of regions can be our decomposition; observe that the assumptions below are similar to \Cref{xhh}.
		
		\begin{definition}
			\label{dh} 
		
		 A trapezoid $\mathfrak{D}$ is \emph{adapted} to $H^*$ if the following five conditions hold. 
		
		\begin{enumerate} 
			\item The boundary $\partial_{\ea} (\mathfrak{D})$ is disjoint with $\overline{\mathfrak{L}}$, unless $\partial_{\ea} (\mathfrak{D}) \subset \partial \mathfrak{P}$ and $\mathfrak{A}$ is tangent to $\partial_{\ea} (\mathfrak{D})$; the same must hold for $\partial_{\we} (\mathfrak{D})$. 
			\item The function $H^*$ is constant along $\partial_{\ea} (\mathfrak{D})$ and along $\partial_{\we} (\mathfrak{D})$. 
			\item There exists $\widetilde{\mathfrak{t}} \in [\mathfrak{t}_1, \mathfrak{t}_2]$ such that one of the following two conditions holds. 
			\begin{enumerate} 
				\item For $t \in [\mathfrak{t}_1, \widetilde{\mathfrak{t}}]$, the set $I_t$ consists of one nonempty interval, and for $t \in (\widetilde{\mathfrak{t}}, \mathfrak{t}_2]$ the set $I_t$ consists of two nonempty dijsoint intervals.
				\item For $t \in [\mathfrak{t}_1, \widetilde{\mathfrak{t}})$, the set $I_t$ consists of two nonempty disjoint intervals, and for $t \in (\widetilde{\mathfrak{t}}, \mathfrak{t}_2]$ the set $I_t$ consists of one nonempty interval.
			\end{enumerate} 
			\item  Any tangency location of $\mathfrak{A} \cap \overline{\mathfrak{D}}$ is of the form $\max I_t$ or $\min I_t$, for some $t \in (\mathfrak{t}_1, \mathfrak{t}_2)$. Moreover, at most one is of the form $\max I_t$, and at most one is of the form $\min I_t$.
			\item We have $\mathfrak{t}_1 - \mathfrak{t}_2 \leq \mathfrak{c}$, where $\mathfrak{c} = \mathfrak{c} (\mathfrak{P}) > 0$ is given by \Cref{estimategamma}.
		\end{enumerate} 
	
		\end{definition} 
	
		\begin{figure}
			
			\begin{center}		
				
				\begin{tikzpicture}[
					>=stealth,
					auto,
					style={
						scale = .52
					}
					]
					
					\draw[black, ultra thick] (-7.5, -.866) -- (-3.5, 2.584)-- (5.45, 2.584) node[right]{$\partial \mathfrak{P}$}-- (5.45, -.866);
					
					\filldraw[fill = gray!40!white] (-3.25, 1.75) -- (4.5, 1.75) -- (4.5, -.25) -- (-5.25, -.25) -- (-3.25, 1.75); 
					
					\draw[black] (-3, -.866) arc (180:120:3.464);
					\draw[black] (3.928, -.866) arc (0:60:3.464);
					\draw[black] (-1.278, 2.134) arc (120:60:3.464);
					\draw[black] (-1, -.866) arc (-60:0:1.732);
					\draw[black] (.75, -.866) arc (240:180:1.732);

					\filldraw[fill = black] (-.125, .75) circle [radius = .1] node[above, scale = .7]{$u$};
					
					\draw[thick] (-3.25, 1.75) -- (4.5, 1.75) node[above, scale = .7]{$\mathfrak{D} (u)$} -- (4.5, -.25) -- (-5.25, -.25) -- (-3.25, 1.75);
					
					\draw[] (-.125, -1) circle[radius = 0] node[]{$\vdots$};
					
					\draw[dashed] (-4.5, .75) node[left, scale = .7]{$\ell_0$}-- (5, .75); 
					
					\filldraw[fill = black] (-2.6, .75) circle [radius = .075] node[left = 2, above, scale = .6]{$u_1$};
					\filldraw[fill = black] (3.52, .75) circle [radius = .075] node[right = 2, above, scale = .6]{$u_2$};
					
				\end{tikzpicture}
				
			\end{center}
			
			\caption{\label{rj1} Shown above is an example of the trapezoid $\mathfrak{D} (u)$ from \Cref{udomain}; only part of the polygon $\mathfrak{P}$ and its liquid region $\mathfrak{L}$ are depicted.}
			
		\end{figure}
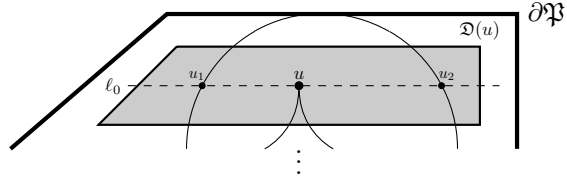

		\begin{lem} 
			
		\label{udomain} 
		
		If $u \in \overline{\mathfrak{L}}$ is not a tangency location of $\mathfrak{A}$, then there exists a trapezoid $\mathfrak{D} (u)$ adapted to $H^*$, containing $u$ in its interior. 
		
		\end{lem} 
	
		\begin{proof} 
			
			Let $u = (x_0, y_0)$ and $\ell_0 = \{ y = y_0 \} \subset \mathbb{R}^2$ denote the line through $u$ parallel to the $x$-axis. To create $\mathfrak{D} (u) = \mathfrak{D} (\mathfrak{a}, \mathfrak{b}; \mathfrak{t}_1, \mathfrak{t}_2)$, we will first specify segments containing its east and west boundaries, and then specify $(\mathfrak{t}_1, \mathfrak{t}_2)$ to make it sufficiently ``short'' (that is, with $\mathfrak{t}_2 - \mathfrak{t}_1$ small). 
			
			To that end, first assume that $u$ is a cusp of $\mathfrak{A}$; we refer to \Cref{rj1} for a depiction. Let $x_1 \in \mathbb{R}$ be maximal and $x_2 \in \mathbb{R}$ be minimal such that $x_1 < x_0 < x_2$; $u_1 = (x_1, y_0) \in \mathfrak{A}$; and $u_2 = (x_2, y_0) \in \mathfrak{A}$. By \Cref{pa}, neither $u_1$ nor $u_2$ is a cusp of $\mathfrak{A}$. If $u_1 \in \partial \mathfrak{P}$, then it is a (non-horizontal) tangency location of $\mathfrak{A}$, so it lies along a side of $\partial \mathfrak{P}$ with slope $1$ or $\infty$. We may then let this side contain the west boundary of $\mathfrak{D}(u)$. If instead $u_1 \notin \partial \mathfrak{P}$, then there exists a real number $\varepsilon = \varepsilon (\mathfrak{P}, u) \in (0, 1)$ such that $(x, y_0) \in \mathfrak{P} \setminus \overline{\mathfrak{L}}$ for each $x_1 - \varepsilon^{1/2} < x < x_1$, and such that either $\big( (x_1 - \varepsilon^{1/2}, x_2 + \varepsilon^{1/2}) \times (y_0, y_0 + \varepsilon) \big) \cap \mathfrak{L}$ or $\big( (x_1 - \varepsilon^{1/2}, x_2 + \varepsilon^{1/2}) \times (y_0 - \varepsilon, y_0) \big) \cap \mathfrak{L}$ is connected (see \Cref{rj1}). 
			
			Next recall from the first statement in \Cref{pla} that, on $\mathfrak P\setminus \mathfrak L(\mathfrak P)$, $\nabla H^*$ is piecewise constant, taking values in $\big\{(0,0), (1,0), (1,-1) \big\}$.
			If $(x_1,y_0)\in \fA$ is a continuity point of $\nabla H^*$,  then (upon decreasing $\varepsilon$ if necessary) there exists $\lambda = \lambda (\mathfrak{P}, u) > 0 \in (0, \varepsilon)$ such that the disk $\mathfrak{B}_{\lambda} (x_1 - \varepsilon^{1/2}, y_0)$ does not intersect $\overline{\mathfrak{L}}$, and $\nabla H^* (x_1, y_0) $ is constant on $\mathfrak{B}_{\lambda} (x_1 - \varepsilon^{1/2}, y_0)$. Then, depending on whether $\nabla H^* (x_1, y_0) = (1, -1)$ or $\nabla H^* (x_1, y_0) \in \big\{ (0, 0), (1, 0) \big\}$, the west boundary of $\mathfrak{D} (u)$ is contained in the segment obtained as the intersection between $\mathfrak{B}_{\lambda} (x_1 - \varepsilon^{1/2}, y_0)$ and the line passing through $(x_1 - \varepsilon^{1/2}, y_0)$ with slope $1$ or $\infty$, respectively; then, $H^*$ is constant along this line.
			
			If $(x_1,y_0)\in \fA$ is a discontinuity point of $\nabla H^*$ then, by \Cref{pa}, $(x_1,y_0)\in \fA$ is a tangency location. From our choice of $x_1$, $(x_1,y_0)$ cannot be a horizontal tangency location. Thus its tangent line has slope $1$ or $\infty$. For $\varepsilon$ small enough, the part of the tangent line between $y=y_0-\varepsilon$ and $y=y_0+\varepsilon$ is contained in the $\mathfrak P\setminus \mathfrak L(\mathfrak P)$. By the relations \eqref{fh} between $\nabla H^*$ and the complex slope, and those \eqref{lqq} between the complex slope and the slope of the tangent line of the arctic curve, if the tangent line has slope $1$ then $\nabla H^*\in \big\{ (0,0),(1,-1) \big\}$, and if the tangent line has slope $\infty$, then $\nabla H^*\in \big\{ (0,0), (1,0) \big\}$. In both cases $H^*$ is constant along the tangent line. This again specifies a segment containing the west boundary of $\mathfrak{D} (u)$ where $H^*$ is a constant along it, and one containing its east boundary can be specified similarly. 
						
			In either case (whether $(x_1, y_0)$ is a continuity or discontinuity point of $\nabla H^*$), we let $\mathfrak{t}_1 = y_0 - \lambda_0$ and $\mathfrak{t}_2 = y_0 + \lambda_0$, where $\lambda_0$ is chosen sufficiently small so that the east and west boundary of $\mathfrak{D} (u)$ are contained in the segments specified above, and so that $\mathfrak{D}$ satisfies the third, fourth, and fifth conditions of \Cref{dh}. This determines $\mathfrak{D} (u)$, which contains $u$ in its interior and is quickly seen to be adapted to $H^*$.
			
			The proof is similar if instead $u$ is not a cusp of $\mathfrak{A}$, so we only outline it. If $u \in \mathfrak{L}$, then the above reasoning applies, unless either $u_1$ or $u_2$ is a cusp of $\mathfrak{A}$. Assuming for example that the former is, there exists a trapzeoid $\mathfrak{D} (u_1)$ adapted to $H^*$ that contains $u_1$ in its interior. Then $\mathfrak{D} (u_1)$ must also contain $u$ in its interior, since $\ell_0$ pases through $u$ before intersecting $\mathfrak{A}$; so, we may set $\mathfrak{D} (u) = \mathfrak{D} (u_1)$. If instead $u \in \mathfrak{A}$ and is not a cusp, then the above reasoning (in the case when $u$ is cusp) again applies, with the mild modification that we allow $u = u_1$ or $u = u_2$, depending on whether $u$ is a left or right boundary point of $\mathfrak{A}$, respectively.
		\end{proof} 
		
		Now, for each point $u \in \overline{\mathfrak{P}}$, in the following, we define an open subset $\mathfrak{R} (u) \subset \mathfrak{P}$ such that $u \in \mathfrak{R} (u)$ if $u \in \mathfrak{P}$, and $u \in \overline{\mathfrak{R} (u)}$ if $u \in  \partial \mathfrak{P}$. In each case, $\mathfrak{R} (u)$ will be the union of at most two trapezoids (intersected with $\mathfrak{P}$). We always assume (by applying a small shift if necessary) that the north and south boundaries of any such trapezoid does not contain any cusps or tangency locations of $\mathfrak{A}$, except for possibly if $u$ is a horizontal tangency location.

		\begin{enumerate}
			\item If $u \in \overline{\mathfrak{L}}$ is not a horizontal tangency location of $\mathfrak{A}$, then let $\mathfrak{R} (u) \subseteq \mathfrak{P}$ denote a trapezoid adapted to $H^*$, containing $u$ in its interior.
			\item If $u \in \mathfrak{A}$ is horizontal tangency location on $\mathfrak{A} \cap \partial \mathfrak{P}$, then let $\mathfrak{R} (u) \subseteq \mathfrak{P}$ denote a trapezoid adapted to $H^*$, such that $u$ is in the interior of either $\partial_{\north} (\mathfrak{D})$ or $\partial_{\so} (\mathfrak{D})$. 
			\item If $u \in \mathfrak{A}$ is a horizontal tangency location on $\mathfrak{A}$ not in $\partial \mathfrak{P}$, then let $\mathfrak{R} (u) = \mathfrak{D}_1 (u) \cup \mathfrak{D}_2 (u)$. Here, $\mathfrak{D}_1 = \mathfrak{D}_1 (u)$ and $\mathfrak{D}_2 = \mathfrak{D}_2 (u)$ are trapzeoids adapted to $H^*$, such that $u$ is in the interior of $\partial_{\north} (\mathfrak{D}_2)$ and of $\partial_{\so} (\mathfrak{D}_1)$, and such that either $\mathfrak{D}_1$ or $\mathfrak{D}_2$ is disjoint with $\mathfrak{L}$.
			\item If $u \notin \overline{\mathfrak{L}}$, then let $\mathfrak{R} (u) = \mathfrak{D} (u) \cap \mathfrak{P}$, for some trapezoid $\mathfrak{D} (u) \subset \mathbb{R}^2$ containing $u$ in its interior, such that $\mathfrak{D} (u)$ is disjoint with $\overline{\mathfrak{L}}$.
		
		\end{enumerate}

		\begin{figure}
			
			\begin{center}		
				
				\begin{tikzpicture}[
					>=stealth,
					auto,
					style={
						scale = .475
					}
					]

					\filldraw[fill = gray!40!white, thick] (2, 0) -- (6, 0) -- (6.5, .5) -- (2, .5) -- (2, 0);
					\filldraw[fill = gray!40!white, thick] (10.5, 1.25) -- (10.5, 2) -- (11.25, 2.75) -- (15, 2.75) -- (14.25, 2) -- (14, 2) -- (14, 1.25) -- (10.5, 1.25);
					\draw[fill = gray!40!white, thick] (23.65, 1.5) -- (23.65, 2.75) -- (26, 2.75) -- (25.25, 2) -- (24, 2) -- (24, 1.5) -- (23.65, 1.5);
					\draw[dashed] (23.65, 1.5) -- (23.65, 2.75) -- (26, 2.75) -- (24.75, 1.5) -- (23.65, 1.5);
					
					\draw[black, ultra thick] (.5, 2) node[left]{$\partial \mathfrak{P}$}-- (.5, 0) -- (6, 0) -- (8, 2);
					
					\draw[black] (1.25, 1.2) node[above, scale = .7]{$\mathfrak{A}$} arc (-135:-45:4);
					
					\draw (5, 1.9) circle[radius = 0] node[]{$\vdots$};
					\filldraw[fill = black] (4, 0) circle [radius = .1] node[below, scale = .7]{$u$};

					\draw[black, ultra thick] (14, -1) -- (14, 2) -- (17, 2) -- (18, 3) node[right]{$\partial \mathfrak{P}$};
					\draw[black, very thick, dashed] (10, 2) -- (14, 2);
					\draw[] (13.45, -1) node[left, scale = .7]{$\mathfrak{A}$} arc (-45:150:1.75);
					
					\draw[] (12, -.8) circle[radius = 0] node[]{$\vdots$};
					\draw[] (15.5, 3.125) circle[radius = 0] node[]{$\vdots$};
					\filldraw[fill = black] (12.25, 2) circle [radius = .1] node[above, scale = .7]{$u$};
					
					\draw[black, ultra thick] (24, -1) -- (24, 2) -- (27, 2) -- (28, 3) node[right]{$\partial \mathfrak{P}$};
					
					\draw[black] (23.45, -1) node[left, scale = .7]{$\mathfrak{A}$} arc (-45:150:1.75);
										
					\draw[] (22.125, -.75) circle[radius = 0] node[]{$\vdots$};
					\draw[] (26.25, 3.125) circle[radius = 0] node[]{$\vdots$};
					\filldraw[fill = black] (24, 2) circle [radius = .1] node[above, scale = .7]{$u$};	
					
				\end{tikzpicture}
				
			\end{center}
			
			\caption{\label{rj2} Shown to the left, middle, and right are examples of the $\mathfrak{R} (u)$ (shaded) when $u \in \partial \mathfrak{P}$ is a horizontal tangency location of $\mathfrak{A}$, when $u \notin \mathfrak{A}$ is a horizontal tangency location of $\mathfrak{A}$, and when $u \notin \overline{\mathfrak{L}}$, respectively. In all cases, only part of the polygon $\mathfrak{P}$ and its liquid region $\mathfrak{L}$ are depicted.}
			
		\end{figure}
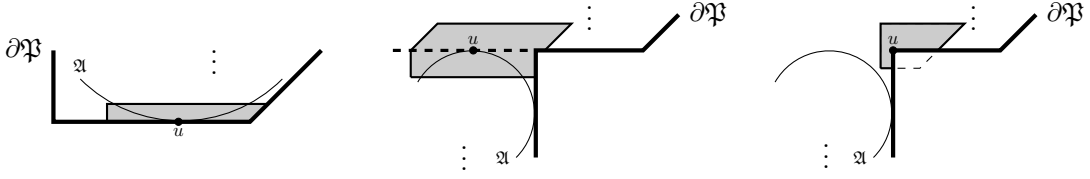

		\noindent We may further assume (after applying a small shift, if necessary) that $n \mathfrak{R} (u) \subset \mathbb{T}$, for each $u \in \mathfrak{P}$. The existence of these regions $\mathfrak{R}_i$ follows from \Cref{udomain} in the first case and is quickly verified from the definitions in all other cases.\footnote{For the third, we are using the fact that no horizontal tangency location of $\mathfrak{A}$ is also a cusp of $\mathfrak{A}$ (and that $\mathfrak{A}$ has no tacnode singularities), as stipulated by \Cref{pa}.} We refer to \Cref{rj1} for a depiction in the first case and to \Cref{rj2} for depictions in the remaining three cases.
		
		Since $\overline{\mathfrak{P}}$ is compact, the $\mathfrak{R} (u)$ are open, and $\bigcup_{u \in \mathfrak{P}} \overline{\mathfrak{R} (u)} = \overline{\mathfrak{P}}$, there exists a finite subcover $\bigcup_{i = 1}^k \overline{\mathfrak{R}_i} = \overline{\mathfrak{P}}$; here, each $\mathfrak{R}_i = \mathfrak{R} (u_i)$ for some $u_i \in \overline{\mathfrak{P}}$. In what follows, we fix such a cover and let $\mathfrak{t}_0 < \mathfrak{t}_1 < \cdots < \mathfrak{t}_m$ denote all real numbers for which either a north or south boundary of some $\mathfrak{R}_i$ lies along a line $\{ y = \mathfrak{t}_j \}$. Observe that, if such a boundary of some $\mathfrak{R}_i$ lies along $\{ y = \mathfrak{t}_0 \}$ or $\{ y = \mathfrak{t}_m \}$, then it must lie along $\partial \mathfrak{P}$. Moreover, since the $\mathfrak{t}_j$ are pairwise distinct, there exists a constant $\varepsilon_0 = \varepsilon_0 (\mathfrak{P}) > 0$ such that, for each $1\leq j\leq m$, 
		\begin{flalign}
			\label{tj1tjtj1} 
			\displaystyle\min_{1\leq j\leq m} (\mathfrak{t}_j -  \mathfrak{t}_{j - 1}) \geq \varepsilon_0 (\mathfrak{t}_m - \mathfrak{t}_0).
		\end{flalign}

		Further let $\mathsf{R}_i = n \mathfrak{R}_i \subset \mathbb{T}$ for each $1\leq i\leq k$. Observe, since the $\mathfrak{R}_i$ are open and cover $\mathfrak{P}$, that any interior vertex of $\mathsf{P}$ is an interior vertex of some $\mathsf{R}_i$. Thus, we may consider the alternating dynamics (from \Cref{r1r2r}) on $\mathsf{P}$ with respect to $(\mathsf{R}_1, \mathsf{R}_2, \ldots , \mathsf{R}_k)$. In particular, let us fix the height function $\mathsf{H}_0 \in \mathscr{G} (\mathsf{h})$ by setting $\mathsf{H}_0 (u) = \big\lfloor n H^* (n^{-1} u) \big\rfloor$, for each $u \in \mathsf{P}$; observe that $\mathsf{H}_0 (u) = nH^* ( n^{-1} u)$ for each $u \in \mathsf{P} \setminus (n \cdot \mathfrak{L})$, by \Cref{p:frozenr}. Then, run the alternating dynamics on $\mathscr{G} (\mathsf{h})$ with initial state $\mathsf{H}_0$. For each integer $r \geq 0$, let $\mathsf{H}_r \in \mathscr{G} (\mathsf{h})$ denote the state of this Markov chain at time $r$; define its scaled version $H_r \in \Adm (\mathfrak{D}; h)$ by $H_r (u) = n^{-1} \mathsf{H}_r (nu)$ for each $u \in \overline{\mathfrak{D}}$.

		\subsection{Proof of \Cref{mh}} 
		
		\label{TilingEstimate} 
		
		In this section we establish \Cref{mh}. To that end, recalling the notation from \Cref{Decomposition}, we define for any integer $r \geq 0$ the events
		\begin{flalign}
			 \label{fr1} 
			 \begin{aligned} 
			\mathscr{F}_r & = \bigcap_{u \in \overline{\mathfrak{P}}} \Big\{ \big| H_r (u) - H^* (u) \big| < n^{\delta - 1} \Big\} \cap \bigcap_{u \in \overline{\mathfrak{P}} \setminus \mathfrak{L}_+^{\delta} (\mathfrak{P})} \big\{ H_r (u) = H^* (u)\big\}; \\
			\mathscr{F}_{\infty} & = \bigcap_{u \in \overline{\mathfrak{P}}} \Big\{ \big| \mathsf{H} (n u) - n H^* (u) \big| < n^{\delta} \Big\} \cap \bigcap_{u \in \overline{\mathfrak{P}} \setminus \mathfrak{L}_+^{\delta} (\mathfrak{P})} \big\{ \mathsf{H} (nu) = n H^* (u)\big\}.
			\end{aligned} 
		\end{flalign}
		
		\noindent Then, \Cref{mh} indicates that $\mathscr{F}_{\infty}$ should hold with overwhelming probability. As $r$ becomes large, the alternating dynamics tend to stationarity, so $n H_r (u)$ converges in law to $\mathsf{H} (nu)$; hence, it instead suffices to show that $\mathscr{F}_r$ likely holds for large $r$. We would like to proceed inductively, by showing that $\mathscr{F}_r$ implies $\mathscr{F}_{r + 1}$ with high probability. It is not transparent to us how to do this directly; we instead prove a stronger version of this implication involving the notion of tilting, which we recall from \Cref{lambdamud}. 
		
		We first require some notation. Recall that $\varepsilon_0 = \varepsilon_0 (\mathfrak{P}) > 0$ satisfies \eqref{tj1tjtj1}, and set 
		\begin{flalign}
			\label{nui} 
			\nu_i = 1 - \bigg( \displaystyle\frac{\varepsilon_0}{10} \bigg)^{i + 1}, \qquad 0\leq i\leq m.
		\end{flalign} 
	
		\noindent Also recall the subset $\mathfrak{L}_-^{\delta} = \mathfrak{L}_-^{\delta} (\mathfrak{P}) \subset \mathfrak{L}$ from \eqref{ldelta}. 
		
		For any real number $z \geq n^{-1}$, define the functions
		\begin{flalign}
			\label{functionx} 
			\varkappa (z) = n^{\delta / 12 -1} z^{-1}; \qquad \varpi (z) = n^{\delta / 12 - 1} \log (nz).
		\end{flalign}
		
		\noindent The explicit forms of $\varkappa$ and $\varpi$ above will not be central for our purposes, but a useful point will be that $A \varkappa (z) + \varpi (z)$ is minimized when $z = A$.
		
		Now, recall the subset $\mathfrak{L}_-^{\delta} = \mathfrak{L}_-^{\delta} (\mathfrak{P}) \subset \mathfrak{L}$ from \eqref{ldelta}, and for any integers $r \geq 0$ and $0\leq i\leq m$ define the events 
		\begin{flalign}
			\label{ei1ei2r}
			\begin{aligned} 
			\mathscr{E}_r^{(1)} (i) & = \big\{ \text{The edge of $H_r$ is $\nu_i n^{\delta / 240 - 2/3}$-tilted with respect to $H^*$ at level $\mathfrak{t}_i$} \big\}; \\
			\mathscr{E}_r^{(2)} (i) & = \bigcap_{(x, \mathfrak{t}_i) \in \mathfrak{L}_-^{\delta / 4}} \big\{ \text{At $(x, \mathfrak{t}_i)$, $H_r$ is $(\nu_i n^{\delta / 240 - 2/3}; 0)$-tilted with respect to $H^*$} \big\}.
			\end{aligned} 
		\end{flalign}
		
		\noindent For any $x \in \mathbb{R}$ such that $(x, \mathfrak{t}_i) \in \mathfrak{L}_-^{\delta / 4}$, further define the event 
		\begin{flalign}
			\label{ei3x0} 
			\begin{aligned}
				& \mathscr{E}_r^{(3)} (i; x) =  \bigg\{ \displaystyle\sup_{z \geq n^{-1}} \Big( H^* (x, \mathfrak{t}_i) + \nu_i \varkappa (z) \Omega_{\mathfrak{t}_i} (x) - \varpi (z) \Big) \\
				& \qquad \qquad \qquad \qquad \qquad \leq H_r (x, \mathfrak{t}_i) \leq \displaystyle\inf_{z  \geq n^{-1}} \Big( H^* (x, \mathfrak{t}_j) - \nu_i \varkappa (z) \Omega_{\mathfrak{t}_i} (x) + \varpi (z) \Big) \bigg\},
			\end{aligned} 
		\end{flalign}
		
		\noindent where we recall $\Omega_s (x)$ from \eqref{omegasx}, and let 
		\begin{flalign*} 
			\mathscr{E}_r^{(3)} (i) = \bigcap_{(x, \mathfrak{t}_i) \in \mathfrak{L}_-^{\delta / 4}} \mathscr{E}_r^{(3)} (i; x).
		\end{flalign*}
		
		\noindent Then define the events 
		\begin{flalign}
			\label{afe}
			\mathscr{E}_r (i) & = \mathscr{E}_r^{(1)} (i) \cap \mathscr{E}_r^{(2)} (i) \cap \mathscr{E}_r^{(3)} (i); \qquad \mathscr{E}_r = \bigcap_{i = 0}^m \mathscr{E}_r (i); \qquad \mathscr{A}_r = \mathscr{E}_r \cap \mathscr{F}_r
		\end{flalign} 
	
		Under this notation, we have the following proposition. 
		
		\begin{prop}
		
		\label{ar1ar} 
		
		For any real number $D > 1$, there exists a constant $C = C(\mathfrak{P}, D) > 1$ such that the following holds whenever $n > C$. For any integer $r \geq 0$, we have $\mathbb{P} (\mathscr{A}_{r + 1}) \geq \mathbb{P} (\mathscr{A}_r) - n^{-D}$. 
		\end{prop} 
		
		Given \Cref{ar1ar}, we can quickly establish \Cref{mh}.

		\begin{proof}[Proof of \Cref{mh}] 
			
			First observe that $\mathscr{A}_0$ holds deterministically, since $H_0 (u) = n^{-1} \mathsf{H}_0 (nu)$ and $\mathsf{H} (nu) = \big\lfloor n H^* (u) \big\rfloor$ hold for each $u \in \overline{\mathfrak{D}}$ (and $H_0 (u) = H^* (u)$ for each $u \in \mathfrak{P} \setminus \mathfrak{L}$, by \Cref{p:frozenr}). Then, inductively applying \Cref{ar1ar}, with the $D$ there equal to $2D + 15$ here, yields $\mathbb{P} \big( \mathscr{A}_{n^{40}} \big) \geq 1 - n^{-2D}$ for sufficiently large $n$. By the definition \eqref{afe} of $\mathscr{A}_r$, this implies $\mathbb{P} \big( \mathscr{F}_{n^{40}} \big) \geq 1 - n^{-2D}$. Since $\diam \mathsf{R} = n \diam \mathfrak{R}$, the bound \Cref{r1r2estimate} on the mixing time for the alternating dynamics (together with the definitions \eqref{fr1} of $\mathscr{F}_r$ and $\mathscr{F}_{\infty}$) gives 
			\begin{flalign*} 
				\mathbb{P} ( \mathscr{F}_{\infty}) \geq \mathbb{P} \big( \mathscr{F}_{n^{40}} \big) - e^{-n} \geq 1 - n^{-2D} - e^{-n} \geq 1 - n^{-D},
			\end{flalign*}
		
			\noindent for sufficiently large $n$. Since this holds for any $D$ (if $n$ is sufficiently large), this implies the theorem.
		\end{proof}

		\subsection{Proof of \Cref{ar1ar}}
		
		\label{ProofEstimateHr}

		In this section we establish \Cref{ar1ar}.

		\begin{proof}[Proof of \Cref{ar1ar}]
			
			Throughout this proof, we restrict to the event $\mathscr{A}_r$; it suffices to show that $\mathscr{A}_{r + 1}$ holds with overwhelming probability. In what follows, we will frequently use the equality
			\begin{flalign}
				\label{hruhru} 
				H_r (u) = H^* (u), \qquad \text{if $u \in \overline{\mathfrak{P}} \setminus \overline{\mathfrak{L}}$ is bounded away from $\mathfrak{L}$},
			\end{flalign}
		
			\noindent which holds since we have restricted to the event $\mathscr{F}_r \subseteq \mathscr{A}_r$ from \eqref{fr1}. 
			
			Updating $\mathsf{H}_r$ to $\mathsf{H}_{r + 1}$ involves resampling it on a subdomain $\mathsf{R}_j = n \mathfrak{R}_j$, for some index $1\leq j\leq m$, in the decomposition $\mathsf{P} = \bigcup_{j = 1}^k \mathsf{R}_j$. Recall from \Cref{Decomposition} that $\mathfrak{R}_j = \mathfrak{R} (u_j)$, for some $u_j \in \overline{\mathfrak{P}}$, and that there are four possible cases for $\mathfrak{R} (u_j)$, depending on whether $u_j \in \overline{\mathfrak{L}}$ is not a tangency location of $\mathfrak{A}$; $u_j \in \mathfrak{A}$ is a tangency location of $\mathfrak{A}$ that lies on $\partial \mathfrak{P}$; $u_j \in \mathfrak{A}$ is a tangency location that does not lie on $\partial \mathfrak{P}$; or $u \notin \overline{\mathfrak{L}}$ is outside the liquid region. We will address each of these cases.
			
			To that end, first consider the fourth case when $u_j \notin \overline{\mathfrak{L}}$. Then, $\mathfrak{R}_j \subseteq \mathfrak{P} \setminus \overline{\mathfrak{L}}$; in particular, it is bounded away from $\overline{\mathfrak{L}}$, so \eqref{hruhru} implies $\mathsf{H}_r (nu) = n H^* (u)$ for each $u \in \mathfrak{R}_j$. Since $\nabla H^* (u) \in \big\{ (1, 0), (1, -1), (1, -1) \big\}$ for almost every $u \notin \mathfrak{L}$ (by the first statement of \Cref{pla}), it follows that there is only one height function on $\mathsf{R}_j$ with boundary data $\mathsf{H}_r |_{\partial \mathsf{R}_j}$, that is, this domain is frozen. This implies that $\mathsf{H}_{r + 1} = \mathsf{H}_r$; so, each of the estimates involved in the definitions of the events in \eqref{fr1}, \eqref{ei1ei2r}, \eqref{ei3x0} for $H_{r + 1}$ follow from their counterparts for $H_r$ guaranteed by $\mathscr{A}_r$. Thus, $\mathscr{A}_{r + 1}$ holds deterministically in this case. 
			
			Next, we consider the first or second case, namely, when $u_j \in \overline{\mathfrak{L}}$ and is not a horizontal tangency location outside of $\partial \mathfrak{P}$. If $u_j$ is a horizontal tangency location of $\mathfrak{A}$, then we assume $u_j \in \partial_{\north} (\mathfrak{R}_j)$, as the proof when $u_j \in \partial_{\so} (\mathfrak{R}_j)$ is entirely analogous by rotation (see \Cref{northsouthd}). Then $\mathfrak{R}_j = \mathfrak{R} (u_j)$ is a double-sided trapezoid satisfying the conditions of \Cref{assumptiond}, since it is adapted with respect to $H^*$ (recall \Cref{dh}). By \eqref{afe}, suffices to show that $\mathscr{E}_{r + 1}$ and $\mathscr{F}_{r + 1}$ both hold with overwhelming probability. 
			
			We begin with the former. Fix an index $1\leq i_0\leq m$; we must show $\mathscr{E}_{r + 1} (i_0)$ holds with overwhelming probability. Define indices $1\leq i, i' \leq m$ so that $\partial_{\so} (\mathfrak{R}_j)$ and $\partial_{\north} (\mathfrak{R}_j)$ are contained in the horizontal lines $\{ y = \mathfrak{t}_i \}$ and $\{ y = \mathfrak{t}_{i'} \}$, respectively. Without loss of generality, we assume that  $i < i'$. Since the update from $H_r$ to $H_{r + 1}$ only affects its restriction to $\mathfrak{R}_j \subset \big\{ (x, y) \in \mathfrak{P} : \mathfrak{t}_i < y < \mathfrak{t}_{i'} \}$, for $i_0 \notin (i, i')$ the event $\mathscr{E}_{r + 1} (i_0)$ holds deterministically if $\mathscr{E}_r (i_0)$ does. Hence, we may assume that $i < i_0 < i'$. In what follows, we further denote the restrictions $h_r = H_r |_{\mathfrak{R}_j}$ and $\mathsf{h}_r = \mathsf{H}_r |_{\mathsf{R}_j}$, 
			
			We first verify that $\mathscr{E}_{r + 1}^{(1)} (i_0)$ and $\mathscr{E}_{r+1}^{(2)} (i_0)$ from \eqref{ei1ei2r} both hold with overwhelming probability, by suitably applying \Cref{hhestimate1}. To that end, observe that \Cref{assumptiondomega} holds with the parameters $(\varepsilon, \varsigma, \delta; \mathfrak{D}, \mathsf{D}; \mathfrak{s}; \mathfrak{t}_1, \mathfrak{t}_2; \widetilde{\mathsf{h}}, \widetilde{\mathsf{H}}; \xi_1, \xi_2; \zeta_1, \zeta_2; \mu)$ there equal to 
			\begin{flalign}
				\label{1parameters}
				\begin{aligned} 
			 \bigg( \varepsilon_0, \Big( \displaystyle\frac{\varepsilon_0}{20} \Big)^{m + 1}, \frac{\delta}{4}; \mathfrak{R}_j, \mathsf{R}_j; \mathfrak{t}_{i_0}; \mathfrak{t}_i, \mathfrak{t}_{i'}; \mathsf{h}_r, \mathsf{H}_{r + 1}; & \nu_i n^{\delta / 240 - 2/3}, \nu_{i'} n^{\delta / 240 - 2/3}; \\
			& \nu_i n^{\delta / 240 - 2/3}, \nu_{i'} n^{\delta / 240 - 2 / 3}; 0 \bigg),
			\end{aligned} 
				\end{flalign} 
			
			\noindent here. To see this, first observe that $\mathsf{h}_r$ is constant along the east and west boundaries of $\partial \mathsf{R}_j$. Indeed, \eqref{hruhru} and the fact that $\partial_{\ea} (\mathfrak{R}_j)$ and $\partial_{\we} (\mathfrak{R}_j)$ are either subsets of $\partial \mathfrak{P}$ or bounded away from $\overline{\mathfrak{L}}$, together imply that $H_r = H^*$ along $\partial_{\ea} (\mathfrak{R}_j) \cup \partial_{\we} (\mathfrak{R}_j)$. In particular, $\mathsf{h}_r (nv) = n h_r (v) = n h(v) = n H^* (v)$ holds for each $v \in \partial_{\ea} (\mathfrak{R}_j) \cup \partial_{\we} (\mathfrak{R}_j)$. So, since $\mathfrak{R}_j$ is adapted to $H^*$, $h$ is constant both $\partial_{\ea} (\mathfrak{R}_j)$ and $\partial_{\we} (\mathfrak{R}_j)$; thus, $h_r$ is as well. Next, the inequalities on $(\mathfrak{t}_1, \mathfrak{s}, \mathfrak{t}_2)$ with respect to $\varepsilon$, and on $(\xi_1, \xi_2)$ and $(\zeta_1, \zeta_2)$ with respect to $\varsigma$, in \Cref{assumptiondomega} follow from \eqref{tj1tjtj1} and \eqref{nui}. Moreover, the edge-tiltedness for $H_r$ with respect to $H^*$ along $\partial_{\so} (\mathfrak{R}_j) \cup \partial_{\north} (\mathfrak{R}_j)$ is a consequence of our restriction to the event $\mathscr{E}_r^{(1)} (i) \cap \mathscr{E}_r^{(1)} (i') \subseteq \mathscr{A}_r$.  Similarly, the bulk-tiltedness of $H_r$ with respect to $H^*$ at each $(x, \mathfrak{t}_i), (x, \mathfrak{t}_{i'}) \in \mathfrak{L}_-^{\delta / 4}$ holds follows from our restriction to $\mathscr{E}_r^{(2)} (i) \cap \mathscr{E}_r^{(2)} (i') \subseteq \mathscr{A}_r$. 
			
			Thus, \Cref{hhestimate1} applies. Under the choice of parameters \eqref{1parameters}, the $\zeta$ in \Cref{hhestimate1} equals 
			\begin{flalign}
				\label{nu1} 
				\nu' = \displaystyle\frac{\varepsilon_0}{2} \nu_i + \bigg( 1 - \displaystyle\frac{\varepsilon_0}{2} \bigg) \nu_{i'} < \nu_{i_0},
			\end{flalign}
	
			\noindent where to deduce the last inequality we used \eqref{nui} and the fact that $i < i_0 < i'$. In particular, \Cref{hhestimate1} implies with overwhelming probability that the edge of $H_{r + 1}$ is $\nu_{i_0} n^{\delta / 240 - 2/3}$-tilted with respect to $H^*$ at level $\mathfrak{t}_{i_0}$, and that $H_{r + 1}$ is $(\nu_{i_0} n^{\delta / 240 - 2/3}; 0)$-tilted with respect to $H^*$ at any $(x, \mathfrak{t}_{i_0}) \in \mathfrak{L}_i^{\delta / 4}$. So, by \eqref{ei1ei2r}, $\mathscr{E}_{r + 1}^{(1)} (i_0)$ and $\mathscr{E}_{r + 1}^{(2)} (i_0)$ both hold with overwhelming probability. 
			
			Next let us show that $\mathscr{E}_{r + 1}^{(3)} (i_0)$ holds with overwhelming probability. To that end, we fix $x \in \mathbb{R}$ such that $(x, \mathfrak{t}_{i_0}) \in \mathfrak{L}_-^{\delta / 4}$, and set
			\begin{flalign}
				\label{lambda}  
				\lambda = - \nu' \Omega_{\mathfrak{t}_{i_0}} (x) \geq n^{-1}
			\end{flalign} 
		
			\noindent where the latter inequality follows from \Cref{omegad} (and we recall $\nu'$ from \eqref{nu1}). We will first use \Cref{hhestimate2} to show with overwhelming probability that
			\begin{flalign}
			\label{hx0i0lambda}
				H^* (x, \mathfrak{t}_{i_0}) + \nu' \varkappa (\lambda) \Omega_{\mathfrak{t}_{i_0}} (x) - \varpi (\lambda) \leq H_{r + 1} (x, \mathfrak{t}_{i_0}) \leq H^* (x, \mathfrak{t}_{i_0}) - \nu' \varkappa (\lambda) \Omega_{\mathfrak{t}_{i_0}} (x) + \varpi (\lambda),
			\end{flalign}
		
			\noindent which since $A \varkappa (z) + \varpi (z)$ is minimized at $z = A$ implies 
			\begin{flalign}
				\label{hx0i0x} 
				\begin{aligned}
				\displaystyle\sup_{z \geq n^{-1}} \Big( H^* & (x, \mathfrak{t}_{i_0}) + \nu' \varkappa (z) \Omega_{\mathfrak{t}_{i_0}} (x) - \varpi (z) \Big) \\
				& \leq H_{r + 1} (x, \mathfrak{t}_{i_0}) \leq \displaystyle\inf_{z \geq n^{-1}} \Big( H^* (x, \mathfrak{t}_{i_0}) + \nu' \varkappa (z) \Omega_{\mathfrak{t}_{i_0}} (x) - \varpi (z) \Big).
				\end{aligned} 
			\end{flalign}
		
			\noindent We will then deduce that the event $\mathscr{E}_{r + 1}^{(3)} (i_0)$ likely holds by taking a union bound over $x$. 
			
			To implement this, first observe that \Cref{assumptiondomega} applies, with the parameters 
			\begin{flalign*} 
				(\varepsilon, \varsigma, \delta; \mathfrak{D}, \mathsf{D}; \mathfrak{s}; \mathfrak{t}_1, \mathfrak{t}_2; \widetilde{\mathsf{h}}, \widetilde{\mathsf{H}}; \xi_1, \xi_2; \zeta_1, \zeta_2; \mu),
			\end{flalign*} 
		
			\noindent there equal to 
			\begin{flalign}
				\label{parameters1} 
					\bigg( \varepsilon_0, \Big( \displaystyle\frac{\varepsilon_0}{20} \Big)^{m + 1}, \frac{\delta}{4}; \mathfrak{R}_j, \mathsf{R}_j; \mathfrak{t}_{i_0}; \mathfrak{t}_i, \mathfrak{t}_{i'}; \mathsf{h}_r, \mathsf{H}_{r + 1}; & \nu_i \varkappa (\lambda), \nu_{i'} \varkappa (\lambda); \nu_i n^{\delta / 240 - 2/3}, \nu_{i'} n^{\delta / 240 - 2 / 3}; \varpi (\lambda) \bigg),
			\end{flalign}
			
			\noindent here, where we recall that $\lambda = - \nu' \Omega_{\mathfrak{t}_{i_0}} (x) \geq n^{-1}$. The verification that this assumption holds is very similar to that in the previous setting, except that the $(\xi; \mu)$-tiltedness condition now follows from the fact that we restricted to the event $\mathscr{E}_r^{(3)} (i) \cap \mathscr{E}_r^{(3)} (i') \subseteq \mathscr{A}_r$. To verify that the parameters \eqref{parameters1} satisfy the inequalities stipulated in \Cref{hhestimate2}, let $d = \min \big\{ |x - x'| : (x', \mathfrak{t}_{i_0}) \in \mathfrak{A} \big\}$; then \Cref{omegad} implies the existence of a constant $c = c(\mathfrak{P}) > 0$ such that $cd^{1/2} < \lambda < c^{-1} d^{1/2}$. Under our choice \eqref{lambda} of $\lambda$ and the definitions \eqref{functionx} of $\varkappa$ and $\varpi$, we have
			\begin{flalign*} 
				\varkappa (\lambda) \geq c n^{\delta / 12 - 1} d^{-1/2} & \geq n^{\delta / 16 - 2/3} d^{-1/2}; \qquad \varkappa (\lambda) \leq c^{-1} n^{\delta / 12 - 1} d^{-1/2} \leq n^{ - 2/3},
			\end{flalign*} 
		
			\noindent where in the second inequality we used the fact that $d \geq n^{\delta / 4 - 2/3}$, since $(x, \mathfrak{t}_{i_0}) \in \mathfrak{L}_-^{\delta / 4}$. 
			
			Thus, we may apply \Cref{hhestimate2} to deduce that $H_{r + 1}$ is $\big( \nu' \varkappa (\lambda), \varpi (\lambda)\big)$-tilted at $(x, \mathfrak{t}_{i_0})$ with overwhelming probability, where $\nu'$ is given by \eqref{nu1}. This implies with overwhelming probability that \eqref{hx0i0lambda} holds, which as mentioned above yields \eqref{hx0i0x} with overwhelming probability. This applies to a single $x$, so from a union bound, it follows that \eqref{hx0i0x} holds simultaneously for all $x \in n^{-2} \mathbb{Z}$ such that $(x, \mathfrak{t}_{i_0}) \in \mathfrak{L}_-^{\delta / 4}$. From the $1$-Lipschitz property of $H_{r + 1}$, we deduce with overwhelming probability that, for all $x \in \mathbb{R}$ such that $(x, \mathfrak{t}_{i_0}) \in \mathfrak{L}_-^{\delta / 4}$, we have  
			\begin{flalign*}
				\displaystyle\sup_{z \geq n^{-1}} \Big( H^* & (x, \mathfrak{t}_{i_0}) + \nu' \varkappa (z) \Omega_{\mathfrak{t}_{i_0}} (x) - \varpi (z) - n^{-2} \Big) \\
				& \leq H_{r + 1} (x, \mathfrak{t}_{i_0}) \leq \displaystyle\inf_{z \geq n^{-1}} \Big( H^* (x, \mathfrak{t}_{i_0}) + \nu' \varkappa (z) \Omega_{\mathfrak{t}_{i_0}} (x) - \varpi (z) + n^{-2} \Big).
			\end{flalign*} 
	
			\noindent Using the fact \eqref{nu1} that $\nu' < \nu_{i_0}$ and that $\Omega_{\mathfrak{t}_{i_0}} (x) \geq n^{-1}$ for $(x, \mathfrak{t}_{i_0}) \in \mathfrak{L}_-^{\delta / 4}$, it follows that $\mathscr{E}_{r + 1}^{(3)} (i_0)$ holds with overwhelming probability. 
			
			This shows that the event $\mathscr{E}_r^{(1)} (i_0) \cap \mathscr{E}_r^{(2)} (i_0) \cap \mathscr{E}_r^{(3)} (i_0)$ holds with overwhelming probability, for any index $0\leq i_0 \leq m$. By a union bound, we deduce that the event $\mathscr{E}_r$ (from \eqref{afe}) holds with overwhelming probability. So, it remains to show that $\mathscr{F}_r$ does as well.			
			
			This will follow from an application of \Cref{estimategamma}; let us verify that $(h, H^*; \mathsf{h}_r; \mathfrak{R}_j, \mathsf{R}_j)$ satisfies the constraints on $(h, H^*;\mathsf{h}; \mathfrak{D}, \mathsf{D})$ listed in \Cref{xhh} and \Cref{xhh2}. First observe, if $u_j \in \partial_{\north} (\mathfrak{R}_j)$ is a tangency location for $\mathfrak{A}$, then $\partial_{\north} (\mathfrak{R}_j) \subset \partial \mathfrak{P}$. Thus, $\mathsf{H}_{r + 1} (nv) = n h(v)$, and $\partial_{\north} (\mathfrak{R}_j)$ is packed with respect to $h$. Otherwise, $\mathfrak{L}$ extends beyond $\partial_{\north} (\mathfrak{R}_j)$ and $\partial_{\so} (\mathfrak{R}_j)$, and so $H^*$ admits an extension beyond the north and south boundaries of $\mathfrak{R}_j$. This shows that $\mathfrak{R}_j$ satisfies the second condition on $\mathfrak{D}$ in \Cref{xhh}. The first, third, and fourth  follow from the adaptedness of $\mathfrak{R}_j$ to $H^*$; the fifth follows from taking the $(\mathfrak{P}, 1)$ there equal to $(\mathfrak{P}, 1)$ here.
						
			To verify that it satisfies \Cref{xhh2}, observe by \eqref{hruhru}, the fact that $\partial_{\ea} (\mathfrak{R}_j)$ and $\partial_{\we} (\mathfrak{R}_j)$ are either subsets of $\partial \mathfrak{P}$ or bounded away from $\overline{\mathfrak{L}}$ gives $H_r = H^*$ along $\partial_{\ea} (\mathfrak{R}_j) \cup \partial_{\we} (\mathfrak{R}_j)$. In particular, $\mathsf{h}_r (nv) = n h_r (v) = n h(v)$ for each $v \in \partial_{\ea} (\mathfrak{R}_j) \cup \partial_{\we} (\mathfrak{R}_j)$. Since $\mathfrak{R}_j$ is adapted to $H^*$, $h$ is constant both $\partial_{\ea} (\mathfrak{R}_j)$ and $\partial_{\we} (\mathfrak{R}_j)$.	Further observe since we have restricted to $\mathscr{E}_r^{(1)} (i) \cap \mathscr{E}_r^{(1)} (i')$ that the edge of $H_r$ is $n^{\delta / 240 - 2/3}$-tilted with respect to $H^*$. This implies that $\mathsf{h}_r (nv) = n h_r (v) = n h(v)$ for any $v \notin \mathfrak{L}$ and $\dist (v, \mathfrak{A}) \geq \mathcal{O} (n^{\delta / 200 - 2/3})$. In particular, this holds for any $v \notin \mathfrak{L}_+^{\delta / 2}$, thereby verifying the second property listed in \Cref{xhh2}; the verification of the third is very similar.
			
			It remains to verify the first one, namely that $\big| \mathsf{h}_r (nv) - n h (nv) \big| \leq n^{\delta / 2}$ for each $v \in \partial \mathfrak{R}_j$. This is already implied by the second property for $v \in \partial_{\ea} (\mathfrak{R}_j) \cup \partial_{\we} (\mathfrak{R}_j)$, so we must show it for $v \in \partial_{\north} (\mathfrak{R}_j) \cup \partial_{\so} (\mathfrak{R}_j)$. We only consider $v \in \partial_{\so} (\mathfrak{R}_j)$, since the case when $v \in \partial_{\north} (\mathfrak{R}_j)$ is entirely analogous. To that end, observe since since we restricted to the event $\mathscr{E}_r^{(3)} (i)$ that 
			\begin{flalign*}
				 \displaystyle\sup_{z \geq n^{-1}} & \Big( H^* (x, \mathfrak{t}_i) + \nu_i \varkappa (z) \Omega_{\mathfrak{t}_i} (x) - \varpi (z) \Big) \\
				& \qquad \qquad \qquad \leq H_r (x, \mathfrak{t}_i) \leq \displaystyle\inf_{z  \geq n^{-1}} \Big( H^* (x, \mathfrak{t}_j) - \nu_i \varkappa (z) \Omega_{\mathfrak{t}_i} (x) + \varpi (z) \Big),
			\end{flalign*}
		
			\noindent holds for each $(x, \mathfrak{t}_i) \in \mathfrak{L}_-^{\delta / 4}$. In particular, we my take $z = -\Omega_{\mathfrak{t}_i} (x) \geq n^{-1}$, which by \eqref{functionx} gives
			\begin{flalign*}
				 \big| H_r (x, \mathfrak{t}_i) - H^* (x, \mathfrak{t}_i) \big| \leq n^{\delta / 12 - 1} + n^{\delta / 12 - 1} \log n \leq 2n^{\delta / 12 - 1} \log n \leq n^{\delta / 2 - 1},
			\end{flalign*}
			
			\noindent for any $(x, \mathfrak{t}_i) \in \partial_{\rm so}(\mathfrak{R}_j)$. This yields the first property listed in \Cref{xhh2}, so that assumption holds. In particular, \Cref{estimategamma} applies (with the $(\mathsf{h}, \mathsf{H})$ there equal to $(\mathsf{h}_r, \mathsf{H}_{r + 1})$ here), implying that $\mathscr{F}_{r + 1}$ holds with overwhelming probability. Hence, $\mathscr{E}_{r + 1} \cap \mathscr{F}_{r + 1} = \mathscr{A}_{r + 1}$ holds with overwhelming probabilty upon restricting to $\mathscr{A}_r$, thereby establishing the proposition if $\mathfrak{R}_j$ is of the first or second type listed in \Cref{Decomposition}. 
			
			It remains to consider the case when $\mathfrak{R}_j = \mathfrak{R} (u_j)$ is of the third type listed in \Cref{Decomposition}, that is, when $u_j$ is a horizontal tangency location of $\mathfrak{A}$ that does not lie on $\partial \mathfrak{P}$. Then, $\mathfrak{R}_j = \mathfrak{D}_1 \cup \mathfrak{D}_2$ for two trapezoids $\mathfrak{D}_1 = \mathfrak{D}_1 (u_j)$ and $\mathfrak{D}_2 (u_j)$ that contain $u_j$ in the interiors of their south and north boundaries, respectively; set $\mathsf{D}_1 = n \mathfrak{D}_1 \subset \mathbb{T}$ and $\mathsf{D}_2 = n \mathfrak{D}_2 \subset \mathbb{T}$. Either $\mathfrak{D}_1$ or $\mathfrak{D}_2$ is disjoint with $\mathfrak{L}$; let us assume the former is, as the proof in the alternative case is entirely analogous. 
			
			Then, the top three boundaries $\partial_{\north} (\mathfrak{D}_1) \cup \partial_{\ea} (\mathfrak{D}_1) \cup \partial_{\ea} (\mathfrak{D}_1)$ of $\mathfrak{D}_1$ are bounded away from $\overline{\mathfrak{D}}$. By \eqref{hruhru}, this yields $h_r (v) = H^* (v)$ for any $v \in \partial_{\north} (\mathfrak{D}) \cup \partial_{\ea} (\mathfrak{D}) \cup \partial_{\we} (\mathfrak{D})$. Moreover $\nabla H^*$ is constant,\footnote{This holds since $\mathfrak{A}$ has no singularities that are simultaneously cusps and tangency locations; recall \Cref{pa}.} and an element of $\{ (0, 0), (1, 0), (1, -1) \}$, along $\partial_{\north} (\mathfrak{D}) \cup \partial_{\ea} (\mathfrak{D}) \cup \partial_{\we} (\mathfrak{D})$. Hence, the same holds for $\nabla H_r$, from which it follows that there is only one height function on $\mathsf{D}_1$ with boundary data $\mathsf{H}_r |_{\mathsf{D}_1}$. Hence, $\mathsf{H}_{r + 1} (nv) = n H^* (v) = \mathsf{H}_r (nv)$, for each $v \in \mathfrak{D}_1$. The inequalities \eqref{fr1}, \eqref{ei1ei2r}, and \eqref{ei3x0} hold deterministically inside $\mathfrak{D}_1$.
			
			Furthermore, the fact that $\mathsf{H}_r (nv) = n H^* (v)$ for $v \in \mathfrak{D}_1$ implies that the $\partial_{\so} (\mathfrak{D}_1)$ is packed with respect to $h_r$; thus, the same statement holds for $\partial_{\north} (\mathfrak{D}_2)$. Therefore, the same reasoning as applied above in the second case for $\mathfrak{R}_j$ listed in \Cref{Decomposition} (when $u$ is a horizontal tangency location of $\mathfrak{A}$ that lies on $\partial \mathfrak{P}$) applies to show that the inequalities \eqref{fr1}, \eqref{ei1ei2r}, and \eqref{ei3x0} hold with overwhelming probability inside $\mathfrak{D}_2$. It follows that these inequalities hold with overwhelming probability on $\mathfrak{D}_1 \cup \mathfrak{D}_2 = \mathfrak{R}_j$. This implies that $\mathscr{A}_{r + 1}$ holds with overwhelming probability, which yields the proposition.
		\end{proof}

	\section{Existence of Tilted Height Functions}

	\label{WalkLimit}
	
	In this section we establish \Cref{omegaxizeta}. Instead of directly producing the tilted height function $\widehat{H}^*$ as described there, we first in \Cref{Functionfalpha} define its complex slope as a solution to the equation \eqref{q0f}, with the function $Q_0$ there modified in an explicit way. In \Cref{Compareg1galpha} we solve this equation perturbatively, and compare its solution to the original one. We then define the tilted height function from its complex slope using \eqref{fh} and establish \Cref{omegaxizeta} in \Cref{ProofFunction}.

	\subsection{Modifying \texorpdfstring{$Q$}{}}
	
	\label{Functionfalpha}
	
	In this section we introduce a function $\mathcal{F}_t (x; \alpha_0)$ that will eventually be the complex slope for our tilted height profile. Throughout this section, we adopt the notation from \Cref{omegaxizeta}, recalling in particular the complex slope $f_t (x)$ associated with $H^*$, defined for $(x, t) \in \mathfrak{L} = \mathfrak{L} (\mathfrak{P})$.
	
	The function $\mathcal{F}_t (x; \alpha_0)$ may be interpreted as the solution to the complex Burgers equation
		\begin{flalign}
			\label{ftxalpha0} 
			\partial_t \mathcal{F}_t (x; \alpha_0) + \partial_x \mathcal{F}_t (x; \alpha_0) \displaystyle\frac{\mathcal{F}_t (x; \alpha_0)}{\mathcal{F}_t (x; \alpha_0) + 1} = 0, \qquad \text{with initial data $\mathcal{F}_0 (x; \alpha_0) = \alpha f_{\mathfrak{t}_0} (x)$},
		\end{flalign}
		
		\noindent for suitable real numbers $\alpha_0 = 1 + \mathcal{O} \big( |\xi_1 + \xi_2| \big)$ and $\mathfrak{t}_0$; stated alternatively, we first time-shift the solution $f_t (x)$ of \eqref{ftx} by $\mathfrak{t}_0$ and then multiply its initial data by a ``drift'' $\alpha_0$. However, making this precise would involve justifying the existence and uniqueness of a solution to \eqref{ftxalpha0}. To circumvent this, we instead define $\mathcal{F}_t (x; \alpha_0)$ as the solution to an ``$\alpha_0$-deformation'' of the equation \eqref{q0alphaf}. 
	
		To implement this, we define real numbers $\mathfrak{t}_0$ and $\alpha_0$ by
	\begin{flalign}
		\label{t0alpha} 
		\mathfrak{t}_0 = \displaystyle\frac{\xi_2 \mathfrak{t}_1 - \xi_1 \mathfrak{t}_2}{\xi_2 - \xi_1}; \qquad \alpha_0 = \displaystyle\frac{\xi_2 - \xi_1}{\mathfrak{t}_2 - \mathfrak{t}_1} + 1 = \displaystyle\frac{\xi_1}{\mathfrak{t}_1 - \mathfrak{t}_0} + 1 = \displaystyle\frac{\xi_2}{\mathfrak{t}_2 - \mathfrak{t}_0} + 1,
	\end{flalign} 

	\noindent so that $\alpha_0 = 1 + \mathcal{O} \big( |\xi_1 + \xi_2| \big)$. We next introduce a time-shifted variant of $f_t (x)$ given by
	\begin{flalign}
		\label{fsx1} 
		\mathcal{F}_s (x; 1) = f_{s + \mathfrak{t}_0} (x), \qquad \text{whenever $(x, s + \mathfrak{t}_0) \in \overline{\mathfrak{L}}$}.
	\end{flalign}

	\noindent Although the complex slope for the tilted height function $\widehat{H}^*$ from \Cref{omegaxizeta} will eventually be related to an $\alpha_0$-deformation of $\mathcal{F}_s (x; 1)$ with $\alpha_0$ given explicitly by \eqref{t0alpha}, it will be useful to define this deformation (denoted by $\mathcal{F}_t (x; \alpha)$) for any $\alpha \in \mathbb{R}$ with $|\alpha - 1|$ sufficiently small. To that end, recalling the rational function $Q: \mathbb{C}^2 \rightarrow \mathbb{C}$ satisfying \eqref{e:qfh} with respect to $f_t$, we also define its time-shift $\mathcal{Q}_1 $and $\alpha$-deformation $\mathcal{Q}_{\alpha}$, for any $\alpha \in \mathbb{R}$, by
	\begin{flalign*}
		\mathcal{Q}_1 (u, v) = Q \bigg( u, v - \displaystyle\frac{\mathfrak{t}_0 u}{u + 1} \bigg); \qquad \mathcal{Q}_{\alpha} (u, v) = \displaystyle\frac{u + 1}{\alpha^{-1} u + 1} \mathcal{Q}_1 (\alpha^{-1} u, v),
	\end{flalign*}
	
	\noindent observing in particular that \eqref{e:qfh} and \eqref{fsx1} together imply 
	\begin{flalign} 
		\label{q1equation} 
		\mathcal{Q}_1 \bigg( \mathcal{F}_t (x; 1), x - \frac{t \mathcal{F}_t (x; 1)}{\mathcal{F}_t (x; 1) + 1}\bigg) = 0. 
	\end{flalign} 
	
	If $|\alpha - 1|$ is sufficiently small (in a way only dependent on $\mathfrak{P}$), this implies the existence of an analytic function $\mathcal{F}_t (x; \alpha)$, defined for $(x, t)$ in an open subset of $\mathfrak{L}$, that is continuous in $\alpha$ and satisfies
	\begin{flalign}
		\label{qalphaequation}
		\mathcal{Q}_{\alpha} \bigg( \mathcal{F}_t (x; \alpha), x - \displaystyle\frac{t \mathcal{F}_t (x; \alpha)}{\mathcal{F}_t (x; \alpha) + 1} \bigg) = 0.
	\end{flalign}
	
	\noindent For example, recall from the third part of \Cref{pa1} that there exists a real analytic function $\mathcal{Q}_{0; 1}$, locally defined around any solution of \eqref{q1equation} (obtained by solving \eqref{q1equation}), such that 
	\begin{flalign*}
		\mathcal{Q}_{0; 1} \big( \mathcal{F}_t (x; 1) \big) = x \big( \mathcal{F}_t (x; 1) + 1 \big) - t \mathcal{F}_t (x; 1).
	\end{flalign*} 
	
	\noindent Then, we may set
	\begin{flalign}
		\label{q0alphaz} 
		\mathcal{Q}_{0; \alpha} (u) = \displaystyle\frac{u + 1}{\alpha^{-1}u + 1} \mathcal{Q}_{0; 1} (\alpha^{-1} u),
	\end{flalign}
	
	\noindent for $u$ in the domain of $\mathcal{Q}_{0; 1}$ and let $\mathcal{F}_t (x; \alpha)$ denote the root of 
	\begin{flalign}
		\label{q0alphaf}
		\mathcal{Q}_{0; \alpha} \big( \mathcal{F}_t (x; \alpha) \big) = x \big( \mathcal{F}_t (x; \alpha) + 1 \big) - t \mathcal{F}_t (x; \alpha),
	\end{flalign}
	
	\noindent  chosen so that it is continuous in $\alpha$; it is directly verified that it satisfies \eqref{qalphaequation}. Such a root is well-defined on a nonempty open subset of $\mathfrak{L}$ that contains no double root of \eqref{qalphaequation}, or equivalently of \eqref{q0alphaf} (such a subset exists for $|\alpha - 1|$ sufficiently small). Since any double root of \eqref{qalphaequation} is real, we may extend $\mathcal{F}_t (x; \alpha)$ to the $\alpha$-deformed liquid region $\mathfrak{L}_{\alpha}$ and its arctic curve $\mathfrak{A}_{\alpha}$, defined by
	\begin{flalign*}
		\mathfrak{L}_{\alpha} = \big\{ (x, t) \in \mathbb{R}^2 : \mathcal{F}_t (x; \alpha) \in \mathbb{H}^- \big\}; \qquad \mathfrak{A}_{\alpha} = \partial \mathfrak{L}_{\alpha},
	\end{flalign*}
	
	\noindent where we observe for $|\alpha - 1|$ sufficiently small that $\mathfrak{L}_{\alpha}$ is simply connected since $\mathfrak{L}$ (and thus $\mathfrak{L}_1$) is. 
	
	We next have the below lemma that, given a solution of an equation of the type \eqref{q0alphaf}, evaluates its derivatives with respect to $x$ and $t$, and shows that it satisfies the complex Burgers equation. Its proof essentially follows from a Taylor expansion and will be provided in \Cref{Equation} below.

	\begin{lem}
		
		\label{fderivativeq} 
		
		Fix $(x_0, t_0) \in \mathbb{R}^2$;  let $\mathcal{F}_t (x) $ denote a function that is real analytic in a neighborhood of $(x_0, t_0)$, and let $\mathcal{Q}_0$ denote a function that is real analytic in a neighborhood of $\mathcal{F}_{t_0} (x_0)$. Assume in a neighborhood of $(x_0, t_0)$ that 
		\begin{flalign}
			\label{q0f2}
			\mathcal{Q}_0 \big( \mathcal{F}_t (x) \big) = x \big( \mathcal{F}_t (x) + 1 \big) - t \mathcal{F}_t (x).
		\end{flalign}
		
		\noindent Then for all $(x ,t)$ in a neighborhood of $(x_0, t_0)$ we have 
		\begin{flalign}
			\label{ftxderivativeq}
			\partial_x \mathcal{F}_t (x) = \displaystyle\frac{\mathcal{F}_t (x) + 1}{\mathcal{Q}_0' \big( \mathcal{F}_t (x) \big) - x + t}; \qquad \partial_t \mathcal{F}_t (x) = - \displaystyle\frac{\mathcal{F}_t (x)}{\mathcal{Q}_0' \big( \mathcal{F}_t (x) \big) - x + t},
		\end{flalign}
		\noindent and in particular
		\begin{flalign}
			\label{xfequation} 
			\partial_t \mathcal{F}_t (x) + \partial_x \mathcal{F}_t (x) \displaystyle\frac{\mathcal{F}_t (x)}{\mathcal{F}_t (x) + 1} = 0.
		\end{flalign}
	\end{lem}

	\subsection{Comparison Between \texorpdfstring{$\mathcal{F}_t (x; \alpha)$}{} and \texorpdfstring{$\mathcal{F}_t (x; 1)$}{}}

	\label{Compareg1galpha} 
	
	In this section we provide two estimates comparing $\mathcal{F}_t (x; \alpha)$ and $\mathcal{F}_t (x; 1)$. The first (given by \Cref{ftalphaft1} below) compares their logarithms, where in what follows, we take the branch of the logarithm to be so that $\Imaginary \log u \in [-\pi, \pi)$. The second (given by \Cref{x0fq} below) compares the endpoints of their arctic boundaries. These results will eventually be the sources of the quantities $\Omega$ and $\Upsilon$ from \eqref{omegasx}.
	
	\begin{lem} 
		
		\label{ftalphaft1} 
		
		Suppose $|\alpha - 1|$ is sufficiently small, and $v = (x, t) \in \mathfrak{L}_1 \cap \mathcal{L}_{\alpha}$ is any point bounded away from a cusp or tangency location of $\mathfrak{A}_1$. Setting $d = \dist (v, \mathfrak{A}_1)$ we have the following estimates, where the implicit constants in the error below only depend on the first three derivatives of $\mathcal{Q}_{0; 1}$ at $\mathcal{F}_t (x; 1)$. 
		
		\begin{enumerate} 
			\item If $d \geq |\alpha - 1|$, then 
		\begin{flalign*}
			 \log \mathcal{F}_t (x; \alpha) - \log \alpha - \log \mathcal{F}_t (x; 1) = t (1 - \alpha) \partial_x \bigg( \displaystyle\frac{\mathcal{F}_t (x; 1)}{\mathcal{F}_t (x; 1) + 1} \bigg) \Big( 1 + \mathcal{O} \big( d^{-1} |\alpha - 1| + |\alpha - 1|^{1/2} \big) \Big).
		\end{flalign*}
		
		\item For any $d > 0$, we have $\big| \log \mathcal{F}_t (x; \alpha) - \log \alpha - \log \mathcal{F}_t (x; 1) \big| = \mathcal{O} \big( |\alpha - 1|^{1/2} \big)$.
		\end{enumerate} 
		
	\end{lem} 
	
	\begin{proof}
		
	Letting $\mathcal{F} = \mathcal{F}_t (x; 1)$ and $\mathcal{F}' = \mathcal{F}_t (x; \alpha)$, we have by \eqref{q0alphaf} that 
	\begin{flalign}
		\label{q01f0alphaf}
		\mathcal{Q}_{0; 1} (\mathcal{F}) = x (\mathcal{F} + 1) - t \mathcal{F}; \qquad \mathcal{Q}_{0; \alpha} (\mathcal{F}') = x (\mathcal{F}' + 1) - t\mathcal{F}'.
	\end{flalign}

	\noindent Next, \eqref{q0alphaz} implies for $( u, x - tu/(u + 1) )$ in a neighborhood of $( \mathcal{F}', x - t \mathcal{F}'/(\mathcal{F}' + 1))$ that
	\begin{flalign*}
		\mathcal{Q}_{0; \alpha} (u) = \displaystyle\frac{u + 1}{\alpha^{-1} u + 1} \mathcal{Q}_{0; 1} (\alpha^{-1} u).
	\end{flalign*}
	
	\noindent Letting $\widetilde{\mathcal{F}} = \alpha^{-1} \mathcal{F}'$, we deduce from the second statement of \eqref{q01f0alphaf} that 
	\begin{flalign*}
		\mathcal{Q}_{0; 1} (\widetilde{\mathcal{F}}) + \displaystyle\frac{(\alpha - 1) \widetilde{\mathcal{F}}}{\widetilde{\mathcal{F}} + 1} \mathcal{Q}_{0; 1} (\widetilde{\mathcal{F}}) = \mathcal{Q}_{0; \alpha} (\mathcal{F}') = x (\widetilde{\mathcal{F}} + 1) - t \widetilde{\mathcal{F}}+ (\alpha - 1) (x - t) \widetilde{\mathcal{F}}.
	\end{flalign*}

	\noindent Together with the first statement of \eqref{q01f0alphaf}, this gives 
	\begin{flalign}
		\label{qfqf2} 
		\mathcal{Q}_{0; 1} (\widetilde{\mathcal{F}}) - \mathcal{Q}_{0; 1} (\mathcal{F}) = (\widetilde{\mathcal{F}} - \mathcal{F})(x - t) + \displaystyle\frac{(1 - \alpha) \widetilde{\mathcal{F}}}{\widetilde{\mathcal{F}} + 1} \mathcal{Q}_{0; 1} (\widetilde{\mathcal{F}}) + (\alpha - 1) (x - t) \widetilde{\mathcal{F}}.
	\end{flalign}
	
	\noindent By a Taylor expansion, we have 
	\begin{flalign*} 
		\mathcal{Q}_{0; 1} (\widetilde{\mathcal{F}}) - \mathcal{Q}_{0; 1} (\mathcal{F}) = (\widetilde{\mathcal{F}} - \mathcal{F}) \mathcal{Q}_{0; 1}' (\mathcal{F}) + \frac{1}{2} (\widetilde{\mathcal{F}} - \mathcal{F})^2 \mathcal{Q}_{0; 1}'' (\mathcal{F}) + \mathcal{O} \big( |\widetilde{\mathcal{F}} - \mathcal{F}|^3 \big),	
	\end{flalign*} 

	\noindent and which by \eqref{qfqf2} implies, since $\mathcal{F}$ is bounded away from $\{ -1, 0, \infty \}$ (as $(x, t)$ is bounded away from a tangency location of $\mathfrak{A}_1$), that
	\begin{flalign}
		\label{qffestimate} 
		\begin{aligned}
		(\widetilde{\mathcal{F}} - \mathcal{F}) \big( \mathcal{Q}_{0; 1}' (& \mathcal{F} )- x + t \big) + \displaystyle\frac{1}{2} (\widetilde{\mathcal{F}} - \mathcal{F})^2 \mathcal{Q}_{0; 1}'' (\mathcal{F}) \\
		& = (\alpha - 1) \bigg( (x - t) \mathcal{F} - \displaystyle\frac{\mathcal{F}}{\mathcal{F} + 1} \mathcal{Q}_{0; 1} (\mathcal{F}) \bigg) + \mathcal{O} \big( |\alpha - 1| |\widetilde{\mathcal{F}} - \mathcal{F}| + |\widetilde{\mathcal{F}} - \mathcal{F}|^3 \big) \\
		& = \displaystyle\frac{(1 - \alpha) t\mathcal{F}}{\mathcal{F} + 1} + \mathcal{O} \big( |\alpha - 1| |\widetilde{\mathcal{F}} - \mathcal{F}| + |\widetilde{\mathcal{F}} - \mathcal{F}|^3 \big),
		\end{aligned} 
	\end{flalign}
	
	\noindent where to deduce the last equality we used the first statement of \eqref{q01f0alphaf}. 
	
	Now, if $d < \delta$ for some $\delta = \delta (\mathfrak{P}) > 0$, then $\mathcal{Q}_{0; 1}'' (\mathcal{F})$ is bounded away from $0$, since $(x, t)$ is bounded away from a cusp of $\mathfrak{A}_1$; so, \eqref{qffestimate} implies $|\widetilde{\mathcal{F}} - \mathcal{F}| = \mathcal{O} \big( |\alpha - 1|^{1/2} \big)$. If instead $d \geq \delta$, then $\Imaginary \mathcal{Q}_{0; 1}' (\mathcal{F})$ is bounded away from $0$, which implies that $|\widetilde{\mathcal{F}} - \mathcal{F}| = \mathcal{O} \big( |\alpha - 1| \big)$. In either case, this gives $|\widetilde{\mathcal{F}} - \mathcal{F}| = \mathcal{O} \big( |\alpha - 1|^{1/2} \big)$, which verifies the second statement of the lemma.
	
	Next, observe from the first statement of \eqref{ftxderivativeq} that 
	\begin{flalign}
		\label{qf2} 
		\frac{1}{\mathcal{Q}_{0; 1}' (\mathcal{F}) - x + t} = \displaystyle\frac{\partial_x \mathcal{F}}{\mathcal{F} + 1}.
	\end{flalign}	

	By the square root decay of $\mathcal{F}_t (x)$ as $(x, t)$ nears $\mathfrak{A}_1$ (see \Cref{derivativeh}  below), there exist constants $c = c(\mathfrak{P}) > 0$ and $C = C(\mathfrak{P}) > 1$ such that $c d^{-1/2} < \partial_x \mathcal{F}_t (x; 1) < C d^{-1/2}$. After decreasing $c$ and increasing $C$ if necessary, \eqref{qf2} then implies
	\begin{flalign*} 
		c d^{1/2} < \mathcal{Q}_{0; 1}' (\mathcal{F}) - x + t < C d^{1/2}.
	\end{flalign*} 
	
	\noindent This, together with \eqref{qffestimate}; the bound $|\widetilde{\mathcal{F}} - \mathcal{F}| = \mathcal{O} \big( |\alpha - 1|^{1/2} \big)$; and the fact that $\mathcal{Q}_{0; 1}'' (\mathcal{F})$ is bounded away from $0$ for $(x, t)$ sufficiently close to $\mathfrak{A}_1$ quicky gives for $d \geq |\alpha - 1|$ that
	\begin{flalign*}
		\frac{\widetilde{\mathcal{F}} - \mathcal{F}}{\mathcal F} = t (1 - \alpha) \displaystyle\frac{\partial_x \mathcal{F}}{(\mathcal{F} + 1)^2} \Big( 1 + \mathcal{O} \big( d^{-1} |\alpha - 1| + |\alpha - 1|^{1 / 2} \big) \Big).
	\end{flalign*}

	\noindent It follows that
	\begin{flalign*}
		\log \bigg( \displaystyle\frac{\widetilde{\mathcal{F}}}{\mathcal{F}} \bigg) & =  \displaystyle\frac{\widetilde{\mathcal{F}}}{\mathcal{F}} - 1 + \mathcal{O} \big( (\widetilde{\mathcal{F}} - \mathcal{F})^2 \big) \\
		& = t (1 - \alpha) \Bigg( \displaystyle\frac{\partial_x \mathcal{F}_t (x; 1)}{\big(\mathcal{F}_t (x; 1) + 1 \big)^2} \Bigg) \Big( 1 + \mathcal{O} \big( d^{-1} |\alpha - 1| + |\alpha - 1|^{1/2} \big) \Big),
	\end{flalign*}

	\noindent which implies the first statement of the lemma.
	\end{proof}

Next we show that if no cusp or tangency location of $\mathcal{\fA}_1$ is of the form $(x,t)$ with $x\in \mathbb R$, then the time slices $\big\{ x: (x, t) \in \mathfrak{L}_1 \big\}$ and $\big\{ x: (x, t) \in \mathfrak{L}_\alpha \big\}$ contain the same number of intervals; we also estimate the distance between their endpoints. 

	\begin{lem}
		
		\label{x0fq}
		
		The following holds for $|\alpha - 1|$ sufficiently small. 
		\begin{enumerate}
		\item
		Let $(x_0, t) \in \mathfrak{A}_1$ denote a right (or left) boundary point of $\mathfrak{A}_1$, bounded away from a cusp or horizontal tangency location of $\mathfrak{A}_1$. Then, there exists $(x_0', t) \in \mathfrak{A}_{\alpha}$, which is a right (or left, respectively) boundary point of $\mathfrak{A}_{\alpha}$ so that 
		\begin{flalign}\label{e:x0change}
			x_0' - x_0 = t (\alpha - 1) \displaystyle\frac{\mathcal{F}_t (x_0)}{\big( \mathcal{F}_t (x_0) + 1\big)^2} + \mathcal{O} \big( |\alpha - 1|^2 \big).
		\end{flalign}
		\item
			 Fix $t \in \mathbb{R}$; suppose that $\big\{ x: (x, t) \in \mathfrak{L}_1 \big\}$ is a union of $k \geq 1$ disjoint open intervals $(x_1, x_1') \cup (x_2, x_2') \cup \cdots \cup (x_k, x_k')$, and that it is bounded away from a cusp or horizontal tangency location of $\mathfrak{A}_1$. 
		  Then, $\big\{ x : (x, t) \in \mathfrak{L}_\alpha\big\}$ is also a union of $k$ disjoint open intervals $(\widehat{x}_1, \widehat{x}_1') \cup (\widehat{x}_2, \widehat{x}_2') \cup \cdots \cup (\widehat{x}_k, \widehat{x}_k')$. Moreover, for any index $j \in [1, k]$, we have
		\begin{flalign}\begin{split}
			\label{xjdiff}
		&\widehat{x}_j - x_j = t (\alpha - 1) \displaystyle\frac{\mathcal{F}_t (x_j)}{\big( \mathcal{F}_t (x_j) + 1\big)^2} + \mathcal{O} \big( |\alpha - 1|^2 \big)\\
		& \widehat{x}_j' - x_j' = t (\alpha - 1) \displaystyle\frac{\mathcal{F}_t (x'_j)}{\big( \mathcal{F}_t (x'_j) + 1\big)^2} + \mathcal{O} \big( |\alpha - 1|^2 \big), 
		\end{split}\end{flalign}
\end{enumerate}

	\end{lem}

	\begin{proof} 
	For the first statement, we will only prove the case that $(x_0,t)$ is a left boundary point of $\fA_1$, as the proof of the case that it is a right boundary point is entirely analogous.
		 Observe by the second and third parts of \Cref{pa1} that $(x, t) \in \mathfrak{A}_{\alpha}$ if and only if 
		\begin{flalign*}
			\mathcal{Q}_{0; \alpha} \big( \mathcal{F}_t (x; \alpha) \big) = x \big( \mathcal{F}_t (x; \alpha) + 1 \big) - t \mathcal{F}_t (x; \alpha); \qquad \mathcal{Q}_{0; \alpha}' \big( \mathcal{F}_t (x; \alpha) \big) = x - t,
		\end{flalign*}
	
		\noindent that is, if and only if 
		\begin{flalign}
			\label{q0alphaa} 
			\mathcal{Q}_{0; \alpha} \big( \mathcal{F}_t (x; \alpha) \big) = \big( \mathcal{F}_t (x; \alpha) + 1 \big) \mathcal{Q}_{0; \alpha}' \big( \mathcal{F}_t (x; \alpha) \big) + t; \qquad x = \mathcal{Q}_{0; \alpha}' \big( \mathcal{F}_t (x; \alpha) \big) + t.
		\end{flalign}
		
		\noindent In particular, \eqref{q0alphaa} holds if $(x, \alpha) = (x_0, 1)$. Let us produce a solution to \eqref{q0alphaa} for $\alpha$ close to $1$ by perturbing around $x_0$. To that end, abbreviate $\mathcal{F} = \mathcal{F} (x_0; 1)$; let $x$ be close to $x_0$; and abbreviate $\mathcal{F}' = \mathcal{F} (x; \alpha)$. Then recall from the first statement of \eqref{q0alphaf} and \eqref{q0alphaz} that for $|\alpha - 1|$ sufficiently small and $\big( u, x -tu/(u + 1) \big)$ in a neighborhood of $( \mathcal{F}, x - t \mathcal{F}/(\mathcal{F} + 1))$, we have
	\begin{flalign*}
		\mathcal{Q}_{0; \alpha} (u) = \displaystyle\frac{u + 1}{\alpha^{-1} u + 1} \mathcal{Q}_{0; 1} (\alpha^{-1} u) = \alpha \mathcal{Q}_{0; 1} (\alpha^{-1} u) + \displaystyle\frac{\alpha (1 - \alpha)}{u + \alpha} \mathcal{Q}_{0; 1} (\alpha^{-1} u),
	\end{flalign*} 

	\noindent so that
	\begin{flalign}
		\label{q0alphaderivative}
		\mathcal{Q}_{0; \alpha}' (u) = \mathcal{Q}_{0; 1}' (\alpha^{-1} u) + \displaystyle\frac{\alpha - 1}{u + \alpha} \bigg( \displaystyle\frac{\alpha}{u + \alpha} \mathcal{Q}_{0; 1} (\alpha^{-1} u) - \mathcal{Q}_{0; 1}' (\alpha^{-1} u) \bigg).
	\end{flalign}
	
	\noindent Thus, denoting $\widetilde{\mathcal{F}} = \alpha^{-1} \mathcal{F}'$, the first statement of \eqref{q0alphaa} holds if and only if 
	\begin{flalign*}
		\displaystyle\frac{\alpha \widetilde{\mathcal{F}} + 1}{\widetilde{\mathcal{F}} + 1} \mathcal{Q}_{0; 1} (\widetilde{\mathcal{F}}) = (\alpha \widetilde{\mathcal{F}} + 1) \Bigg( \mathcal{Q}_{0; 1}' (\widetilde{\mathcal{F}}) + \displaystyle\frac{\alpha - 1}{\alpha \widetilde{\mathcal{F}} + \alpha} \bigg( \displaystyle\frac{1}{\widetilde{\mathcal{F}} + 1} \mathcal{Q}_{0; 1} (\widetilde{\mathcal{F}}) - \mathcal{Q}_{0; 1}' (\widetilde{\mathcal{F}}) \bigg) \Bigg) + t.
	\end{flalign*}
	
	\noindent This is equivalent to 
	\begin{flalign*}
		\mathcal{Q}_{0; 1} (\widetilde{\mathcal{F}}) & = (\widetilde{\mathcal{F}} + \alpha^{-1}) \mathcal{Q}_{0; 1}' (\widetilde{\mathcal{F}}) + \displaystyle\frac{\alpha - 1}{\alpha \widetilde{\mathcal{F}} + \alpha} \mathcal{Q}_{0; 1} (\widetilde{\mathcal{F}}) + \displaystyle\frac{t (\widetilde{\mathcal{F}} + 1)}{\alpha \widetilde{\mathcal{F}} + 1} \\
		& = (\widetilde{\mathcal{F}} + 1) \mathcal{Q}_{0; 1}' (\widetilde{\mathcal{F}}) + t + (\alpha - 1) \Bigg( \displaystyle\frac{\mathcal{Q}_{0; 1} (\widetilde{\mathcal{F}})}{\alpha \widetilde{\mathcal{F}} + \alpha} - \displaystyle\frac{t \widetilde{\mathcal{F}}}{\alpha \widetilde{\mathcal{F}} + 1} - \alpha^{-1} \mathcal{Q}_{0; 1}' (\widetilde{\mathcal{F}}) \Bigg),
	\end{flalign*}
	
	\noindent which upon subtracting from the $(x, \alpha) = (x_0, 1)$ case of \eqref{q0alphaa} is true if and only if  
	\begin{flalign}
		\label{q01fq0alphaf2}
		\begin{aligned}
		\mathcal{Q}_{0; 1} (\mathcal{F}) - (\mathcal{F} + 1) \mathcal{Q}_{0; 1}' (\mathcal{F}) - \mathcal{Q}_{0; 1} ( & \widetilde{\mathcal{F}}) - (\widetilde{\mathcal{F}} + 1) \mathcal{Q}_{0; 1}' (\widetilde{\mathcal{F}}) \\
		& = (1 - \alpha) \Bigg( \displaystyle\frac{\mathcal{Q}_{0; 1} (\widetilde{\mathcal{F}})}{\alpha \widetilde{\mathcal{F}} + \alpha} - \displaystyle\frac{t \widetilde{\mathcal{F}}}{\alpha \widetilde{\mathcal{F}} + 1} - \alpha^{-1} \mathcal{Q}_{0; 1}' (\widetilde{\mathcal{F}}) \Bigg). 
		\end{aligned} 
	\end{flalign}
	
	Now observe that the derivative of $\mathcal{Q}_{0; 1} (z) - (z + 1) \mathcal{Q}_{0; 1}' (z)$ is $-(z + 1) \mathcal{Q}_{0; 1}'' (z)$, which is bounded away from $0$ for $z = \mathcal{F}$, since $(x, t)$ is bounded away from a horizontal tangency location or singularity of $\mathfrak{A}$. Hence, the implicit function theorem implies that, for $|\alpha - 1|$ sufficiently small, \eqref{q01fq0alphaf2} admits a solution with $|\widetilde{\mathcal{F}} - \mathcal{F}| = \mathcal{O} \big( |\alpha - 1| \big)$. Taylor expanding then gives
	\begin{flalign}
		\label{falphaf1} 
		\begin{aligned} 
		(\widetilde{\mathcal{F}} - \mathcal{F}) (\mathcal{F} + 1) \mathcal{Q}_{0; 1}'' (\mathcal{F}) & = \displaystyle\frac{1 - \alpha}{\mathcal{F} + 1} \big( \mathcal{Q}_{0; 1} (\mathcal{F}) - (\mathcal{F} + 1) \mathcal{Q}_{0; 1}' (\mathcal{F}) - t\mathcal{F} \big) + \mathcal{O} \big( |\alpha - 1|^2 \big) \\
		& = \displaystyle\frac{t (\alpha - 1) (\mathcal{F} - 1)}{\mathcal{F} + 1} + \mathcal{O} \big( |\alpha - 1|^2 \big),
		\end{aligned} 
	\end{flalign}
	
	\noindent where in the last equality we used the first statement of \eqref{q0alphaa}. Inserting this into the second statement of \eqref{q0alphaa} and applying \eqref{q0alphaderivative} yields
	\begin{flalign*}
		x - x_0 & = \mathcal{Q}_{0; \alpha}' (\mathcal{F}') - \mathcal{Q}_{0; 1}' (\mathcal{F}) \\
		& = \mathcal{Q}_{0; 1}' (\widetilde{\mathcal{F}}) - \mathcal{Q}_{0; 1}' (\mathcal{F}) + \displaystyle\frac{\alpha - 1}{\alpha \widetilde{\mathcal{F}} + \alpha} \bigg( \displaystyle\frac{\mathcal{Q}_{0; 1} (\widetilde{\mathcal{F}})}{\widetilde{\mathcal{F}} + 1} - \mathcal{Q}_{0; 1}' (\widetilde{\mathcal{F}}) \bigg) \\
		& = (\widetilde{\mathcal{F}} - \mathcal{F}) \mathcal{Q}_{0; 1}'' (\mathcal{F}) + \displaystyle\frac{\alpha - 1}{\mathcal{F} + 1} \bigg( \displaystyle\frac{\mathcal{Q}_{0; 1} (\mathcal{F})}{\mathcal{F} + 1} - \mathcal{Q}_{0; 1}' (\mathcal{F}) \bigg) + \mathcal{O} \big( |\alpha - 1|^2 \big) \\
		& = \displaystyle\frac{t (\alpha - 1) \mathcal{F}}{(\mathcal{F} + 1)^2} + \mathcal{O} \big( |\alpha - 1|^2 \big),
	\end{flalign*}

	\noindent where in the last equality we applied \eqref{falphaf1} and first statement of \eqref{q0alphaa}. 
	Setting $x_0' = x$ implies the first statement of \eqref{x0fq}.
	
	The second statement of \Cref{x0fq} that $\big\{ x : (x, y) \in \mathfrak{L}_\alpha\big\}$ is a union of $k$ disjoint open intervals follows from the first statement and the fact that $\mathcal L_\alpha$ is continuous in $\alpha$. The estimates \eqref{xjdiff} follow from \eqref{e:x0change} and the assumption that $\big\{ x: (x, t) \in \mathfrak{L}_1 \big\}$ is bounded away from a cusp or horizontal tangency location of $\mathfrak{A}_1$.
	\end{proof}

	\subsection{Proof of \Cref{omegaxizeta}}
	
	\label{ProofFunction} 
	
	In this section we establish \Cref{omegaxizeta}; throughout, we recall the notation from that proposition and from \Cref{Functionfalpha}. In particular, $\mathfrak{t}_0$ and $\alpha_0$ are given by \eqref{t0alpha}. We first define a height function $H_{\alpha}$ from $\mathcal{F}_t (x; \alpha)$ for $\alpha \in \mathbb{R}$ with $|\alpha - 1|$ sufficiently small, whose complex slope is given by $\mathcal{F}_{t - \mathfrak{t}_0} (x; \alpha)$. The eventual tilted height function $\widehat{H}^* : \overline{\mathfrak{D}} \rightarrow \mathbb{R}$ given by \Cref{omegaxizeta} will be defined by setting $\widehat{H}^* = H_{\alpha_0}$. 
	
	To that end, first set $H_{\alpha} (v_0) = H^* (v_0)$ for $v_0 = \big( \mathfrak{a} (\mathfrak{t}_1), \mathfrak{t}_1 \big)$ equal to the southwest corner of $\partial \mathfrak{D}$. To define $H_{\alpha}$ on the remainder of $\overline{\mathfrak{D}}$, it suffices to fix its gradient almost everywhere. Similarly to in \eqref{fh}, for any $u = (x, s) \in \overline{\mathfrak{D}}$ such that $(x, s - \mathfrak{t}_0) \in \mathfrak{L}_{\alpha}$, define $\nabla H_{\alpha} (u) = \big( \partial_x H_{\alpha} (u), \partial_y H_{\alpha} (u) \big)$ by setting
		\begin{flalign}
			\label{halphau}
			\partial_x H_{\alpha} (u) = - \pi^{-1} \arg^* \mathcal{F}_{s - \mathfrak{t}_0} (x; \alpha); \qquad \partial_y H_{\alpha} (u) = \pi^{-1} \arg^* \big( \mathcal{F}_{s - \mathfrak{t}_0} (x; \alpha) + 1 \big),
		\end{flalign}
	
		\noindent where we observe that we are implementing the same shift by $\mathfrak{t}_0$ as in \eqref{fsx1}. 
		
		If instead $u = (x, s) \in \overline{\mathfrak{D}}$ satisfies $(x, s - \mathfrak{t}_0) \notin \mathfrak{L}_{\alpha}$, then define $\nabla H_{\alpha} (u)$ as follows. 
		\begin{enumerate} 
			\item For $\alpha = 1$, set $\nabla H_{\alpha} (u) = \nabla H^* (u) \in \big\{ (0, 0), (1, 0), (1, -1) \big\}$.
			\item For $\alpha \ne 1$ with $|\alpha - 1|$ sufficiently small, define $\nabla H_{\alpha} (u)$ so that it remains among the three frozen slopes $\big\{ (0, 0), (1, 0), (1, -1) \big\}$, and is continuous in $\alpha$ at almost every $u \in \overline{\mathfrak{D}}$.
		\end{enumerate} 
	
		For $|\alpha - 1|$ sufficiently small, the existence of $\nabla H_{\alpha} (u)$ satisfying these properties follows from the fact that $\mathfrak{A}_{\alpha} = \partial \mathfrak{L}_{\alpha}$ deforms continuously in $\alpha$ (by \Cref{x0fq}); alternatively, it can be viewed as a consequence of \cite[Remark 8.6]{DMCS}. This determines $\nabla H_{\alpha} (u)$ for almost all $u \in \overline{\mathfrak{D}}$, fixing $H_{\alpha} : \overline{\mathfrak{D}} \rightarrow \mathbb{R}$.
		
		We then define $\widehat{H}^* : \overline{\mathfrak{D}} \rightarrow \mathbb{R}$ and $\widehat{h} : \partial \mathfrak{D} \rightarrow \mathbb{R}$ by setting 
		\begin{flalign*} 
			\widehat{H}^* = H_{\alpha_0}; \qquad \widehat{h} = \widehat{H}^* |_{\partial \mathfrak{D}}.
		\end{flalign*} 
		
		\noindent By \Cref{fderivativeq}, the complex slope $\mathcal{F}_t (x; \alpha_0)$ associated with $\widehat{H}^*$ satisfies the complex Burgers equation \eqref{ftx}. 
		We will later use \cite[Remark 8.6]{DMCS} or \cite[Theorem 8.3]{DMCS} to show that $\widehat{H}^* \in \Adm (\mathfrak{D}; \widehat{h})$ is a maximizer of $\mathcal{E}$.

				 Recalling the liquid region $\widehat{\mathfrak{L}}$ associated with $\widehat{H}^*$ from \Cref{omegaxizeta}, observe that 
		\begin{flalign}
			\label{ll} 
			\widehat{\mathfrak{L}} = \big\{ u \in \overline{\mathfrak{D}} : u - (0, \mathfrak{t}_0) \in \mathfrak{L}_{\alpha} \big\}.
		\end{flalign} 
	
		Given these points, we can now establish \Cref{omegaxizeta}. In the below, we will frequently use the identity  
		\begin{flalign} 
			\label{omegayalphat} 
			(\alpha_0 - 1) (y - \mathfrak{t}_0) = \omega (y), \qquad \text{for any $y \in \mathbb{R}$},
		\end{flalign}
	
		\noindent which follows from \eqref{omegay} and \eqref{t0alpha}.
		
		\begin{proof}[Proof of \Cref{omegaxizeta}]
			
			We will verify that $\widehat{H}^*$ satisfies the second, third, first, and fourth statement of the proposition, in this order. Throughout this proof, we fix $y \in \{ \mathfrak{t}_1, \mathfrak{s}, \mathfrak{t}_2 \}$.

		We first establish the second statement of the proposition. To show \eqref{hxyhxy2}, observe from the $(t, \alpha) = (y - \mathfrak{t}_0, \alpha)$ case of the second statement of \Cref{x0fq} that
		the time slice $\{x:(x,y-\mathfrak t_0)\in \mathfrak L_\alpha\}$ 
		is also a union of $k$ disjoint open intervals $(\widehat{x}_1, \widehat{x}_1') \cup (\widehat{x}_2, \widehat{x}_2') \cup \cdots \cup (\widehat{x}_k, \widehat{x}_k')$. Moreover, for any index $1\leq i\leq k$, we have
		\begin{flalign}
			\label{xialphaxi1} 
			\begin{aligned}
			\widehat x_i  - x_i & = (y - \mathfrak{t}_0) (\alpha - 1) \displaystyle\frac{\mathcal{F}_{y - \mathfrak{t}_0} (x_i; 1)}{\big( \mathcal{F}_{y - \mathfrak{t}_0} (x_i; 1) + 1 \big)^2} + \mathcal{O} \big( |\alpha - 1|^2 \big) \\
			& = (y - \mathfrak{t}_0) (\alpha - 1) \displaystyle\frac{f_y (x_i)}{\big( f_y (x_i) + 1 \big)^2} + \mathcal{O} \big( |\alpha - 1|^2 \big) = (y - \mathfrak{t}_0) (\alpha - 1) \Upsilon_{y} (x_i) + \mathcal{O} \big( |\alpha - 1|^2 \big),
			\end{aligned} 
		\end{flalign}
		
		\noindent where to deduce the second equality we used \eqref{fsx1} and to deduce the third we used the definition \eqref{omegasx} of $\Upsilon_y$. Implementing similar reasoning for the difference of $\widehat x_i'-x_i'$; setting $\alpha = \alpha_0$ in \eqref{xialphaxi1}; applying \eqref{omegayalphat} and \eqref{ll}; and using the fact that $|\alpha_0 - 1| = \mathcal{O} \big( |\xi_1| + |\xi_2| \big)$, we deduce \eqref{hxyhxy2}. 
		
		We next show $\widehat{H}^* (u) = H^* (u)$ for $u \in \mathfrak{D}$ with $u \notin \mathfrak{L} \cup \widehat{\mathfrak{L}}$. We will more generally prove $H_{\alpha} (u) = H^* (u)$ for $(x, y) = u \in \mathfrak{D}$ satisfying $(x, y - \mathfrak{t}_0) \notin \mathfrak{L}_1 \cup \mathfrak{L}_{\alpha}$, for $|\alpha - 1|$ sufficiently small. From this, $\widehat{H}^* (u) = H^* (u)$ would follow by taking $\alpha = \alpha_0$.
		
		First, we verify that $H_{\alpha} (u)$ is constant for $u \in \partial_{\ea} (\mathfrak{D})$ and for $u \in \partial_{\we} (\mathfrak{D})$; we only consider the case $u \in \partial_{\ea} (\mathfrak{D})$ as the proof is entirely analogous if $u \in \partial_{\we} (\mathfrak{D})$. By \Cref{assumptiond}, $\partial_{\ea} (\mathfrak{D})$ is either disjoint with $\overline{\mathfrak{L}}$ or is a subset of $\partial \mathfrak{P}$. In the former case, $\partial_{\ea} (\mathfrak{D})$ is bounded away from $\overline{\mathfrak{L}}$, so the first statement in \Cref{x0fq}  implies for $|\alpha - 1|$ sufficiently small that $\partial_{\ea} (\mathfrak{D})$ is disjoint from $\mathfrak{L}_{\alpha} + (0, \mathfrak{t}_0)$. Since $H_{\alpha}$ is Lipschitz; $\nabla H_1 (u) = \nabla H^* (u)$; the gradient $\nabla H_{\alpha} (u) \big \{ (0, 0), (1, 0), (1, -1) \}$ and is continuous in $\alpha$ for almost every $u \notin \mathfrak{L}_{\alpha} + (0, \mathfrak{t}_0)$; and $H^*$ is constant along $\partial_{\ea} (\mathfrak{D})$, it follows that $H_{\alpha}$ is also constant along $\partial_{\ea} (\mathfrak{D})$. 
		
		If instead $\partial_{\ea} (\mathfrak{D})$ and $\overline{\mathfrak{L}}$ are not disjoint, then $\partial_{\ea} (\mathfrak{D}) \subset \partial \mathfrak{P}$, so $\partial_{\ea} (\mathfrak{D})$ is tangent to $\overline{\mathfrak{L}}$ at some point $(x_0, y_0)$. In particular, the shift $\partial_{\ea} (\mathfrak{D}) - (0, \mathfrak{t}_0)$ is tangent to $\mathfrak{A}_1$ at $(x_0, y_0 - \mathfrak{t}_0)$. Then, it is quickly verified from \eqref{q0alphaz} and \eqref{q0alphaf} that $\partial_{\ea} (\mathfrak{D}) - (0, \mathfrak{t}_0)$ remains tangent to $\mathfrak{A}_{\alpha}$. In particular, $\partial_{\ea} (\mathfrak{D})$ is disjoint from the open set $\mathfrak{L}_{\alpha} + (0, \mathfrak{t}_0)$. Then following the same reasoning as above gives that $H_{\alpha}$ is constant along $\partial_{\ea} (\mathfrak{D})$.

		This verifies that $H_{\alpha}$ and $H^*$ are constant along $\partial_{\ea} (\mathfrak{D})$; by similar reasoning, they are also constant along $\partial_{\we} (\mathfrak{D})$. This, together with the fact that they coincide at the southwest vertex of $\overline{\mathfrak{D}}$, implies that $H_{\alpha} = H^*$ along $\partial_{\we} (\mathfrak{D})$. Since the definition of $\nabla H_{\alpha}$ implies that $\nabla H_{\alpha} (x, y) = \nabla H^* (x, y)$ whenever $(x, y - \mathfrak{t}_0) \notin \mathfrak{L}_1 \cup \mathfrak{L}_{\alpha}$, to show $H_{\alpha} (x, y) = H^* (x, y)$ for $(x, y - \mathfrak{t}_0) \notin \mathfrak{L}_1 \cup \mathfrak{L}_{\alpha}$, it suffices to show that height differences are ``conserved along liquid regions'' of $\partial_{\north} (\mathfrak{D})$ and $\partial_{\so} (\mathfrak{D})$. More specifically, let us fix a left and right endpoint $\big( E_1 (\alpha), y - \mathfrak{t}_0 \big), \big( E_2 (\alpha), y - \mathfrak{t}_0 \big) \in \mathfrak{A}_{\alpha}$, respectively, such that $E_j (\alpha)$ is continuous in $\alpha$ for each $j \in \{ 1 , 2\}$ and such that $\big( x, y - \mathfrak{t}_0 \big) \in \mathfrak{L}_{\alpha}$ for each $E_1 (\alpha) < x < E_2 (\alpha)$. Abbreviating $E_1 = E_1 (1)$ and $E_2 = E_2 (1)$, we must show that 
		\begin{flalign}
			\label{halphae1e2} 
			\begin{aligned}
			& \Big( H_{\alpha} \big( E_2 (\alpha), y \big) - H_{\alpha} \big( E_1 (\alpha), y \big) \Big) - \big( H^* (E_2, y) - H^* (E_1, y) \big) \\
			& \qquad \qquad \qquad = \big( E_2 (\alpha) - E_2 \big) \partial_x H^* (E_2, y) - \big( E_1 (\alpha) - E_1 \big) \partial_x H^* (E_1, y). 
			\end{aligned} 
		\end{flalign}
		
		\noindent Indeed, the equality $H_{\alpha} (u) = H^* (u)$ for $(x, y) = u \in \overline{\mathfrak{D}}$ satisfying $(x, y - \mathfrak{t}_0) \notin \mathfrak{L}_1 \cup \mathfrak{L}_{\alpha}$ would then follow from the fact that $\nabla H_{\alpha} (u) = \nabla H^* (u)$ for all such $u$ and the fact that $H_{\alpha}$ and $H^*$ coincide at the northwest and southwest corners of $\overline{\mathfrak{D}}$. 
		
		Observe that \eqref{halphae1e2} holds at $\alpha = 1$, since $H_1 = H^*$. Thus, it suffices to show that the derivatives with respect to $\alpha$ of both sides of \eqref{halphae1e2} are equal, namely, 
		\begin{flalign}
			\label{halphae1e2derivative} 
			\partial_{\alpha} \Big( H_{\alpha} \big( E_2 (\alpha), y \big) - H_{\alpha} \big( E_1 (\alpha), y \big) \Big) = \partial_x H^* (E_2, y) \partial_{\alpha} E_2 (\alpha) - \partial_x H^* (E_1, y) \partial_{\alpha} E_1 (\alpha).
		\end{flalign}
	
		\noindent To do this, observe from \eqref{halphau} that the left side of \eqref{halphae1e2derivative} is given by
		\begin{flalign*} 
			\partial_{\alpha} \Big( H_{\alpha} \big( E_2 (\alpha), y \big) - H_{\alpha} \big( E_1 (\alpha), y \big) \Big) & = \partial_{\alpha} \displaystyle\int_{E_1 (\alpha)}^{E_2 (\alpha)} \partial_x H_{\alpha} (x, y) \mathrm{d} x \\
			 & = - \pi^{-1} \partial_{\alpha} \bigg( \displaystyle\int_{E_1 (\alpha)}^{E_2 (\alpha)} \Imaginary \log \mathcal{F}_{y - \mathfrak{t}_0} (x; \alpha) \mathrm{d} x \bigg).
		\end{flalign*} 
	
		\noindent Thus,
		\begin{flalign}
			\label{derivativehualpha}
			\begin{aligned} 
			&\partial_{\alpha} \Big( H_{\alpha} \big( E_2 (\alpha), y \big) - H_{\alpha} \big( E_1 (\alpha), y \big) \Big)  = - \pi^{-1} \Imaginary \displaystyle\int_{E_1 (\alpha)}^{E_2 (\alpha)}  \partial_{\alpha}  \log \mathcal{F}_{y - \mathfrak{t}_0} (x; \alpha)  \mathrm{d} x\\
			&+ \pi^{-1} \Imaginary  \Big( \log \mathcal{F}_{y - \mathfrak{t}_0} \big( E_1 (\alpha); \alpha \big) \Big) \partial_{\alpha} E_1 (\alpha)  - \pi^{-1} \Imaginary \Big( \log \mathcal{F}_{y - \mathfrak{t}_0} \big( E_2 (\alpha); \alpha \big) \Big) \partial_{\alpha} E_2 (\alpha). 
			\end{aligned} 
		\end{flalign} 
	
		\noindent Letting $|\alpha - 1|$ tend to $0$ in the first statement of \Cref{ftalphaft1} gives
		\begin{flalign*}
			\partial_{\alpha} \log \mathcal{F}_{y - \mathfrak{t}_0} (x; \alpha) = (y - \mathfrak{t}_0) \partial_x \bigg( \displaystyle\frac{\mathcal{F}_t (x; \alpha)}{\mathcal{F}_t (x; \alpha) + 1} \bigg) + 1.
		\end{flalign*}
	
		\noindent Inserting this into \eqref{derivativehualpha}; using the fact that $\mathcal{F}_{y - \mathfrak{t}_0} \big( E_i (\alpha); \alpha \big) \in \mathbb{R}$; using the equality (following from \eqref{halphau} and the continuity of $\nabla H_{\alpha}$ almost everywhere in $\mathfrak{L}_{\alpha} + (0, \mathfrak{t}_0)$) 
		\begin{flalign*} 
			\pi^{-1} \Imaginary \log \mathcal{F}_{y - \mathfrak{t}_0} \big( E_i (\alpha); \alpha \big) = - \partial_x H_{\alpha} \big( E_i (\alpha), y \big) = - \partial_x H^* (E_i, y);
		\end{flalign*} 
	
		\noindent and using the fact that 
		\begin{flalign*}
			\Imaginary \displaystyle\int_{E_1 (\alpha)}^{E_2 (\alpha)} & \Bigg( (y - \mathfrak{t}_0) \partial_x \bigg( \displaystyle\frac{\mathcal{F}_{y - \mathfrak{t}_0} (x; \alpha)}{\mathcal{F}_{y - \mathfrak{t}_0} (x; \alpha) + 1} \bigg) + 1 \Bigg) \mathrm{d}x \\
			&  = (y - \mathfrak{t}_0) \Imaginary \Bigg( \displaystyle\frac{\mathcal{F}_{y - \mathfrak{t}_0} \big( E_2 (\alpha) \big)}{\mathcal{F}_{y - \mathfrak{t}_0} \big( E_2 (\alpha) \big) + 1} - \displaystyle\frac{\mathcal{F}_{y - \mathfrak{t}_0} \big( E_1 (\alpha) \big)}{\mathcal{F}_{y - \mathfrak{t}_0} \big( E_1 (\alpha) \big) + 1} + E_2 (\alpha) - E_1 (\alpha) \Bigg) = 0,
		\end{flalign*}
	
		\noindent then yields \eqref{halphae1e2derivative}. As mentioned above, this implies \eqref{halphae1e2} and hence that $\widehat{H}^* (x, y) = H^* (x, y)$ for $(x, y) \in \overline{\mathfrak{D}}$ with $(x, y) \notin \mathfrak{L} \cup \widehat{\mathfrak{L}}$. This verifies that $\widehat{H}^*$ satisfies the second statement of the proposition. 
		
		To show that $\widehat{H}^*$ satisfies the third statement, we must verify \eqref{hxdelta} for sufficiently small $\Delta$. Let us assume in what follows that $(x, y)$ and $(\widehat{x}, y)$ are right endpoints of $\mathfrak{L}$ and $\widehat{\mathfrak{L}}$, respectively, as the proof in the case when they are left endpoints is entirely analogous. Then we may assume that $\Delta \leq 0$, for otherwise the fact that $\partial_x H^* (x, y) = \partial_x \widehat{H}^* (\widehat{x}, y)$ (which holds by the almost everywhere continuity of $\nabla H_{\alpha}$ in $\alpha$, together with the facts that $(x, y) \notin \mathfrak{L}$ and $(\widehat{x}, y)	\notin \widehat{\mathfrak{L}}$) implies $H^* (x + \Delta, y) - H^* (x, y) = \Delta \partial_x H^* (x, y) = \Delta \partial_x \widehat{H}^* (\widehat{x}, y) = \widehat{H}^* (\widehat{x} + \Delta, y) - \widehat{H}^* (\widehat{x}, y)$. 
				
		So we must show \eqref{hxdelta} for $\Delta \leq 0$. Since it holds at $\Delta = 0$, and since $|\alpha - 1| = \mathcal{O} \big( |\xi_1| + |\xi_2| \big)$, it suffices to show that 
		\begin{flalign*}
		 \partial_x \widehat{H}^* (\widehat{x} + \Delta, y) = \partial_x H^* (x + \Delta, y) + \mathcal{O} \big( |\Delta|^{1/2} |\alpha_0 - 1| + |\Delta| \big),
		\end{flalign*}  
	
		\noindent for sufficiently small $\Delta$. In view of \eqref{halphau}, this is equivalent to showing that
		\begin{flalign}
			\label{fyt01} 
			\begin{aligned} 
			\bigg| \Imaginary \displaystyle\frac{\mathcal{F}_{y - \mathfrak{t}_0} (\widehat{x} + \Delta; \alpha_0)}{ \mathcal{F}_{y - \mathfrak{t}_0} (\widehat{x}; \alpha_0)}  -  \Imaginary \displaystyle\frac{\mathcal{F}_{y - \mathfrak{t}_0} (x + \Delta; 1)}{\mathcal{F}_{y - \mathfrak{t}_0} (x; 1)} \bigg| = \mathcal{O} \big( |\Delta|^{1/2} |\alpha_0 - 1| + |\Delta| \big),
			\end{aligned} 
		\end{flalign}
		
		\noindent To that end, observe from \eqref{ff0} that 
		\begin{flalign}
			\label{fyt02} 
			\begin{aligned} 
			 \mathcal{F}_{y - \mathfrak{t}_0} (x + \Delta; 1) - \mathcal{F}_{y - \mathfrak{t}_0} (x; 1)  & = \bigg( \displaystyle\frac{2 \big( \mathcal{F}_{y - \mathfrak{t}_0} (x; 1) + 1\big)}{\mathcal{Q}_{0; 1}'' \big( \mathcal{F}_{y - \mathfrak{t}_0} (x; 1) \big)} \bigg)^{1/2} |\Delta|^{1/2} + \mathcal{O} \big( |\Delta|^{3/2} \big); \\
			 \mathcal{F}_{y - \mathfrak{t}_0} (\widehat{x} + \Delta; \alpha_0) - \mathcal{F}_{y - \mathfrak{t}_0} (\widehat{x}; \alpha_0) & =  \bigg( \displaystyle\frac{2 \big( \mathcal{F}_{y - \mathfrak{t}_0} (\widehat{x}; \alpha_0) + 1\big)}{\mathcal{Q}_{0; \alpha_0}'' \big( \mathcal{F}_{y - \mathfrak{t}_0} (\widehat{x}; \alpha_0) \big)} \bigg)^{1/2} |\Delta|^{1/2} + \mathcal{O} \big( |\Delta|^{3/2} \big).
			\end{aligned} 
		\end{flalign}
	
		\noindent  Now observe, for uniformly bounded $u \in \mathbb{C}$, we have
		\begin{flalign}
			\label{fy1} 
			\big| \mathcal{F}_{y - \mathfrak{t}_0} (\widehat{x}; \alpha_0) - \mathcal{F}_{y - \mathfrak{t}_0} (x; 1) \big| = \mathcal{O} \big( |\alpha_0 - 1| \big); \qquad \big| \mathcal{Q}_{0; \alpha_0}'' (u) - \mathcal{Q}_{0; 1}'' (u) \big| = \mathcal{O} \big( |\alpha_0 - 1| \big),
		\end{flalign}
		
		\noindent where the former estimate follows from \eqref{falphaf1} and the latter from \eqref{q0alphaz}. Applying the two approximations in \eqref{fyt02} and \eqref{fy1} (with the fact that $\mathcal{F}_{y - \mathfrak{t}_0} (x, 1)$ and $\mathcal{Q}_{0; 1}'' \big( \mathcal{F}_{y - \mathfrak{t}_0} (x; 1) \big)$ are bounded away from $0$) then yields \eqref{fyt01}. As mentioned above, this implies \eqref{hxyhxy2}, thereby verifying that $\widehat{H}^*$ satisfies the third statement of the proposition. 
		
		We next show that $\widehat{H}^*$ satisfies the first statement of the proposition. To that end, observe from the second and third statements of the proposition (together with \Cref{derivativeh}), it suffices to show for any $(x, y), (x', y) \in \mathfrak{L} \cap \widehat{\mathfrak{L}}$ such that $(z, y) \in \mathfrak{L} \cap \widehat{\mathfrak{L}}$ for each $z \in [x, x']$ that 
		\begin{flalign}
			\label{hxy1} 
		\widehat{H}^* (x', y) - \widehat{H}^* (x, y) = H^* (x', y) - H^* (x, y) + \omega (y) \big( \Omega_y (x') - \Omega_y (x) \big) + \mathcal{O} \big( |\alpha_0 - 1|^{3/2} \big).
		\end{flalign}
	
		\noindent This will follow from a suitable application of \Cref{ftalphaft1}. Indeed, we have
		\begin{flalign*}
			\big( \widehat{H}^* (x', y) - H^* & (x', y) \big) - \big( \widehat{H}^* (x, y) - H^* (x, y) \big) \\
			& = \displaystyle\int_x^{x'} \big( \partial_z \widehat{H}^* (z, y) - \partial_z H^* (z, y) \big) \mathrm{d} z \\ 
			& = \pi^{-1} \displaystyle\int_x^{x'} \Imaginary \big( \log \mathcal{F}_{y - \mathfrak{t}_0} (z; 1) - \log \mathcal{F}_{y - \mathfrak{t}_0} (z; \alpha_0) \big) \mathrm{d} z \\
			& = \pi^{-1} (y - \mathfrak{t}_0) (\alpha_0 - 1) \Imaginary \displaystyle\int_x^{x'} \partial_z \bigg( \displaystyle\frac{\mathcal{F}_{y - \mathfrak{t}_0} (z; 1)}{\mathcal{F}_{y - \mathfrak{t}_0} (z; 1) + 1} \bigg) \mathrm{d} z + \mathcal{O} \big( |\alpha_0 -1|^{3/2} \big) \\
			& = \pi^{-1} \omega (y) \Imaginary \bigg( \displaystyle\frac{f_y ({x'})}{f_y ({x'}) + 1} - \displaystyle\frac{f_y (x)}{f_y (x) + 1} \bigg) + \mathcal{O} \big( |\alpha_0 - 1|^{3/2} \big) \\
			& = \omega (y) \big( \Omega_y ({x'}) - \Omega_y (x) \big) + \mathcal{O} \big( |\alpha_0 - 1|^{3/2} \big),
		\end{flalign*}
	
		\noindent where the second equality follows from \eqref{halphau}; the third from \Cref{ftalphaft1}; the fourth from \eqref{fsx1} and \eqref{omegayalphat}; and the fifth from \eqref{omegasx}. This confirms \eqref{hxy1} and thus the first statement of the proposition.
		
		To establish the fourth statement of the proposition, we must verify that $\mathfrak{D}$ satisfies the five assumptions listed in \Cref{xhh} with respect to $\widehat{H}^*$. We have already verified the first, namely, that $\widehat{H}^*$ is constant along $\partial_{\ea} (\mathfrak{D})$ and $\partial_{\we} (\mathfrak{D})$. That the third holds follows from \Cref{x0fq}, together with the fact that $\mathfrak{D}$ satisfies the third assumption in \Cref{xhh} with respect to $H^*$ (by \Cref{assumptiond}). That the fourth holds follows from the continuity of $\mathcal{F}$ in $\alpha$, together with the fact that $\mathfrak{D}$ satisfies the fourth assumption in \Cref{xhh} with respect to $H^*$. The fifth follows from the fact that $\mathcal{F}_t (x; \alpha_0)$ was defined to satisfy \eqref{qalphaequation}. To verify the second, first suppose that $\partial_{\north} (\mathfrak{D}) \cap \mathfrak{L}$ is connected and non-empty. Then, $\mathfrak{L}$ admits an extension beyond the north boundary of $\mathfrak{D}$, and so $H^*$ admits an extension to some time $\mathfrak{t}' > \mathfrak{t}_2$. That $\widehat{H}^*$ does as well then follows from \Cref{x0fq} and the continuity of $H_{\alpha}$ in $\alpha$. 
		
		If instead $\partial_{\north} (\mathfrak{D})$ is packed with respect to $h$, then we must show that $\widehat{h} (v) = H^* (v)$ for each $v \in \partial_{\north} (\mathfrak{D})$. To that end, observe since the northwest corner $\big( \mathfrak{a} (\mathfrak{t}_2), \mathfrak{t}_2 \big)$ and northeast corner $\big( \mathfrak{b} (\mathfrak{t}_2), \mathfrak{t}_2 \big)$ of $\overline{\mathfrak{D}}$ are outside of (and thus bounded away from) $\overline{\mathfrak{L}}$, it follows from the second statement of the proposition that $H^*$ and $\widehat{H}^*$ coincide on them. Since $\partial_{\north} (\mathfrak{D})$ is packed with respect to $h$, we must have that $\partial_x H^* (v) = 1$ for each $v \in \partial_{\north} (\mathfrak{D})$. Since $\widehat{H}^*$ is $1$-Lipschitz, it follows that $\widehat{H}^* (v) = H^* (v)$ for each $v \in \partial_{\north} (\mathfrak{D})$, implying that $\partial_{\north} (\mathfrak{D})$ is packed with respect to $\widehat{h}$. This shows that $\widehat{H}^*$ satisfies the four properties listed by the proposition.

		It remains to show that $\widehat{H}^* \in \Adm (\mathfrak{D}; \widehat{h})$ is a maximizer of $\mathcal{E}$. This follows from checking the  frozen star ray property in \cite[Definition 8.2]{DMCS}, from which the claim follows by \cite[Remark 8.6]{DMCS} or \cite[Theorem 8.3]{DMCS}. We only briefly outline this verification here, as it is very similar to what was done in \cite{DMCS}.	Item i), ii) and iii) of the frozen star ray property in \cite[Definition 8.2]{DMCS} are quickly verified in our setting, since the domain $\mathfrak{D}$ satisfies \Cref{xhh}. Item iv) in \cite[Definition 8.2]{DMCS} may fail to hold in our setting since, if we parametrize the leftmost and rightmost arctic boundaries of $\widehat H^*$ by $\widehat E_1(t)=\inf \big\{x: (x,t)\in \widehat{\mathfrak L} \big\}$ by $\widehat E_2(t)=\sup \big\{x: (x,t)\in \widehat{\mathfrak L} \big\}$, then the two frozen regions $\big\{(x,t): \mathfrak a(t)\leq x \leq \widehat E_1(t) \big\}$ and $\big\{(x,t):  \widehat E_1(t)\leq x \leq \mathfrak b(t)\big\}$ may not be covered by the family of rays as in \cite[Definition 8.2]{DMCS}.  However, in this case the proof of \cite[Theorem 8.3]{DMCS} still applies. Indeed, in the region not covered by those rays, one can extend the function $\Phi(z)$ appearing in \cite[Proposition 8.1]{DMCS} continuously to be a constant function, and the proofs of \cite[Proposition 8.1, Theorem 8.3, and Remark 8.6]{DMCS} still continue to go through. This idea was already applied in a completely analogous way in the short proof of \cite[Theorem 8.4]{DMCS} (where the same phenomenon arose), so we refer there for a more detailed discussion.
	\end{proof}

\appendix
\section{Complex Burgers Equation}

\label{Equation} 

In this section we collect some quantitative results on the complex Burgers equation and establish \Cref{lqq} and  \Cref{fderivativeq}. We recall the polygonal domain $\mathfrak P$ from Definition \ref{p}; the associated boundary height function $h: \partial \mathfrak P\mapsto \mathbb R$; the maximizer $H^* \in \Adm (\mathfrak{P}; h)$ of $\mathcal{E}$ from \eqref{hmaximum}; the complex slope $f_t (x)$ associated with $H^*$ through \eqref{fh}; and the liquid region $\mathfrak{L} = \mathfrak{L} (\mathfrak{P}; h)$ and arctic boundary $\mathfrak{A} = \mathfrak{A} (\mathfrak{P}; h)$ from \eqref{al}. We further recall that there exists an analytic function $Q_0$ satisfying \eqref{q0f}. If $(x, t) \in \mathfrak{A}$ and $f_t (x)$ is a triple root of\eqref{q0f}, then $(x, t)$ is a \emph{cusp} of $\mathfrak{A}$. If the tangent line through a point $u \in \mathfrak{A}$ to $\mathfrak{A}$ has slope in the set $\{ 0, 1, \infty \}$, then $u$ is a \emph{tangency location} of $\mathfrak{A}$; in this case, the \emph{slope} of $u$ is the slope of this tangent line.

	Given this notation, the following result provides the behavior of the complex slope $f_t (x)$ along generic points of $\mathfrak{A}$. It follows quickly from \cite[Section 1.6]{LSCE} and \cite[Theorem 1.8(c)]{DMCS}, together with the analyticity of $f (x, t)$, and it can be also deduced directly from \eqref{q0f} by a Taylor expansion; so, its proof is omitted. 
	
	\begin{lem}[{\cite{LSCE,DMCS}}]
	
	\label{fdomain}
	
	For any fixed $(x_0, t) \in \mathfrak{A}$ which is not a tangency location or cusp of $\mathfrak{A}$, there exists a constant $C = C (x_0, t, \mathfrak{P}) > 0$ and a neighborhood $\mathfrak{U} \subset \overline{\mathfrak{P}}$ of $(x_0, t)$ such that the following holds. For $(x, t) \in \mathfrak{U} \cap \mathfrak{L}$, we have
			\begin{flalign*} 
				f_t (x) = f_t (x_0) + (C x - C x_0)^{1/2} + \mathcal{O} \big( |x - x_0|^{3/2} \big),
			\end{flalign*} 
		
			\noindent where the implicit constant in the error only depends on $\mathfrak{P}$ and the distance from $(x, t)$ to a tangency location or cusp of $\mathfrak{A}$. Here, the branch of the root is so that it lies in $\mathbb{H}^-$.

	\end{lem}

	\begin{rem}
		
		\label{derivativeh} 
		
		Suppose $u = (x, t) \in \mathfrak{L}$ is bounded away from a cusp or tangency location of $\mathfrak{A}$; let $d_u = \inf \big\{ |x - x_0| : (x_0, t) \in \mathfrak{A} \big\}$. Then, the fact that $f_t (x_0) \in \mathbb{R}$ if $(x_0, t) \in \mathfrak{A}$ (from \Cref{fh}), \Cref{fdomain}, and \eqref{fh} together imply that there exists a small constant $\mathfrak{c} > 0$ such that $\mathfrak{c} d^{1/2} < \big| \partial_x H^* (x, t) - \partial_x H^* (x_0, t) \big| < \mathfrak{c}^{-1} d^{1/2}$.
		
	\end{rem}

	Now let us show \Cref{lqq}.

\begin{proof}[Proof of \Cref{lqq}]
	
	Throughout this proof, let $(\widetilde{x}, \widetilde{t}) \in \mathfrak{A}$ be some point on the arctic boundary $\mathfrak{A} = \mathfrak{A} (\mathfrak{P})$ near $(x, t)$. Abbreviating $f = f_t (x)$ and $\widetilde{f} = f_{\widetilde{t}} (\widetilde{x})$, \eqref{q0f} and the third part of \Cref{pa1} together imply 
	\begin{flalign}
		\label{q0fq0f}
		Q_0 (f) = x (f + 1) - tf; \quad Q_0' ( f )= x - t; \quad Q_0 (\widetilde{f}) = x (\widetilde{f} + 1) - t \widetilde{f}; \quad Q_0' (\widetilde{f}) = x - t.
	\end{flalign}
	
	\noindent From this, we deduce
	\begin{flalign*}
		Q_0(f) = (f + 1) Q_0' (f) + t; \qquad Q_0(\widetilde{f}) = (\widetilde{f} + 1) Q_0' (\widetilde{f}) + \widetilde{t},
	\end{flalign*}
	
	\noindent which together imply
	\begin{flalign}
		\label{qfqf} 
		Q_0(\widetilde{f}) - Q_0(f) = (f + 1) \big( Q_0' (\widetilde{f}) - Q_0' (f) \big) + (\widetilde{f} - f) Q_0' (\widetilde{f}) + \widetilde{t} - t.
	\end{flalign} 
	
	We will first use \eqref{qfqf} to approximately solve for $\widetilde{f}$ in terms of $(f, \widetilde{t}, t)$ and then use \eqref{q0fq0f} to solve for $\widetilde{x}$ in terms of $(x, f, \widetilde{t}, t)$. To that end, applying a Taylor expansion in \eqref{qfqf} (and using the fact that $(\widetilde{x}, \widetilde{t})$ is close to $(x, t)$, which implies that $\widetilde{f}$ is close to $f$ by the analyticity of $f_t (x)$ for $(x, t) \in \mathfrak{L} (\mathfrak{P})$) gives 
	\begin{flalign*}
		& (\widetilde{f} - f) Q_0' (f) + \displaystyle\frac{(\widetilde{f} - f)^2}{2} Q_0'' (f) + \mathcal{O} \big( |\widetilde{f} - f|^3 \big) \\
		& \qquad \qquad \quad  = (f + 1) \Big( (\widetilde{f} - f) Q_0'' (f) + \displaystyle\frac{Q_0''' (f)}{2} (\widetilde{f} - f)^2 + \mathcal{O} \big( |\widetilde{f} - f|^3 \big) \Big) \\
		& \qquad \qquad \qquad \quad + (\widetilde{f} - f) Q_0' (f) + (\widetilde{f} - f)^2 Q_0'' (f) + \mathcal{O} \big( |\widetilde{f} - f|^3 \big) + \widetilde{t} - t.
	\end{flalign*} 
	
	\noindent This implies 
	\begin{flalign*}
		(\widetilde{f} - f) (f + 1) Q_0'' (f) + \displaystyle\frac{(\widetilde{f} - f)^2}{2} \big( Q_0'' (f) + (f + 1) Q_0''' (f) \big) + \mathcal{O} \big( |\widetilde{f} - f|^3 \big)= t - \widetilde{t}.	
	\end{flalign*}
	
	\noindent In particular, it follows that $\widetilde{f} - f = \mathcal{O} \big( |\widetilde{t} - t| \big)$ and, more precisely, that
	\begin{flalign}
		\label{ftft}
		\widetilde{f} - f = \displaystyle\frac{t - \widetilde{t}}{(f + 1) Q_0'' (f)} - \displaystyle\frac{(t - \widetilde{t})^2}{2 (f + 1)^3 Q_0'' (f)^3} \big( Q_0'' (f) + (f + 1) Q_0''' (f) \big) + \mathcal{O} \big( |t - \widetilde{t}|^3 \big).
	\end{flalign}
	
	Since \eqref{q0fq0f} indicates that 
	\begin{flalign*} 
		x = Q_0' (f) + t; \qquad \widetilde{x} = Q_0' (\widetilde{f}) + \widetilde{t},
	\end{flalign*} 
	
	\noindent subtracting; applying \eqref{ftft}; and Taylor expanding again gives
	\begin{flalign*}
		\widetilde{x} - x = Q_0' (\widetilde{f}) - Q_0' (f) + \widetilde{t} - t & = (\widetilde{f} - f) Q_0'' (f) + (\widetilde{f} - f)^2 \displaystyle\frac{Q_0''' (f)}{2} + \widetilde{t} - t + \mathcal{O} \big( |\widetilde{f} - f|^3 \big) \\
		& = \displaystyle\frac{f (\widetilde{t} - t)}{f + 1} -  \displaystyle\frac{(\widetilde{t} - t)^2}{2 (f + 1)^3 Q_0'' (f)} + \mathcal{O} \big( |\widetilde{t} - t|^3 \big),
	\end{flalign*}
	
	\noindent which implies the lemma upon matching with \eqref{xyql}. 
\end{proof}

Next we establish \Cref{fderivativeq}.

\begin{proof}[Proof of \Cref{fderivativeq}]
	
	It suffices to establish \eqref{ftxderivativeq}, as \eqref{xfequation} follows from it. Since the derivation of both statements of \eqref{ftxderivativeq} are similar, we only establish the former. To that end, let $(x', t)$ be some point close to $(x, t)$, and set $\mathcal{F} = \mathcal{F}_t (x')$. Then, \eqref{q0f2} implies
	\begin{flalign*}
		\mathcal{Q}_0 (\mathcal{F}) = x (\mathcal{F} + 1) - t \mathcal{F}; \qquad \mathcal{Q}_0 (\mathcal{F}') = x' (\mathcal{F}' + 1) - t \mathcal{F}',
	\end{flalign*}
	
	\noindent from which it follows that 
	\begin{flalign*}
		\mathcal{Q}_0 (\mathcal{F}') - \mathcal{Q}_0 (\mathcal{F}) = (\mathcal{F}' - \mathcal{F}) (x - t) + (x' - x) (\mathcal{F}' + 1).
	\end{flalign*}
	
	\noindent From a Taylor expansion, we deduce
	\begin{flalign*}
		(\mathcal{F}' - \mathcal{F}) \big( \mathcal{Q}_0' (\mathcal{F}) - x + t) = (x' - x) (\mathcal{F} + 1) + \mathcal{O} \big( |\mathcal{F}' - \mathcal{F}|^2 + |x' - x| |\mathcal{F}' - \mathcal{F}| \big).
	\end{flalign*} 
	
	\noindent Letting $|x' - x|$ tend to $0$, it follows that  
	\begin{flalign*} 
		\partial_x \mathcal{F}_t (x) = \displaystyle\lim_{x' \rightarrow x} \displaystyle\frac{\mathcal{F}' - \mathcal{F}}{x' - x} = \displaystyle\frac{\mathcal{F} + 1}{\mathcal{Q}_0' (\mathcal{F}) - x + t},
	\end{flalign*}  
	
	\noindent which yields the first statement of \eqref{ftxderivativeq}.
\end{proof}

\section{Proof of \Cref{p:frozenr}}

\label{HeightInteger} 

In this section we establish \Cref{p:frozenr}; throughout, we adopt the notation from that proposition. Let us begin by recalling some results from \cite{MCFARS,DMCS}. There exist two functions $m, M\in \Adm (\mathfrak{P}; h)$ (sometimes called \emph{obstacles}) such that
\begin{flalign*} 
	m(z) \le H(z) \le M(z), \qquad \text{for any $H\in \Adm (\mathfrak{P})$ and $z \in \mathfrak P$}.
\end{flalign*} 

\noindent They are explicitly given by (see \cite[Equation (2.23)]{DMCS}) 
\begin{align}\label{e:defmM}
m(v)=\max_{u\in \partial \mathfrak{P}} \bigg(-\max_{p\in \overline {\mathcal T}}\langle p, u-v\rangle+h(u) \bigg),\quad M(v)=\min_{u\in \partial \mathfrak{P}} \bigg( \max_{p\in \overline {\mathcal T}}\langle p, v-u\rangle+h(u) \bigg),
\end{align} 
where we recall the triangle $\mathcal T$ from \eqref{t}. The following result due to \cite{MCFARS} indicates continuity properties for the gradient $\nabla H^*$ of the maximizer $H^*$ of $\mathcal{E}$ on $\mathfrak{P}$, as well as convexity properties for its arctic boundary. In what follows, for any direction $\omega \in \mathbb{R}^2 \setminus \big\{ (0, 0) \big\}$, the graph $G \subset \mathbb{R}^2$ of a (possibly discontinuous) function (whose domain is possibly disconnected or empty) is said to be convex (or concave) in the $\omega$ direction, if the following holds. Let $\rho_{\omega} : \mathbb{R}^2 \rightarrow \mathbb{R}^2$ denote the rotation such that $\rho_{\omega} (\omega) \in \mathbb{R}_{> 0} \cdot (0, 1)$ points vertically upwards. Then, each connected component of $\rho_{\omega} (G)$ is convex (or concave, respectively). 

\begin{prop}[{\cite[Theorem 1.3 and 4.2]{MCFARS}}]\label{p:H*reg}
	
	The following two statements hold. 
	
\begin{enumerate}
\item The gradient $\nabla H^*$ exists and is continuous on the set $\big\{ z \in \mathfrak P: m(z)<H^*(z)<M(z) \big\}$.
\item Fix a real number $c \in \mathbb{R}$; a vertex $p_0 \in \big\{ (0, 0), (1, 0), (1, -1) \big\} \in \overline{\mathcal{T}}$; and a direction $\omega \in \mathbb{R}^2 \setminus \big\{ (0, 0) \big\}$ such that
\begin{align}
\omega\cdot(p-p_0)>0, \quad \text{for all } p \in \overline{\mathcal T}\setminus \{p_0\}.
\end{align}

\noindent Let $S$ denote the interior of the set $\{ z \in \mathfrak P: H^*(z)=c+p_0\cdot z\}$; assume that $S\neq \emptyset$. Then $\partial S\cap \mathfrak P$ consists of the union of a convex graph (by above) and a concave graph (by below) in the $\omega$ direction.

\end{enumerate}
\end{prop}

The following lemma, which is a quick consequence of \Cref{p}, provides integrality properties for the boundary height function $h : \partial \mathfrak{P} \rightarrow \mathbb{R}$. 

\begin{lem} 

\label{vh2} 	

The following two statements hold. 

\begin{enumerate}
	
	\item \label{heightxyinteger} For any $(x, t) \in (n^{-1} \cdot \mathbb{Z})^2$, we have $h(x, t) \in n^{-1} \cdot \mathbb{Z}$. 
	\item \label{height1} Along any edge of $\partial \mathfrak{P}$ of slope $0$, we have $\partial_x h = 1$. 
\item \label{height00} Along any edge of $\partial \mathfrak{P}$ of slope $\infty$ or $1$, $h$ is constant and takes value $h \in n^{-1} \cdot \mathbb{Z}$.
\end{enumerate}

\end{lem} 

\begin{proof} 
	
	Since (by \Cref{p}) $\mathfrak{P}$ is a polygonal domain and $(0, 0) \in \partial \mathfrak{P}$, each corner vertex of the polygon $\mathfrak{P}$ lies on $(n^{-1} \cdot \mathbb{Z})^2$. The fact that $h$ is constant along any edge of $\partial \mathfrak{P}$ of slope $\infty$ or $1$, and that $\partial_x h = 1$ along any edge of $\partial \mathfrak{P}$ of slope $0$, follows from the fact that $\mathfrak{P}$ is polygonal (together with the conventions relating tilings to height functions described in \Cref{FunctionWalks}); this confirms \Cref{height1} and the first part of \Cref{height00} of the lemma. Using this, with the facts that $h(0,0) = 0$ and that each corner vertex of $\mathfrak{P}$ lies on $(n^{-1} \cdot \mathbb{Z})^2$, it follows that $h (x, t) \in n^{-1} \cdot \mathbb{Z}$ whenever $(x, t) \in (n^{-1} \cdot \mathbb{Z})^2$. This confirms \Cref{heightxyinteger} of the lemma. Then the second part of \Cref{height00} of the lemma follows from the fact that $h$ is constant along each edge of $\partial \mathfrak{P}$ with slope $\infty$ or $1$, and again the fact that each corner vertex of $\mathfrak{P}$ lies on $(n^{-1} \cdot \mathbb{Z})^2$. 
\end{proof}

The below lemma states, at lattice points $v \in (n^{-1} \cdot \mathbb{Z})^2$, that $n \cdot m(v)$ and $n \cdot M(v)$ are integers.

\begin{lem}\label{l:int}
For any point $v\in \overline{\mathfrak P}\cap (n^{-1} \cdot \bZ)^2$, we have $n \cdot m(v) \in \mathbb{Z}$ and $n \cdot M(v)\in \bZ$.
\end{lem}

\begin{proof}
	
The proof follows from analyzing the two formulas \eqref{e:defmM}. The proofs are the same for $m$ and $M$, so we will only give the proof that $n \cdot M(v)\in \bZ$. By \eqref{e:defmM},
\begin{align}
	\label{e:Mv2}
M(v)=\min_{u\in \partial \mathfrak{P}} \bigg( \max_{p\in \overline {\mathcal T}}\langle p, v-u\rangle+h(u) \bigg).
\end{align} 

\noindent Let the minimum over $u \in \partial \mathfrak{P}$ in \eqref{e:Mv2} be attained at $u^* \in \partial \mathfrak{P}$ (if there are multiple such $u^*$, then we select one arbitrarily). Further observe that the maximum over ${p\in \overline {\mathcal T}}$ in \eqref{e:Mv2} is attained at a vertex $p^* \in \big\{ (0, 0), (1, 0), (1, -1) \big\}$ of the triangle $\mathcal T$. Without loss of generality (for the proofs in the remaining cases are entirely analogous), we assume that 
\begin{flalign*} 
	p^*=(0,0), \qquad \text{so that} \qquad M(v) = h(u^*). 
\end{flalign*} 

\noindent Then, 
\begin{flalign*}
	\displaystyle\max \Big\{ \big\langle (1,0), v-u^* \big\rangle, \big\langle (1,-1), v-u^* \big\rangle \Big\} \le  \langle p^*, v-u^* \rangle = 0,
\end{flalign*} 

\noindent so that 
\begin{flalign} 
	\label{vuu} 
	v-u^* \in \big\{ (x, y) \in \mathbb{R}_{\le 0} \times \mathbb{R} : x - y \le 0 \big\}.
\end{flalign}

Now let us consider several cases, depending on the side $\ell = \ell (u^*)$ of $\partial \mathfrak{P}$ that $u^*$ lies on. If the slope of $\ell$ is either $1$ or $\infty$, then \Cref{height00} of \Cref{vh2} yields $M(v)=h(u^*)\in n^{-1} \cdot \bZ$. Otherwise, $\ell$ has slope $0$ (that is, $\ell$ is a horizontal edge of $\partial \mathfrak{P}$), and $u^*$ is not a corner vertex of $\mathfrak{P}$. Observe in this case that $u^* \cdot (0, 1) \in n^{-1} \cdot \mathbb{Z}$, that is, the second coordinate of $u^*$ is in $n^{-1} \cdot \mathbb{Z}$ (as the second coordinate of any horizontal edge of $\partial \mathfrak{P}$ is in $n^{-1} \cdot \mathbb{Z}$). 

We claim that either $v - u^* \in \mathbb{R}_{\ge 0} \cdot (0, 1)$ or $v-u^* \in \mathbb{R}_{\ge 0} \cdot (-1, 1)$. Otherwise, by \eqref{vuu}, there would exist some $u' \in \big\{ u + \mathbb{R}_{< 0} \cdot (1,0) \big\} \cap \ell$ (that is, on $\ell$ and to the left of $u$), such that $v - u' \in \big\{ (x, y) \in \mathbb{R}_{\le 0} \times \mathbb{R} : x-y \le 0 \big\}$. Since $\big\langle p, (-1, 0) \big\rangle \le 0$ for each $p \in \overline{\mathcal{T}}$, it follows that $\max_{p \in \overline{\mathcal{T}}} \langle p, v-u' \rangle \le \max_{p \in \overline{\mathcal{T}}} \langle p, v-u^* \rangle = 0$. Together with the fact that $h(u') = h(u^*) + (1, 0) \cdot (u' - u) < h(u^*)$ (by \Cref{height1} of \Cref{vh2}), this yields 
\begin{flalign*}
	\max_{p \in \overline{\mathcal{T}}} \langle p, v-u' \rangle + h(u') < \max_{p \in \overline{\mathcal{T}}} \langle p, v-u^* \rangle + h(u^*),
\end{flalign*} 

\noindent which is a contradiction. Hence $v - u^* \in \mathbb{R}_{\ge 0} \cdot (0, 1)$ or $v - u^* \in \mathbb{R}_{\ge 0} \cdot (-1, 1)$. 

In the first situation $v-u^* \in \mathbb{R}_{\ge 0} \cdot (0, 1)$, the first coordinates of $v$ and $u^*$ coincide. Thus, the first coordinate of $v$ is in $n^{-1} \cdot \mathbb{Z}$; the same therefore holds for $u^*$. Together with the fact that $u^* \cdot (0, 1) \in n^{-1} \cdot \mathbb{Z}$, it follows that $u^* \in (n^{-1} \cdot \mathbb{Z})^2$, from which \Cref{heightxyinteger} of \Cref{vh2} yields $M(v) = h(u^*) \in n^{-1} \cdot \bZ$. In the second situation $v - u^* \in \mathbb{R}_{\ge 0} \cdot (0, 1)$, we have $u^* \cdot (1, 1) = v \cdot (1, 1) \in n^{-1} \cdot \mathbb{Z}$. Again together with the fact that $u^* \cdot (0, 1) \in n^{-1} \cdot \mathbb{Z}$, this yields $u^* \in (n^{-1} \cdot \mathbb{Z})^2$, and so \Cref{heightxyinteger} of \Cref{vh2} gives $M(v) = h(u^*) \in n^{-1} \cdot \mathbb{Z}$. This finishes the proof of \Cref{l:int}.
\end{proof}

Now we can establish \Cref{p:frozenr}.

\begin{proof}[Proof of \Cref{p:frozenr}]
	
	We begin by establishing the first statement of the proposition; observe that it suffices to address the case when $(x,t) \in \mathfrak P\setminus \overline{\mathfrak L}$, for then the extension to the case $(x,t) \in \partial \mathfrak{L}$ would follow from the continuity of $H^*$. To that end, recall from \Cref{pla} that $\nabla H^* (x, t) \in \big\{ (0,0), (1,0), (1,-1) \big\}$, as $(x, t) \in \mathfrak{P} \setminus \overline{\mathfrak{L}}$. Since $\nabla H^*$ is continuous at $(x, t)$, there exists a small neighborhood $\mathcal N(x, t) \subset \mathfrak{P}$ of $(x,t)$ such that 
 $\nabla H^*$ is constant on $\mathcal N(x, t)$, on which it takes one of the values $\big\{ (0,0), (1,0), (1,-1) \big\}$. We assume in what follows that $\nabla H^*(u)=(0,0)$ for each $u\in \mathcal N(x, t)$, for the other cases can be proven in an entirely analogous way.

Let $S$ denote the interior of the connected component of $\big\{u: H^*(u)=H^*(x, t) \big\}$ containing $(x, t)$; then $S$ is nonempty, and we may assume (after taking a subset of $\mathcal{N}(x,t)$ if necessary) that $\mathcal N(x, t)\subseteq S$. Take any direction $\omega \in \mathbb{R}^2$ such that
\begin{align}
\omega \cdot p>0, \quad \text{for all } p \in \overline{\mathcal T}\setminus \big\{ (0,0) \big\}.
\end{align}
This is equivalent to the argument of $\omega$ being in $(-\pi/2, \pi/4)$; we may assume in what follows that the argument of $\omega$ is in a small neighborhood of $-\pi/8$. Then $H^*$ is non-decreasing in the $\omega$ direction, and $S$ is bounded (in the $\omega$ direction) between two Lipschitz curves $\mathcal C_{\rm top}$ and $\mathcal C_{\rm btm}$. We can assume that the left and right boundary of $S$ (in $\omega$ direction) are given by points $l$ and $r$ respectively (as if they are given by segments in $\omega$ direction, we can slightly perturb $\omega$ to avoid such non-generic situation). In this way, the two curves $\mathcal C_{\rm top}$ and $\mathcal C_{\rm btm}$ are from $l$ to $r$.

First observe that $\mathcal C_{\rm top}, \mathcal C_{\rm btm} \not\subset \mathfrak P$. Indeed, the second statement of \Cref{p:H*reg} would otherwise imply that $\mathcal C_{\rm top}$ is convex and $\mathcal C_{\rm btm}$ is concave in the $\omega$ direction. Hence, $\mathcal{C}_{\mathrm{top}}$ lies below the line connecting $l$ to $r$ in the $\omega$ direction, while $\mathcal{C}_{\mathrm{top}}$ lies above this line, which is impossible since $\mathcal{C}_{\mathrm{top}}$ bounds $S$ from above and $\mathcal{C}_{\mathrm{\btm}}$ bounds $S$ from below in the $\omega$ direction.  Thus, we must instead have $\mathcal C_{\rm top} \cup \mathcal C_{\rm btm}\not\subset \mathfrak P$, meaning that there exists some point $q \in \mathcal (\mathcal C_{\rm top}\cup \mathcal C_{\rm btm})\cap\partial \mathfrak P\neq \emptyset$, with $q \notin \mathcal{C}_{\mathrm{top}} \cap \mathcal{C}_{\mathrm{btm}}$. If $q$ belongs to a vertical boundary edge or a boundary edge with slope $1$, then $H^*(x, t)=h(q)\in n^{-1} \cdot \bZ$ by \Cref{height00} of \Cref{vh2}. 

Otherwise, $q$ belongs to a horizontal edge $\ell$ of $\partial \mathfrak P$, but it is not a corner vertex of $\mathfrak P$. Without loss of generality, we assume (by rotating $\mathfrak{P}$ if necessary) that $q\in \mathcal C_{\rm btm}$, meaning that $q \notin \mathcal{C}_{\mathrm{top}}$. We can take a small $\varepsilon>0$ such that the short interval $\big[q-(\varepsilon,0), q+(\varepsilon,0) \big]$ also belongs to the same edge $\ell$ of $\partial \mathfrak P$. Since $q \in \mathcal{C}_{\mathrm{btm}} \setminus (\mathcal{C}_{\mathrm{btm}} \cap \mathcal{C}_{\mathrm{top}})$, the point $q$ is in the interior of $\mathcal{C}_{\mathrm{btm}}$, and so we can take $\varepsilon$ small enough so that there exists some $r\geq 0$ such that $q+(\varepsilon,0)+r\omega\in S$; see \Cref{f:regionS} for a depiction. This yields a contradiction, because
 \begin{align*}
H^* \big(q+(\varepsilon,0)+r\omega \big) \geq H^* \big(q+(\varepsilon,0) \big)=h^* \big(q+(\varepsilon,0) \big) >h^*(q)=H^* \big(q+(\varepsilon,0)+r\omega \big),
\end{align*}
where in the first statement we used the fact that, in the $\omega$ direction, $H^*$ is non-decreasing; in the second, we used that $\big[ q-(\varepsilon,0), q+(\varepsilon,0) \big]$  belongs to a horizontal boundary edge of $\mathfrak P$; in the third, we used \Cref{height1} of \Cref{vh2}; and in the fourth, we used that $q,  q+(\varepsilon,0)+r\omega\in \overline{S}$. This confirms that $H^* (x, t) \in n^{-1} \cdot \mathbb{Z}$ and finishes the proof of the first statement in \Cref{p:frozenr}.

\begin{figure}[t]
\includegraphics[width=1.2\textwidth, trim=0cm 4cm 0cm 1cm]{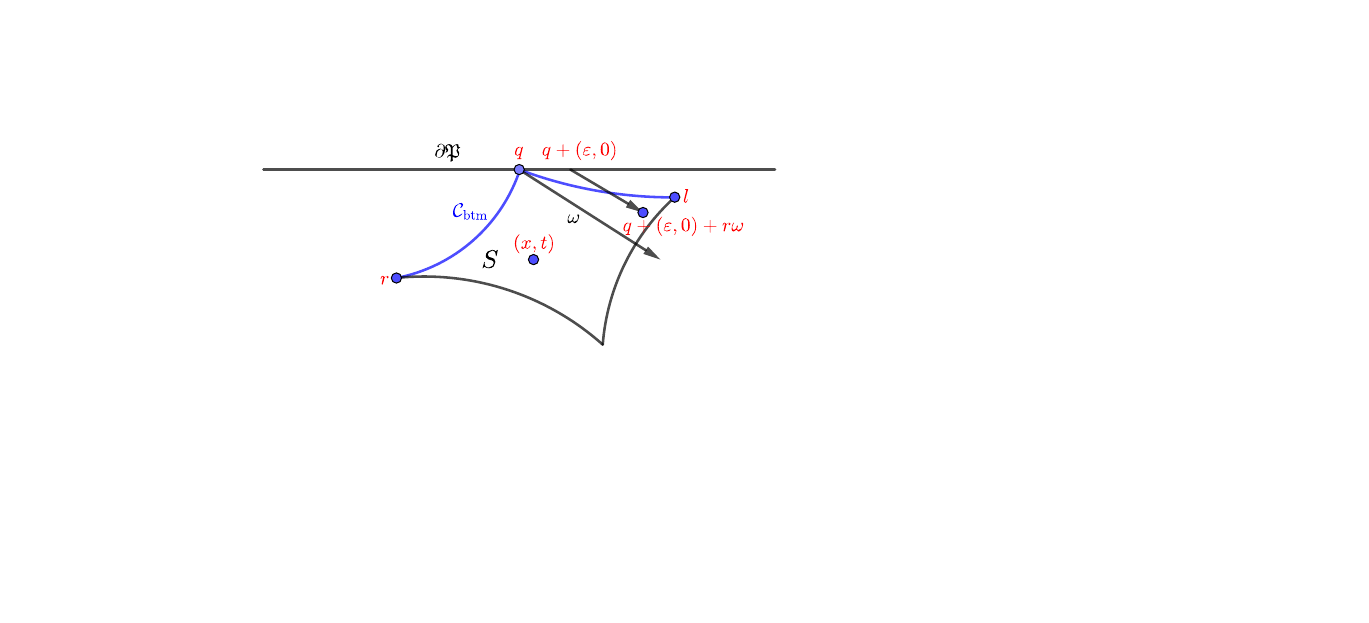}
\caption{If $q\in \mathcal C_{\rm btm}\cap \partial \fP$ belongs to a horizontal boundary edge of $\fP$, then we can take $\varepsilon$ small enough so that there exists some $r\geq 0$ such that $q+(\varepsilon,0)+r\omega\in S$.}
\label{f:regionS}
\end{figure}

To establish the second statement, we set $\Lambda= \big\{u\in \overline{\mathfrak P}: H^*(u)=m(u) \text{ or } H^*(u)=M(u) \big\}$. If $(x,t) \in \Lambda$ then it follows from \Cref{l:int} that $n \cdot H^*(x,t)\in \bZ$. Otherwise $(x,t)\in \mathfrak P\setminus(\overline{\mathfrak L}\cup \Lambda)$, and so the first statement in \Cref{p:H*reg} implies that $\nabla H^*$ is continuous at $(x,t)$.  Setting $\nabla H^* (x,t) = (s, r) \in \big\{ (0, 0), (1, 0), (1, -1) \big\}$, the first statement of \Cref{p:frozenr} implies that $n (H^*(x,t)-sx-rt)\in \bZ$. By our assumption $(x,t) \in (n^{-1} \cdot \mathbb{Z})^2$, so $n(sx+rt)\in \bZ$ and we conclude that $n \cdot H^*(x,t)\in \bZ$. This finishes the second statement of \Cref{p:frozenr}.

\end{proof}

\newcommand{\etalchar}[1]{$^{#1}$}


\begin{thebibliography}{BKMM07}
	
	\bibitem[ADPZ20]{DMCS}
	K.~Astala, E.~Duse, I.~Prause, and X.~Zhong.
	\newblock Dimer models and conformal structures.
	\newblock Preprint, ar{X}iv:2004.02599, 2020.
	
	\bibitem[AEK19]{ajanki2019stability}
	O.~H. Ajanki, L.~Erd{\H{o}}s, and T.~Kr{\"u}ger.
	\newblock Stability of the matrix {D}yson equation and random matrices with
	correlations.
	\newblock {\em Probab. Theory Related Fields}, 173(1):293--373, 2019.
	
	\bibitem[AFvM10]{PWC}
	M.~Adler, P.~L. Ferrari, and P.~van Moerbeke.
	\newblock Airy processes with wanderers and new universality classes.
	\newblock {\em Ann. Probab.}, 38(2):714--769, 2010.
	
	\bibitem[AG22]{ERT}
	A. Aggarwal and V. Gorin.
	\newblock Gaussian unitary ensemble in random lozenge tilings.
	\newblock {\em Probability Theory and Related Fields}, 184(3-4):1139--1166,
	2022.
	
	\bibitem[Agg]{ULS}
	A.~Aggarwal.
	\newblock Universality for lozenge tilings local statistics.
	\newblock To appear on \emph{Ann. Math.} preprint, ar{X}iv:1907.09991.
	
	\bibitem[AJvM18a]{THCAF}
	M.~Adler, K.~Johansson, and P.~van Moerbeke.
	\newblock Lozenge tilings of hexagons with cuts and asymptotic fluctuations: a
	new universality class.
	\newblock {\em Math. Phys. Anal. Geom.}, 21(1):Paper No. 9, 53, 2018.
	
	\bibitem[AJvM18b]{TNSDP}
	M.~Adler, K.~Johansson, and P.~van Moerbeke.
	\newblock Tilings of non-convex polygons, skew-{Y}oung tableaux and
	determinantal processes.
	\newblock {\em Comm. Math. Phys.}, 364(1):287--342, 2018.
	
	\bibitem[AvM18]{PDTP}
	M.~Adler and P.~van Moerbeke.
	\newblock Probability distributions related to tilings of non-convex polygons.
	\newblock {\em J. Math. Phys.}, 59(9):091418, 21, 2018.
	
	\bibitem[BEK{\etalchar{+}}14]{alex2014isotropic}
	A.~Bloemendal, L.~Erd{\H{o}}s, A.~Knowles, H.-T. Yau, and J.~Yin.
	\newblock Isotropic local laws for sample covariance and generalized {W}igner
	matrices.
	\newblock {\em Electron. J. Probab.}, 19, 2014.
	
	\bibitem[BES17]{bao2017convergence}
	Z.~Bao, L.~Erd{\H{o}}s, and K.~Schnelli.
	\newblock Convergence rate for spectral distribution of addition of random
	matrices.
	\newblock {\em Adv. Math.}, 319:251--291, 2017.
	
	\bibitem[BES20]{bao2020spectral}
	Z.~Bao, L.~Erd{\H{o}}s, and K.~Schnelli.
	\newblock Spectral rigidity for addition of random matrices at the regular
	edge.
	\newblock {\em J. Funct. Anal.}, 279(7):108639, 2020.
	
	\bibitem[BGR10]{borodin2010q}
	Alexei Borodin, Vadim Gorin, and Eric~M Rains.
	\newblock $q$-distributions on boxed plane partitions.
	\newblock {\em Sel. Math.}, 16(4):731--789, 2010.
	
	\bibitem[BKMM07]{DP}
	J.~Baik, T.~Kriecherbauer, K.~T.-R. McLaughlin, and P.~D. Miller.
	\newblock {\em Discrete orthogonal polynomials}, volume 164 of {\em Annals of
		Mathematics Studies}.
	\newblock Princeton University Press, Princeton, NJ, 2007.
	
	\bibitem[CEP96]{LSRT}
	H.~Cohn, N.~Elkies, and J.~Propp.
	\newblock Local statistics for random domino tilings of the {A}ztec diamond.
	\newblock {\em Duke Math. J.}, 85(1):117--166, 1996.
	
	\bibitem[CH14]{PLE}
	I.~Corwin and A.~Hammond.
	\newblock Brownian {G}ibbs property for {A}iry line ensembles.
	\newblock {\em Invent. Math.}, 195(2):441--508, 2014.
	
	\bibitem[CK01]{ECM}
	R.~Cerf and R.~Kenyon.
	\newblock The low-temperature expansion of the {W}ulff crystal in the 3{D}
	{I}sing model.
	\newblock {\em Comm. Math. Phys.}, 222(1):147--179, 2001.
	
	\bibitem[CKP01]{VPT}
	H.~Cohn, R.~Kenyon, and J.~Propp.
	\newblock A variational principle for domino tilings.
	\newblock {\em J. Amer. Math. Soc.}, 14(2):297--346, 2001.
	
	\bibitem[CLP98]{TSP}
	H.~Cohn, M.~Larsen, and J.~Propp.
	\newblock The shape of a typical boxed plane partition.
	\newblock {\em New York J. Math.}, 4:137--165, 1998.
	
	\bibitem[DJM16]{TCP}
	E.~Duse, K.~Johansson, and A.~Metcalfe.
	\newblock The cusp-{A}iry process.
	\newblock {\em Electron. J. Probab.}, 21:Paper No. 57, 50, 2016.
	
	\bibitem[DM18]{UEFDIPS}
	E.~Duse and A.~Metcalfe.
	\newblock Universal edge fluctuations of discrete interlaced particle systems.
	\newblock {\em Ann. Math. Blaise Pascal}, 25(1):75--197, 2018.
	
	\bibitem[DNV19]{UCLE}
	D.~Dauvergne, M.~Nica, and B.~Vir\'{a}g.
	\newblock Uniform convergence to the {A}iry line ensemble.
	\newblock Preprint, ar{X}iv:1907.10160, 2019.
	
	\bibitem[DSS10]{MCFARS}
	D.~De~Silva and O.~Savin.
	\newblock Minimizers of convex functionals arising in random surfaces.
	\newblock {\em Duke Math. J.}, 151(3):487--532, 2010.
	
	\bibitem[EKYY13]{erdHos2013local}
	L.~Erd{\H{o}}s, A.~Knowles, H.-T. Yau, and J.~Yin.
	\newblock The local semicircle law for a general class of random matrices.
	\newblock {\em Electron. J. Probab.}, 18, 2013.
	
	\bibitem[EYY12]{erdHos2012rigidity}
	L.~Erd{\H{o}}s, H.-T. Yau, and J.~Yin.
	\newblock Rigidity of eigenvalues of generalized {W}igner matrices.
	\newblock {\em Adv. Math.}, 229(3):1435--1515, 2012.
	
	\bibitem[FS03]{SFFC}
	P.~L. Ferrari and H.~Spohn.
	\newblock Step fluctuations for a faceted crystal.
	\newblock {\em J. Statist. Phys.}, 113(1-2):1--46, 2003.
	
	\bibitem[Gor08]{gorin2008nonintersecting}
	V.~Gorin.
	\newblock Nonintersecting paths and the hahn orthogonal polynomial ensemble.
	\newblock {\em Funct. Anal. Appl.}, 42(3):180, 2008.
	
	\bibitem[Gor17]{gorin2017bulk}
	V.~Gorin.
	\newblock Bulk universality for random lozenge tilings near straight boundaries
	and for tensor products.
	\newblock {\em Comm. Math. Phys.}, 354(1):317--344, 2017.
	
	\bibitem[Gor21]{RT}
	V.~Gorin.
	\newblock {\em Lectures on random lozenge tilings}, volume 193 of {\em
		Cambridge Studies in Advanced Mathematics}.
	\newblock Cambridge University Press, 2021.
	
	\bibitem[HKR18]{he2018isotropic}
	Y.~He, A.~Knowles, and R.~Rosenthal.
	\newblock Isotropic self-consistent equations for mean-field random matrices.
	\newblock {\em Probab. Theory Related Fields}, 171(1):203--249, 2018.
	
	\bibitem[Hua20]{HFRTNRW}
	J.~Huang.
	\newblock Height fluctuations for random lozenge tilings through
	nonintersecting walks.
	\newblock Preprint, ar{X}iv:2011.01751, 2020.
	
	\bibitem[Hua21]{H1}
	J.~Huang.
	\newblock Edge statistics for lozenge tilings of polygons, {I}: Concentration
	of height function on strip domains.
	\newblock Preprint, arXiv:2108.12872, 2021.
	
	\bibitem[Joh00]{SFRM}
	K.~Johansson.
	\newblock Shape fluctuations and random matrices.
	\newblock {\em Comm. Math. Phys.}, 209(2):437--476, 2000.
	
	\bibitem[Joh02]{NPRTRM}
	K.~Johansson.
	\newblock Non-intersecting paths, random tilings and random matrices.
	\newblock {\em Probab. Theory Related Fields}, 123(2):225--280, 2002.
	
	\bibitem[Joh05]{ACP}
	K.~Johansson.
	\newblock The arctic circle boundary and the {A}iry process.
	\newblock {\em Ann. Probab.}, 33(1):1--30, 2005.
	
	\bibitem[Joh18]{EFLS}
	K.~Johansson.
	\newblock Edge fluctuations of limit shapes.
	\newblock In {\em Current developments in mathematics 2016}, pages 47--110.
	Int. Press, Somerville, MA, 2018.
	
	\bibitem[KO07]{LSCE}
	R.~Kenyon and A.~Okounkov.
	\newblock Limit shapes and the complex {B}urgers equation.
	\newblock {\em Acta Math.}, 199(2):263--302, 2007.
	
	\bibitem[KY13]{knowles2013isotropic}
	A.~Knowles and J.~Yin.
	\newblock The isotropic semicircle law and deformation of {W}igner matrices.
	\newblock {\em Comm. Pure Appl. Math.}, 66(11):1663--1749, 2013.
	
	\bibitem[Las19]{LLTS}
	B. Laslier.
	\newblock Local limits of lozenge tilings are stable under bounded boundary
	height perturbations.
	\newblock {\em Probab. Theory Related Fields}, 173(3-4):1243--1264, 2019.
	
	\bibitem[LPW09]{CMT}
	D.~A. Levin, Y. Peres, and E.~L. Wilmer.
	\newblock {\em Markov chains and mixing times}.
	\newblock American Mathematical Society, Providence, RI, 2009.
	\newblock With a chapter by James G. Propp and David B. Wilson.
	
	\bibitem[OR03]{CFPALG}
	A.~Okounkov and N.~Reshetikhin.
	\newblock Correlation function of {S}chur process with application to local
	geometry of a random 3-dimensional {Y}oung diagram.
	\newblock {\em J. Amer. Math. Soc.}, 16(3):581--603, 2003.
	
	\bibitem[OR06]{RM}
	A.~Okounkov and N.~Reshetikhin.
	\newblock The birth of a random matrix.
	\newblock {\em Mosc. Math. J.}, 6(3):553--566, 588, 2006.
	
	\bibitem[OR07]{RSPPP}
	A.~Okounkov and N.~Reshetikhin.
	\newblock Random skew plane partitions and the {P}earcey process.
	\newblock {\em Comm. Math. Phys.}, 269(3):571--609, 2007.
	
	\bibitem[Pet14]{ARS}
	L.~Petrov.
	\newblock Asymptotics of random lozenge tilings via {G}elfand-{T}setlin
	schemes.
	\newblock {\em Probab. Theory Related Fields}, 160(3-4):429--487, 2014.
	
	\bibitem[Pet15]{AURTPF}
	L.~Petrov.
	\newblock Asymptotics of uniformly random lozenge tilings of polygons.
	{G}aussian free field.
	\newblock {\em Ann. Probab.}, 43(1):1--43, 2015.
	
	\bibitem[PS02]{SIDP}
	M.~Pr\"{a}hofer and H.~Spohn.
	\newblock Scale invariance of the {PNG} droplet and the {A}iry process.
	\newblock volume 108, pages 1071--1106. 2002.
	\newblock Dedicated to David Ruelle and Yasha Sinai on the occasion of their
	65th birthdays.
	
	\bibitem[PW13]{EUM}
	Y.~Peres and P.~Winkler.
	\newblock Can extra updates delay mixing?
	\newblock {\em Comm. Math. Phys.}, 323(3):1007--1016, 2013.
	
	\bibitem[RT00]{ADCC}
	D. Randall and P. Tetali.
	\newblock Analyzing {G}lauber dynamics by comparison of {M}arkov chains.
	\newblock volume~41, pages 1598--1615. 2000.
	\newblock Probabilistic techniques in equilibrium and nonequilibrium
	statistical physics.
	
\end{thebibliography}
\end{document}